\documentclass[12pt]{article}
\usepackage{amssymb,amsmath}
\usepackage{cancel}
\usepackage{palatino}
\usepackage{longtable}
\usepackage{dsfont}
\usepackage[hmargin=3cm,vmargin=3cm]{geometry}
\usepackage{color}
\numberwithin{equation}{section}
\newtheorem{theorem}{Theorem}[section]   
     
\newtheorem{definition}[theorem]{Definition}
\newtheorem{defi}[theorem]{Definition}

\newtheorem{prop}[theorem]{Proposition}
\newtheorem{lemma}[theorem]{Lemma}
\newtheorem{rmk}[theorem]{Remark}

\newtheorem{example-notation}[theorem]{Example-Notation}
\newtheorem{remark}[theorem]{Remark}

\def\Tr{\mathrm{Tr}}

\def\d{\partial}
\def\n{\noindent}
\def\f{\frac}
\def\dna{d_{\nabla}}

\def\QEDclosed{\mbox{\rule[0pt]{1.3ex}{1.3ex}}} 

\def\QED{\QEDclosed} 
\def\endproof{\hspace*{\fill}~\QED\par\endtrivlist\unskip}
\newcommand{\eqa}{\begin{eqnarray}}
\newcommand{\eeqa}{\end{eqnarray}}
\newcommand{\beq}{\begin{equation}}
\newcommand{\eeq}{\end{equation}}

\allowdisplaybreaks

\begin{document}
\title{Integrable systems, Nijenhuis geometry and Lauricella bi-flat structures}
\author{Paolo Lorenzoni$^1$ and Sara Perletti$^2$\\
{\small $^1$ Dipartimento di Matematica e Applicazioni}\\
{\small Universit\`a di Milano-Bicocca,}
{\small Via Roberto Cozzi 53, I-20125 Milano, Italy}\\
{\small paolo.lorenzoni@unimib.it}\\
{\small $^2$ Dipartimento di Matematica e Applicazioni}\\
{\small Universit\`a di Milano-Bicocca,}
{\small Via Roberto Cozzi 53, I-20125 Milano, Italy}\\
{\small s.perletti1@unimib.it}
}

\date{}
\maketitle
\vspace{-0.2in}

\begin{abstract} 
Combining the construction of integrable systems of hydrodynamic type starting from the Fr\"{o}licher-Nijenhuis bicomplex $(d,d_L)$ associated
 with a $(1,1)$-tensor field $L$ with vanishing Nijenhuis torsion with the construction of flat structures starting from integrable systems of hydrodynamic type  we define multi-parameter families of bi-flat structures $(\nabla,e,\circ,\nabla^*,*,E)$ associated with Fr\"{o}licher-Nijenhuis bicomplexes.   We call these structures Lauricella bi-flat structures since in the $n$-dimensional semisimple case $(n-1)$ flat coordinates of $\nabla$ are Lauricella functions. 
\end{abstract}    
 
 \tableofcontents
 
 \section{Introduction}
Given a tensor field $L$ of  type $(1,1)$ on a manifold $M$ with vanishing Nijenhuis torsion 
it is possible to define a bi-differential complex $(d,d_L,\Omega(M))$ called the  Fr\"{o}licher-Nijenhuis bicomplex on the Grasmann algebra $\Omega(M)$ of differential forms on $M$ (see \cite{FN}). The differential $d$ is the usual de Rham differential while the differential $d_L$ is defined as
 $$(d_L \omega)(X_0, \dots, X_k)=\sum_{i=0}^k (-1)^i (LX_i)(\omega(X_0, \dots, \hat{X}_i, \dots, X_k))+$$
$$+\sum_{0\leq i<j\leq k}(-1)^{i+j}\omega([X_i, X_j]_L, X_0, \dots, \hat{X}_i, \dots, \hat{X}_j, \dots X_k),$$
where  $\omega\in\Omega^k(M)$ and  
$$[X_i,X_j]_L=[LX_i, X_j]+[X_i, LX_j]-L[X_i,X_j].$$
For $L=I$ the vector field $[X_i,X_j]_L$ reduces to the commutator $[X_i,X_j]$ and the differential $d_L$ to $d$.

The fact that $d^2_L=0$ is equivalent to the vanishing of the Nijenhuis torsion, i.e. for any pair of vector fields $X$ and $Y$ we have
\beq\label{nijenhis}
[LX,LY]-L\,[X,LY]-L\,[LX,Y]+L^2\,[X,Y]=0.
\eeq
The differentials $d$ and $d_L$ anticommute. Indeed taking into account that also the tensor field $L+I$ 
 has vanishing Nijenhuis torsion we have 
\begin{eqnarray*}
d^2_{L+I}=(d+d_L)^2=d\cdot d_L+d_L\cdot d=0.
\end{eqnarray*}
In other words the pair $(d,d_L)$ defines a bidifferential complex. 
Such complex plays an important role in the theory of integrable systems, in both finite \cite{Magri} and infinite dimensional
 case.
 
 In the latter case it has been proved in \cite{LM} that starting from the Fr\"{o}licher-Nijenhuis bicomplex it is possible to construct an integrable hierachy of
  quasilinear systems of PDEs of the form
$$u^i_{t_n}=(V_n({\bf u}))^i_ju^j_x.$$
The tensor fields $V_n$ which define this hierarchy are polynomials in $L$
  $$V_n=L^n+a_0L^{n-1}+a_1L^{n-2}+\cdots+a_{n-1}I.$$
  The coefficients of these polynomials are obtained recursively from $a_0$ by means of a generalized Lenard-Magri chain.    
It has been proved in \cite{Limrn} that, in the case $L={\rm diag}(u_1,...,u_n)$ and $a_0=\sum_k\epsilon^k u_k$, the above hierarchy can be identified with the principal hierarchy of a bi-flat F-manifold. Bi-flat F-manifolds are a generaliziation of Dubrovin-Frobenius manifolds. The flat pencil of metrics associated with a Dubrovin-Frobenius manifold is replaced by a pair of flat connections satisfying suitable compatibility conditions. Many results in the theory of Dubrovin-Frobenius manifolds can be generalized in this more general setting. We mention:
\begin{itemize}
\item the relation with reflection groups \cite{KMS15,ALcomplex,KMS}.
\item the relation with Painlevé transcendents \cite{ALjgp,Limrn,ALmulti,KMS15,KM}. We point out that this relation survives also in the regular 3-dimensional non semisimple cases  leading to a correspondence with solutions of the full family of Painlev\'e IV and V depending on the number of Jordan blocks appearing in the Jordan form of the operator of multiplication by the Euler vector field, while in the case of Dubrovin-Frobenius manifolds the same cases are parametrized by elementary functions \cite{LP22}.   
\item the relation with (generalization of) cohomological field theories and (generalization of) Givental group action \cite{ABLR1}.
\item the relation with dispersive integrable hierarchies \cite{BR,ABLR2}.
\end{itemize}
 
The aim of this paper is to extend the construction of \cite{Limrn} in the non semisimple regular case. The extension is based on two main facts:
\begin{enumerate}
\item The construction of integrable hierarchies of hydrodynamic type starting from a $(1,1)$-tensor field $L$ with vanishing Nijenhuis torsion
 does not require that $L$ is diagonalizable.
\item In the theory of bi-flat F-manifolds there is a natural operator with vanishing Nijenhuis torsion: it is the operator of multiplication by the Euler
 vector field. Moreover in the regular case there are special coordinates
 found by David and Hertling in \cite{DH} where the unit vector field, the Euler vector field and the product 
  have a canonical form. 
\end{enumerate}
The above facts suggest to proceed in the following way:
\begin{itemize}
\item Following \cite{LM} for each David-Hertling canonical form of $L$ and for a suitable choice of the function $a_0$ we can construct integrable hierarchies of hydrodynamic type. Assuming that $a_0$ is a weighted sum of the traces of the blocks appearing in the David-Hertling canonical form of $E\circ$, the freedom reduces to a choice of $r$ parameters (the weights) where $r$ is the number of the blocks.
\item Then, following \cite{LP} we can impose that the flows of this hierarchy are symmetries of the principal
 hierarchy associated with a bi-flat F-manifold. 
\end{itemize}
The main result of the paper is that for any choice of the weights there is a unique associated bi-flat structure. Following \cite{ALmulti} we call bi-flat F-manifolds obtained in this way \emph{Lauricella bi-flat F-manifolds} since in the semisimple case they are related to the theory of Lauricella functions \cite{Lauricella} and Lauricella connections \cite{Looijenga}.
 
The paper is organized as follows: in Section 2 we recall the main facts about the theory of integrable systems of hydrodynamic type, in Section 3 we recall
 the definition of  (bi)-flat F-manifolds and the associated principal hierarchy; in Section 4 we recall the construction of integrable systems of hydrodynamic type
  by means of Fr\"{o}licher-Nijenhuis bicomplex and in Section 5 we recall the definition of Lauricella bi-flat F-manifolds in the semisimple case. Section 6 is devoted to describe non semisimple regular Lauricella bi-flat F-manifolds in dimensions $2,3,4,5$. This section plays an important role since it suggests a strategy to   prove the main result of the paper, first in the case of a single Jordan block of arbitrary size (Section 7) and finally in the general regular case (Section 8).   
   Even if the results of section 7 follow from the general theorem proved in Section 8, we have decided to keep it
    since it contains the essential ideas of the general proof without the extra  non trivial difficulties due to the high number of subcases that one has to consider in the case of an arbitrary number of Jordan blocks.
\newline
\newline
\noindent{\bf Acknowledgements}.  P.~L. is supported by funds of H2020-MSCA-RISE-2017 Project No. 778010 IPaDEGAN. We acknowledge the support of  INFN   (Istituto Nazionale di Fisica Nucleare) by IS-CSN4 Mathematical Methods of Nonlinear Physics. Authors are also thankful to GNFM (Gruppo Nazionale di Fisica Matematica) for supporting activities that contributed to the research reported in this paper.
 
\section{Integrable systems of hydrodynamic type}
Integrable diagonal systems of hydrodynamic type 
\begin{equation}
\label{hts}
u^i_t=v^i(u)u^i_x\qquad i=1,...,n.
\end{equation}
have been studied by Tsarev in \cite{ts}. Assuming $v^i\ne v^j$ Tsarev proved that all the information about the integrability of such systems is contained in the $n(n-1)$ functions
\begin{equation}\label{CHsymb}
\Gamma^i_{ij}=\f{\d_j v^i}{v^j-v^i},\qquad i\ne j.
\end{equation}
The system is integrable iff these functions satisfy the conditions 
\begin{equation}
\label{semih}
\partial_j\left(\f{\d_k v^i}{v^k-v^i}\right)=\partial_k\left(\f{\d_j v^i}{v^j-v^i}\right),\qquad\forall i\ne j\ne k\ne i.
\end{equation}
Systems satisfying the condition \eqref{semih} are called semi-Hamiltonian systems or rich systems. They possess infinitely many symmetries (depending on $n$ functions of a single variable)
\begin{equation}
\label{hts2}
u^i_{\tau}=w^i(u)u^i_x\qquad i=1,...,n.
\end{equation}
obtained solving the linear system
\beq\label{sym}
\d_j w^i=\Gamma^i_{ij}(w^j-w^i)
\eeq
and infinitely many densities of conservation laws 
 obtained solving the linear system
\begin{equation}
\label{cl}
\d_i\d_j h-\Gamma^i_{ij}\d_i h-\Gamma^j_{ji}\d_j h=0, \qquad i\neq j.
\end{equation}

Let us consider now a general system of hydrodynamic type
\beq\label{sht}
u^i_t=V^i_j({\bf u})u^j_x,\qquad i=1,...,n.
\eeq
Assuming that at each point the $(1,1)$-tensor field $V$ has pairwise distict eigenvalues, the diagonalizability of the system is equivalent to the vanishing of the Haantjes tensor of $V$
\cite{H} and the semi-Hamiltonian condition \eqref{shder} is equivalent to the vanishing of a tensor field, called the \emph{semi-Hamiltonian tensor}  \cite{PSS}. Sevennec observed that semi-Hamiltonian systems can be equivalently defined as diagonalizable systems of hydrodynamic type that can be written as systems of conservation laws \cite{Sevennec}.

The main equations of Tsarev's theory can be also formulated in terms of a family of torsionless connections.
\begin{defi}
Given an affinor $V$, a torsionless connection $\nabla$ satisfying 
\beq\label{tsarevconnection}
d_{\nabla}V=0
\eeq
will be called a \emph{Tsarev's connection} associated with $V$.
\end{defi}

Tsarev's connections are not uniquely defined: 
in the Riemann invariants $(r^1,...,r^n)$ where $V={\rm diag}(v^1,...,v^n)$ the above condition is equivalent to $\Gamma^i_{jk}=0$ for pairwise distinct indices and to \eqref{CHsymb}. Moreover $\Gamma^i_{ji}=\Gamma^i_{ij}$ due to the vanishing of the torsion.
 All the remaining Christoffel symbols $\Gamma^i_{jj}$ and $\Gamma^i_{ii}$ are free. To prove this fact we have to spell out the condition
$$(\dna V)^i_{jk}=\d_j V^i_k+\Gamma^i_{mj}V^m_k-\d_k V^i_j-\Gamma^i_{mk}V^m_j=0$$
in the Riemann invariants. If all indices are distinct we obtain 
\[(\dna V)^i_{jk}=\Gamma^i_{kj}(v^k-v^j)=0\] 
and as a consequence, taking into account that the characteristic velocities are pairwise distinct, we obtain $\Gamma^i_{kj}=0$ for $i\ne j\ne k\ne i$. If $i=k$ (or $i=j$) we get 
$$(\dna V)^i_{ji}=\d_j v^i+\Gamma^i_{ij}(v^i-v^j)=0.$$

\begin{theorem}
A diagonalizable system of hydrodynamic type with pairwise distinct characteristic velocities is \emph{semi-Hamiltonian} iff the Tsarev's connections
 associated with $V$ satisfy the condition $\dna^2 W=0$ for any matrix $W$
 commuting with $V$.
\end{theorem}

\noindent
\emph{Proof}. By straightforward computation we get 
 $$[\dna^2 W]_{ijk}^l =R^l_{ijn}W^n_k+R^l_{jkn}W^n_i+R^l_{kin}W^n_j,$$
 where $R$ is the Riemann tensor of $\nabla$:
\[R^k_{ijl}=\d_j\Gamma^k_{il}- \d_i\Gamma^k_{jl}+\Gamma^k_{js}\Gamma^s_{il}-\Gamma^k_{is}\Gamma^s_{jl}.\]
In the Riemann invariants the set of matrices $W$ are diagonal and the
condition $\dna^2 W=0$  reads
 $$[\dna^2 W]_{ijk}^l =R^l_{ijk}w^k+R^l_{jki}w^i+R^l_{kij}w^j=0$$
 for any $(w^1,...,w^n)$. Due to arbitrariness of $W$ this is equivalent to $R^l_{ijk}=0$ for pairwise distinct indices $i,j,k$. If at least two of
  the three  indices $i,j,k$ are equal the condition is automatically satisfied since $R^l_{jik}=-R^l_{ijk}$. Thus assuming the indices $i,j,k$ pairwise
   distinct we need to consider the case $l=i$ (the cases $l=j$ and $l=k$ are equivalent). Due to arbitrariness of $W$ we obtain the conditions
 \[R^i_{jki}=0,\quad R^i_{ijk}=0.\]
The second condition, also known as \emph{Darboux-Tsarev system}, reads
\begin{equation}\label{sh}
\d_i\Gamma^k_{kj}+\Gamma^k_{ki}\Gamma^k_{kj}-\Gamma^k_{kj}\Gamma^j_{ji}
-\Gamma^k_{ik}\Gamma^i_{ij}=0,\qquad\forall i\ne j\ne k\ne i
\end{equation}
while the first condition reads
\begin{equation}
\label{shder}
\partial_j\Gamma^i_{ik}=\partial_k\Gamma^i_{ij},\qquad\forall i\ne j\ne k\ne i.
\end{equation}
Clearly \eqref{shder} follows from \eqref{sh}. Remarkably, if the functions $\Gamma^i_{ij}$ are given by \eqref{CHsymb} both conditions \eqref{semih} and  \eqref{sh} are equivalent to \eqref{shder} due to the identity \cite{ts}
\begin{equation}\label{tsarevid}
\d_i\Gamma^k_{kj}+\Gamma^k_{ki}\Gamma^k_{kj}-\Gamma^k_{kj}\Gamma^j_{ji}
-\Gamma^k_{ki}\Gamma^i_{ij}=\f{v^i-v^k}{v^j-v^i}\left[\d_j\left(\f{\d_i v^k}{v^i-v^k}\right)-\d_i\left(\f{\d_j v^k}{v^j-v^k}\right)\right].
\end{equation} 

\endproof

The Sevennec's result can be formulated as follows.
\begin{theorem}
Let $V$ be a $(1,1)$-tensor field with distinct eigenvalues and with vanishing Haantjies tensor, then 
$V$ defines a semi-Hamiltonian system iff among the associated  Tsarev's connections there is a flat
 connection. 
\end{theorem}
 
\noindent
\emph{Proof}.  Assume that $\nabla$ is a flat Tsarev's connection. In flat coordinates the condition $d_{\nabla}V=0$ reads
 $\d_kV^i_j=\d_jV^i_k$ which implies that in flat coordinates (locally) we have $V^i_j=\d_j X^i$ and thus $V^i_j u^j_x=\d_x X^i$. 
  Since Riemann invariants exist by hypothesis the  Sevennec's result implies that the sistem defined by $V$ is semi-Hamiltonian.
  
 Assume now that $V$ defines a semihamiltonian system then due to Sevennec's result there exist coordinates $(u^1,...,u^n)$ where  $V^i_j u^j_x=\d_x X^i$.  This implies $\d_kV^i_j=\d_jV^i_k$. Define $\nabla$ in the coordinates  $(u^1,...,u^n)$ as $\Gamma^i_{jk}=0$. Then the condition  $\d_kV^i_j=\d_jV^i_k$ can be written as $d_{\nabla}V=0$. In other words $\nabla$ is a flat Tsarev's connection.
 
\endproof  

The symmetries 
\beq\label{symsht}
u^i_\tau=W^i_j({\bf u})u^j_x,\qquad i=1,...,n.
\eeq
of the system \eqref{sht}  are defined by $(1,1)$-tensor fields $W({\bf u})$ commuting with $V$ and satisfying the condition
\beq\label{symmetries}
d_{\nabla}W=0.
\eeq
The full hierarchy is thus defined by the solutions of the system \eqref{symmetries}.

\begin{rmk}
In the general case  the symmetries are defined by $(1,1)$-tensor fields $W({\bf u})$ commuting with $V$ and satisfying the condition
\beq\label{symmetries gen}
V^s_i(d_{\nabla}W)^l_{js}+V^s_j(d_{\nabla}W)^l_{is}=0.
\eeq 
If $V$ is diagonalizable with distict eigenvalues, then \eqref{symmetries gen}  is equivalent to \eqref{symmetries} but in general
 \eqref{symmetries gen} is a weaker condition.
 \end{rmk}

\section{F-manifolds and integrable hierarchies}

\subsection{Flat and bi-flat F-manifolds}

F-manifolds have been introduced by Hertling and Manin in \cite{HM}.
\begin{definition}\label{defFmani}
An \emph{F-manifold} is a triple $(M,\circ,e)$, where $M$ is a manifold, $\circ$ is a commutative associative bilinear product  on the module of (local) vector fields, satisfying the following identity:
\begin{align}
&[X\circ Y,W\circ Z]-[X\circ Y, Z]\circ W-[X\circ Y, W]\circ Z-X\circ [Y, Z \circ W]+X\circ [Y, Z]\circ W \label{HMeq1free}\\
&+X\circ [Y, W]\circ Z-Y\circ [X,Z\circ W]+Y\circ [X,Z]\circ W+Y\circ [X, W]\circ Z=0,\notag
\end{align}
for all local vector fields $X,Y,W, Z$, where $[X,Y]$ is the Lie bracket, and $e$ is a distinguished vector field on $M$ such that $e\circ X=X$ for all local vector fields $X$. 
\end{definition}

In this paper we will consider F-manifolds equipped with some additional structures.
\begin{defi}\cite{manin}
A flat F-manifold $(M, \circ, \nabla, e)$ is an F-manifold $M$ equipped with a connection
 $\nabla$ related to the product $\circ$ and to the unit vector field $e$ by the following axiioms:
\begin{itemize}
\item {\bf Axiom 1} (the connection and the product). The one parameter family of connections
$$\nabla-\lambda\circ$$
is flat and torsionless for any $\lambda$.
\item
{\bf Axiom 2} (the connection and the vector field). The vector field $e$ is covariantly constant: $\nabla e=0$.
\end{itemize}
\end{defi}
Bi-flat F-manifolds are manifolds equipped with two ``compatible'' flat structures. They are defined in the following way. 

\begin{defi}\cite{ALjgp}
A \emph{bi-flat}  F-manifold $(M,\nabla,\nabla^{*},\circ,*,e,E)$
 is a manifold $M$ equipped with a pair
 of connections $\nabla$ and $\nabla^{*}$, a pair of products $\circ$ and $*$ on the tangent spaces $T_u M$ and  a pair of vector fields $e$ and  $E$ s.t.:
\begin{itemize}
\item $(\nabla,\circ,e)$ defines a flat structure on $M$.
\item $(\nabla^{*},*,E)$ defines a flat structure on $M$.
\item The two structures are related by the following conditions
$$X*Y:=(E\circ)^{-1}X\circ Y,\qquad[e,E]=e,\qquad  {\rm Lie}_E \circ=\circ,\qquad(d_{\nabla}-d_{\nabla^{*}})(X\,\circ)=0,$$
where $X$ and $Y$ are arbitrary vector fields and at a generic point the operator $E\circ$ is assumed to be invertible. 
\end{itemize}
\end{defi}
Notice that not all the axioms are independent. For instance the compatibility between the dual connection and the dual product follows from the other axioms \cite{ALmulti}. Notice the dual connection is defined only at the points where the operators $E\circ$ is invertible. At these points the condition 
\beq\label{ahe1}
(d_{\nabla}-d_{\nabla^{*}})(X\,\circ)=0\qquad\forall X
\eeq
 is equivalent to the condition
\beq\label{ahe2}
(d_{\nabla}-d_{\nabla^{*}})(X\,*)=0\qquad\forall X.
\eeq
This implies
\beq\label{dualfromnatural}
\Gamma^{*k}_{ij} =\Gamma^k_{ij}- c^{*l}_{ji}\nabla_l E^k.
\eeq
Moreover the flatness of the dual connection follows from the linearity of the Euler vector field (see \cite{ALcomplex} for the semisimple  case and \cite{KMS} for the general case).

\begin{remark}
The manifolds in the above definitions are real or complex manifolds. In the first case all the geometric data are supposed to be smooth.
 In the latter case $TM$ is intended as the holomorphic tangent bundle and
all the geometric data are supposed to be holomorphic. 
\end{remark}

\subsection{The principal hierarchy}
Given a (bi)-flat F-manifold one can define an integrable hierarchy starting from solutions of the equation \eqref{symmetries} (in this
 case $\nabla$ is given). Looking for solutions of the form $W=X\circ$ we  obtain the condition
\beq\label{admvf-intri}
\left(\nabla_Z X\right)\circ W=\left(\nabla_W X\right)\circ Z
\eeq
for all pairs $(Z,W)$ of vector fields, that is, in local coordinates,
\beq\label{admvf}
c^i_{jm}\nabla_k X^m=c^i_{km}\nabla_j X^m\ .
\eeq 
Let us define now the vector fields $X_{(p,l)}$ as follows: the vector fields $X_{(p,-1)}$  are covariantly constant with respect to $\nabla$, while the others are obtained through the recurrence relation:
\beq\label{recurrenceeq1}
\nabla X_{(p,l+1)}=X_{(p,l)}\circ. 
\eeq 
It is easy to check (see \cite{LPR} for details) that they are solutions of the system \eqref{admvf-intri}. As a consequence the corresponding flows
\[u^i_{t_(p,l)}=c^i_{jk}X^k_{(p,l)}u^j_x,\qquad i=1,...,n\]
commute and define an integrable hierarchy called \emph{the principal hierarchy}.

\section{Fr\"olicher-Nijenhuis bicomplex and integrable systems}
We recall now a construction of integrable hierarchies starting from the  Fr\"{o}licher-Nijenhuis  bicomplex \cite{LM}. 
Let $L$ be a  tensor field of type $(1,1)$ with vanishing Nijenhuis torsion. In \cite{LM} (see also \cite{L2006}) it has been proved that given any solution of 
 the equation
\beq
dd_La_0=0
\eeq
the tensor fields of type $(1,1)$  defined recursively by 
\begin{eqnarray*}
V_{k+1}=V_k L-a_k I,\qquad da_{k+1}=V^*_{k+1} da_0=d_L a_k -a_k da_0
\end{eqnarray*}
starting from the identity $I$, that is
\begin{eqnarray*}
&&V_0=I\\
&&V_1=L-a_0 I\\
&&V_{2}=L^2-a_0 L-a_1 I,\\
&&\cdots=\cdots
\end{eqnarray*}
 define an integrable hierarchy of hydrodynamic type. Remarkably this construction does not require that $L$ is diagonalizable.

\subsection{Examples} 
\subsubsection{Generalized $\epsilon$-system} 
The system of hydrodynamic type
\beq
\begin{bmatrix}
u^1_{t_1} \cr
u^2_{t_1}\cr
\vdots \cr
u^n_{t_1}
\end{bmatrix}
=\begin{bmatrix}
u^1-\sum_{k=1}^n\epsilon_k u^k & 0 & \dots & 0\cr
0 & u^2-\sum_{k=1}^n\epsilon_k u^k & \dots & 0\cr
\vdots & \ddots & \ddots & \vdots\cr
0 & \dots & 0 & u^n-\sum_{k=1}^n\epsilon_k u^k
\end{bmatrix}
\begin{bmatrix}
u^1_x \cr
u^2_x \cr
\vdots \cr
u^n_x
\end{bmatrix}
\eeq
has been obtained in \cite{pavlov} as finite component reduction of an infinite hydrodynamic chain. It
can be written as ${\bf u}_{t_1}=(L-a_0 I){\bf u}_x$ with $a_0=\sum_{k=1}^n\epsilon_k u^k$ and
\beq\label{L}
L=\begin{bmatrix}
u^1 & 0 & \dots & 0\cr
0 & u^2 & \dots & 0\cr
\vdots & \ddots & \ddots & \vdots\cr
0 & \dots & 0 & u^n
\end{bmatrix}.
\eeq
For specific values of the constants
 $\epsilon_i$ it provides well-known examples of integrable systems of hydrodynamic type. The above hierarchy 
  is related to the principal hierarchy associated
 with Lauricella bi-flat F-manifolds \cite{LP,Limrn}. 

\subsubsection{Kodama-Konopelchenko system}
 The system of hydrodynamic type \cite{KK}
\beq
\begin{bmatrix}
u^1_{t_1} \cr
u^2_{t_1}\cr
\vdots \cr
u^{n-1}_{t_1}\cr
u^n_{t_1}
\end{bmatrix}
=\begin{bmatrix}
u^1 & 1 & 0 & \dots & 0\cr
0 & u^1 & 1 & \dots & 0\cr
\vdots & \ddots & \ddots & \ddots & \vdots\cr
0 & \dots & 0 & u^1 & 1 \cr
0 & \dots & 0 & 0 & u^1
\end{bmatrix}
\begin{bmatrix}
u^1_x \cr
u^2_x \cr
\vdots \cr
u^{n-1}_x \cr
u^n_x
\end{bmatrix}
\eeq
can be written as ${\bf u}_{t_1}=(L-a_0 I){\bf u}_x$ with $a_0=-u^1$ and
$$L=
\begin{bmatrix}
0 & 1 & 0 & \dots & 0\cr
0 & 0 & 1 & \dots & 0\cr
\vdots & \ddots & \ddots & \ddots & \vdots\cr
0 & \dots & 0 & 0 & 1 \cr
0 & \dots & 0 & 0 & 0
\end{bmatrix}.
$$ 
Clearly $L$ has vanishing Nijenhuis torsion and $dd_L a_0=0$. Applying the first step of the recursive procedure we have
\begin{eqnarray*}
&&\d_1 a_1=-a_0\d_1a_0=-u^1\\
&&\d_2 a_1=\d_1 a_0-a_0\d_2a_0=-1\\
&&\d_3a_0=0\\
&&\vdots\\
&&\d_na_0=0.\\
\end{eqnarray*}
This implies (up to an inessential constant) $a_1=-u^2-\f{(u^1)^2}{2}$. Therefore the first commuting flow ${\bf u}_t=(L^2-a_0 L-a_1 I){\bf u}_x$
is given by
\beq
\begin{bmatrix}
u^1_{t_2} \cr
u^2_{t_2} \cr
\vdots \cr
u^{n-1}_{t_2}\cr
u^n_{t_2}
\end{bmatrix}
=\begin{bmatrix}
u^2+\f{(u^1)^2}{2} & u^1 & 1 & \dots & 0\cr
0 & u^2+\f{(u^1)^2}{2} & u^1 & \dots & 0\cr
\vdots & \ddots & \ddots & \ddots & \vdots\cr
0 & \dots & 0 & u^2+\f{(u^1)^2}{2} & u^1 \cr
0 & \dots & 0 & 0 & u^2+\f{(u^1)^2}{2}
\end{bmatrix}
\begin{bmatrix}
u^1_x \cr
u^2_x \cr
\vdots \cr
u^{n-1}_x \cr
u^n_x
\end{bmatrix}
\eeq 
The higher flows can be obtained in a similar way. 

\section{From integrable hierarchies to Lauricella bi-flat F-manifolds}
Given an F-manifold with  Euler vector field $(\circ,e,E)$ the operator $E\circ$ has vanishing Nijenhuis torsion (see for instance \cite{ALimrn}). Among Tsarev connections of the associated
 integrable system we consider those satisfying the additional conditions
  (in \cite{LP} they are called \emph{natural connections}):
 \[\nabla_j e^i=0,\qquad \nabla_kc^i_{jl}=\nabla_jc^i_{kl},\qquad\forall i,j,k,l=1,...,n.\]
  In the next two subsections we will show that for special choices of the function $a_0$ the natural connections defined in this way are flat. Moreover
   it is possible to define a second compatible flat structure.  

\subsection{Semisimple Lauricella  bi-flat structure}
Let us recall the following therem of \cite{Limrn} (see also \cite{LP} for the special case $\epsilon_1=\epsilon_2=...=\epsilon_n$).
\begin{theorem} 
For any choice of  $\epsilon_1,\epsilon_2,\dots,\epsilon_n$ there exists a unique semisimple bi-flat structure $(\nabla,\nabla^{*},\circ,*,e,E)$ with canonical coordinates $\{u^1,...,u^n\}$ such that $L=E\circ$ and 
\beq\label{maincond}
d_{\nabla}(L-a_0 I)=0,
\eeq
where $a_0=\sum_{k=1}^n\epsilon_k u^k$. Moreover, in canonical coordinates this structure is given by
\begin{eqnarray*}
e&=&\sum_{k=1}^n \partial_k,\quad E=\sum_{k=1}^n u^k\partial_k\\
c^i_{jk}&=&\delta^i_j\delta^i_k\qquad c^{*i}_{jk}=\frac{1}{u^i}\delta^i_j\delta^i_k\qquad\forall i,j,k\\
\Gamma^{(1)i}_{jk}&=&0\quad \Gamma^{(2)i}_{jk}=0\qquad\forall i\ne j\ne k \ne i\\
\Gamma^{(1)i}_{jj}&=&-\Gamma^{(1)i}_{ij}\quad\Gamma^{(2)i}_{jj}=-\f{u^i}{u^j}\Gamma^{(2)i}_{ij}\qquad i\ne j\\
\Gamma^{(1)i}_{ij}&=&\f{\epsilon_j}{u^i-u^j}\quad\Gamma^{(2)i}_{ij}=\f{\epsilon_j}{u^i-u^j}\qquad i\ne j\\
\Gamma^{(1)i}_{ii}&=&-\sum_{l\ne i}\Gamma^{(1)i}_{li},\quad\Gamma^{(2)i}_{ii}=-\sum_{l\ne i}\f{u^l}{u^i}\Gamma^{(2)i}_{li}-\f{1}{u^i}.\
\end{eqnarray*}
\end{theorem}
 This structure is related to the theory of Lauricella functions (see next section) and for this reason we will call it the \emph{semisimple Lauricella bi-flat structure}. This example is also related to the Euler-Poisson-Darboux system \cite{LM,L2006,LP,Limrn}
\beq\label{ddl}
dd_L k=da_0\wedge dk,
\eeq
In particular flat coordinates of $\nabla$ are special solutions of \eqref{ddl}.

\subsection{Non semisimple Lauricella bi-flat F-manifolds}
In this paper we consider a generalization of Lauricella bi-flat structures to the non semisimple regular case. We will use this important result of \cite{DH}
 about the existence of non semisimple canonical coordinates.
\begin{theorem}[\cite{DH}]\label{DavidHertlingth}
Let $(M, \circ, e, E)$ be a regular F-manifold of dimension greater or equal to $2$ with an Euler vector field $E$ of weight one. Furthermore assume that locally around a point $p\in M$, the Jordan canonical form of the operator $L$ has $n$ Jordan blocks $m_1,...,m_n$ with distinct eigenvalues. Then
there exists locally around $p$ a distinguished system of coordinates $\{u^1, \dots, u^{m_1+\dots+m_n}\}$ such that 
\begin{eqnarray} \label{canonical1}
e&=&\f{\partial}{\partial u^1}+\f{\partial}{\partial u^{m_1+1}}+\f{\partial}{\partial u^{m_1+m_2+1}}+\cdots+\f{\partial}{\partial u^{m_1+\cdots+m_{n-1}+1}},\\ \label{canonical2}
E&=&\sum_{s=1}^{m_1+\cdots+m_n}u^{s}\frac{\partial}{\partial u^{s}},\\
c^l_{ij}&=&\delta^l_{i+j-1},\quad i,j,l=m_1+\dots+m_{k-1}+1,...,m_1+\dots+m_{k-1}+m_k\\ \label{canonical3}
c^l_{ij}&=&0,\quad {\rm otherwise}.
\end{eqnarray}
\end{theorem}
We start from the integrable hierarchy associated with the tensor field
 $L=E\circ$ and with $a_0=\sum_{i=1}^n\epsilon_i{\rm Tr}(L_i)$. 
 By definition $L$ contains $n$ Jordan blocks $L_1,...,L_n$ of dimension $m_1,...,m_n$ respectively. Each Jordan block and has the form
\beq\label{Lblock}
L_k=\begin{bmatrix}
u^{k,1} & 0 & \dots & 0\cr
u^{k,2} & u^{k,1} & \dots & 0\cr
\vdots & \ddots & \ddots & \vdots\cr
u^{k,m_k} & \dots & u^{k,2} & u^{k,1}
\end{bmatrix}
\eeq
where $u^{k,l}=u^{m_1+\dots+m_{k-1}+l}$. 
The case $m_1=...=m_n=1$ corresponds to the usual generalized $\epsilon$-system. In the next Section we will prove that for any choice of  $\epsilon_1,\epsilon_2,\dots,\epsilon_n$ and $m_1,...,m_n$ and up to $n=5$ there exists a unique bi-flat F-structure such that $d_{\nabla}(L-a_0I)=0$. In Section 7 we will consider the case $n=1$ and $m_1$ arbitrary and in the final Section 8 we will show that for any choice of  $\epsilon_1,\epsilon_2,\dots,\epsilon_n$ there exists a unique regular bi-flat structure such that $L=E\circ$ and
\beq\label{maincond2}
d_{\nabla}(L-a_0I)=0.
\eeq
\begin{rmk}
By straightforward computation one gets
\[\left(d_{\nabla}(L-a_0I)\right)^i_{jk}\d_j a_0=L^i_j\nabla_i(da_0)_k-L^i_k\nabla_i(da_0)_j.\]
Therefore the condition \eqref{maincond2} implies
\[L^i_j\nabla_i(da_0)_k-L^i_k\nabla_i(da_0)_j=0.\]
\end{rmk}

\section{Bi-flat Lauricella structures in dimension $2,3,4,5$}
In this Section we provide a complete classification of non-semisimple bi-flat F-manifold structures in 2,3,4 and 5 dimensions.   
\subsection{$2$-dimensional case}
\subsubsection{$2\times 2$ Jordan block}
\beq
L=\begin{bmatrix}
u^{1} & 0\cr
u^{2} & u^{1} \cr
\end{bmatrix},\quad
e=\f{\partial}{\partial u^1}
\eeq
The non vanishing Christoffel symbol of $\nabla^{(1)}$ is $\Gamma^2_{22}=-\f{2\epsilon_1}{u^2}$.

\subsection{$3$-dimensional case}
\subsubsection{$3\times 3$ Jordan block}
\beq
L=\begin{bmatrix}
u^{1} & 0 & 0\cr
u^{2} & u^{1} &  0\cr
u^{3} & u^{2} & u^{1}
\end{bmatrix},\quad
e=\f{\partial}{\partial u^1},\quad a_0=3\epsilon_1u^1
\eeq
The non vanishing Christoffel symbols $\Gamma^i_{jk}$ (up to exchange of $j$ with $k$) are
\begin{eqnarray*}
\Gamma^2_{22}=\Gamma^3_{23}=-\f{3\epsilon_1}{u^2},\,\Gamma^3_{22}=\f{3\epsilon_1u^3}{(u^2)^2}
\end{eqnarray*}

\subsubsection{$2\times 2+1\times 1$ Jordan blocks}
\beq\label{Lblock}
L=\begin{bmatrix}
u^{1} & 0 & 0\cr
u^{2} & u^{1} &  0\cr
0 & 0 & u^{3}
\end{bmatrix},\quad
e=\f{\partial}{\partial u^1}+\f{\partial}{\partial u^3},\quad a_0=2\epsilon_1 u^1+\epsilon_3u^3
\eeq
The non vanishing Christoffel symbols $\Gamma^i_{jk}$ (up to exchange of $j$ with $k$) are
\begin{eqnarray*}
&&\Gamma^2_{22}=-\f{2\epsilon_1}{u^2},\,\Gamma^1_{13}=\Gamma^2_{23}=-\Gamma^1_{11}=-\Gamma^1_{33}=-\Gamma^2_{12}=\f{\epsilon_3}{u^1-u^3}\\
&&\Gamma^3_{11}=\Gamma^3_{33}=-\Gamma^3_{13}=\f{2\epsilon_1}{u^1-u^3},\,\Gamma^2_{11}=\Gamma^2_{33}=-\Gamma^2_{13}
=\f{\epsilon_3u^2}{(u^1-u^3)^2}
\end{eqnarray*}
\subsection{$4$-dimensional case}
\subsubsection{$4\times 4$ Jordan block}
\beq
L=\begin{bmatrix}
u^{1} & 0 & 0 & 0\cr
u^{2} & u^{1} & 0 & 0\cr
u^3 & u^2 & u^1 & 0\cr
u^{4} & u^3 & u^{2} & u^{1}
\end{bmatrix},\quad
e=\f{\partial}{\partial u^1},\quad a_0=4\epsilon_1u^1
\eeq
The non vanishing Christoffel symbols $\Gamma^i_{jk}$ (up to exchange of $j$ with $k$) are
\begin{eqnarray*}
&&\Gamma^2_{22}=\Gamma^3_{23}=\Gamma^4_{33}=\Gamma^4_{24}=-\f{4\epsilon_1}{u^2},\,\Gamma^3_{22}=\Gamma^4_{23}=\f{4\epsilon_1u^3}{(u^2)^2},\,\Gamma^4_{22}=\f{4\epsilon_1(u^2u^4-(u^3)^2)}{(u^2)^3}.
\end{eqnarray*}
\subsubsection{$3\times 3+1\times 1$ Jordan blocks}
\beq
L=\begin{bmatrix}
u^{1} & 0 & 0 & 0\cr
u^{2} & u^{1} & 0 & 0\cr
u^3 & u^2 & u^1 & 0\cr
0 & 0 & 0 & u^{4}
\end{bmatrix},\quad
e=\f{\partial}{\partial u^1}+\f{\partial}{\partial u^4},\quad a_0=3\epsilon_1u^1+\epsilon_4u^4
\eeq
The non vanishing Christoffel symbols $\Gamma^i_{jk}$ (up to exchange of $j$ with $k$) are
\begin{eqnarray*}
&&\Gamma^2_{22}=\Gamma^3_{23}=-\f{3\epsilon_1}{u^2},\,\Gamma^3_{11}=\Gamma^3_{44}=-\Gamma^3_{14}
=\f{\epsilon_4u^3}{(u^1-u^4)^2}-\f{\epsilon_4(u^2)^2}{(u^1-u^4)^3}\\
&&\Gamma^4_{11}=\Gamma^4_{44}=-\Gamma^4_{14}=\f{3\epsilon_1}{u^1-u^4},\,\\&&\Gamma^3_{12}=-\Gamma^3_{24}=-\Gamma^2_{14}=\Gamma^2_{11}=\Gamma^2_{44}=\f{\epsilon_4u^2}{(u^1-u^4)^2}\\
&&\Gamma^3_{34}=-\Gamma^3_{13}=-\Gamma^2_{12}=-\Gamma^1_{11}=\Gamma^1_{14}=-\Gamma^1_{44}=\f{\epsilon_4}{u^1-u^4},\,\\&&\Gamma^3_{22}=\f{3\epsilon_1u^3}{(u^2)^2}-\f{\epsilon_4}{u^1-u^4}
\end{eqnarray*}

\subsubsection{$2\times 2+2\times 2$ Jordan blocks}
\beq
L=\begin{bmatrix}
u^{1} & 0 & 0 & 0\cr
u^{2} & u^{1} & 0 & 0\cr
0 & 0 & u^3 & 0\cr
0 & 0 & u^{4} & u^{3}
\end{bmatrix},\quad
e=\f{\partial}{\partial u^1}+\f{\partial}{\partial u^3},\quad a_0=2\epsilon_1u^1+2\epsilon_3u^3
\eeq
The non vanishing Christoffel symbols $\Gamma^i_{jk}$ (up to exchange of $j$ with $k$) are
\begin{eqnarray*}
&&\Gamma^2_{22}=-\f{2\epsilon_1}{u^2},\,\Gamma^4_{44}=-\f{2\epsilon_3}{u^4},\,\,\Gamma^1_{13}=
-\Gamma^2_{12}=-\Gamma^1_{11}=-\Gamma^1_{33}=\f{2\epsilon_3}{u^1-u^3},\\
&&\Gamma^3_{11}=\Gamma^4_{34}=\Gamma^3_{33}=-\Gamma^3_{13}=-\Gamma^4_{14}=
\f{2\epsilon_1}{u^1-u^3},\\
&&\Gamma^2_{11}=\Gamma^2_{33}=-\Gamma^2_{13}=\f{2\epsilon_3u^2}{(u^1-u^3)^2},\,\Gamma^4_{11}=\Gamma^4_{33}=-\Gamma^4_{13}
=\f{2\epsilon_1u^4}{(u^1-u^3)^2}
\end{eqnarray*}

\subsubsection{$2\times 2+1\times 1+1\times 1$ Jordan blocks}
\beq
L=\begin{bmatrix}
u^{1} & 0 & 0 & 0\cr
u^{2} & u^{1} & 0 & 0\cr
0 & 0 & u^3 & 0\cr
0 & 0 & 0 & u^{4}
\end{bmatrix},\quad
e=\f{\partial}{\partial u^1}+\f{\partial}{\partial u^3}+\f{\partial}{\partial u^4},\quad a_0=2\epsilon_1u^1+\epsilon_3u^3+\epsilon_4u^4
\eeq
The non vanishing Christoffel symbols $\Gamma^i_{jk}$ (up to exchange of $j$ with $k$) are
\begin{eqnarray*}
&&\Gamma^2_{22}=-\f{2\epsilon_1}{u^2},\,\Gamma^1_{13}=\Gamma^2_{23}=-\Gamma^1_{33}=\f{\epsilon_3}{u^1-u^3},\,
\Gamma^1_{14}=\Gamma^2_{24}=-\Gamma^1_{44}=\f{\epsilon_4}{u^1-u^4},\\
&&\Gamma^3_{34}=-\Gamma^3_{44}=\f{\epsilon_4}{u^3-u^4},\,
\Gamma^4_{34}=-\f{\epsilon_3}{u^3-u^4},\,\Gamma^3_{13}=-\Gamma^3_{11}=\f{-2\epsilon_1}{u^1-u^3},\\
&&\Gamma^4_{14}=\f{-2\epsilon_1}{u^1-u^4},\,\Gamma^1_{11}=\Gamma^2_{12}=\Gamma^1_{33}+\Gamma^1_{44},\,\Gamma^2_{11}=\Gamma^2_{33}+\Gamma^2_{44},\\
&&\Gamma^2_{33}=-\Gamma^2_{13}=\f{\epsilon_3u^2}{(u^1-u^3)^2},\,\Gamma^2_{44}=-\Gamma^2_{14}=\f{\epsilon_4u^2}{(u^1-u^4)^2},\\
&&\Gamma^3_{33}=\f{2\epsilon_1}{u^1-u^3}-\f{\epsilon_4}{u^3-u^4},\,\,\Gamma^4_{44}=\f{2\epsilon_1}{u^1-u^4}
+\f{\epsilon_3}{u^3-u^4}
\end{eqnarray*}
\subsection{$5$-dimensional case}
\subsubsection{$5\times 5$ Jordan block}
\beq
L=\begin{bmatrix}
u^{1} & 0 & 0 & 0 & 0\cr
u^{2} & u^{1} & 0 & 0 & 0\cr
u^3 & u^2 & u^1 & 0 & 0\cr
u^{4} & u^3 & u^{2} & u^{1} &0 \cr
u^{5} & u^4 & u^{3} & u^{2} &u^1
\end{bmatrix},\quad
e=\f{\partial}{\partial u^1},\quad a_0=5\epsilon_1u^1
\eeq
The non vanishing Christoffel symbols $\Gamma^i_{jk}$ (up to exchange of $j$ with $k$) are
\begin{eqnarray*}
&&\Gamma^2_{22}=\Gamma^3_{23}=\Gamma^4_{33}=\Gamma^4_{24}=\Gamma^5_{25}=\Gamma^5_{34}=-\f{5\epsilon_1}{u^2},\\
&&\Gamma^3_{22}=\Gamma^4_{23}=\Gamma^5_{24}=\Gamma^5_{33}=\f{5\epsilon_1u^3}{(u^2)^2},\,\Gamma^4_{22}=\Gamma^5_{23}
=\f{5\epsilon_1(u^2u^4-(u^3)^2)}{(u^2)^3},\\
&&\Gamma^5_{22}=\f{5\epsilon_1((u^2)^2u^5-2u^2u^3u^4+(u^3)^3)}{(u^2)^4}
\end{eqnarray*}

\subsubsection{$4\times 4+1\times 1$ Jordan blocks}
\beq
L=\begin{bmatrix}
u^{1} & 0 & 0 & 0 & 0\cr
u^{2} & u^{1} & 0 & 0 & 0\cr
u^3 & u^2 & u^1 & 0 & 0\cr
u^4 & u^3 & u^2 & u^1 & 0\cr
0 & 0 & 0 & 0 & u^{5}
\end{bmatrix},\quad
e=\f{\partial}{\partial u^1}+\f{\partial}{\partial u^5},\quad a_0=4\epsilon_1u^1+\epsilon_5u^5.
\eeq
The non vanishing Christoffel symbols $\Gamma^i_{jk}$ (up to exchange of $j$ with $k$) are
\begin{eqnarray*}
&&\Gamma^2_{22}= \Gamma^3_{23}=\Gamma^4_{24}=\Gamma^4_{33}= -4\f{\epsilon_1}{u^2},\\ 
&&\Gamma^1_{15}=\Gamma^2_{25}=\Gamma^3_{35}=\Gamma^4_{45}=\f{\epsilon_5}{u^1-u^5}\\
&&\Gamma^1_{55}=\Gamma^2_{12}=\Gamma^1_{11}=\Gamma^3_{13}=\Gamma^4_{14}=\Gamma^4_{41}=-\f{\epsilon_5}{u^1-u^5},\\
&&\Gamma^5_{11}=\Gamma^5_{55}=-\Gamma^5_{15}=\f{4\epsilon_1}{u^1-u^5},\\
&&\Gamma^2_{11}=\Gamma^2_{55}=\Gamma^3_{12}=\Gamma^4_{13}=-\Gamma^2_{15}=-\Gamma^3_{25}=-\Gamma^4_{35}=
-\Gamma^4_{53}=\f{\epsilon_5u^2}{(u^1-u^5)^2},\\ 
&&\Gamma^3_{11}=\Gamma^3_{55}=\Gamma^4_{12}=-\Gamma^3_{15}= -\Gamma^4_{25}= 
\f{\epsilon_5u^3}{(u^1-u^5)^2}-\f{\epsilon_5(u^2)^2}{(u^1-u^5)^3},\\ 
&&\Gamma^3_{22}=\Gamma^4_{23}=\f{4\epsilon_1u^3}{(u^2)^2}-\f{\epsilon_5}{u^1-u^5},\\ 
&&\Gamma^4_{11}= \Gamma^4_{55}=-\Gamma^4_{15}=\f{\epsilon_5(u^2)^3}{(u^1-u^5)^4}-\f{2\epsilon_5u^2u^3}{(u^1-u^5)^3}
+\f{\epsilon_5u^4}{(u^1-u^5)^2},\\
&&\Gamma^4_{22}=\f{4\epsilon_1(u^2u^4-(u^3)^2)}
{(u^2)^3}+\f{\epsilon_5u^2}{(u^1-u^5)^2}.
\end{eqnarray*}

\subsubsection{$3\times 3+2\times 2$ Jordan blocks}
\beq
L=\begin{bmatrix}
u^{1} & 0 & 0 & 0 & 0\cr
u^{2} & u^{1} & 0 & 0 &0\cr
u^{3} & u^{2} & u^1 & 0 &0\cr
0 & 0 & 0 & u^4 & 0\cr
0 & 0 & 0 & u^{5} & u^{4}
\end{bmatrix},\quad
e=\f{\partial}{\partial u^1}+\f{\partial}{\partial u^4}\quad a_0=3\epsilon_1u^1+2\epsilon_4u^4.
\eeq
The non vanishing Christoffel symbols $\Gamma^i_{jk}$ (up to exchange of $j$ with $k$) are
\begin{eqnarray*}
&&\Gamma^1_{14} = \Gamma^2_{24} = \Gamma^3_{34} = -\Gamma^1_{11} = -\Gamma^1_{44} = -\Gamma^2_{12} = 
-\Gamma^3_{13} = \f{2\epsilon_4}{u^1-u^4},\\ 
&&\Gamma^2_{11} = \Gamma^2_{44} =\Gamma^3_{12} = -\Gamma^2_{14}=-\Gamma^3_{24}= \f{2 u^2\epsilon_4}{(u^1-u^4)^2},\, 
\Gamma^2_{22} =\Gamma^3_{23} = -\f{3\epsilon_1}{u^2},\\ 
&&\Gamma^5_{55} = -\f{2\epsilon_4}{u^5},\,\Gamma^3_{22}= 
\f{3\epsilon_1 u^3}{(u^2)^2}-\f{2\epsilon_4}{u^1-u^4},\\ 
&&\Gamma^3_{11}=\Gamma^3_{44}=-\Gamma^3_{14} =\f{2\epsilon_4(u^1u^3-(u^2)^2-u^3u^4)}{(u^1-u^4)^3},\\
&&\Gamma^4_{11}=\Gamma^4_{44}=-\Gamma^4_{14}=-\Gamma^5_{15}=\Gamma^5_{45}=\f{3\epsilon_1}{u^1-u^4},\\
&&\Gamma^5_{11}=\Gamma^5_{44}=-\Gamma^5_{14}=\f{3u^5\epsilon_1}{(u^1-u^4)^2},  
\end{eqnarray*}

\subsubsection{$3\times 3+1\times 1+1\times 1$ Jordan blocks}
\beq
L=\begin{bmatrix}
u^{1} & 0 & 0 & 0 & 0\cr
u^{2} & u^{1} & 0 & 0 &0\cr
u^{3} & u^{2} & u^1 & 0 &0\cr
0 & 0 & 0 & u^4 & 0\cr
0 & 0 & 0 & 0 & u^{5}
\end{bmatrix},\quad
e=\f{\partial}{\partial u^1}+\f{\partial}{\partial u^4}+\f{\partial}{\partial u^5},\quad a_0=3\epsilon_1u^1+\epsilon_4u^4+\epsilon_5u^5.
\eeq
The non vanishing Christoffel symbols $\Gamma^i_{jk}$ (up to exchange of $j$ with $k$) are
\begin{eqnarray*}
&&\Gamma^3_{23}= \Gamma^2_{22}=-\f{3\epsilon_1}{u^2},\,\Gamma^1_{14}=\Gamma^2_{24}=\Gamma^3_{34}=-\Gamma^1_{44}
=\f{\epsilon_4}{u^1-u^4},\\    
&&\Gamma^1_{15}=\Gamma^2_{25}=\Gamma^3_{35}=-\Gamma^1_{55}= \f{\epsilon_5}{u^1-u^5},\, 
\Gamma^4_{45}=-\Gamma^4_{55}=\f{\epsilon_5}{u^4-u^5},\\ 
&&\Gamma^5_{44}=-\Gamma^5_{45}=\f{\epsilon_4}{u^4-u^5},\,
\Gamma^4_{11}=-\Gamma^4_{14}=\f{3\epsilon_1}{u^1-u^4},\,
\Gamma^5_{11}=-\Gamma^5_{15}=\f{3\epsilon_1}{u^1-u^5},\\
&&\Gamma^2_{11}=\Gamma^3_{12}=\f{\epsilon_4u^2}{(u^1-u^4)^2}+\f{\epsilon_5u^2}{(u^1-u^5)^2},\\ 
&&\Gamma^2_{44}=-\Gamma^2_{14}=-\Gamma^3_{24}=\f{u^2\epsilon_4}{(u^1-u^4)^2},\, 
\Gamma^2_{55}=-\Gamma^2_{15}=-\Gamma^3_{25}=\f{u^2\epsilon_5}{(u^1-u^5)^2},\\ 
&&\Gamma^1_{11}=\Gamma^2_{12}=\Gamma^3_{13}=-\f{\epsilon_4}{u^1-u^4}-\f{\epsilon_5}{u^1-u^5},\\
&&\Gamma^3_{11}=\f{\epsilon_4(-(u^2)^2+(u^1-u^4)u^3)}{(u^1-u^4)^3}+\f{\epsilon_5(-(u^2)^2+(u^1-u^5)u^3)}{(u^1-u^5)^3},\\  
&&\Gamma^3_{44}=-\Gamma^3_{14}=\f{\epsilon_4(u^1u^3-(u^2)^2-u^3u^4)}{(u^1-u^4)^3},\, 
\Gamma^3_{55}=-\Gamma^3_{15}=\f{\epsilon_5(u^1u^3-(u^2)^2-u^3u^5)}{(u^1-u^5)^3},\\ 
&&\Gamma^4_{44}=\f{3\epsilon_1}{u^1-u^4}-\f{\epsilon_5}{u^4-u^5},\,
\Gamma^5_{55}=\f{3\epsilon_1}{u^1-u^5}+\f{\epsilon_4}{u^4-u^5},\\ 
&&\Gamma^3_{22}=\f{3u^3\epsilon_1}{(u^2)^2}-\f{\epsilon_4}{u^1-u^4}-\f{\epsilon_5}{u^1-u^5}. 
\end{eqnarray*}

\subsubsection{$2\times 2+2\times 2+1\times 1$ Jordan blocks}
\beq
L=\begin{bmatrix}
u^{1} & 0 & 0 & 0& 0\cr
u^{2} & u^{1} & 0 & 0 & 0\cr
0 & 0 & u^3 & 0 & 0 \cr
0 & 0 & u^4 & u^3 & 0 \cr
0 & 0 & 0 & 0 & u^{5}
\end{bmatrix},\quad
e=\f{\partial}{\partial u^1}+\f{\partial}{\partial u^3}+\f{\partial}{\partial u^5},\quad a_0=2\epsilon_1u^1+2\epsilon_3u^3+\epsilon_5u^5.
\eeq
The non vanishing Christoffel symbols $\Gamma^i_{jk}$ (up to exchange of $j$ with $k$) are
\begin{eqnarray*}
&&\Gamma^1_{11}=\Gamma^2_{12}=-\f{2\epsilon_3}{u^1-u^3}-\f{\epsilon_5}{u^1-u^5},\,\Gamma^3_{33}=\Gamma^4_{34}=\f{2\epsilon_1}{u^1-u^3}-\f{\epsilon_5}{u^3-u^5},\\
&&-\Gamma^1_{13}=\Gamma^1_{33}= -\f{2\epsilon_3}{u^1-u^3},\,\Gamma^1_{15}=\Gamma^2_{25}=-\Gamma^1_{55}=\f{\epsilon_5}{u^1-u^5},\\ 
&&\Gamma^2_{11}=\f{2u^2\epsilon_3}{(u^1-u^3)^2}+\f{u^2\epsilon_5}{(u^1-u^5)^2},\, 
\Gamma^2_{33}=-\Gamma^2_{13}=\f{2u^2\epsilon_3}{(u^1-u^3)^2},\\
&&\Gamma^2_{55}=-\Gamma^2_{15}=\f{u^2\epsilon_5}{(u^1-u^5)^2},\,\Gamma^2_{22}=-\f{2\epsilon_1}{u^2},\, 
\Gamma^4_{44}=-\f{2\epsilon_3}{u^4},\,\Gamma^2_{23}=\f{2\epsilon_3}{u^1-u^3},\\
&&\Gamma^3_{11}=-\Gamma^3_{13}=-\Gamma^4_{14}=\f{2\epsilon_1}{u^1-u^3},\, 
\Gamma^3_{35}=\Gamma^4_{45}=-\Gamma^3_{55}=\f{\epsilon_5}{u^3-u^5},\\ 
&&\Gamma^4_{11}=-\Gamma^4_{13}=\f{2\epsilon_1u^4}{(u^1-u^3)^2},\, 
\Gamma^4_{33}=\f{2u^4\epsilon_1}{(u^1-u^3)^2}+\f{u^4\epsilon_5}{(u^3-u^5)^2},\\
&&\Gamma^4_{55}=-\Gamma^4_{35}=\f{u^4\epsilon_5}{(u^3-u^5)^2},\,  
\Gamma^5_{11}=-\Gamma^5_{15}=\f{2\epsilon_1}{u^1-u^5},\, 
\Gamma^5_{33}=-\Gamma^5_{35}=\f{2\epsilon_3}{u^3-u^5},\\  
&&\Gamma^5_{55}=\f{2\epsilon_1}{u^1-u^5}+\f{\epsilon_3}{u^3-u^5}
\end{eqnarray*}

\subsubsection{$2\times 2+1\times 1+1\times 1+1\times1$ Jordan blocks}
\beq
L=\begin{bmatrix}
u^{1} & 0 & 0 & 0 & 0\cr
u^{2} & u^{1} & 0 & 0 & 0\cr
0 & 0 & u^3 & 0 & 0\cr
0 & 0 & 0 & u^{4} & 0\cr
0 & 0 & 0 & 0 & u^5
\end{bmatrix},\quad
e=\f{\partial}{\partial u^1}+\f{\partial}{\partial u^3}+\f{\partial}{\partial u^4}+\f{\partial}{\partial u^5},
\quad a_0=2\epsilon_1u^1+\epsilon_3u^3+\epsilon_4u^4+\epsilon_5u^5.
\eeq
The non vanishing Christoffel symbols $\Gamma^i_{jk}$ (up to exchange of $j$ with $k$) are
\begin{eqnarray*}
&&\Gamma^1_{11}=\Gamma^2_{12}=-\f{\epsilon_3}{u^1-u^3}-\f{\epsilon_4}{u^1-u^4}-\f{\epsilon_5}{u^1-u^5},\\ 
&&\Gamma^1_{13}=-\Gamma^1_{33}=\f{\epsilon_3}{u^1-u^3},\,\Gamma^1_{14}=\f{\epsilon_4}{u^1-u^4},\, 
\Gamma^1_{15}=-\Gamma^1_{55}=\f{\epsilon_5}{u^1-u^5},\,\Gamma^1_{44}=-\f{\epsilon_4}{u^1-u^4},\\
&&\Gamma^2_{11}=\f{\epsilon_3u^2}{(u^1-u^3)^2}+\f{\epsilon_4u^2}{(u^1-u^4)^2}+\f{\epsilon_5u^2}{(u^1-u^5)^2},\\ 
&&\Gamma^2_{13}=-\f{u^2\epsilon_3}{(u^1-u^3)^2},\, 
\Gamma^2_{14}=-\f{u^2\epsilon_4}{(u^1-u^4)^2},\,\Gamma^2_{15}=-\f{u^2\epsilon_5}{(u^1-u^5)^2},\\ 
&&\Gamma^2_{22}=-\f{2\epsilon_1}{u^2},\,\Gamma^2_{23}=\f{\epsilon_3}{u^1-u^3},\,\Gamma^2_{24}=\f{\epsilon_4}{u^1-u^4},\, 
\Gamma^2_{25}=\f{\epsilon_5}{u^1-u^5},\\ 
&&\Gamma^2_{33}=\f{u^2\epsilon_3}{(u^1-u^3)^2},\, 
\Gamma^2_{44}=\f{u^2\epsilon_4}{(u^1-u^4)^2},\, 
\Gamma^2_{55}=\f{u^2\epsilon_5}{(u^1-u^5)^2},\,\Gamma^3_{11}=-\Gamma^3_{13}=\f{2\epsilon_1}{u^1-u^3},\\
&&\Gamma^3_{33}=\f{2\epsilon_1}{u^1-u^3}-\f{\epsilon_4}{u^3-u^4}-\f{\epsilon_5}{u^3-u^5},\\
&&\Gamma^3_{34}=-\Gamma^3_{44}=\f{\epsilon_4}{u^3-u^4},\, 
\Gamma^3_{35}=-\Gamma^3_{55}=\f{\epsilon_5}{u^3-u^5},\\
&&\Gamma^4_{11}=-\Gamma^4_{14}=\f{2\epsilon_1}{u^1-u^4},\,
\Gamma^4_{33}=-\Gamma^4_{34}=\f{\epsilon_3}{u^3-u^4},\\ 
&&\Gamma^4_{44}=\f{2\epsilon_1}{u^1-u^4}+\f{\epsilon_3}{u^3-u^4}-\f{\epsilon_5}{u^4-u^5},\\ 
&&\Gamma^4_{45}=-\Gamma^4_{55}=\f{\epsilon_5}{u^4-u^5},\,
\Gamma^5_{11}=-\Gamma^5_{15}=\f{2\epsilon_1}{u^1-u^5},\\
&&\Gamma^5_{33}=-\Gamma^5_{35}=\f{\epsilon_3}{u^3-u^5},\,
\Gamma^5_{44}=-\Gamma^5_{45}=\f{\epsilon_4}{u^4-u^5},\\
&&\Gamma^5_{55}=\f{2\epsilon_1}{u^1-u^5}+\f{\epsilon_3}{u^3-u^5}+\f{\epsilon_4}{u^4-u^5}.
\end{eqnarray*}   
  
  \endproof

\section{The case of a Jordan block of arbitrary size}
Looking at the formulas for the case of a single Jordan block we observe that in order to pass from the $n\times n$ Jordan block with $n\epsilon_1=\epsilon$ to the $(n+1)\times (n+1)$ Jordan blok with $(n+1)\epsilon_1=\epsilon$ there is a simple rule. The new non vanishing Christoffel symbols
 are $\Gamma^{n+1}_{ij}$ with $i,j\ne 1$. The Christoffel simbol $\Gamma^{n+1}_{22}$ is given by the formula
\beq\label{mainide}
\Gamma^{n+1}_{22}=-\f{1}{u^2}\sum_{s=1}^{n-1}u^{2+s}\Gamma^{n+1-s}_{22}
\eeq
while the Christoffel simbols $\Gamma^{n+1}_{ij}$ with $i$ and $j$ different from $1$ and not simultaneously equal to $2$ can be written in terms of the Christoffel symbols  $\Gamma^{n}_{ij}$:
\beq\label{steps}
 \Gamma^{n+1}_{ij}= \Gamma^{n}_{i-1,j}=\Gamma^{n}_{i,j-1}
\eeq
provided that $i-1\ne 1$ or $j-1\ne 1$ (all the Christofell symbols with $i$ or $j$ equal to $1$ vanish). Notice that from the above definition it follows immediately that all the non vanishing Christoffel symbols can be obtained recursively starting from $\Gamma^2_{22}=-\f{\epsilon}{u^2}$. Indeed applying the relation \eqref{steps} $i+j-4$ times we obtain (of course if $n-i-j<-3$ the Christoffel symbols vanish)
\beq
\Gamma^{n+1}_{ij}= \Gamma^{n+5-i-j}_{22}
\eeq
and more in general (the above property holds for all $n$)
\begin{eqnarray}
\label{mainide_bis}
\Gamma^{k}_{ij}&=& \Gamma^{k+4-i-j}_{22}\qquad {\rm if}\,\,k-i-j\ge -2,\,\,i,j\ne 1\\
\Gamma^{k}_{ij}&=& 0\qquad {\rm if}\,\,k-i-j<-2.
\end{eqnarray}
\subsection{Technical lemmas}
Using the above remarks one can easily prove by induction the following lemmas.
\begin{lemma}
The Christoffel symbols $\Gamma^k_{ij}$ associated with the $n\times n$ Jordan block, defined recursively as explained above starting from the $2\times 2$ Jordan block satisfy the following identity:
\beq
\f{\partial\Gamma^k_{ij}}{\partial u^l}=\f{\partial\Gamma^{k-1}_{ij}}{\partial u^{l-1}},\qquad l>2.
\eeq
\end{lemma}

\noindent
{\it Proof}. It is sufficient to prove the lemma in the case $i=j=2$. Indeed, if $\Gamma^k_{ij}$ vanishes it is easy to check that also  $\Gamma^{k-1}_{ij}$  
 vanishes. If $\Gamma^k_{ij}$ does not vanish then there exists $2\le h\le n+1$ satisfying $i+j-k=4-h$. Moreover  $\Gamma^k_{ij}=\Gamma^h_{22}$
 and $\Gamma^{k-1}_{ij}=\Gamma^{h-1}_{22}$. For $h=2$ we have to prove that
\beq
\f{\partial\Gamma^2_{22}}{\partial u^l}=\f{\partial\Gamma^{1}_{22}}{\partial u^{l-1}},\qquad l>2.
\eeq
The l.h.s. vanishes since $\Gamma^2_{22}$ depends only on $u^2$ while the r.h.s vanishes since $\Gamma^1_{22}=0$. For $h>2$ we have to prove that
\beq
\f{\partial\Gamma^h_{22}}{\partial u^l}=\f{\partial\Gamma^{h-1}_{22}}{\partial u^{l-1}},\qquad l>2.
\eeq
For $h=3$ it is true. Assume that it is true for a given $h\ge 3$. Then using \eqref{mainide} and the inductive hypothesis we get
$$
\f{\partial\Gamma^h_{22}}{\partial u^l}=-\f{1}{u^2}\sum_{s=1}^{h-2}u^{2+s}\f{\partial \Gamma^{h-s}_{22}}{\partial u^l}
=-\f{1}{u^2}\sum_{s=1}^{h-3}u^{2+s}\f{\partial \Gamma^{h-s}_{22}}{\partial u^l}
=-\f{1}{u^2}\sum_{s=1}^{h-3}u^{2+s}\f{\partial \Gamma^{h-s-1}_{22}}{\partial u^{l-1}}
=\f{\partial\Gamma^{h-1}_{22}}{\partial u^{l-1}}.
$$

\endproof

\begin{lemma}
The Christoffel symbols $\Gamma^k_{ij}$ associated with the $n\times n$ Jordan block, defined recursively as explained above starting from the $2\times 2$ Jordan block satisfy the following identities:
\begin{eqnarray}
&&\sum_{k=1}^n\Gamma^i_{jk}u^k=0,\qquad i\ne j\quad {\rm or}\quad i=1\quad{\rm or}\quad j=1\\
&&\sum_{k=1}^n\Gamma^i_{jk}u^k=-\epsilon ,\qquad i=j\ne 1
\end{eqnarray}
\end{lemma}
For $i=1$ or $j=1$ the first identity is trivially satisfied since all the summands vanish. Thus we can assume both indices different from $1$. For $i=j=2$  the
 only term surviving in the sum is $\Gamma^2_{22}u^2$ and the result follows from the definition of  $\Gamma^2_{22}$. For $j=2$ and $i\ne 2$ the result
 follows from the formula \eqref{mainide}. Indeed, we can rewrite the identity
$$
\Gamma^{i}_{22}=-\f{1}{u^2}\sum_{s=1}^{i-2}u^{2+s}\Gamma^{i-s}_{22}
$$
as
$$\sum_{s=2}^{i}u^{s}\Gamma^{i}_{2s}=0.$$
Taking into account that for $s>i$ $\Gamma^{i}_{2s}=0$ and $\Gamma^{i}_{21}=0$ we have
$$\sum_{s=1}^{n}u^{s}\Gamma^{i}_{2s}=0.$$
The cases $i>2$ and $j>2$ can be reduced to the above cases using the identity $\Gamma^i_{jk}=\Gamma^{i-1}_{j-1,k}$. For instance, if $i=j$ applying
 $(i-2)$-times this identity we reduce to the case $i=j=2$.

\endproof 

\begin{lemma}
The connection $\nabla$ associated with the $n\times n$ Jordan block satisfies the condition
\beq
\nabla_jE^i=(1-\epsilon)\delta^i_j+\epsilon e^ie^j.
\label{Lemma7.3J=1}
\eeq
\end{lemma}

\n
\emph{Proof}.  Since $\nabla_jE^i=\delta^i_j+\Gamma^i_{jk}u^k$ the result follows from the previous lemma.

\endproof

\begin{lemma}
The component of $E^{-1}$ are defined recursively by
\beq
(E^{-1})^1=\f{1}{u^1},\qquad (E^{-1})^{n+1}=-\f{1}{u^1}\sum_{s=1}^{n}u^{1+s}(E^{-1})^{n+1-s}
\eeq
\end{lemma}

\n
\emph{Proof}. By definition $E^{-1}\circ E=e$. In canonical coordinates we obtain
\begin{eqnarray*}
(E^{-1})^{i+1-k}u^k=\delta^i_1.
\end{eqnarray*}
In the one component case we obtain $(E^{-1})^1=\f{1}{u^1}$. 
In the $(n+1)$-component case we obtain
\begin{eqnarray*}
&&(E^{-1})^{1}u^1=1\\
&&(E^{-1})^{2}u^1+(E^{-1})^{1}u^2=0\\
&&\qquad\vdots\\
&&(E^{-1})^{n+1}u^1+(E^{-1})^{n}u^2+\cdots+(E^{-1})^{1}u^{n+1}=0.\\
\end{eqnarray*}
The first $n$ equations coincide with the equations for the $n$-component case.
\endproof

\begin{lemma}
The dual connection $\nabla^*$ is defined by
\beq
\Gamma^{*k}_{ij}=\Gamma^{k}_{ij}+(\epsilon-1)(E^{-1})^{k-i-j+2}-\epsilon\,(E^{-1})^1\,\delta^k_1\delta^1_i\delta^1_j
\eeq
where it is understood that $(E^{-1})^{k-i-j+2}=0$ if $k-i-j<-1$.
\end{lemma}

\n
\emph{Proof}. From \eqref{dualfromnatural} and \eqref{Lemma7.3J=1} it follows that
\begin{align}
	\Gamma^{*k}_{ij} =\Gamma^k_{ij}- c^{*l}_{ji}((1-\epsilon)\delta^k_l+\epsilon e^ke^l) 
	=\Gamma^k_{ij}- (1-\epsilon)c^{*k}_{ji}-\epsilon c^{*1}_{ji} \delta^k_1.
	\notag
\end{align}
By definition 
$$c^{*i}_{jk}=c^i_{jl}c^l_{km}(E^{-1})^m=\delta^i_{j+l-1}\delta^l_{k+m-1}(E^{-1})^m=\delta^i_{j+l-1}(E^{-1})^{l+1-k}
=(E^{-1})^{2+i-j-k}.$$
Substituting in $\Gamma^{*k}_{ij}$ we get the result.

\endproof

\subsection{The proof of the theorem}
From the above lemmas it follows that 
\begin{theorem}
The connection $\nabla$ associated with the $n\times n$ Jordan block, the product $\circ$ with structure constants $c^k_{ij}=\delta^k_{i+j-1}$, the
 vector field $e$ of components $e^i=\delta^i_1$, the dual connection defined by the Christoffel symbols $\Gamma^{*k}_{ij}=\Gamma^{k}_{ij}+(\epsilon-1)(E^{-1})^{k-i-j+2}-\epsilon\,(E^{-1})^1\,\delta^k_1\delta^1_i\delta^1_j$, the dual product $*$ with structure constants $c^{*i}_{jk}=(E^{-1})^{2+i-j-k}$ and the vector field of components $E^i=u^i$
 define a unique bi-flat structure.
\end{theorem}

\n
{\it Proof}. 

\begin{itemize}
\item {\bf Step 1: flatness of $\nabla$}. We already know that the connection $\nabla$ is flat for $n=2,3,4,5$. We need to prove that if  
 the connection  $\nabla$ associated with the $n\times n$ Jordan block is flat, that is
\beq
\partial_k\Gamma^i_{hj}-\partial_h\Gamma^i_{kj}-\sum_{l=1}^n(\Gamma^i_{hl}\Gamma^l_{kj}-\Gamma^i_{kl}\Gamma^l_{hj})=0, \qquad i,j,h,k=1,...,n
\eeq 
then also the connection associated with the $(n+1)\times (n+1)$ Jordan block is flat:
\beq
R^i_{hkj}=\partial_k\Gamma^i_{hj}-\partial_h\Gamma^i_{kj}-\sum_{l=1}^{n+1}(\Gamma^i_{hl}\Gamma^l_{kj}-\Gamma^i_{kl}\Gamma^l_{hj})=0,\qquad i,j,h,k=1,...,n+1.
\eeq
Since  $\Gamma^i_{j,n+1}=0$ if $i\ne n+1$ we need only to check the case $i=n+1$. We have $3$ interesting subcases: $h>2,k>2,j>2$, $h=2,k>2,j>2$ and $h=j=2$,$k>2$. In the first case using the lemma and the definition of the Christoffel symbols we have
 (notice that $\Gamma^2_{ij}=0$ if $i$ or $j$ are greater than $2$)
\begin{eqnarray*}
R^{n+1}_{hkj}&=&\partial_{k-1}\Gamma^{n}_{hj}-\partial_{h-1}\Gamma^n_{kj}-\sum_{l=3}^{n}(\Gamma^{n+1}_{hl}\Gamma^l_{kj}-\Gamma^{n+1}_{kl}\Gamma^l_{hj})=\\
&=&\partial_{k-1}\Gamma^{n-1}_{h-1,j}-\partial_{h-1}\Gamma^{n-1}_{k-1,j}-\sum_{l=3}^{n}(\Gamma^{n-1}_{h-1,l-1}\Gamma^{l-1}_{k-1,j}
-\Gamma^{n-1}_{k-1,l-1}\Gamma^{l-1}_{h-1,j})\\
&=&\partial_{k-1}\Gamma^{n-1}_{h-1,j}-\partial_{h-1}\Gamma^{n-1}_{k-1,j}-\sum_{l=1}^{n}(\Gamma^{n-1}_{h-1,l}\Gamma^{l}_{k-1,j}
-\Gamma^{n-1}_{k-1,l-1}\Gamma^{l-1}_{h-1,j})\\
&=&R^{n-1}_{h-1,k-1,j}.
\end{eqnarray*}
The quantity $R^{n-1}_{h-1,k-1,j}$ vanishes by hypothesis if $j=1,...,n$. For $j=n+1$ it vanishes since each term in the sum contains a Christoffel symbol of the form $\Gamma^r_{s,n+1}$ with $r<n+1$. In the second case we have
\begin{eqnarray*}
R^{n+1}_{2kj}&=&\partial_k\Gamma^{n+1}_{2j}-\partial_2\Gamma^{n+1}_{kj}-\sum_{l=3}^{n+1}(\Gamma^{n+1}_{2l}\Gamma^l_{kj}-\Gamma^{n+1}_{kl}\Gamma^l_{2j})=\\
&=&\partial_{k-1}\Gamma^{n}_{2j}-\partial_2\Gamma^{n}_{k-1,j}-\sum_{l=3}^{n+1}(\Gamma^{n}_{2,l-1}\Gamma^{l-1}_{k-1,j}-\Gamma^{n}_{k-1,l}\Gamma^l_{2j})=\\
&=&\partial_{k-1}\Gamma^{n}_{2j}-\partial_2\Gamma^{n}_{k-1,j}-\sum_{l=1}^{n}(\Gamma^{n}_{2l}\Gamma^{l}_{k-1,j}-\Gamma^{n}_{k-1,l}\Gamma^l_{2j})=\\
&=&R^{n}_{2,k-1,j}=0.
\end{eqnarray*}
Finally in the last case we have
\begin{eqnarray*}
R^{n+1}_{2k2}&=&\partial_k\Gamma^{n+1}_{22}-\partial_2\Gamma^{n+1}_{k2}-\sum_{l=3}^{n+1}\Gamma^{n+1}_{2l}\Gamma^l_{k2}+\sum_{l=2}^{n+1}\Gamma^{n+1}_{kl}\Gamma^l_{22}=\\
&=&\partial_{k-1}\Gamma^{n}_{22}-\partial_2\Gamma^{n}_{k-1,2}-\sum_{l=3}^{n+1}\Gamma^{n}_{2,l-1}\Gamma^{l-1}_{k-1,2}+\sum_{l=2}^{n}\Gamma^{n}_{k-1,l}\Gamma^l_{22}=\\
&=&\partial_{k-1}\Gamma^{n}_{22}-\partial_2\Gamma^{n}_{k-1,2}-\sum_{l=1}^{n}\Gamma^{n}_{2l}\Gamma^{l}_{k-1,2}+\sum_{l=1}^{n}\Gamma^{n}_{k-1,l}\Gamma^l_{22}=\\
&=&R^{n}_{2,k-1,2}=0.
\end{eqnarray*}
\item {\bf Step 2: flatness of $e$}. It is equivalent to $\Gamma^i_{j1}=0$.
\item {\bf Step 3: compatibility of $\nabla$ and $\circ$}. In canonical coordinates this means that
$$\Gamma^k_{i+j-1,l}-\Gamma^{k-i+1}_{lj}=\Gamma^k_{l+j-1,i}-\Gamma^{k-l+1}_{ij}.$$
Let us prove this condition by induction. We already know that it is satisfied up to $n=5$. Let us suppose that it is satisfied for a given $n$. We have to prove that 
$$\Gamma^{n+1}_{i+j-1,l}-\Gamma^{n+2-i}_{lj}=\Gamma^{n+1}_{l+j-1,i}-\Gamma^{n+2-l}_{ij},$$
where it is understood that $\Gamma^k_{ij}=0$ if $i,j$ or $k$ are greater than $n+1$.  
If $i=1$ or $l=1$ the above condition is trivially satisfied. If $i,l\ne 1$ we have $\Gamma^{n+1}_{i+j-1,l}=\Gamma^{n+1}_{l+j-1,i}$ (if $i+j\ge n+1$ both sides vanish). If $j=1$ the remaining condition is trivially satisfied. If also $j\ne 1$ the remaining conditon is satisfied since $\Gamma^{n+2-i}_{lj}=\Gamma^{n+1}_{l+i-1,j}$ and $\Gamma^{n+2-l}_{ij}=\Gamma^{n+1}_{i+l-1,j}$ (if $i+l\ge n+2$ both Christoffel symbols vanish). 
\item {\bf Step 4: linearity of the Euler vector field}. We need to prove $\nabla\nabla E=0$. In local coordinates we have:
\begin{eqnarray*}
\nabla_i\nabla_j E^k&=&\Gamma^k_{il} \nabla_j E^l-\Gamma^l_{ik} \nabla_j E^k=\\
&=&\Gamma^k_{il}((1-\epsilon)\delta^l_j+\epsilon e^le^j)-\Gamma^l_{ij}((1-\epsilon)\delta^k_l+\epsilon e^ke^l)=0.
\end{eqnarray*}
\item {\bf Step 5: properties of the dual structure}. The flatness of $\nabla^*$ and the additional properties follows from the linearity of $E$ and from the definition of $\nabla^*$ and $*$.
\item {\bf Step 6: Uniqueness}. The connection $\nabla$  is uniquely determined by the conditions 
\begin{eqnarray*}
\nabla_j e^i&=&\d_je^i+\Gamma^i_{jl}e^l=0\\
(d_{\nabla}V)^k_{ij}&=&\frac{\partial V^k_j}{\partial u^i}+\Gamma^k_{il}V^l_j-\frac{\partial V^k_i}{\partial u^j}-\Gamma^k_{jl}V^l_i=0.
\end{eqnarray*}
Indeed, in the case of a single Jordan block in David-Hertling coordinates the affinor $V$ has the form
\beq
V=\begin{bmatrix}
V^1_1 & 0 & 0 & \dots & 0\cr
u^{2} & V^2_2 & 0 & \dots & 0\cr
u^3 & u^2 & V^3_3 & \dots & 0\cr
\vdots & \vdots & \ddots & \ddots &0 \cr
u^{n} & u^{n-1} & \dots & u^2 &V^n_n
\end{bmatrix},
\eeq

where $V^1_1=V^2_2=...=V^n_n=(1-n\epsilon)u_1$. Denoting by $\dots$ terms which do not contain any Christoffel symbols
 we have that : 
 \newline
 \newline
- The vanishing of $(d_{\nabla} V)^i_{n,n-1}$ defines uniquely $\Gamma^i_{nn}$.
\newline
\newline
- Using the previous condition the vanishing of $(d_{\nabla} V)^i_{n-2,n}$ defines uniquely $\Gamma^i_{n-1,n}$.
\newline
\newline
- More in general using the previous conditions the vanishing of $(d_{\nabla} V)^i_{k,n}$ defines uniquely $\Gamma^i_{k+1,n}$.
\newline
 \newline
- Similarly the vanishing of $(d_{\nabla} V)^i_{n-2,n-1}$ defines uniquely $\Gamma^i_{n-1,n-1}$ and the vanishing of
 $(d_{\nabla} V)^i_{k,n-1}$ defines uniquely $\Gamma^i_{k+1,n-1}$.
 \newline
 \newline
In general,  the vanishing of $(d_{\nabla} V)^i_{n-j-1,n-j}$ defines uniquely $\Gamma^i_{n-j,n-j}$ 
 and the vanishing of $(d_{\nabla} V)^i_{k,n-j}$ defines uniquely $\Gamma^i_{k+1,n-j}$ taking into account all the previous conditions
  starting from $j=n-1,k=n$. In this way get
  all the Christoffel symbols apart from $\Gamma^i_{1j}=\Gamma^i_{j1}$ which vanish due to the condition $\nabla e=0$. This means that the connection
   constructed above is unique.
\end{itemize}

\section{The case of an arbitrary number of  Jordan blocks}
Theorem $7.6$ can be extended to the general case where the operator $L=E\,\circ$ has a block diagonal form. In order to do so, the crucial Lemmas $7.1$-$7.5$ must be suitably extended too and some new preliminar results must be taken into account.

Let $(M,\circ,e,E)$ be a regular F-manifold of dimension $n\geq2$ with an Euler vector field $E$ of weight one. Around a point $p\in M$, let the canonical form of the operator $L=E\,\circ$ have $r$ Jordan blocks $L_1,\dots,L_r$ of sizes $m_1,\dots,m_r$ with distinct eigenvalues. Let us define $a_0=\overset{r}{\underset{\alpha=1}{\sum}}\epsilon_\alpha\,\Tr(L_\alpha)$. Our final goal is to prove that for any choice of $\epsilon_1,\dots,\epsilon_r$ there exists a unique regular bi-flat F-structure with non-semisimple canonical coordinates such that $L=E\,\circ$ and $d_\nabla(L-a_0\,I)=0$.

Any set of coordinates $u^1,\dots,u^n$ for $M$ can be re-labelled by means of the following notation: for each $\alpha\in\{2,\dots,r\}$ and for each $j\in\{1,\dots,m_\alpha\}$ we write
\begin{equation}
	\label{relabellingcoordinates}
	j(\alpha)=m_1+\dots+m_{\alpha-1}+j
\end{equation}
(for $\alpha=1$ we set $j(\alpha)=j$) so that $u^{j(\alpha)}$ denotes the $j$-th coordinate associated to the $\alpha$-th Jordan block. From now on, we will write $u^i$ when seeing the coordinate as running from $1$ to the dimension of the manifold and we will write $u^{i(\alpha)}$ when in need to highlight the Jordan block to which the coordinate refers. According to this notation, $\partial_i$ and $\partial_{i(\alpha)}$ will denote the partial derivative with respect to $u^i$ and $u^{i(\alpha)}$ respectively.

Let us recall that in these coordinates we have
\begin{itemize}
	\item[$\bullet$] $e=\overset{r}{\underset{\alpha=1}{\sum}}\,\partial_{1(\alpha)}$;
	\item[$\bullet$] $E=\overset{n}{\underset{s=1}{\sum}}\,u^s\,\partial_{s}=\overset{r}{\underset{\sigma=1}{\sum}}\,\overset{m_\sigma}{\underset{s=1}{\sum}}\,u^{s(\sigma)}\,\partial_{s(\sigma)}$;
	\item[$\bullet$] $c^{k(\gamma)}_{i(\alpha)j(\beta)}=\delta^\gamma_\alpha\delta^\gamma_\beta\delta^k_{i+j-1}$ for any $\alpha,\beta,\gamma\in\{1,\dots,r\}$, $i\in\{1,\dots,m_\alpha\}$, $j\in\{1,\dots,m_\beta\}$, $k\in\{1,\dots,m_\gamma\}$.
\end{itemize}

Since we seek a bi-flat F-manifold structure, we reasonably start by assuming $\nabla e=0$ and $\nabla_ic^l_{jk}=\nabla_jc^l_{ik}$ for every $i,j,k,l\in\{1,\dots,n\}$.

\begin{remark}
	Due to regularity condition we are implicitly assuming that $u^{2(\alpha)}\ne 0$ and $u^{1(\alpha)}\ne u^{1(\beta)}$ if $\alpha\ne\beta$.
\end{remark}

\subsection{The Christoffel symbols}
\begin{prop}
	For every $\alpha,\beta\in\{1,\dots,r\}$ and $i\in\{1,\dots,m_\alpha\}$, $j\in\{1,\dots,m_\beta\}$ we have
	\begin{equation}
		\label{nablae=0}
		\overset{r}{\underset{\sigma=1}{\sum}}\,\Gamma^{i(\alpha)}_{1(\sigma)j(\beta)}=0.
	\end{equation}
\end{prop}
\begin{proof}
	By writing $\nabla e=0$ in canonical coordiantes we get
	\begin{align}
		\notag
		0&=\nabla_{j(\beta)}e^{i(\alpha)}=\partial_{j(\beta)}e^{i(\alpha)}+\overset{r}{\underset{\sigma=1}{\sum}}\,\overset{m_\sigma}{\underset{k=1}{\sum}}\,\Gamma^{i(\alpha)}_{j(\beta)k(\sigma)}e^{k(\sigma)}=\overset{r}{\underset{\sigma=1}{\sum}}\,\Gamma^{i(\alpha)}_{j(\beta)1(\sigma)}
	\end{align}
	as $e^{k(\sigma)}=\delta^k_1$, for every $\alpha,\beta\in\{1,\dots,r\}$ and $i\in\{1,\dots,m_\alpha\}$, $j\in\{1,\dots,m_\beta\}$.
\end{proof}
\begin{prop}
	Let $\alpha$, $\beta$, $\gamma$ be pairwise distinct. The following relations hold:
	\begin{itemize}
		\item[$(i)$] $\Gamma^{l(\beta)}_{i(\alpha)(j+k)(\alpha)}=\Gamma^{l(\beta)}_{(i+k)(\alpha)j(\alpha)}$;
		\item[$(ii)$] $\Gamma^{k(\beta)}_{i(\alpha)j(\beta)}=\Gamma^{l(\beta)}_{i(\alpha)m(\beta)}$ when $k-j=l-m$;
		\item[$(iii)$] $\Gamma^{k(\alpha)}_{(i-1)(\alpha)j(\alpha)}=\Gamma^{k(\alpha)}_{i(\alpha)(j-1)(\alpha)}$ when the lower indices are both different from $1$ and not simultaneously equal to $2$;
		\item[$(iv)$] $\Gamma^{k(\alpha)}_{i(\alpha)j(\alpha)}=\Gamma^{(k+1)(\alpha)}_{i(\alpha)(j+1)(\alpha)}$ when $j\neq1$;
		\item[$(v)$] $\Gamma^{k(\alpha)}_{i(\alpha)j(\beta)}=-\Gamma^{(k-i+1)(\alpha)}_{1(\beta)j(\beta)}$;
		\item[$(vi)$] $\Gamma^{k(\gamma)}_{i(\alpha)j(\beta)}=0$.
	\end{itemize}
	The above quantities must be considered when the indices do not exceed the size of the corresponding Jordan block.
\end{prop}
\begin{proof}
	The condition $\nabla_ic^l_{jk}=\nabla_jc^l_{ik}$ $\forall i,j,k,l\in\{1,\dots,n\}$ can be rewritten as
	\begin{equation}
		\nabla_{i(\alpha)}c^{l(\delta)}_{j(\beta)k(\gamma)}=\nabla_{j(\beta)}c^{l(\delta)}_{i(\alpha)k(\gamma)}\label{nablasymm}
	\end{equation}
	for every $\alpha,\beta,\gamma,\delta\in{1,\dots,r}$ and $i\in\{1,\dots,m_\alpha\}$, $j\in\{1,\dots,m_\beta\}$, $k\in\{1,\dots,m_\gamma\}$, $l\in\{1,\dots,m_\delta\}$. This yields
	\begin{align}
		0&=\nabla_{i(\alpha)}c^{l(\delta)}_{j(\beta)k(\gamma)}-\nabla_{j(\beta)}c^{l(\delta)}_{i(\alpha)k(\gamma)}\notag\\&=\xcancel{\partial_{i(\alpha)}c^{l(\delta)}_{j(\beta)k(\gamma)}}+\Gamma^{l(\delta)}_{i(\alpha)s(\sigma)}c^{s(\sigma)}_{j(\beta)k(\gamma)}\cancel{-\Gamma^{s(\sigma)}_{i(\alpha)j(\beta)}c^{l(\delta)}_{s(\sigma)k(\gamma)}}-\Gamma^{s(\sigma)}_{i(\alpha)k(\gamma)}c^{l(\delta)}_{j(\beta)s(\sigma)}\notag\\&\xcancel{-\partial_{j(\beta)}c^{l(\delta)}_{i(\alpha)k(\gamma)}}-\Gamma^{l(\delta)}_{j(\beta)s(\sigma)}c^{s(\sigma)}_{i(\alpha)k(\gamma)}\cancel{+\Gamma^{s(\sigma)}_{j(\beta)i(\alpha)}c^{l(\delta)}_{s(\sigma)k(\gamma)}}+\Gamma^{s(\sigma)}_{j(\beta)k(\gamma)}c^{l(\delta)}_{i(\alpha)s(\sigma)}\notag\\&=\delta_{\beta\gamma}\Gamma^{l(\delta)}_{i(\alpha)(j+k-1)(\beta)}-\delta^\delta_\beta\Gamma^{(l-j+1)(\beta)}_{i(\alpha)k(\gamma)}-\delta_{\alpha\gamma}\Gamma^{l(\delta)}_{j(\beta)(i+k-1)(\alpha)}+\delta^\delta_\alpha\Gamma^{(l-i+1)(\alpha)}_{j(\beta)k(\gamma)}\notag
	\end{align}
	where we used the Einstein convention when summing both over the greek indices labelling the Jordan blocks and over the latin indices describing the inner coordinates. Therefore we got
	\begin{equation}
		\delta_{\beta\gamma}\Gamma^{l(\delta)}_{i(\alpha)(j+k-1)(\beta)}-\delta^\delta_\beta\Gamma^{(l-j+1)(\beta)}_{i(\alpha)k(\gamma)}-\delta_{\alpha\gamma}\Gamma^{l(\delta)}_{j(\beta)(i+k-1)(\alpha)}+\delta^\delta_\alpha\Gamma^{(l-i+1)(\alpha)}_{j(\beta)k(\gamma)}=0\label{prop2}
	\end{equation}
	for all suitable indices. Since condition \eqref{nablasymm} is symmetric with respect to the exhange of $\alpha$ and $\beta$, the only relevant cases are the following:
	\begin{itemize}
		\item[$1.$] $\alpha=\beta=\gamma=\delta$
		\item[$2.$] $\alpha=\beta=\gamma\neq\delta$
		\item[$3.$] $\alpha=\beta=\delta\neq\gamma$
		\item[$4.$] $\alpha=\delta=\gamma\neq\beta$
		\item[$5.$] $\alpha=\beta\neq\gamma=\delta$
		\item[$6.$] $\alpha=\gamma\neq\beta=\delta$
		\item[$7.$] $\alpha=\beta\neq\gamma\neq\delta\neq\alpha$
		\item[$8.$] $\alpha=\gamma\neq\beta\neq\delta\neq\alpha$
		\item[$9.$] $\alpha=\delta\neq\gamma\neq\beta\neq\alpha$
		\item[$10.$] $\gamma=\delta\neq\alpha\neq\beta\neq\gamma$
		\item[$11.$] $\alpha$, $\beta$, $\gamma$, $\delta$ are pairwise distinct.
	\end{itemize}
	\textbf{Case 1: $\alpha=\beta=\gamma=\delta$.} Condition \eqref{prop2} reads
	\begin{align}
		\notag
		\Gamma^{l(\alpha)}_{i(\alpha)(j+k-1)(\alpha)}-\Gamma^{(l-j+1)(\alpha)}_{i(\alpha)k(\alpha)}-\Gamma^{l(\alpha)}_{j(\alpha)(i+k-1)(\alpha)}+\Gamma^{(l-i+1)(\alpha)}_{j(\alpha)k(\alpha)}=0.
	\end{align}
	For $k=0$ we get $\Gamma^{l(\alpha)}_{i(\alpha)(j-1)(\alpha)}-\Gamma^{l(\alpha)}_{j(\alpha)(i-1)(\alpha)}=0$ where the indexed make sense for $i,j>1$, which proves $(iii)$. For $j=2$ we get
	\begin{align}
		0&=\cancel{\Gamma^{l(\alpha)}_{i(\alpha)(k+1)(\alpha)}}-\Gamma^{(l-1)(\alpha)}_{i(\alpha)k(\alpha)}\cancel{-\Gamma^{l(\alpha)}_{2(\alpha)(i+k-1)(\alpha)}}+\Gamma^{(l-i+1)(\alpha)}_{2(\alpha)k(\alpha)}\notag\\&=-\Gamma^{(l-1)(\alpha)}_{i(\alpha)k(\alpha)}+\Gamma^{(l-i+1)(\alpha)}_{2(\alpha)k(\alpha)}
		\notag
	\end{align}
	for $i\neq1$ (otherwise we couldn't apply $(iii)$ to cancel out the two terms in the first line), which proves $(iv)$.
	\\\textbf{Case 2: $\alpha=\beta=\gamma\neq\delta$.} Condition \eqref{prop2} reads
	\begin{equation}
		\Gamma^{l(\delta)}_{i(\alpha)(j+k-1)(\alpha)}-\Gamma^{l(\delta)}_{j(\alpha)(i+k-1)(\alpha)}=0
	\end{equation}
	which proves $(i)$.
	\\\textbf{Case 3: $\alpha=\beta=\delta\neq\gamma$.} Condition \eqref{prop2} reads
	\begin{equation}
		-\Gamma^{(l-j+1)(\alpha)}_{i(\alpha)k(\gamma)}+\Gamma^{(l-i+1)(\alpha)}_{j(\alpha)k(\gamma)}=0
	\end{equation}
	which proves $(ii)$.
	\\\textbf{Case 6: $\alpha=\gamma\neq\beta=\delta$.} Condition \eqref{prop2} reads
	\begin{equation}
		-\Gamma^{(l-j+1)(\beta)}_{i(\alpha)k(\alpha)}-\Gamma^{l(\beta)}_{j(\beta)(i+k-1)(\alpha)}=0
	\end{equation}
	which (by relabelling the indices) proves $(v)$.
	\\\textbf{Case 9: $\alpha=\delta\neq\gamma\neq\beta\neq\alpha$.} Condition \eqref{prop2} reads
	\begin{equation}
		\Gamma^{(l-i+1)(\alpha)}_{j(\beta)k(\gamma)}=0
	\end{equation}
	which proves $(vi)$.
	\\In all the remaining cases condition \eqref{prop2} is trivially satisfied.
\end{proof}
Coherently with our final goal, let us furthermore assume
\begin{equation}
	\label{dnablaoriginale}
	d_\nabla(L-a_0\,I)=0.
\end{equation}
This yields
\begin{align}
	0&=\big(d_\nabla(L-a_0\,I)\big)^{i(\alpha)}_{j(\beta)k(\gamma)}\notag\\&=\partial_{j(\beta)}(L-a_0\,I)^{i(\alpha)}_{k(\gamma)}-\partial_{k(\gamma)}(L-a_0\,I)^{i(\alpha)}_{j(\beta)}\notag\\&+\Gamma^{i(\alpha)}_{j(\beta)l(\delta)}(L-a_0\,I)^{l(\delta)}_{k(\gamma)}-\Gamma^{i(\alpha)}_{k(\gamma)l(\delta)}(L-a_0\,I)^{l(\delta)}_{j(\beta)}\notag\\&=\cancel{\delta^\alpha_\gamma\delta^\alpha_\beta\delta^{i-k+1}_j}-\delta^\alpha_\gamma\delta^i_k m_\beta\epsilon_\beta\delta^1_j\cancel{-\delta^\alpha_\beta\delta^\alpha_\gamma\delta^{i-j+1}_k}+\delta^\alpha_\beta\delta^i_j m_\gamma\epsilon_\gamma\delta^1_k\notag\\&+\overset{r}{\underset{\delta=1}{\sum}}\,\overset{m_\delta}{\underset{l=1}{\sum}}\,\Gamma^{i(\alpha)}_{j(\beta)l(\delta)}\delta_{\delta\gamma}u^{(l-k+1)(\gamma)}\mathds{1}_{\{l\geq k\}}\notag\\&-\overset{r}{\underset{\delta=1}{\sum}}\,\overset{m_\delta}{\underset{l=1}{\sum}}\,\Gamma^{i(\alpha)}_{k(\gamma)l(\delta)}\delta_{\delta\beta}u^{(l-j+1)(\beta)}\mathds{1}_{\{l\geq j\}}\notag\\&=\delta^\alpha_\beta\delta^i_j m_\gamma\epsilon_\gamma\delta^1_k-\delta^\alpha_\gamma\delta^i_k m_\beta\epsilon_\beta\delta^1_j\notag\\&+\overset{m_\gamma}{\underset{l=k}{\sum}}\,\Gamma^{i(\alpha)}_{j(\beta)l(\gamma)}u^{(l-k+1)(\gamma)}-\overset{m_\beta}{\underset{l=j}{\sum}}\,\Gamma^{i(\alpha)}_{k(\gamma)l(\beta)}u^{(l-j+1)(\beta)}
	\notag
\end{align}
as
\begin{align}
	(L-a_0\,I)^{a(\eta)}_{b(\mu)}&=L^{a(\eta)}_{b(\mu)}-a_0\,\delta^{a(\eta)}_{b(\mu)}\notag\\&=\delta^{\eta}_{\mu}\,\overset{m_\eta}{\underset{s=1}{\sum}}\,u^{s(\eta)}\,\delta^{a+b-1}_{s}-\delta^{\eta}_{\mu}\,\overset{r}{\underset{\alpha=1}{\sum}}\,m_\alpha\epsilon_\alpha u^{1(\alpha)}.
	\notag
\end{align}
Therefore \eqref{dnablaoriginale} amounts to
\begin{align}
	\label{dnabla0}
	&\delta^\alpha_\beta\delta^i_j m_\gamma\epsilon_\gamma\delta^1_k-\delta^\alpha_\gamma\delta^i_k m_\beta\epsilon_\beta\delta^1_j+\overset{m_\gamma}{\underset{l=k}{\sum}}\,\Gamma^{i(\alpha)}_{j(\beta)l(\gamma)}u^{(l-k+1)(\gamma)}-\overset{m_\beta}{\underset{l=j}{\sum}}\,\Gamma^{i(\alpha)}_{k(\gamma)l(\beta)}u^{(l-j+1)(\beta)}=0.
\end{align}

\begin{prop}
	Let $\alpha$, $\beta$, $\gamma$ be pairwise distinct. Then
	\begin{itemize}
		\item[$1.$] $\Gamma^{k(\alpha)}_{i(\alpha)j(\alpha)}=\begin{cases}
			\Gamma^{(k-i-j+4)(\alpha)}_{2(\alpha)2(\alpha)}\qquad&\text{if}\,\,\,k-i-j\geq-2\\
			0&\text{if}\,\,\,k-i-j\leq-3
		\end{cases}\qquad$when $i,j>1$
		\item[$2.$] $\Gamma^{2(\alpha)}_{2(\alpha)2(\alpha)}=-\frac{m_\alpha\epsilon_{\alpha}}{u^{2(\alpha)}}$
		\item[$3.$] $\Gamma^{1(\alpha)}_{1(\alpha)j(\alpha)}=0\quad$when or $j\geq2$
		\item[$4.$] $\Gamma^{k(\alpha)}_{i(\alpha)j(\beta)}=0\quad$when $k<i$ or $j\geq2$
		\item[$5.$] $\Gamma^{i(\alpha)}_{i(\alpha)1(\beta)}=\frac{m_\beta\epsilon_\beta}{u^{1(\alpha)}-u^{1(\beta)}}$
		\item[$6.$] $\Gamma^{(i+h)(\alpha)}_{i(\alpha)1(\beta)}=-\frac{1}{u^{1(\alpha)}-u^{1(\beta)}}\,\overset{h+1}{\underset{s=2}{\sum}}\,\Gamma^{(i+h-s+1)(\alpha)}_{i(\alpha)1(\beta)}\,u^{s(\alpha)}\quad$for $h\geq1$
		\item[$7.$] $\Gamma^{n(\alpha)}_{2(\alpha)2(\alpha)}=\Gamma^{(n-2)(\alpha)}_{1(\alpha)1(\alpha)}-\Gamma^{2(\alpha)}_{2(\alpha)2(\alpha)}\frac{u^{n(\alpha)}}{u^{2(\alpha)}}-\frac{1}{u^{2(\alpha)}}\overset{n-3}{\underset{l=1}{\sum}}\big(\Gamma^{(l+2)(\alpha)}_{2(\alpha)2(\alpha)}-\Gamma^{l(\alpha)}_{1(\alpha)1(\alpha)}\big)u^{(n-l)(\alpha)}$$\qquad$ for $n\geq3$\footnote{The last summation is not to be considered for $n=2,3$.}.
	\end{itemize}
\end{prop}
\begin{proof}
	By virtue of Proposition 8.2, for suitable values of $s$ and $t$ we can write
	\begin{equation}
		\notag
		\Gamma^{k(\alpha)}_{i(\alpha)j(\alpha)}=\Gamma^{(k+s+t)(\alpha)}_{(i+s)(\alpha)(j+t)(\alpha)}.
	\end{equation}
	In particular, for $s=2-i$ and $t=2-j$ we get
	\begin{equation}
		\notag
		\Gamma^{k(\alpha)}_{i(\alpha)j(\alpha)}=\Gamma^{(k-i-j+4)(\alpha)}_{2(\alpha)2(\alpha)}
	\end{equation}
	which proves $(1.)$ (for $k-i-j\leq-3$ the quantity vanishes automatically as the upper index $k-i-j-4\leq0$).
	
	By taking $\alpha=\beta=\gamma$ in \eqref{dnabla0} we get
	\begin{align}
		m_\alpha\epsilon_\alpha(\delta^i_j\delta^1_k-\delta^i_k \delta^1_j)+\overset{m_\alpha}{\underset{l=k}{\sum}}\,\Gamma^{i(\alpha)}_{j(\alpha)l(\alpha)}u^{(l-k+1)(\alpha)}-\overset{m_\alpha}{\underset{l=j}{\sum}}\,\Gamma^{i(\alpha)}_{k(\alpha)l(\alpha)}u^{(l-j+1)(\alpha)}&=0
		\label{abc0}
	\end{align}
	which, for $k=1$ and $i=j=m_\alpha$, gives
	\begin{align}
		0&=m_\alpha\epsilon_\alpha+\overset{m_\alpha}{\underset{l=1}{\sum}}\,\Gamma^{m_\alpha(\alpha)}_{m_\alpha(\alpha)l(\alpha)}u^{l(\alpha)}-\Gamma^{m_\alpha(\alpha)}_{1(\alpha)m_\alpha(\alpha)}u^{1(\alpha)}\notag\\&=m_\alpha\epsilon_\alpha\cancel{+\Gamma^{m_\alpha(\alpha)}_{m_\alpha(\alpha)1(\alpha)}u^{1(\alpha)}}+\overset{m_\alpha}{\underset{l=2}{\sum}}\,\Gamma^{m_\alpha(\alpha)}_{m_\alpha(\alpha)l(\alpha)}u^{l(\alpha)}\cancel{-\Gamma^{m_\alpha(\alpha)}_{1(\alpha)m_\alpha(\alpha)}u^{1(\alpha)}}\notag\\&\overset{1.}{=}m_\alpha\epsilon_\alpha+\overset{m_\alpha}{\underset{l=2}{\sum}}\,\Gamma^{(4-l)(\alpha)}_{2(\alpha)2(\alpha)}u^{l(\alpha)}\mathds{1}_{\{l\leq2\}}=m_\alpha\epsilon_\alpha+\Gamma^{2(\alpha)}_{2(\alpha)2(\alpha)}u^{2(\alpha)}
		\notag
	\end{align}
	which proves $(2.)$.
	
	By taking $k=1$ and $i=j\geq2$ in \eqref{abc0} we get
	\begin{align}
		0&=m_\alpha\epsilon_\alpha+\overset{m_\alpha}{\underset{l=1}{\sum}}\,\Gamma^{i(\alpha)}_{i(\alpha)l(\alpha)}u^{l(\alpha)}-\overset{m_\alpha}{\underset{l=i}{\sum}}\,\Gamma^{i(\alpha)}_{1(\alpha)l(\alpha)}u^{(l-i+1)(\alpha)}\notag\\&=m_\alpha\epsilon_\alpha+\overset{m_\alpha}{\underset{l=2}{\sum}}\,\Gamma^{i(\alpha)}_{i(\alpha)l(\alpha)}u^{l(\alpha)}-\overset{m_\alpha}{\underset{l=i+1}{\sum}}\,\Gamma^{i(\alpha)}_{1(\alpha)l(\alpha)}u^{(l-i+1)(\alpha)}
		\notag
	\end{align}
	where in the first summation all terms vanish but the one for $l=2$, by virtue of (1.), as $i-i-l\leq-3$ for each $l\geq3$. Therefore we have
	\begin{align}
		0&=m_\alpha\epsilon_\alpha+\Gamma^{i(\alpha)}_{i(\alpha)2(\alpha)}u^{2(\alpha)}-\overset{m_\alpha}{\underset{l=i+1}{\sum}}\,\Gamma^{i(\alpha)}_{1(\alpha)l(\alpha)}u^{(l-i+1)(\alpha)}\notag\\&\overset{1.}{=}m_\alpha\epsilon_\alpha+\Gamma^{2(\alpha)}_{2(\alpha)2(\alpha)}u^{2(\alpha)}-\overset{m_\alpha}{\underset{l=i+1}{\sum}}\,\Gamma^{i(\alpha)}_{1(\alpha)l(\alpha)}u^{(l-i+1)(\alpha)}\notag\\&\overset{2.}{=}\cancel{m_\alpha\epsilon_\alpha}+\cancel{\Gamma^{2(\alpha)}_{2(\alpha)2(\alpha)}u^{2(\alpha)}}-\overset{m_\alpha}{\underset{l=i+1}{\sum}}\,\Gamma^{i(\alpha)}_{1(\alpha)l(\alpha)}u^{(l-i+1)(\alpha)}
		\notag
	\end{align}
	that is
	\begin{align}
		\label{8.3_3}
		\overset{m_\alpha}{\underset{l=i+1}{\sum}}\,\Gamma^{i(\alpha)}_{1(\alpha)l(\alpha)}u^{(l-i+1)(\alpha)}=0.
	\end{align}
	For $i=m_\alpha-1$ we get $\Gamma^{(m_\alpha-1)(\alpha)}_{1(\alpha)m_\alpha(\alpha)}u^{2(\alpha)}=0$, thus $\Gamma^{(m_\alpha-1)(\alpha)}_{1(\alpha)m_\alpha(\alpha)}=0$. Let us prove, by induction, that $\Gamma^{(m_\alpha-s)(\alpha)}_{1(\alpha)m_\alpha(\alpha)}=0$ for every integer $s\geq1$. Given an integer $h\geq2$, let us suppose it holds for each $s\leq h-1$. By taking $i=m_\alpha-h$ in \eqref{8.3_3} we get
	\begin{align}
		\notag
		\overset{m_\alpha}{\underset{l=m_\alpha-h+1}{\sum}}\,\Gamma^{(m_\alpha-h)(\alpha)}_{1(\alpha)l(\alpha)}u^{(l-m_\alpha+h+1)(\alpha)}=0
	\end{align}
	that, by virtue of Proposition 8.2, is
	\begin{align}
		\notag
		\overset{m_\alpha}{\underset{l=m_\alpha-h+1}{\sum}}\,\Gamma^{(2m_\alpha-h-l)(\alpha)}_{1(\alpha)m_\alpha(\alpha)}u^{(l-m_\alpha+h+1)(\alpha)}=0
	\end{align}
	where the only term surviving in the summation is the one for $l=m_\alpha$, as for each $l\leq m_\alpha-1$ we have $\Gamma^{(2m_\alpha-h-l)(\alpha)}_{1(\alpha)m_\alpha(\alpha)}=\Gamma^{(m_\alpha-s)(\alpha)}_{1(\alpha)m_\alpha(\alpha)}$ for $s=h+l-m_\alpha\leq h-1$. Therefore we proved
	\begin{align}
		\label{8.3u}
		\Gamma^{k(\alpha)}_{1(\alpha)j(\alpha)}=0\qquad\text{for }k<j.
	\end{align}
	In particular, $\Gamma^{1(\alpha)}_{1(\alpha)j(\alpha)}=0$ for each $j\geq2$, which proves $(3.)$.
	
	By virtue of Proposition 8.2 we have $\Gamma^{k(\alpha)}_{i(\alpha)j(\beta)}=-\Gamma^{(k-i+1)(\alpha)}_{1(\beta)j(\beta)}$ which trivially vanishes whenever $k<i$. In order to complete the proof of point $(4.)$ we then have to show that $\Gamma^{k(\alpha)}_{i(\alpha)j(\beta)}=0$, given $k\geq i$ and $j\geq2$. By taking $\beta\neq\gamma=\alpha$ in \eqref{dnabla0} we get
	\begin{align}
		\label{bneqca0}
		&-\delta^i_k m_\beta\epsilon_\beta\delta^1_j+\overset{\cancelto{i}{m_\alpha}}{\underset{l=k}{\sum}}\,\Gamma^{i(\alpha)}_{j(\beta)l(\alpha)}u^{(l-k+1)(\alpha)}-\overset{m_\beta}{\underset{l=j}{\sum}}\,\Gamma^{i(\alpha)}_{k(\alpha)l(\beta)}u^{(l-j+1)(\beta)}=0
	\end{align}
	which, by taking $j\geq2$, yields
	\begin{align}
		\label{dnabla0_8.3_4}
		&\overset{i}{\underset{l=k}{\sum}}\,\Gamma^{i(\alpha)}_{j(\beta)l(\alpha)}u^{(l-k+1)(\alpha)}-\overset{m_\beta}{\underset{l=j}{\sum}}\,\Gamma^{i(\alpha)}_{k(\alpha)l(\beta)}u^{(l-j+1)(\beta)}=0.
	\end{align}
	We are going to prove that
	\begin{align}
		\label{8.3_4thind}
		\Gamma^{i(\alpha)}_{j(\beta)k(\alpha)}=0\qquad\forall i\geq k
	\end{align}
	by a double procedure of induction: an external one over $i$ and a nested one over $j$. Let us first take $j=m_\beta$ and show
	\begin{align}
		\label{8.16induzesterna_1}
		\Gamma^{i(\alpha)}_{m_\beta(\beta)k(\alpha)}=0\qquad\forall\,i\geq k.
	\end{align}
	By taking $i=k$ and $j=m_\beta$ in \eqref{dnabla0_8.3_4} we get
	\begin{align}
		\notag
		&\Gamma^{k(\alpha)}_{m_\beta(\beta)k(\alpha)}\big(u^{1(\alpha)}-u^{1(\beta)}\big)=0
	\end{align}
	which yields $\Gamma^{k(\alpha)}_{m_\beta(\beta)k(\alpha)}=0$. Given an integer $p\geq1$, let us suppose
	\begin{align}
		\label{8.16induzesterna_2}
		\Gamma^{(k+s)(\alpha)}_{m_\beta(\beta)k(\alpha)}=0\qquad\forall\,s\leq p-1
	\end{align}
	and show $\Gamma^{(k+p)(\alpha)}_{m_\beta(\beta)k(\alpha)}=0$. By taking $i=k+p$ and $j=m_\beta$ in \eqref{dnabla0_8.3_4} we get
	\begin{align}
		\notag
		&\overset{k+p}{\underset{l=k}{\sum}}\,\Gamma^{(k+p)(\alpha)}_{m_\beta(\beta)l(\alpha)}u^{(l-k+1)(\alpha)}-\Gamma^{(k+p)(\alpha)}_{k(\alpha)m_\beta(\beta)}u^{1(\beta)}=0
	\end{align}
	where in the summation only the term for $l=k$ survives, as 
	\begin{align}
		\Gamma^{(k+p)(\alpha)}_{m_\beta(\beta)l(\alpha)}&\overset{Prop.8.2}{=}\Gamma^{(2k+p-l)(\alpha)}_{m_\beta(\beta)k(\alpha)}\overset{\eqref{8.16induzesterna_2}}{=}0
		\notag
	\end{align}
	for each $l\geq k+1$. This implies
	\begin{align}
		\notag
		&\Gamma^{(k+p)(\alpha)}_{k(\alpha)m_\beta(\beta)}(u^{1(\alpha)}-u^{1(\beta)})=0
	\end{align}
	thus $\Gamma^{(k+p)(\alpha)}_{k(\alpha)m_\beta(\beta)}=0$, proving \eqref{8.16induzesterna_1}. Given an integer $q\geq1$, let us now suppose that
	\begin{align}
		\label{8.16induzesterna_3}
		\Gamma^{i(\alpha)}_{(m_\beta-t)(\beta)k(\alpha)}=0\qquad\forall\,i\geq k,\,\forall\,t\leq q-1
	\end{align}
	and show
	\begin{align}
		\label{8.16induzesterna_4}
		\Gamma^{i(\alpha)}_{(m_\beta-q)(\beta)k(\alpha)}=0\qquad\forall\,i\geq k.
	\end{align}
	By taking $i=k$ and $j=m_\beta-q$ in \eqref{dnabla0_8.3_4} we get
	\begin{align}
		&\Gamma^{k(\alpha)}_{(m_\beta-q)(\beta)k(\alpha)}u^{1(\alpha)}-\overset{m_\beta}{\underset{l=m_\beta-q}{\sum}}\,\Gamma^{k(\alpha)}_{k(\alpha)(m_\beta-q)(\beta)}u^{(l-m_\beta+q+1)(\beta)}=0
		\notag
	\end{align}
	where in the summation only the term for $l=m_\beta-q$ survives, as
	\begin{align}
		\Gamma^{k(\alpha)}_{k(\alpha)l(\beta)}&\overset{\eqref{8.16induzesterna_3}}{=}0
		\notag
	\end{align}
	for each $l\geq m_\beta-q+1$. This yields
	\begin{align}
		&\Gamma^{k(\alpha)}_{(m_\beta-q)(\beta)k(\alpha)}(u^{1(\alpha)}-u^{1(\beta)})=0
		\notag
	\end{align}
	thus $\Gamma^{k(\alpha)}_{(m_\beta-q)(\beta)k(\alpha)}=0$, proving \eqref{8.16induzesterna_4} for $i=k$. Given an integer $p\geq1$, let us suppose \eqref{8.16induzesterna_4} to hold whenever $i\leq k+p-1$, that is
	\begin{align}
		\label{8.16induzesterna_5}
		\Gamma^{(k+s)(\alpha)}_{(m_\beta-q)(\beta)k(\alpha)}=0\qquad\forall\,s\leq p-1.
	\end{align}
	We want to show that $\Gamma^{(k+p)(\alpha)}_{(m_\beta-q)(\beta)k(\alpha)}=0$. By taking $i=k+p$ and $j=m_\beta-q$ in \eqref{dnabla0_8.3_4} we get
	\begin{align}
		\notag
		&\overset{k+p}{\underset{l=k}{\sum}}\,\Gamma^{(k+p)(\alpha)}_{(m_\beta-q)(\beta)l(\alpha)}u^{(l-k+1)(\alpha)}-\overset{m_\beta}{\underset{l=m_\beta-q}{\sum}}\,\Gamma^{(k+p)(\alpha)}_{k(\alpha)l(\beta)}u^{(l-m_\beta+q+1)(\beta)}=0
	\end{align}
	where in the first summation only the term for $l=k$ survives, as 
	\begin{align}
		\Gamma^{(k+p)(\alpha)}_{(m_\beta-q)(\beta)l(\alpha)}&\overset{Prop.8.2}{=}\Gamma^{(2k+p-l)(\alpha)}_{(m_\beta-q)(\beta)k(\alpha)}\overset{\eqref{8.16induzesterna_5}}{=}0
		\notag
	\end{align}
	for each $l\geq k+1$, and in the second summation only the term for $l=m_\beta-q$ survives, as
	\begin{align}
		\Gamma^{(k+p)(\alpha)}_{k(\alpha)l(\beta)}&\overset{\eqref{8.16induzesterna_3}}{=}0
		\notag
	\end{align}
	for each $l\geq m_\beta-q+1$. It follows that
	\begin{align}
		\notag
		&\Gamma^{(k+p)(\alpha)}_{(m_\beta-q)(\beta)k(\alpha)}\big(u^{1(\alpha)}-u^{1(\beta)}\big)=0
	\end{align}
	thus $\Gamma^{(k+p)(\alpha)}_{(m_\beta-q)(\beta)k(\alpha)}=0$, proving \eqref{8.16induzesterna_4} and consequently $(4.)$.
	
	By taking $i=k$ and $j=1$ in \eqref{bneqca0} we get
	\begin{align}
		0&=-m_\beta\epsilon_\beta+\Gamma^{i(\alpha)}_{1(\beta)i(\alpha)}u^{1(\alpha)}-\overset{m_\beta}{\underset{l=1}{\sum}}\,\Gamma^{i(\alpha)}_{i(\alpha)l(\beta)}u^{l(\beta)}\notag\\&\overset{(4.)}{=}-m_\beta\epsilon_\beta+\Gamma^{i(\alpha)}_{1(\beta)i(\alpha)}\big(u^{1(\alpha)}-u^{1(\beta)}\big)
		\notag
	\end{align}
	which proves $(5.)$.
	
	By taking $i=k+h$ with $h\geq1$ and $j=1$ in \eqref{bneqca0} we get
	\begin{align}
		0&=\overset{k+h}{\underset{l=k}{\sum}}\,\Gamma^{(k+h)(\alpha)}_{1(\beta)l(\alpha)}u^{(l-k+1)(\alpha)}-\overset{m_\beta}{\underset{l=1}{\sum}}\,\Gamma^{(k+h)(\alpha)}_{k(\alpha)l(\beta)}u^{l(\beta)}\notag\\&\overset{(4.)}{=}\overset{k+h}{\underset{l=k+1}{\sum}}\,\Gamma^{(k+h)(\alpha)}_{1(\beta)l(\alpha)}u^{(l-k+1)(\alpha)}+\Gamma^{(k+h)(\alpha)}_{k(\alpha)1(\beta)}\big(u^{1(\alpha)}-u^{1(\beta)}\big).
		\notag
	\end{align}
	Thus
	\begin{align}
		\Gamma^{(k+h)(\alpha)}_{1(\beta)k(\alpha)}&=-\frac{1}{u^{1(\alpha)}-u^{1(\beta)}}\overset{k+h}{\underset{l=k+1}{\sum}}\,\Gamma^{(k+h)(\alpha)}_{1(\beta)l(\alpha)}u^{(l-k+1)(\alpha)}\notag\\&\overset{s=l-k+1}{=}-\frac{1}{u^{1(\alpha)}-u^{1(\beta)}}\overset{h+1}{\underset{s=2}{\sum}}\,\Gamma^{(k+h)(\alpha)}_{1(\beta)(s+k-1)(\alpha)}u^{s(\alpha)}\notag\\&\overset{Prop.8.2}{=}-\frac{1}{u^{1(\alpha)}-u^{1(\beta)}}\overset{h+1}{\underset{s=2}{\sum}}\,\Gamma^{(k+h-s+1)(\alpha)}_{1(\beta)k(\alpha)}u^{s(\alpha)}
		\notag
	\end{align}
	which proves $(6.)$.
	
	By taking $k=1$, $j=2$ and $i\geq3$ in \eqref{abc0} we get
	\begin{align}
		0&=\overset{m_\alpha}{\underset{l=1}{\sum}}\,\Gamma^{i(\alpha)}_{2(\alpha)l(\alpha)}u^{l(\alpha)}-\overset{m_\alpha}{\underset{l=2}{\sum}}\,\Gamma^{i(\alpha)}_{1(\alpha)l(\alpha)}u^{(l-1)(\alpha)}\notag\\&\overset{(1.)}{\underset{(4.)}{=}}\overset{i}{\underset{l=1}{\sum}}\,\Gamma^{i(\alpha)}_{2(\alpha)l(\alpha)}u^{l(\alpha)}-\overset{i}{\underset{l=2}{\sum}}\,\Gamma^{i(\alpha)}_{1(\alpha)l(\alpha)}u^{(l-1)(\alpha)}\notag\\&=\cancel{\Gamma^{i(\alpha)}_{2(\alpha)1(\alpha)}u^{1(\alpha)}}+\overset{i}{\underset{l=2}{\sum}}\,\Gamma^{i(\alpha)}_{2(\alpha)l(\alpha)}u^{l(\alpha)}\cancel{-\Gamma^{i(\alpha)}_{1(\alpha)2(\alpha)}u^{1(\alpha)}}-\overset{i}{\underset{l=3}{\sum}}\,\Gamma^{i(\alpha)}_{1(\alpha)l(\alpha)}u^{(l-1)(\alpha)}\notag\\&=\Gamma^{i(\alpha)}_{2(\alpha)2(\alpha)}u^{2(\alpha)}+\overset{i-1}{\underset{l=3}{\sum}}\,\Gamma^{i(\alpha)}_{2(\alpha)l(\alpha)}u^{l(\alpha)}+\Gamma^{i(\alpha)}_{2(\alpha)i(\alpha)}u^{i(\alpha)}-\Gamma^{i(\alpha)}_{1(\alpha)3(\alpha)}u^{2(\alpha)}\notag\\&-\overset{i}{\underset{l=4}{\sum}}\,\Gamma^{i(\alpha)}_{1(\alpha)l(\alpha)}u^{(l-1)(\alpha)}\notag\\&\overset{Prop.8.2}{\underset{(1.)}{=}}\big(\Gamma^{i(\alpha)}_{2(\alpha)2(\alpha)}-\Gamma^{(i-2)(\alpha)}_{1(\alpha)1(\alpha)}\big)u^{2(\alpha)}+\overset{i-1}{\underset{l=3}{\sum}}\,\Gamma^{(i-l+2)(\alpha)}_{2(\alpha)2(\alpha)}u^{l(\alpha)}+\Gamma^{2(\alpha)}_{2(\alpha)2(\alpha)}u^{i(\alpha)}\notag\\&-\overset{i}{\underset{l=4}{\sum}}\,\Gamma^{(i-l+1)(\alpha)}_{1(\alpha)1(\alpha)}u^{(l-1)(\alpha)}.
		\notag
	\end{align}
	By means of changing the variables in summations ($s=i-l$ in the first one, $s=i-l+1$ in the second one) we obtain
	\begin{align}
		&\big(\Gamma^{i(\alpha)}_{2(\alpha)2(\alpha)}-\Gamma^{(i-2)(\alpha)}_{1(\alpha)1(\alpha)}\big)u^{2(\alpha)}+\Gamma^{2(\alpha)}_{2(\alpha)2(\alpha)}u^{i(\alpha)}+\overset{i-3}{\underset{s=1}{\sum}}\big(\Gamma^{(s+2)(\alpha)}_{2(\alpha)2(\alpha)}-\Gamma^{s(\alpha)}_{1(\alpha)1(\alpha)}\big)u^{(i-s)(\alpha)}=0
		\notag
	\end{align}
	thus
	\begin{align}
		\Gamma^{i(\alpha)}_{2(\alpha)2(\alpha)}=&\Gamma^{(i-2)(\alpha)}_{1(\alpha)1(\alpha)}-\frac{1}{u^{2(\alpha)}}\bigg(\Gamma^{2(\alpha)}_{2(\alpha)2(\alpha)}u^{i(\alpha)}+\overset{i-3}{\underset{s=1}{\sum}}\big(\Gamma^{(s+2)(\alpha)}_{2(\alpha)2(\alpha)}-\Gamma^{s(\alpha)}_{1(\alpha)1(\alpha)}\big)u^{(i-s)(\alpha)}\bigg)
		\notag
	\end{align}
	which proves $(7.)$.
\end{proof}
The Christoffel symbols can be divided in four categories:
\begin{itemize}
	\item[$1.$] $\Gamma^{k(\gamma)}_{i(\alpha)j(\beta)}$ with $\alpha\neq\beta\neq\gamma\neq\alpha$
	\item[$2.$] $\Gamma^{k(\alpha)}_{i(\beta)j(\alpha)}$ with $\alpha\neq\beta$
	\item[$3.$] $\Gamma^{k(\alpha)}_{i(\beta)j(\beta)}$ with $\alpha\neq\beta$
	\item[$4.$] $\Gamma^{k(\alpha)}_{i(\alpha)j(\alpha)}$.
\end{itemize}
By means of Propositions 8.1, 8.2 and 8.3 we are able to determine all of them, as the following result shows.
\begin{prop}
	Let $\alpha$, $\beta$, $\gamma$ be pairwise distinct. Then
	\begin{itemize}
		\item[$1.$] $\Gamma^{k(\gamma)}_{i(\alpha)j(\beta)}=0$ for each value of $i,j,k$
		\item[$2.$] $\Gamma^{k(\alpha)}_{i(\beta)j(\alpha)}=0$ for every $j,k$ when $i\geq2$ and
		\begin{align}
			\Gamma^{k(\alpha)}_{1(\beta)j(\alpha)}&=\Gamma^{(k-j+1)(\alpha)}_{1(\beta)1(\alpha)}=\begin{cases}
				0\qquad&\text{if }k<j\\
				\frac{m_\beta\epsilon_\beta}{u^{1(\alpha)}-u^{1(\beta)}}\qquad&\text{if }k=j\\
				-\frac{1}{u^{1(\alpha)}-u^{1(\beta)}}\,\overset{k-j+1}{\underset{s=2}{\sum}}\,\Gamma^{(k-j-s+2)(\alpha)}_{1(\beta)1(\alpha)}\,u^{s(\alpha)}\qquad&\text{if }k>j
			\end{cases}
			\notag
		\end{align}
		when $i=1$
		\item[$3.$] $\Gamma^{k(\alpha)}_{i(\beta)j(\beta)}=0$ for each $k$ when $i+j\geq3$ and
		\begin{align}
			\Gamma^{k(\alpha)}_{1(\beta)1(\beta)}&=-\Gamma^{k(\alpha)}_{1(\beta)1(\alpha)}=
			\begin{cases}
				-\frac{m_\beta\epsilon_\beta}{u^{1(\alpha)}-u^{1(\beta)}}\qquad&\text{if }k=1\\
				\frac{1}{u^{1(\alpha)}-u^{1(\beta)}}\,\overset{k}{\underset{s=2}{\sum}}\,\Gamma^{(k-s+1)(\alpha)}_{1(\beta)1(\alpha)}\,u^{s(\alpha)}\qquad&\text{if }k>1
			\end{cases}
			\notag
		\end{align}
		when $i=j=1$
		\item[$4.$] $\Gamma^{k(\alpha)}_{i(\alpha)j(\alpha)}$ is determinated by
		\begin{align}
			\Gamma^{k(\alpha)}_{i(\alpha)j(\alpha)}&=
			\begin{cases}
				0\qquad&\text{if }k-i-j\leq-3\\
				\Gamma^{(k-i-j+4)(\alpha)}_{2(\alpha)2(\alpha)}\qquad&\text{if }k-i-j\geq-2
			\end{cases}
			\notag
		\end{align}
		when $i,j\geq2$ where $$\Gamma^{2(\alpha)}_{2(\alpha)2(\alpha)}=-\frac{m_\alpha\epsilon_{\alpha}}{u^{2(\alpha)}}$$ and
		\begin{align}
			\Gamma^{n(\alpha)}_{2(\alpha)2(\alpha)}&=\Gamma^{(n-2)(\alpha)}_{1(\alpha)1(\alpha)}-\Gamma^{2(\alpha)}_{2(\alpha)2(\alpha)}\frac{u^{n(\alpha)}}{u^{2(\alpha)}}-\frac{1}{u^{2(\alpha)}}\overset{n-3}{\underset{l=1}{\sum}}\big(\Gamma^{(l+2)(\alpha)}_{2(\alpha)2(\alpha)}-\Gamma^{l(\alpha)}_{1(\alpha)1(\alpha)}\big)u^{(n-l)(\alpha)}
			\notag
		\end{align}
		for $n\ge3$, by
		\begin{align}
			\Gamma^{k(\alpha)}_{1(\alpha)j(\alpha)}&=-\underset{\sigma\neq\alpha}{\sum}\,\Gamma^{k(\alpha)}_{1(\sigma)j(\alpha)}=\begin{cases}
				0\qquad&\text{if }k<j\\
				-\underset{\sigma\neq\alpha}{\sum}\,\frac{m_\sigma\epsilon_\sigma}{u^{1(\alpha)}-u^{1(\sigma)}}\qquad&\text{if }k=j\\
				\underset{\sigma\neq\alpha}{\sum}\,\frac{1}{u^{1(\alpha)}-u^{1(\sigma)}}\,\overset{k-j+1}{\underset{s=2}{\sum}}\,\Gamma^{(k-j-s+2)(\alpha)}_{1(\sigma)1(\alpha)}\,u^{s(\alpha)}\qquad&\text{if }k>j
			\end{cases}
			\notag
		\end{align}
		when $i=1$ and by
		\begin{align}
			\Gamma^{k(\alpha)}_{i(\alpha)1(\alpha)}&=-\underset{\sigma\neq\alpha}{\sum}\,\Gamma^{k(\alpha)}_{1(\sigma)i(\alpha)}=\begin{cases}
				0\qquad&\text{if }k<i\\
				-\underset{\sigma\neq\alpha}{\sum}\,\frac{m_\sigma\epsilon_\sigma}{u^{1(\alpha)}-u^{1(\sigma)}}\qquad&\text{if }k=i\\
				\underset{\sigma\neq\alpha}{\sum}\,\frac{1}{u^{1(\alpha)}-u^{1(\sigma)}}\,\overset{k-i+1}{\underset{s=2}{\sum}}\,\Gamma^{(k-i-s+2)(\alpha)}_{1(\sigma)1(\alpha)}\,u^{s(\alpha)}\qquad&\text{if }k>i
			\end{cases}
			\notag
		\end{align}
		when $j=1$.
	\end{itemize}
\end{prop}
\begin{proof}
	The proof directly follows from Propositions 8.1, 8.2 and 8.3.
\end{proof}
\begin{remark}
	It is easy to observe that $a_0$ is a flat coordinate for $\nabla$, namely $\nabla(da_0)=0$. Indeed, the $i$-th component of $da_0$ is
	\begin{align}
		(da_0)_i&=\partial_i(da_0)=\partial_i\bigg(\underset{\sigma=1}{\overset{r}{\sum}}\,m_\sigma\epsilon_{1(\sigma)}\,u^{1(\sigma)}\bigg)=\underset{\sigma=1}{\overset{r}{\sum}}\,m_\sigma\epsilon_{1(\sigma)}\,\delta^{1(\sigma)}_i
		\notag
	\end{align}
	thus
	\begin{align}
		\nabla_i(da_0)_j&=\xcancel{\partial_i(da_0)_j}-\Gamma^k_{ij}(da_0)_k=-\Gamma^k_{ij}\,\underset{\sigma=1}{\overset{r}{\sum}}\,m_\sigma\epsilon_{1(\sigma)}\,\delta^{1(\sigma)}_k.
		\notag
	\end{align}
	Given $\alpha,\beta\in\{1,\dots,r\}$ this reads
	\begin{align}
		\nabla_{i(\alpha)}(da_0)_{j(\beta)}&=-\underset{\gamma=1}{\overset{r}{\sum}}\,\Gamma^{k(\gamma)}_{i(\alpha)j(\beta)}\,\underset{\sigma=1}{\overset{r}{\sum}}\,m_\sigma\epsilon_{1(\sigma)}\,\delta^{1(\sigma)}_{k(\gamma)}=-\underset{\sigma=1}{\overset{r}{\sum}}\,\Gamma^{1(\sigma)}_{i(\alpha)j(\beta)}\,m_\sigma\epsilon_{1(\sigma)}.
		\label{a_0cp}
	\end{align}
	Let us first consider the case where $\alpha=\beta$. We get
	\begin{align}
		\nabla_{i(\alpha)}(da_0)_{j(\alpha)}&=-\underset{\sigma=1}{\overset{r}{\sum}}\,\Gamma^{1(\sigma)}_{i(\alpha)j(\alpha)}\,m_\sigma\epsilon_{1(\sigma)}=-\Gamma^{1(\alpha)}_{i(\alpha)j(\alpha)}\,m_\alpha\epsilon_{1(\alpha)}-\underset{\sigma\neq\alpha}{\overset{}{\sum}}\,\Gamma^{1(\sigma)}_{i(\alpha)j(\alpha)}\,m_\sigma\epsilon_{1(\sigma)}
		\notag
	\end{align}
	where (by Proposition 8.4) the terms $\Gamma^{1(\alpha)}_{i(\alpha)j(\alpha)}$ and $\Gamma^{1(\sigma)}_{i(\alpha)j(\alpha)}=-\underset{\sigma\neq\alpha}{\overset{}{\sum}}\,\Gamma^{1(\sigma)}_{(i+j-1)(\alpha)1(\sigma)}$ only survive for $i=j=1$. Therefore $\nabla_{i(\alpha)}(da_0)_{j(\alpha)}=0$ trivially whenever at least one among $i$ and $j$ is greater or equal than $2$. When $i=j=1$ we get
	\begin{align}
		\nabla_{1(\alpha)}(da_0)_{1(\alpha)}&=-\Gamma^{1(\alpha)}_{1(\alpha)1(\alpha)}\,m_\alpha\epsilon_{1(\alpha)}-\underset{\sigma\neq\alpha}{\overset{}{\sum}}\,\Gamma^{1(\sigma)}_{1(\alpha)1(\alpha)}\,m_\sigma\epsilon_{1(\sigma)}
		\notag\\&=\underset{\sigma\neq\alpha}{\overset{}{\sum}}\,\Gamma^{1(\alpha)}_{1(\sigma)1(\alpha)}\,m_\alpha\epsilon_{1(\alpha)}+\underset{\sigma\neq\alpha}{\overset{}{\sum}}\,\Gamma^{1(\sigma)}_{1(\alpha)1(\sigma)}\,m_\sigma\epsilon_{1(\sigma)}
		\notag\\&=\underset{\sigma\neq\alpha}{\overset{}{\sum}}\,\bigg(\frac{m_\sigma\epsilon_{1(\sigma)}}{u^{1(\alpha)}-u^{1(\sigma)}}\,m_\alpha\epsilon_{1(\alpha)}-\frac{m_\alpha\epsilon_{1(\alpha)}}{u^{1(\alpha)}-u^{1(\sigma)}}\,m_\sigma\epsilon_{1(\sigma)}\bigg)=0.
		\notag
	\end{align}
	Let us now consider the case of $\alpha\neq\beta$, where \eqref{a_0cp} becomes
	\begin{align}
		\nabla_{i(\alpha)}(da_0)_{j(\beta)}&=-\underset{\sigma=1}{\overset{r}{\sum}}\,\Gamma^{1(\sigma)}_{i(\alpha)j(\beta)}\,m_\sigma\epsilon_{1(\sigma)}=-\Gamma^{1(\alpha)}_{i(\alpha)j(\beta)}\,m_\alpha\epsilon_{1(\alpha)}-\Gamma^{1(\beta)}_{i(\alpha)j(\beta)}\,m_\beta\epsilon_{1(\beta)}
		\notag
	\end{align}
	which trivially vanishes whenever at least one among $i$ and $j$ is greater or equal than $2$, as both $\Gamma^{1(\alpha)}_{i(\alpha)j(\beta)}$ and $\Gamma^{1(\beta)}_{i(\alpha)j(\beta)}$ only survive for $i=j=1$. In this latter case we get
	\begin{align}
		\nabla_{1(\alpha)}(da_0)_{1(\beta)}&=-\Gamma^{1(\alpha)}_{1(\alpha)1(\beta)}\,m_\alpha\epsilon_{1(\alpha)}-\Gamma^{1(\beta)}_{1(\alpha)1(\beta)}\,m_\beta\epsilon_{1(\beta)}
		\notag\\&=-\frac{m_\beta\epsilon_{1(\beta)}}{u^{1(\alpha)}-u^{1(\beta)}}\,m_\alpha\epsilon_{1(\alpha)}+\frac{m_\alpha\epsilon_{1(\alpha)}}{u^{1(\alpha)}-u^{1(\beta)}}\,m_\beta\epsilon_{1(\beta)}=0.
		\notag
	\end{align}
\end{remark}
\begin{remark}
	It is likewise easy to check that $dd_La_0=0$.
\end{remark}

\subsection{Technical lemmas}
\begin{lemma}
	For every choice of $\alpha$, $\beta$, $\gamma$, $\delta\in\{1,\dots,r\}$ we have
	\begin{equation}
		\label{Lemma7.1eq}
		\frac{\partial\Gamma^{k(\gamma)}_{i(\alpha)j(\beta)}}{\partial u^{l(\delta)}}=\frac{\partial\Gamma^{(k-1)(\gamma)}_{i(\alpha)j(\beta)}}{\partial u^{(l-1)(\delta)}}
	\end{equation}
	for all $k\in\{2,\dots,m_\gamma\}$ and $l\in\{3,\dots,m_\delta\}$.
	Moreover, if $\beta\neq\alpha=\gamma=\delta$ then \eqref{Lemma7.1eq} holds for $l=2$ as well.
\end{lemma}
\begin{proof}
	In the wake of the previous result, we consider the following significative cases:
	\begin{itemize}
		\item[$1.$] $\alpha$, $\beta$, $\gamma$ are pairwise distinct
		\item[$2.$] $\alpha=\gamma\neq\beta$
		\item[$3.$] $\alpha=\beta\neq\gamma$
		\item[$4.$] $\alpha=\beta=\gamma$.
	\end{itemize}
	\textbf{Case 1: $\alpha\neq\beta\neq\gamma\neq\alpha$.} Since in this case all the quantities $\Gamma^{k(\gamma)}_{i(\alpha)j(\beta)}$ vanish, \eqref{Lemma7.1eq} holds trivially for each choice of $\delta$.
	\\\textbf{Case 2: $\alpha=\gamma\neq\beta$.} We are going to prove that
	\begin{equation}
		\label{Lemma7.1_2}
		\frac{\partial\Gamma^{k(\alpha)}_{i(\alpha)j(\beta)}}{\partial u^{l(\delta)}}=\frac{\partial\Gamma^{(k-1)(\alpha)}_{i(\alpha)j(\beta)}}{\partial u^{(l-1)(\delta)}}
	\end{equation}
	for all $k\in\{2,\dots,m_\alpha\}$, $l\in\{3,\dots,m_\delta\}$ and $\delta\in\{1,\dots,r\}$. This holds trivially for each $\delta\neq\alpha$, as $\Gamma^{k(\alpha)}_{i(\alpha)j(\beta)}$ does not depend on $u^{l(\delta)}$ and $\Gamma^{(k-1)(\alpha)}_{i(\alpha)j(\beta)}$ does not depend on $u^{(l-1)(\delta)}$ for any $l\geq3$. Moreover, by virtue of Proposition $8.4$, both $\Gamma^{k(\alpha)}_{i(\alpha)j(\beta)}$ and $\Gamma^{(k-1)(\alpha)}_{i(\alpha)j(\beta)}$ vanish if $j\geq2$ or $k<i$. Therefore we are left to show that
	\begin{equation}
		\label{Lemma7.1_2_alpha}
		\frac{\partial\Gamma^{k(\alpha)}_{i(\alpha)1(\beta)}}{\partial u^{l(\alpha)}}=\frac{\partial\Gamma^{(k-1)(\alpha)}_{i(\alpha)1(\beta)}}{\partial u^{(l-1)(\alpha)}}
	\end{equation}
	for $k\geq i$. We are going to prove \eqref{Lemma7.1_2_alpha} by induction over $k$. If $k=i$ we get
	\begin{align}
		\frac{\partial\Gamma^{k(\alpha)}_{i(\alpha)1(\beta)}}{\partial u^{l(\alpha)}}&=\frac{\partial\Gamma^{i(\alpha)}_{i(\alpha)1(\beta)}}{\partial u^{l(\alpha)}}=\frac{\partial\Gamma^{1(\alpha)}_{1(\alpha)1(\beta)}}{\partial u^{l(\alpha)}}=0
		\notag
	\end{align}
	(as $\Gamma^{1(\alpha)}_{1(\alpha)1(\beta)}$ does not depend on $u^{l(\alpha)}$ for any $l\geq3$) and
	\begin{align}
		\frac{\partial\Gamma^{(k-1)(\alpha)}_{i(\alpha)1(\beta)}}{\partial u^{(l-1)(\alpha)}}&=\frac{\partial\Gamma^{(i-1)(\alpha)}_{i(\alpha)1(\beta)}}{\partial u^{(l-1)(\alpha)}}=0
		\notag
	\end{align}
	(as $\Gamma^{(i-1)(\alpha)}_{i(\alpha)1(\beta)}=0$), thus \eqref{Lemma7.1_2_alpha} is verified for $k=i$. Given an integer $h\geq1$, let us suppose that \eqref{Lemma7.1_2_alpha} holds whenever $k\leq i+h-1$. We now want to prove that it holds for $k=i+h$ as well, that is
	\begin{equation}
		\label{Lemma7.1_2_alphaind}
		\frac{\partial\Gamma^{(i+h)(\alpha)}_{i(\alpha)1(\beta)}}{\partial u^{l(\alpha)}}=\frac{\partial\Gamma^{(i+h-1)(\alpha)}_{i(\alpha)1(\beta)}}{\partial u^{(l-1)(\alpha)}}.
	\end{equation}
	The left-hand side term reads
	\begin{align}
		\frac{\partial\Gamma^{(i+h)(\alpha)}_{i(\alpha)1(\beta)}}{\partial u^{l(\alpha)}}&=\frac{\partial}{\partial u^{l(\alpha)}}\bigg(-\frac{1}{u^{1(\alpha)}-u^{1(\beta)}}\,\overset{h+1}{\underset{s=2}{\sum}}\,\Gamma^{(i+h-s+1)(\alpha)}_{i(\alpha)1(\beta)}\,u^{s(\alpha)}\bigg)\notag\\&\overset{l>2}{=}-\frac{1}{u^{1(\alpha)}-u^{1(\beta)}}\,\overset{h+1}{\underset{s=2}{\sum}}\,\frac{\partial}{\partial u^{l(\alpha)}}\bigg(\Gamma^{(i+h-s+1)(\alpha)}_{i(\alpha)1(\beta)}\,u^{s(\alpha)}\bigg)\notag\\&=-\frac{1}{u^{1(\alpha)}-u^{1(\beta)}}\,\overset{h+1}{\underset{s=2}{\sum}}\,\frac{\partial\Gamma^{(i+h-s+1)(\alpha)}_{i(\alpha)1(\beta)}}{\partial u^{l(\alpha)}}\,u^{s(\alpha)}\notag\\&-\frac{1}{u^{1(\alpha)}-u^{1(\beta)}}\,\Gamma^{(i+h-l+1)(\alpha)}_{i(\alpha)1(\beta)}\notag\\&=-\frac{1}{u^{1(\alpha)}-u^{1(\beta)}}\,\overset{h}{\underset{s=2}{\sum}}\,\frac{\partial\Gamma^{(i+h-s+1)(\alpha)}_{i(\alpha)1(\beta)}}{\partial u^{l(\alpha)}}\,u^{s(\alpha)}\notag\\&-\frac{1}{u^{1(\alpha)}-u^{1(\beta)}}\,\frac{\partial\Gamma^{i(\alpha)}_{i(\alpha)1(\beta)}}{\partial u^{l(\alpha)}}\,u^{(h+1)(\alpha)}\notag\\&-\frac{1}{u^{1(\alpha)}-u^{1(\beta)}}\,\Gamma^{(i+h-l+1)(\alpha)}_{i(\alpha)1(\beta)}
		\notag
	\end{align}
	where
	\begin{align}
		\frac{\partial\Gamma^{(i+h-s+1)(\alpha)}_{i(\alpha)1(\beta)}}{\partial u^{l(\alpha)}}&=\frac{\partial\Gamma^{(i+h-s)(\alpha)}_{i(\alpha)1(\beta)}}{\partial u^{(l-1)(\alpha)}}
		\notag
	\end{align}
	by the inductive hypothesis for each $s\geq2$ and
	\begin{align}
		\frac{\partial\Gamma^{i(\alpha)}_{i(\alpha)1(\beta)}}{\partial u^{l(\alpha)}}&=\frac{\partial\Gamma^{1(\alpha)}_{1(\alpha)1(\beta)}}{\partial u^{l(\alpha)}}=0
		\notag
	\end{align}
	for every $l\geq3$. Therefore the left-hand side of \eqref{Lemma7.1_2_alphaind} is
	\begin{align}
		\frac{\partial\Gamma^{(i+h)(\alpha)}_{i(\alpha)1(\beta)}}{\partial u^{l(\alpha)}}&=-\frac{1}{u^{1(\alpha)}-u^{1(\beta)}}\,\overset{h}{\underset{s=2}{\sum}}\,\frac{\partial\Gamma^{(i+h-s)(\alpha)}_{i(\alpha)1(\beta)}}{\partial u^{(l-1)(\alpha)}}\,u^{s(\alpha)}\notag\\&-\frac{1}{u^{1(\alpha)}-u^{1(\beta)}}\,\Gamma^{(i+h-l+1)(\alpha)}_{i(\alpha)1(\beta)}
		\notag
	\end{align}
	which amounts the right-hand side term
	\begin{align}
		\frac{\partial\Gamma^{(i+h-1)(\alpha)}_{i(\alpha)1(\beta)}}{\partial u^{(l-1)(\alpha)}}&=\frac{\partial}{\partial u^{(l-1)(\alpha)}}\bigg(-\frac{1}{u^{1(\alpha)}-u^{1(\beta)}}\,\overset{h}{\underset{s=2}{\sum}}\,\Gamma^{(i+h-s)(\alpha)}_{i(\alpha)1(\beta)}\,u^{s(\alpha)}\bigg)\notag\\&\overset{l>2}{=}-\frac{1}{u^{1(\alpha)}-u^{1(\beta)}}\,\overset{h}{\underset{s=2}{\sum}}\,\frac{\partial\Gamma^{(i+h-s)(\alpha)}_{i(\alpha)1(\beta)}}{\partial u^{(l-1)(\alpha)}}\,u^{s(\alpha)}\notag\\&-\frac{1}{u^{1(\alpha)}-u^{1(\beta)}}\,\Gamma^{(i+h-l+1)(\alpha)}_{i(\alpha)1(\beta)}.
		\notag
	\end{align}
	This proves \eqref{Lemma7.1_2_alpha} for $k\geq i$.
	\\\textbf{Case 3: $\alpha=\beta\neq\gamma$.} For every $k\in\{2,\dots,m_\gamma\}$, $l\in\{3,\dots,m_\delta\}$ and $\delta\in\{1,\dots,r\}$ we have
	\begin{align}
		\frac{\partial\Gamma^{k(\gamma)}_{i(\alpha)j(\alpha)}}{\partial u^{l(\delta)}}&=\frac{\partial\Gamma^{k(\gamma)}_{(i+j-1)(\alpha)1(\alpha)}}{\partial u^{l(\delta)}}=-\frac{\partial\Gamma^{k(\gamma)}_{(i+j-1)(\alpha)1(\gamma)}}{\partial u^{l(\delta)}}\overset{\text{Case }2}{=}-\frac{\partial\Gamma^{(k-1)(\gamma)}_{(i+j-1)(\alpha)1(\gamma)}}{\partial u^{(l-1)(\delta)}}\notag\\&=\frac{\partial\Gamma^{(k-1)(\gamma)}_{(i+j-1)(\alpha)1(\alpha)}}{\partial u^{(l-1)(\delta)}}=\frac{\partial\Gamma^{(k-1)(\gamma)}_{i(\alpha)j(\alpha)}}{\partial u^{(l-1)(\delta)}}.
		\notag
	\end{align}
	\\\textbf{Case 4: $\alpha=\beta=\gamma$.} We are going to show that
	\begin{equation}
		\label{Lemma7.1_4}
		\frac{\partial\Gamma^{k(\alpha)}_{i(\alpha)j(\alpha)}}{\partial u^{l(\delta)}}=\frac{\partial\Gamma^{(k-1)(\alpha)}_{i(\alpha)j(\alpha)}}{\partial u^{(l-1)(\delta)}}
	\end{equation}
	for all $\delta\in\{1,\dots,r\}$, $k\in\{2,\dots,m_\alpha\}$ and $l\in\{3,\dots,m_\delta\}$. If $i=1$ (or equivalently $j=1$) then \eqref{Lemma7.1_4} is verified by means of Case $2$ and \eqref{nablae=0}, as
	\begin{align}
		\frac{\partial\Gamma^{k(\alpha)}_{1(\alpha)j(\alpha)}}{\partial u^{l(\delta)}}&=-\underset{\sigma\neq\alpha}{\sum}\,\frac{\partial\Gamma^{k(\alpha)}_{1(\sigma)j(\alpha)}}{\partial u^{l(\delta)}}=-\underset{\sigma\neq\alpha}{\sum}\,\frac{\partial\Gamma^{(k-1)(\alpha)}_{1(\sigma)j(\alpha)}}{\partial u^{(l-1)(\delta)}}=-\frac{\partial\Gamma^{(k-1)(\alpha)}_{1(\alpha)j(\alpha)}}{\partial u^{(l-1)(\delta)}}
		\notag
	\end{align}
	for every choice of $\delta$ and every $k\in\{2,\dots,m_\alpha\}$, $l\in\{3,\dots,m_\delta\}$. Let us now consider $i,j\geq2$. Without loss of generality, thanks to Proposition $8.4$, we can restrict ourselves to the case where $i=j=2$. If $\delta\neq\alpha$ then
	\begin{equation}
		\frac{\partial\Gamma^{k(\alpha)}_{2(\alpha)2(\alpha)}}{\partial u^{l(\delta)}}=0=\frac{\partial\Gamma^{(k-1)(\alpha)}_{2(\alpha)2(\alpha)}}{\partial u^{(l-1)(\delta)}}
		\notag
	\end{equation}
	as $\Gamma^{k(\alpha)}_{2(\alpha)2(\alpha)}$ only contains the terms $\{u^{s(\alpha)}\,|\,2\leq s\leq k\}$ and $\{u^{1(\sigma)}\,|\,1,\dots,r\}$, where $l(\delta)$ is not included ($l\geq3$). It only remains to prove \eqref{Lemma7.1_4} for $i=j=2$ and $\delta=\alpha$, that is
	\begin{equation}
		\label{Lemma7.1_4_alpha}
		\frac{\partial\Gamma^{k(\alpha)}_{2(\alpha)2(\alpha)}}{\partial u^{l(\alpha)}}=\frac{\partial\Gamma^{(k-1)(\alpha)}_{2(\alpha)2(\alpha)}}{\partial u^{(l-1)(\alpha)}}
	\end{equation}
	for all $k\in\{2,\dots,m_\alpha\}$ and $l\in\{3,\dots,m_\alpha\}$. We will proceed by induction over $k$. For $k=2$ both the left and the right-hand sides vanish, as $\Gamma^{2(\alpha)}_{2(\alpha)2(\alpha)}=-\frac{m_\alpha\epsilon_{\alpha}}{u^{2(\alpha)}}$ does not depend on any of the terms $\{u^{l(\alpha)}\,|\,l\geq3\}$ and $\Gamma^{1(\alpha)}_{2(\alpha)2(\alpha)}=0$. Let us suppose that \eqref{Lemma7.1_4_alpha} holds for all $k\in\{2,\dots,h\}$ (given an integer $2\leq h\leq m_\alpha$) and $l\in\{3,\dots,m_\alpha\}$. We must prove that
	\begin{equation}
		\label{Lemma7.1_4_alpha_thind}
		\frac{\partial\Gamma^{(h+1)(\alpha)}_{2(\alpha)2(\alpha)}}{\partial u^{l(\alpha)}}=\frac{\partial\Gamma^{h(\alpha)}_{2(\alpha)2(\alpha)}}{\partial u^{(l-1)(\alpha)}}
	\end{equation}
	for all $l\in\{3,\dots,m_\alpha\}$. The left-hand side term reads
	\begin{align}
		\frac{\partial\Gamma^{(h+1)(\alpha)}_{2(\alpha)2(\alpha)}}{\partial u^{l(\alpha)}}&=
		\notag\frac{\partial}{\partial u^{l(\alpha)}}\bigg(\Gamma^{(h-1)(\alpha)}_{1(\alpha)1(\alpha)}-\Gamma^{2(\alpha)}_{2(\alpha)2(\alpha)}\frac{u^{(h+1)(\alpha)}}{u^{2(\alpha)}}\notag\\&-\frac{1}{u^{2(\alpha)}}\overset{h-2}{\underset{s=1}{\sum}}\big(\Gamma^{(s+2)(\alpha)}_{2(\alpha)2(\alpha)}-\Gamma^{s(\alpha)}_{1(\alpha)1(\alpha)}\big)u^{(h-s+1)(\alpha)}\bigg)\notag\\&\overset{l>2}{=}\frac{\partial\Gamma^{(h-1)(\alpha)}_{1(\alpha)1(\alpha)}}{\partial u^{l(\alpha)}}-\frac{\Gamma^{2(\alpha)}_{2(\alpha)2(\alpha)}}{u^{2(\alpha)}}\delta^{h+1}_{l}\notag\\&-\frac{1}{u^{2(\alpha)}}\overset{h-2}{\underset{s=1}{\sum}}\bigg(\frac{\partial\Gamma^{(s+2)(\alpha)}_{2(\alpha)2(\alpha)}}{\partial u^{l(\alpha)}}-\frac{\partial\Gamma^{s(\alpha)}_{1(\alpha)1(\alpha)}}{\partial u^{l(\alpha)}}\bigg)u^{(h-s+1)(\alpha)}\notag\\&-\frac{1}{u^{2(\alpha)}}\big(\Gamma^{(h-l+3)(\alpha)}_{2(\alpha)2(\alpha)}-\Gamma^{(h-l+1)(\alpha)}_{1(\alpha)1(\alpha)}\big)
		\notag
	\end{align}
	where by what we said above ($i=1$) we have
	\begin{align}
		\frac{\partial\Gamma^{(h-1)(\alpha)}_{1(\alpha)1(\alpha)}}{\partial u^{l(\alpha)}}&=\frac{\partial\Gamma^{(h-2)(\alpha)}_{1(\alpha)1(\alpha)}}{\partial u^{(l-1)(\alpha)}}\text{,}\notag\\
		\frac{\partial\Gamma^{s(\alpha)}_{1(\alpha)1(\alpha)}}{\partial u^{l(\alpha)}}&=\frac{\partial\Gamma^{(s-1)(\alpha)}_{1(\alpha)1(\alpha)}}{\partial u^{(l-1)(\alpha)}}\qquad\forall\,1\leq s\leq h-2
		\notag
	\end{align}
	and by the inductive hypothesis
	\begin{align}
		\frac{\partial\Gamma^{(s+2)(\alpha)}_{2(\alpha)2(\alpha)}}{\partial u^{l(\alpha)}}=\frac{\partial\Gamma^{(s+1)(\alpha)}_{2(\alpha)2(\alpha)}}{\partial u^{(l-1)(\alpha)}}\qquad\forall\,1\leq s\leq h-2.
		\notag
	\end{align}
	It follows that the left-hand side of \eqref{Lemma7.1_4_alpha_thind} is
	\begin{align}
		\frac{\partial\Gamma^{(h+1)(\alpha)}_{2(\alpha)2(\alpha)}}{\partial u^{l(\alpha)}}&=\frac{\partial\Gamma^{(h-2)(\alpha)}_{1(\alpha)1(\alpha)}}{\partial u^{(l-1)(\alpha)}}-\frac{\Gamma^{2(\alpha)}_{2(\alpha)2(\alpha)}}{u^{2(\alpha)}}\delta^{h+1}_{l}\notag\\&-\frac{1}{u^{2(\alpha)}}\overset{h-2}{\underset{s=1}{\sum}}\bigg(\frac{\partial\Gamma^{(s+1)(\alpha)}_{2(\alpha)2(\alpha)}}{\partial u^{(l-1)(\alpha)}}-\frac{\partial\Gamma^{(s-1)(\alpha)}_{1(\alpha)1(\alpha)}}{\partial u^{(l-1)(\alpha)}}\bigg)u^{(h-s+1)(\alpha)}\notag\\&-\frac{1}{u^{2(\alpha)}}\big(\Gamma^{(h-l+3)(\alpha)}_{2(\alpha)2(\alpha)}-\Gamma^{(h-l+1)(\alpha)}_{1(\alpha)1(\alpha)}\big)\notag\\&=\frac{\partial\Gamma^{(h-2)(\alpha)}_{1(\alpha)1(\alpha)}}{\partial u^{(l-1)(\alpha)}}-\frac{\Gamma^{2(\alpha)}_{2(\alpha)2(\alpha)}}{u^{2(\alpha)}}\delta^{h+1}_{l}\notag\\&-\frac{1}{u^{2(\alpha)}}\frac{\partial\Gamma^{2(\alpha)}_{2(\alpha)2(\alpha)}}{\partial u^{(l-1)(\alpha)}}u^{h(\alpha)}\notag\\&-\frac{1}{u^{2(\alpha)}}\overset{h-2}{\underset{s=2}{\sum}}\bigg(\frac{\partial\Gamma^{(s+1)(\alpha)}_{2(\alpha)2(\alpha)}}{\partial u^{(l-1)(\alpha)}}-\frac{\partial\Gamma^{(s-1)(\alpha)}_{1(\alpha)1(\alpha)}}{\partial u^{(l-1)(\alpha)}}\bigg)u^{(h-s+1)(\alpha)}\notag\\&-\frac{1}{u^{2(\alpha)}}\big(\Gamma^{(h-l+3)(\alpha)}_{2(\alpha)2(\alpha)}-\Gamma^{(h-l+1)(\alpha)}_{1(\alpha)1(\alpha)}\big).
		\notag
	\end{align}
	Since the right-hand side reads,
	\begin{align}
		\frac{\partial\Gamma^{h(\alpha)}_{2(\alpha)2(\alpha)}}{\partial u^{(l-1)(\alpha)}}&=\frac{\partial}{\partial u^{(l-1)(\alpha)}}\bigg(\Gamma^{(h-2)(\alpha)}_{1(\alpha)1(\alpha)}-\Gamma^{2(\alpha)}_{2(\alpha)2(\alpha)}\frac{u^{h(\alpha)}}{u^{2(\alpha)}}\notag\\&-\frac{1}{u^{2(\alpha)}}\overset{h-3}{\underset{s=1}{\sum}}\big(\Gamma^{(s+2)(\alpha)}_{2(\alpha)2(\alpha)}-\Gamma^{s(\alpha)}_{1(\alpha)1(\alpha)}\big)u^{(h-s)(\alpha)}\bigg)\notag\\&\overset{l>2}{=}\frac{\partial\Gamma^{(h-2)(\alpha)}_{1(\alpha)1(\alpha)}}{\partial u^{(l-1)(\alpha)}}-\frac{\Gamma^{2(\alpha)}_{2(\alpha)2(\alpha)}}{u^{2(\alpha)}}\delta^{h}_{l-1}-\frac{u^{h(\alpha)}}{u^{2(\alpha)}}\frac{\partial\Gamma^{2(\alpha)}_{2(\alpha)2(\alpha)}}{\partial u^{(l-1)(\alpha)}}\notag\\&-\frac{1}{u^{2(\alpha)}}\overset{h-3}{\underset{s=1}{\sum}}\bigg(\frac{\partial\Gamma^{(s+2)(\alpha)}_{2(\alpha)2(\alpha)}}{\partial u^{(l-1)(\alpha)}}-\frac{\partial\Gamma^{s(\alpha)}_{1(\alpha)1(\alpha)}}{\partial u^{(l-1)(\alpha)}}\bigg)u^{(h-s)(\alpha)}
		\notag\\&-\frac{1}{u^{2(\alpha)}}\big(\Gamma^{(h-l+3)(\alpha)}_{2(\alpha)2(\alpha)}-\Gamma^{(h-l+1)(\alpha)}_{1(\alpha)1(\alpha)}\big)
		\notag
	\end{align}	
	we get that the difference between the left and the right-hand side terms is
	\begin{align}
		\frac{\partial\Gamma^{(h+1)(\alpha)}_{2(\alpha)2(\alpha)}}{\partial u^{l(\alpha)}}-\frac{\partial\Gamma^{h(\alpha)}_{2(\alpha)2(\alpha)}}{\partial u^{(l-1)(\alpha)}}&=\cancel{\frac{\partial\Gamma^{(h-2)(\alpha)}_{1(\alpha)1(\alpha)}}{\partial u^{(l-1)(\alpha)}}}-\bcancel{\frac{\Gamma^{2(\alpha)}_{2(\alpha)2(\alpha)}}{u^{2(\alpha)}}\delta^{h+1}_{l}}\notag\\&-\cancel{\cancel{\frac{1}{u^{2(\alpha)}}\frac{\partial\Gamma^{2(\alpha)}_{2(\alpha)2(\alpha)}}{\partial u^{(l-1)(\alpha)}}u^{h(\alpha)}}\,}\notag\\&-\frac{1}{u^{2(\alpha)}}\overset{h-2}{\underset{s=2}{\sum}}\bigg(\frac{\partial\Gamma^{(s+1)(\alpha)}_{2(\alpha)2(\alpha)}}{\partial u^{(l-1)(\alpha)}}-\frac{\partial\Gamma^{(s-1)(\alpha)}_{1(\alpha)1(\alpha)}}{\partial u^{(l-1)(\alpha)}}\bigg)u^{(h-s+1)(\alpha)}\notag\\&-\bcancel{\bcancel{\frac{1}{u^{2(\alpha)}}\big(\Gamma^{(h-l+3)(\alpha)}_{2(\alpha)2(\alpha)}-\Gamma^{(h-l+1)(\alpha)}_{1(\alpha)1(\alpha)}\big)}}\notag\\&-\cancel{\frac{\partial\Gamma^{(h-2)(\alpha)}_{1(\alpha)1(\alpha)}}{\partial u^{(l-1)(\alpha)}}}+\bcancel{\frac{\Gamma^{2(\alpha)}_{2(\alpha)2(\alpha)}}{u^{2(\alpha)}}\delta^{h}_{l-1}}+\cancel{\cancel{\frac{u^{h(\alpha)}}{u^{2(\alpha)}}\frac{\partial\Gamma^{2(\alpha)}_{2(\alpha)2(\alpha)}}{\partial u^{(l-1)(\alpha)}}}\,}\notag\\&+\frac{1}{u^{2(\alpha)}}\overset{h-3}{\underset{s=1}{\sum}}\bigg(\frac{\partial\Gamma^{(s+2)(\alpha)}_{2(\alpha)2(\alpha)}}{\partial u^{(l-1)(\alpha)}}-\frac{\partial\Gamma^{s(\alpha)}_{1(\alpha)1(\alpha)}}{\partial u^{(l-1)(\alpha)}}\bigg)u^{(h-s)(\alpha)}
		\notag\\&+\bcancel{\bcancel{\frac{1}{u^{2(\alpha)}}\big(\Gamma^{(h-l+3)(\alpha)}_{2(\alpha)2(\alpha)}-\Gamma^{(h-l+1)(\alpha)}_{1(\alpha)1(\alpha)}\big)}}\notag\\&=-\frac{1}{u^{2(\alpha)}}\overset{h-2}{\underset{s=2}{\sum}}\bigg(\frac{\partial\Gamma^{(s+1)(\alpha)}_{2(\alpha)2(\alpha)}}{\partial u^{(l-1)(\alpha)}}-\frac{\partial\Gamma^{(s-1)(\alpha)}_{1(\alpha)1(\alpha)}}{\partial u^{(l-1)(\alpha)}}\bigg)u^{(h-s+1)(\alpha)}\notag\\&+\frac{1}{u^{2(\alpha)}}\overset{h-3}{\underset{s=1}{\sum}}\bigg(\frac{\partial\Gamma^{(s+2)(\alpha)}_{2(\alpha)2(\alpha)}}{\partial u^{(l-1)(\alpha)}}-\frac{\partial\Gamma^{s(\alpha)}_{1(\alpha)1(\alpha)}}{\partial u^{(l-1)(\alpha)}}\bigg)u^{(h-s)(\alpha)}=0
		\notag
	\end{align}
	by changing the variable in one of the summations. This proves \eqref{Lemma7.1_4_alpha_thind}, thus \eqref{Lemma7.1_4} has been proved for all $\delta\in\{1,\dots,r\}$, $k\in\{2,\dots,m_\alpha\}$ and $l\in\{3,\dots,m_\delta\}$.
	
	Let us now prove that \eqref{Lemma7.1eq} holds for $l=2$ as well, that is
	\begin{equation}
		\frac{\partial\Gamma^{k(\alpha)}_{i(\alpha)j(\beta)}}{\partial u^{2(\alpha)}}=\frac{\partial\Gamma^{(k-1)(\alpha)}_{i(\alpha)j(\beta)}}{\partial u^{1(\alpha)}}
		\notag
	\end{equation}
	for all $k\in\{2,\dots,m_\gamma\}$. As in the proof of Case $2$, this is trivially true when $j\geq2$ or $k<i$. We are going to prove (by induction over $k$) that
	\begin{equation}
		\label{Lemma7.1_Moreover}
		\frac{\partial\Gamma^{k(\alpha)}_{i(\alpha)1(\beta)}}{\partial u^{2(\alpha)}}=\frac{\partial\Gamma^{(k-1)(\alpha)}_{i(\alpha)1(\beta)}}{\partial u^{1(\alpha)}}
	\end{equation}
	for $k\geq i$.
	If $k=i$ we get
	\begin{align}
		\frac{\partial\Gamma^{k(\alpha)}_{i(\alpha)1(\beta)}}{\partial u^{2(\alpha)}}&=\frac{\partial\Gamma^{i(\alpha)}_{i(\alpha)1(\beta)}}{\partial u^{2(\alpha)}}=\frac{\partial\Gamma^{1(\alpha)}_{1(\alpha)1(\beta)}}{\partial u^{2(\alpha)}}=0
		\notag
	\end{align}
	(as $\Gamma^{1(\alpha)}_{1(\alpha)1(\beta)}$ does not depend on $u^{2(\alpha)}$) and
	\begin{align}
		\frac{\partial\Gamma^{(k-1)(\alpha)}_{i(\alpha)1(\beta)}}{\partial u^{1(\alpha)}}&=\frac{\partial\Gamma^{(i-1)(\alpha)}_{i(\alpha)1(\beta)}}{\partial u^{1(\alpha)}}=0
		\notag
	\end{align}
	(as $\Gamma^{(i-1)(\alpha)}_{i(\alpha)1(\beta)}=0$). This proves \eqref{Lemma7.1_Moreover} for $k=i$. Let us suppose that \eqref{Lemma7.1_Moreover} holds whenever $k\leq i+h-1$, for a given integer $h\geq1$. Let us show that it holds when $k=i+h$ as well, that is
	\begin{equation}
		\label{Lemma7.1_Moreover_ind}
		\frac{\partial\Gamma^{(i+h)(\alpha)}_{i(\alpha)1(\beta)}}{\partial u^{2(\alpha)}}=\frac{\partial\Gamma^{(i+h-1)(\alpha)}_{i(\alpha)1(\beta)}}{\partial u^{1(\alpha)}}.
	\end{equation}
	The left-hand side of \eqref{Lemma7.1_Moreover_ind} reads
	\begin{align}
		\frac{\partial\Gamma^{(i+h)(\alpha)}_{i(\alpha)1(\beta)}}{\partial u^{2(\alpha)}}&=\frac{\partial}{\partial u^{2(\alpha)}}\bigg(-\frac{1}{u^{1(\alpha)}-u^{1(\beta)}}\,\overset{h+1}{\underset{s=2}{\sum}}\,\Gamma^{(i+h-s+1)(\alpha)}_{i(\alpha)1(\beta)}\,u^{s(\alpha)}\bigg)\notag\\&=-\frac{1}{u^{1(\alpha)}-u^{1(\beta)}}\,\overset{h+1}{\underset{s=2}{\sum}}\,\frac{\partial\Gamma^{(i+h-s+1)(\alpha)}_{i(\alpha)1(\beta)}}{\partial u^{2(\alpha)}}\,u^{s(\alpha)}\notag\\&-\frac{1}{u^{1(\alpha)}-u^{1(\beta)}}\,\Gamma^{(i+h-1)(\alpha)}_{i(\alpha)1(\beta)}
		\notag
	\end{align}
	where in the first summation only the terms for $s\leq h$ survive, as for $s=h+1$ we get
	\begin{align}
		\frac{\partial\Gamma^{i(\alpha)}_{i(\alpha)1(\beta)}}{\partial u^{2(\alpha)}}&=\frac{\partial\Gamma^{1(\alpha)}_{1(\alpha)1(\beta)}}{\partial u^{2(\alpha)}}=0\text{,}
		\notag
	\end{align}
	and (by the inductive hypothesis)
	\begin{align}
		\frac{\partial\Gamma^{(i+h-s+1)(\alpha)}_{i(\alpha)1(\beta)}}{\partial u^{2(\alpha)}}&=\frac{\partial\Gamma^{(i+h-s)(\alpha)}_{i(\alpha)1(\beta)}}{\partial u^{1(\alpha)}}
		\notag
	\end{align}
	for each $2\leq s\leq h$. Then
	\begin{align}
		\frac{\partial\Gamma^{(i+h)(\alpha)}_{i(\alpha)1(\beta)}}{\partial u^{2(\alpha)}}&=-\frac{1}{u^{1(\alpha)}-u^{1(\beta)}}\,\overset{h}{\underset{s=2}{\sum}}\,\frac{\partial\Gamma^{(i+h-s)(\alpha)}_{i(\alpha)1(\beta)}}{\partial u^{1(\alpha)}}\,u^{s(\alpha)}\notag\\&-\frac{1}{u^{1(\alpha)}-u^{1(\beta)}}\,\Gamma^{(i+h-1)(\alpha)}_{i(\alpha)1(\beta)}.
		\notag
	\end{align}
	The right-hand side of \eqref{Lemma7.1_Moreover_ind} reads
	\begin{align}
		\frac{\partial\Gamma^{(i+h-1)(\alpha)}_{i(\alpha)1(\beta)}}{\partial u^{1(\alpha)}}&=\frac{\partial}{\partial u^{1(\alpha)}}\bigg(-\frac{1}{u^{1(\alpha)}-u^{1(\beta)}}\,\overset{h}{\underset{s=2}{\sum}}\,\Gamma^{(i+h-s)(\alpha)}_{i(\alpha)1(\beta)}\,u^{s(\alpha)}\bigg)\notag\\&=\frac{1}{(u^{1(\alpha)}-u^{1(\beta)})^2}\,\overset{h}{\underset{s=2}{\sum}}\,\Gamma^{(i+h-s)(\alpha)}_{i(\alpha)1(\beta)}\,u^{s(\alpha)}\notag\\&-\frac{1}{u^{1(\alpha)}-u^{1(\beta)}}\,\overset{h}{\underset{s=2}{\sum}}\,\frac{\partial\Gamma^{(i+h-s)(\alpha)}_{i(\alpha)1(\beta)}}{\partial u^{1(\alpha)}}\,u^{s(\alpha)}
		\notag
	\end{align}
	thus their difference is
	\begin{align}
		\frac{\partial\Gamma^{(i+h)(\alpha)}_{i(\alpha)1(\beta)}}{\partial u^{2(\alpha)}}-\frac{\partial\Gamma^{(i+h-1)(\alpha)}_{i(\alpha)1(\beta)}}{\partial u^{1(\alpha)}}&=-\cancel{\frac{1}{u^{1(\alpha)}-u^{1(\beta)}}\,\overset{h}{\underset{s=2}{\sum}}\,\frac{\partial\Gamma^{(i+h-s)(\alpha)}_{i(\alpha)1(\beta)}}{\partial u^{1(\alpha)}}\,u^{s(\alpha)}}\notag\\&-\frac{1}{u^{1(\alpha)}-u^{1(\beta)}}\,\Gamma^{(i+h-1)(\alpha)}_{i(\alpha)1(\beta)}\notag\\&-\frac{1}{(u^{1(\alpha)}-u^{1(\beta)})^2}\,\overset{h}{\underset{s=2}{\sum}}\,\Gamma^{(i+h-s)(\alpha)}_{i(\alpha)1(\beta)}\,u^{s(\alpha)}\notag\\&+\cancel{\frac{1}{u^{1(\alpha)}-u^{1(\beta)}}\,\overset{h}{\underset{s=2}{\sum}}\,\frac{\partial\Gamma^{(i+h-s)(\alpha)}_{i(\alpha)1(\beta)}}{\partial u^{1(\alpha)}}\,u^{s(\alpha)}}\notag\\&=\frac{1}{(u^{1(\alpha)}-u^{1(\beta)})^2}\,\overset{h}{\underset{s=2}{\sum}}\,\Gamma^{(i+h-s)(\alpha)}_{i(\alpha)1(\beta)}\,u^{s(\alpha)}\notag\\&-\frac{1}{(u^{1(\alpha)}-u^{1(\beta)})^2}\,\overset{h}{\underset{s=2}{\sum}}\,\Gamma^{(i+h-s)(\alpha)}_{i(\alpha)1(\beta)}\,u^{s(\alpha)}=0.
		\notag
	\end{align}
\end{proof}
\begin{lemma}
	For each $\alpha$, $\beta\in\{1,\dots,r\}$, $i\in\{1,\dots,m_\alpha\}$, $j\in\{1,\dots,m_\beta\}$ we have
	\begin{align}
		\overset{n}{\underset{k=1}{\sum}}\,\Gamma^{i(\alpha)}_{j(\beta)k}\,u^k=\begin{cases}
			0\qquad&\text{if }i\neq j\\
			-\delta^\alpha_\beta\,\underset{\sigma\neq\alpha}{\sum}\,m_\sigma\epsilon_\sigma+(1-\delta^\alpha_\beta)\,m_\beta\epsilon_\beta\qquad&\text{if }i=j=1\\
			-\delta^\alpha_\beta\,\overset{r}{\underset{\tau=1}{\sum}}\,m_\tau\epsilon_\tau\qquad&\text{if }i=j\neq1.
		\end{cases}
		\notag
	\end{align}
\end{lemma}
\begin{proof}
	Let us first consider $\alpha\neq\beta$ and prove $\overset{n}{\underset{k=1}{\sum}}\,\Gamma^{i(\alpha)}_{j(\beta)k}\,u^k=\delta^i_1\delta_j^1\,m_\beta\epsilon_\beta$. We have
	\begin{align}
		\overset{n}{\underset{k=1}{\sum}}\,\Gamma^{i(\alpha)}_{j(\beta)k}\,u^k&=\overset{r}{\underset{\gamma=1}{\sum}}\,\overset{m_\gamma}{\underset{k=1}{\sum}}\,\Gamma^{i(\alpha)}_{j(\beta)k(\gamma)}\,u^{k(\gamma)}\overset{Prop.8.4}{=}\overset{m_\alpha}{\underset{k=1}{\sum}}\,\Gamma^{i(\alpha)}_{j(\beta)k(\alpha)}\,u^{k(\alpha)}+\overset{m_\beta}{\underset{k=1}{\sum}}\,\Gamma^{i(\alpha)}_{j(\beta)k(\beta)}\,u^{k(\beta)}
		\notag
	\end{align}
	which vanishes automatically when $j\geq2$. Let us then fix $j=1$, thus
	\begin{align}
		\overset{n}{\underset{k=1}{\sum}}\,\Gamma^{i(\alpha)}_{j(\beta)k}\,u^k&=\overset{n}{\underset{k=1}{\sum}}\,\Gamma^{i(\alpha)}_{1(\beta)k}\,u^k=\overset{m_\alpha}{\underset{k=1}{\sum}}\,\Gamma^{i(\alpha)}_{1(\beta)k(\alpha)}\,u^{k(\alpha)}+\overset{m_\beta}{\underset{k=1}{\sum}}\,\Gamma^{i(\alpha)}_{1(\beta)k(\beta)}\,u^{k(\beta)}\notag\\&\overset{Prop.8.4}{=}\overset{i}{\underset{k=1}{\sum}}\,\Gamma^{i(\alpha)}_{1(\beta)k(\alpha)}\,u^{k(\alpha)}-\overset{m_\beta}{\underset{k=1}{\sum}}\,\Gamma^{i(\alpha)}_{1(\alpha)k(\beta)}\,\delta^1_{k}\,u^{k(\beta)}\notag\\&=\overset{i}{\underset{k=1}{\sum}}\,\Gamma^{i(\alpha)}_{1(\beta)k(\alpha)}\,u^{k(\alpha)}-\Gamma^{i(\alpha)}_{1(\alpha)1(\beta)}\,u^{1(\beta)}\notag\\&=\overset{i}{\underset{k=2}{\sum}}\,\Gamma^{i(\alpha)}_{1(\beta)k(\alpha)}\,u^{k(\alpha)}+\Gamma^{i(\alpha)}_{1(\beta)1(\alpha)}(u^{1(\alpha)}-u^{1(\beta)})\notag\\&\overset{Prop.8.4}{=}
		\begin{cases}
			\Gamma^{1(\alpha)}_{1(\beta)1(\alpha)}(u^{1(\alpha)}-u^{1(\beta)})\qquad&\text{if }i=1\\
			\overset{i}{\underset{k=2}{\sum}}\,\Gamma^{(i-k+1)(\alpha)}_{1(\beta)1(\alpha)}\,u^{k(\alpha)}-\overset{i}{\underset{s=2}{\sum}}\,\Gamma^{(i-s+1)(\alpha)}_{1(\beta)1(\alpha)}\,u^{s(\alpha)}\qquad&\text{if }i\geq2
		\end{cases}
		\notag\\&=
		\begin{cases}
			m_\beta\epsilon_\beta\qquad&\text{if }i=1\\
			0\qquad&\text{if }i\geq2
		\end{cases}
		=m_\beta\epsilon_\beta\,\delta^i_1.
		\notag
	\end{align}
	
	In order to complete the proof we must show that
	\begin{align}
		\overset{n}{\underset{k=1}{\sum}}\,\Gamma^{i(\alpha)}_{j(\alpha)k}\,u^k=\begin{cases}
			0\qquad&\text{if }i\neq j\\
			-\underset{\sigma\neq\alpha}{\sum}\,m_\sigma\epsilon_\sigma\qquad&\text{if }i=j=1\\
			-\overset{r}{\underset{\tau=1}{\sum}}\,m_\tau\epsilon_\tau\qquad&\text{if }i=j\neq1.
		\end{cases}
		\notag
	\end{align}
	Let us first consider the case where $i\neq j$. Without loss of generality we assume $i>j$, as $\Gamma^{i(\alpha)}_{j(\alpha)k}=0$ trivially whenever $i<j$ by Proposition 8.4. We have
	\begin{align}
		\overset{n}{\underset{k=1}{\sum}}\,\Gamma^{i(\alpha)}_{j(\alpha)k}\,u^k&=\overset{m_\alpha}{\underset{k=1}{\sum}}\,\Gamma^{i(\alpha)}_{j(\alpha)k(\alpha)}\,u^{k(\alpha)}+\underset{\sigma\neq\alpha}{\sum}\,\overset{m_\sigma}{\underset{k=1}{\sum}}\,\Gamma^{i(\alpha)}_{j(\alpha)k(\sigma)}\,u^{k(\sigma)}\delta^1_k\notag\\&=\Gamma^{i(\alpha)}_{j(\alpha)1(\alpha)}\,u^{1(\alpha)}+\overset{m_\alpha}{\underset{k=2}{\sum}}\,\Gamma^{i(\alpha)}_{j(\alpha)k(\alpha)}\,u^{k(\alpha)}+\underset{\sigma\neq\alpha}{\sum}\,\Gamma^{i(\alpha)}_{j(\alpha)1(\sigma)}\,u^{1(\sigma)}\notag\\&\overset{Prop.8.1}{=}-\underset{\sigma\neq\alpha}{\sum}\,\Gamma^{i(\alpha)}_{j(\alpha)1(\sigma)}\,u^{1(\alpha)}+\overset{m_\alpha}{\underset{k=2}{\sum}}\,\Gamma^{i(\alpha)}_{j(\alpha)k(\alpha)}\,u^{k(\alpha)}+\underset{\sigma\neq\alpha}{\sum}\,\Gamma^{i(\alpha)}_{j(\alpha)1(\sigma)}\,u^{1(\sigma)}\notag\\&\overset{Prop.8.4}{=}-\underset{\sigma\neq\alpha}{\sum}\,\Gamma^{i(\alpha)}_{j(\alpha)1(\sigma)}\,(u^{1(\alpha)}-u^{1(\sigma)})+\overset{i-j+2}{\underset{k=2}{\sum}}\,\Gamma^{(i-j-k+4)(\alpha)}_{2(\alpha)2(\alpha)}\,u^{k(\alpha)}
		\notag
	\end{align}
	where
	\begin{align}
		\overset{i-j+2}{\underset{k=2}{\sum}}\,\Gamma^{i(\alpha)}_{j(\alpha)k(\alpha)}\,u^{k(\alpha)}&=\Gamma^{(i-j+2)(\alpha)}_{2(\alpha)2(\alpha)}\,u^{2(\alpha)}+\overset{i-j+2}{\underset{k=3}{\sum}}\,\Gamma^{(i-j-k+4)(\alpha)}_{2(\alpha)2(\alpha)}\,u^{k(\alpha)}\notag\\&\overset{Prop.8.4}{=}\bigg(\Gamma^{(i-j)(\alpha)}_{1(\alpha)1(\alpha)}-\Gamma^{2(\alpha)}_{2(\alpha)2(\alpha)}\frac{u^{(i-j+2)(\alpha)}}{u^{2(\alpha)}}\notag\\&-\frac{1}{u^{2(\alpha)}}\overset{i-j-1}{\underset{l=1}{\sum}}\big(\Gamma^{(l+2)(\alpha)}_{2(\alpha)2(\alpha)}-\Gamma^{l(\alpha)}_{1(\alpha)1(\alpha)}\big)u^{(i-j+2-l)(\alpha)}\bigg)\,u^{2(\alpha)}\notag\\&+\overset{i-j+2}{\underset{k=3}{\sum}}\,\Gamma^{(i-j-k+4)(\alpha)}_{2(\alpha)2(\alpha)}\,u^{k(\alpha)}\notag\\&=\Gamma^{(i-j)(\alpha)}_{1(\alpha)1(\alpha)}\,u^{2(\alpha)}-\Gamma^{2(\alpha)}_{2(\alpha)2(\alpha)}\,u^{(i-j+2)(\alpha)}\notag\\&-\overset{i-j-1}{\underset{l=1}{\sum}}\Gamma^{(l+2)(\alpha)}_{2(\alpha)2(\alpha)}\,u^{(i-j+2-l)(\alpha)}+\overset{i-j-1}{\underset{l=1}{\sum}}\Gamma^{l(\alpha)}_{1(\alpha)1(\alpha)}\,u^{(i-j+2-l)(\alpha)}\notag\\&+\overset{i-j+2}{\underset{k=3}{\sum}}\,\Gamma^{(i-j-k+4)(\alpha)}_{2(\alpha)2(\alpha)}\,u^{k(\alpha)}
		\notag\\&=\Gamma^{(i-j)(\alpha)}_{1(\alpha)1(\alpha)}\,u^{2(\alpha)}-\Gamma^{2(\alpha)}_{2(\alpha)2(\alpha)}\,u^{(i-j+2)(\alpha)}\notag\\&-\overset{i-j+1}{\underset{k=3}{\sum}}\Gamma^{(i-j-k+4)(\alpha)}_{2(\alpha)2(\alpha)}\,u^{k(\alpha)}+\overset{i-j+1}{\underset{k=3}{\sum}}\Gamma^{(i-j-k+2)(\alpha)}_{1(\alpha)1(\alpha)}\,u^{k(\alpha)}\notag\\&+\overset{i-j+2}{\underset{k=3}{\sum}}\,\Gamma^{(i-j-k+4)(\alpha)}_{2(\alpha)2(\alpha)}\,u^{k(\alpha)}
		\notag\\&=-\cancel{\Gamma^{2(\alpha)}_{2(\alpha)2(\alpha)}\,u^{(i-j+2)(\alpha)}}+\overset{i-j+1}{\underset{k=2}{\sum}}\Gamma^{(i-j-k+2)(\alpha)}_{1(\alpha)1(\alpha)}\,u^{k(\alpha)}\notag\\&+\cancel{\Gamma^{2(\alpha)}_{2(\alpha)2(\alpha)}\,u^{(i-j+2)(\alpha)}}=\overset{i-j+1}{\underset{k=2}{\sum}}\Gamma^{(i-j-k+2)(\alpha)}_{1(\alpha)1(\alpha)}\,u^{k(\alpha)}.
		\notag
	\end{align}
	It follows that
	\begin{align}
		\overset{n}{\underset{k=1}{\sum}}\,\Gamma^{i(\alpha)}_{j(\alpha)k}\,u^k&=-\underset{\sigma\neq\alpha}{\sum}\,\Gamma^{i(\alpha)}_{j(\alpha)1(\sigma)}\,(u^{1(\alpha)}-u^{1(\sigma)})+\overset{i-j+1}{\underset{k=2}{\sum}}\Gamma^{(i-j-k+2)(\alpha)}_{1(\alpha)1(\alpha)}\,u^{k(\alpha)}\notag\\&=\underset{\sigma\neq\alpha}{\sum}\,\overset{i-j+1}{\underset{s=2}{\sum}}\,\Gamma^{(i-j-s+2)(\alpha)}_{1(\alpha)1(\sigma)}\,u^{s(\alpha)}-\underset{\sigma\neq\alpha}{\sum}\,\overset{i-j+1}{\underset{k=2}{\sum}}\Gamma^{(i-j-k+2)(\alpha)}_{1(\alpha)1(\sigma)}\,u^{k(\alpha)}=0.
		\notag
	\end{align}
	Let us now consider the case where $i=j$. We have
	\begin{align}
		\overset{n}{\underset{k=1}{\sum}}\,\Gamma^{i(\alpha)}_{i(\alpha)k}\,u^k&=\overset{m_\alpha}{\underset{k=1}{\sum}}\,\Gamma^{i(\alpha)}_{i(\alpha)k(\alpha)}\,u^{k(\alpha)}+\underset{\sigma\neq\alpha}{\sum}\,\overset{m_\sigma}{\underset{k=1}{\sum}}\,\Gamma^{i(\alpha)}_{i(\alpha)k(\sigma)}\,u^{k(\sigma)}\delta^1_k\notag\\&=\Gamma^{i(\alpha)}_{i(\alpha)1(\alpha)}\,u^{1(\alpha)}+\overset{m_\alpha}{\underset{k=2}{\sum}}\,\Gamma^{i(\alpha)}_{i(\alpha)k(\alpha)}\,u^{k(\alpha)}+\underset{\sigma\neq\alpha}{\sum}\,\Gamma^{i(\alpha)}_{i(\alpha)1(\sigma)}\,u^{1(\sigma)}\notag\\&=-\underset{\sigma\neq\alpha}{\sum}\,\Gamma^{i(\alpha)}_{i(\alpha)1(\sigma)}(u^{1(\alpha)}-u^{1(\sigma)})+\overset{m_\alpha}{\underset{k=2}{\sum}}\,\Gamma^{i(\alpha)}_{i(\alpha)k(\alpha)}\,u^{k(\alpha)}\notag\\&=-\underset{\sigma\neq\alpha}{\sum}\,\Gamma^{1(\alpha)}_{1(\alpha)1(\sigma)}(u^{1(\alpha)}-u^{1(\sigma)})+\overset{m_\alpha}{\underset{k=2}{\sum}}\,\Gamma^{i(\alpha)}_{i(\alpha)k(\alpha)}\,u^{k(\alpha)}.
		\notag
	\end{align}
	If $i=1$ then
	\begin{align}
		\overset{n}{\underset{k=1}{\sum}}\,\Gamma^{1(\alpha)}_{1(\alpha)k}\,u^k&=-\underset{\sigma\neq\alpha}{\sum}\,\Gamma^{1(\alpha)}_{1(\alpha)1(\sigma)}(u^{1(\alpha)}-u^{1(\sigma)})+\overset{m_\alpha}{\underset{k=2}{\sum}}\,\Gamma^{1(\alpha)}_{1(\alpha)k(\alpha)}\,u^{k(\alpha)}
		\notag
	\end{align}
	where
	\begin{align}
		\Gamma^{1(\alpha)}_{1(\alpha)k(\alpha)}&=-\underset{\sigma\neq\alpha}{\sum}\,\Gamma^{1(\alpha)}_{1(\sigma)k(\alpha)}\overset{Prop.8.2}{=}0
		\notag
	\end{align}
	for every $k\geq2$. It follows that
	\begin{align}
		\overset{n}{\underset{k=1}{\sum}}\,\Gamma^{1(\alpha)}_{1(\alpha)k}\,u^k&=-\underset{\sigma\neq\alpha}{\sum}\,\Gamma^{1(\alpha)}_{1(\alpha)1(\sigma)}(u^{1(\alpha)}-u^{1(\sigma)})=-\underset{\sigma\neq\alpha}{\sum}\,m_\sigma\epsilon_\sigma.
		\notag
	\end{align}
	If $i\neq1$ then
	\begin{align}
		\overset{n}{\underset{k=1}{\sum}}\,\Gamma^{i(\alpha)}_{i(\alpha)k}\,u^k&=-\underset{\sigma\neq\alpha}{\sum}\,\Gamma^{1(\alpha)}_{1(\alpha)1(\sigma)}(u^{1(\alpha)}-u^{1(\sigma)})+\overset{m_\alpha}{\underset{k=2}{\sum}}\,\Gamma^{(4-k)(\alpha)}_{2(\alpha)2(\alpha)}\,\delta_k^2\,u^{k(\alpha)}\notag\\&=-\underset{\sigma\neq\alpha}{\sum}\,m_\sigma\epsilon_\sigma+\Gamma^{2(\alpha)}_{2(\alpha)2(\alpha)}\,u^{2(\alpha)}\notag\\&=-\underset{\sigma\neq\alpha}{\sum}\,m_\sigma\epsilon_\sigma-m_\alpha\epsilon_\alpha=-\overset{r}{\underset{\tau=1}{\sum}}\,m_\tau\epsilon_\tau.
		\notag
	\end{align}
\end{proof}
\begin{lemma}
	For each $\alpha$, $\beta\in\{1,\dots,r\}$, $i\in\{1,\dots,m_\alpha\}$, $j\in\{1,\dots,m_\beta\}$ we have
	\begin{align}
		\nabla_{j(\beta)}E^{i(\alpha)}=
		\begin{cases}
		0\qquad&\text{if }i\neq j\\
		\delta^\alpha_\beta\bigg(1-\underset{\sigma\neq\alpha}{\sum}\,m_\sigma\epsilon_\sigma\bigg)+(1-\delta^\alpha_\beta)\,m_\beta\epsilon_\beta\qquad&\text{if }i=j=1\\
		\delta^\alpha_\beta\bigg(1-\overset{r}{\underset{\tau=1}{\sum}}\,m_\tau\epsilon_\tau\bigg)\qquad&\text{if }i=j\neq1.
		\end{cases}
		\notag
	\end{align}
\end{lemma}
\begin{proof}
	Let us recall that
	\begin{align}
		\nabla_{j(\beta)}E^{i(\alpha)}&=\partial_{j(\beta)}E^{i(\alpha)}+\overset{n}{\underset{k=1}{\sum}}\,\Gamma^{i(\alpha)}_{j(\beta)k}\,E^k=\delta^{i(\alpha)}_{j(\beta)}+\overset{n}{\underset{k=1}{\sum}}\,\Gamma^{i(\alpha)}_{j(\beta)k}\,u^k.
		\notag		
	\end{align}
	If $\alpha\neq\beta$ then
	\begin{align}
		\nabla_{j(\beta)}E^{i(\alpha)}&=\delta^i_1\delta_j^1\,m_\beta\epsilon_\beta
		\notag		
	\end{align}
	by Lemma 8.6. Let us now consider the case where $\alpha=\beta$. Here we have
	\begin{align}
		\nabla_{j(\alpha)}E^{i(\alpha)}&=\delta^i_j+\overset{n}{\underset{k=1}{\sum}}\,\Gamma^{i(\alpha)}_{j(\alpha)k}\,u^k
		\notag		
	\end{align}
	which vanishes if $i\neq j$ by Lemma 8.6. Let now consider $i=j$. By virtue of Lemma 8.6, we get
	\begin{align}
		\nabla_{i(\alpha)}E^{i(\alpha)}&=1+\overset{n}{\underset{k=1}{\sum}}\,\Gamma^{i(\alpha)}_{i(\alpha)k}\,u^k\notag\\&=
		\begin{cases}
			1-\underset{\sigma\neq\alpha}{\sum}\,m_\sigma\epsilon_\sigma\qquad&\text{if }i=1\\
			1-\overset{r}{\underset{\tau=1}{\sum}}\,m_\tau\epsilon_\tau\qquad&\text{if }i\neq1.
		\end{cases}
		\notag		
	\end{align}
\end{proof}

\begin{lemma}
	For each $\alpha\in\{1,\dots,r\}$ we have
	\begin{align}
		A^{l(\alpha)}&:=\Gamma^{2(\alpha)}_{2(\alpha)2(\alpha)}\bigg(\frac{u^{3(\alpha)}}{u^{2(\alpha)}}\,u^{l(\alpha)}-u^{(l+1)(\alpha)}\bigg)-\overset{l-1}{\underset{s=2}{\sum}}\,\big(\Gamma^{(s+2)(\alpha)}_{2(\alpha)2(\alpha)}-\Gamma^{s(\alpha)}_{1(\alpha)1(\alpha)}\big)\,u^{(l-s+1)(\alpha)}=0
		\label{atlas}
	\end{align}
	for every $l\in\{3,\dots,m_\alpha-1\}$.
\end{lemma}
\begin{proof}
	We will proceed by induction over $l$. For $l=3$ we get
	\begin{align}
		A^{3(\alpha)}&=\Gamma^{2(\alpha)}_{2(\alpha)2(\alpha)}\bigg(\frac{u^{3(\alpha)}}{u^{2(\alpha)}}\,u^{3(\alpha)}-u^{4(\alpha)}\bigg)-\big(\Gamma^{4(\alpha)}_{2(\alpha)2(\alpha)}-\Gamma^{2(\alpha)}_{1(\alpha)1(\alpha)}\big)\,u^{2(\alpha)}
		\notag
	\end{align}
	where
	\begin{align}
		\Gamma^{4(\alpha)}_{2(\alpha)2(\alpha)}-\Gamma^{2(\alpha)}_{1(\alpha)1(\alpha)}&=-\Gamma^{2(\alpha)}_{2(\alpha)2(\alpha)}\,\frac{u^{4(\alpha)}}{u^{2(\alpha)}}-\frac{1}{u^{2(\alpha)}}\,\big(\Gamma^{3(\alpha)}_{2(\alpha)2(\alpha)}-\Gamma^{1(\alpha)}_{1(\alpha)1(\alpha)}\big)\,u^{3(\alpha)}\notag\\&=-\Gamma^{2(\alpha)}_{2(\alpha)2(\alpha)}\,\frac{u^{4(\alpha)}}{u^{2(\alpha)}}+\frac{u^{3(\alpha)}}{u^{2(\alpha)}}\,\Gamma^{2(\alpha)}_{2(\alpha)2(\alpha)}\frac{u^{3(\alpha)}}{u^{2(\alpha)}}.
		\notag
	\end{align}
	It follows that
	\begin{align}
		A^{3(\alpha)}&=\Gamma^{2(\alpha)}_{2(\alpha)2(\alpha)}\bigg(\cancel{\frac{u^{3(\alpha)}}{u^{2(\alpha)}}\,u^{3(\alpha)}}-\bcancel{u^{4(\alpha)}}\bigg)-\bigg(-\bcancel{\Gamma^{2(\alpha)}_{2(\alpha)2(\alpha)}\,u^{4(\alpha)}}\notag\\&+\cancel{\frac{u^{3(\alpha)}}{u^{2(\alpha)}}\,\Gamma^{2(\alpha)}_{2(\alpha)2(\alpha)}\,u^{3(\alpha)}}\bigg)=0.
		\notag
	\end{align}
	Given an integer $h\geq1$, $h\leq m_\alpha-2$, let us suppose that
	\begin{align}
		\label{LemmaA_ind}
		A^{l(\alpha)}=0\qquad\text{for all }3\leq l\leq h
	\end{align}
	and show $A^{(h+1)(\alpha)}=0$. We have
	\begin{align}
		A^{(h+1)(\alpha)}&=\Gamma^{2(\alpha)}_{2(\alpha)2(\alpha)}\bigg(\frac{u^{3(\alpha)}}{u^{2(\alpha)}}\,u^{(h+1)(\alpha)}-u^{(h+2)(\alpha)}\bigg)-\overset{h}{\underset{s=2}{\sum}}\,\big(\Gamma^{(s+2)(\alpha)}_{2(\alpha)2(\alpha)}-\Gamma^{s(\alpha)}_{1(\alpha)1(\alpha)}\big)\,u^{(h-s+2)(\alpha)}\notag\\&=\Gamma^{2(\alpha)}_{2(\alpha)2(\alpha)}\bigg(\frac{u^{3(\alpha)}}{u^{2(\alpha)}}\,u^{(h+1)(\alpha)}-u^{(h+2)(\alpha)}\bigg)-\overset{h-1}{\underset{s=2}{\sum}}\,\big(\Gamma^{(s+2)(\alpha)}_{2(\alpha)2(\alpha)}-\Gamma^{s(\alpha)}_{1(\alpha)1(\alpha)}\big)\,u^{(h-s+2)(\alpha)}\notag\\&-\big(\Gamma^{(h+2)(\alpha)}_{2(\alpha)2(\alpha)}-\Gamma^{h(\alpha)}_{1(\alpha)1(\alpha)}\big)\,u^{2(\alpha)}
		\notag
	\end{align}
	where
	\begin{align}
		\Gamma^{(h+2)(\alpha)}_{2(\alpha)2(\alpha)}-\Gamma^{h(\alpha)}_{1(\alpha)1(\alpha)}&=-\Gamma^{2(\alpha)}_{2(\alpha)2(\alpha)}\,\frac{u^{(h+2)(\alpha)}}{u^{2(\alpha)}}\notag\\&-\frac{1}{u^{2(\alpha)}}\,\overset{h-1}{\underset{l=1}{\sum}}\,\big(\Gamma^{(l+2)(\alpha)}_{2(\alpha)2(\alpha)}-\Gamma^{l(\alpha)}_{1(\alpha)1(\alpha)}\big)\,u^{(h-l+2)(\alpha)}.
		\notag
	\end{align}
	It follows that
	\begin{align}
		A^{(h+1)(\alpha)}&=\Gamma^{2(\alpha)}_{2(\alpha)2(\alpha)}\bigg(\frac{u^{3(\alpha)}}{u^{2(\alpha)}}\,u^{(h+1)(\alpha)}-\cancel{u^{(h+2)(\alpha)}}\bigg)-\overset{h-1}{\underset{s=2}{\sum}}\,\big(\Gamma^{(s+2)(\alpha)}_{2(\alpha)2(\alpha)}-\Gamma^{s(\alpha)}_{1(\alpha)1(\alpha)}\big)\,u^{(h-s+2)(\alpha)}\notag\\&+\cancel{\Gamma^{2(\alpha)}_{2(\alpha)2(\alpha)}\,u^{(h+2)(\alpha)}}+\overset{h-1}{\underset{l=1}{\sum}}\,\big(\Gamma^{(l+2)(\alpha)}_{2(\alpha)2(\alpha)}-\Gamma^{l(\alpha)}_{1(\alpha)1(\alpha)}\big)\,u^{(h-l+2)(\alpha)}\notag\\&=\Gamma^{2(\alpha)}_{2(\alpha)2(\alpha)}\frac{u^{3(\alpha)}}{u^{2(\alpha)}}\,u^{(h+1)(\alpha)}-\cancel{\overset{h-1}{\underset{s=2}{\sum}}\,\big(\Gamma^{(s+2)(\alpha)}_{2(\alpha)2(\alpha)}-\Gamma^{s(\alpha)}_{1(\alpha)1(\alpha)}\big)\,u^{(h-s+2)(\alpha)}}\notag\\&+\big(\Gamma^{3(\alpha)}_{2(\alpha)2(\alpha)}-\Gamma^{1(\alpha)}_{1(\alpha)1(\alpha)}\big)\,u^{(h+1)(\alpha)}+\cancel{\overset{h-1}{\underset{l=2}{\sum}}\,\big(\Gamma^{(l+2)(\alpha)}_{2(\alpha)2(\alpha)}-\Gamma^{l(\alpha)}_{1(\alpha)1(\alpha)}\big)\,u^{(h-l+2)(\alpha)}}\notag\\&=\Gamma^{2(\alpha)}_{2(\alpha)2(\alpha)}\frac{u^{3(\alpha)}}{u^{2(\alpha)}}\,u^{(h+1)(\alpha)}-\Gamma^{2(\alpha)}_{2(\alpha)2(\alpha)}\frac{u^{3(\alpha)}}{u^{2(\alpha)}}\,u^{(h+1)(\alpha)}=0.
		\notag
	\end{align}
\end{proof}
\begin{lemma}
	For each $\alpha,\sigma\in\{1,\dots,r\}$ with $\alpha\neq\sigma$ we have
	\begin{align}
		\partial_{1(\sigma)}\big(\Gamma^{(l+2)(\alpha)}_{2(\alpha)2(\alpha)}-\Gamma^{l(\alpha)}_{1(\alpha)1(\alpha)}\big)&=0
		\label{ascendente}
	\end{align}
	for every $l\in\{1,\dots,m_\alpha-2\}$.
\end{lemma}
\begin{proof}
	We will proceed by induction over $l$. For $l=1$ we get
	\begin{align}
		\partial_{1(\sigma)}\big(\Gamma^{3(\alpha)}_{2(\alpha)2(\alpha)}-\Gamma^{1(\alpha)}_{1(\alpha)1(\alpha)}\big)&=\partial_{1(\sigma)}\bigg(-\Gamma^{2(\alpha)}_{2(\alpha)2(\alpha)}\,\frac{u^{3(\alpha)}}{u^{2(\alpha)}}\bigg)=0
		\notag
	\end{align}
	as $\Gamma^{2(\alpha)}_{2(\alpha)2(\alpha)}$ does not depend on $u^{1(\sigma)}$ for any choice of $\sigma$. Given an integer $h\geq2$, $h\leq m_\alpha-2$, let us suppose
	\begin{align}
		\label{LemmaAsc_ind}
		\partial_{1(\sigma)}\big(\Gamma^{(l+2)(\alpha)}_{2(\alpha)2(\alpha)}-\Gamma^{l(\alpha)}_{1(\alpha)1(\alpha)}\big)&=0\qquad\text{for all }l\leq h-1
	\end{align}
	and show $\partial_{1(\sigma)}\big(\Gamma^{(h+2)(\alpha)}_{2(\alpha)2(\alpha)}-\Gamma^{h(\alpha)}_{1(\alpha)1(\alpha)}\big)=0$. We have
	\begin{align}
		\Gamma^{(h+2)(\alpha)}_{2(\alpha)2(\alpha)}-\Gamma^{h(\alpha)}_{1(\alpha)1(\alpha)}&=-\Gamma^{2(\alpha)}_{2(\alpha)2(\alpha)}\,\frac{u^{(h+2)(\alpha)}}{u^{2(\alpha)}}\notag\\&-\frac{1}{u^{2(\alpha)}}\,\overset{h-1}{\underset{l=1}{\sum}}\,\big(\Gamma^{(l+2)(\alpha)}_{2(\alpha)2(\alpha)}-\Gamma^{l(\alpha)}_{1(\alpha)1(\alpha)}\big)\,u^{(h-l+2)(\alpha)}
		\notag
	\end{align}
	where neither $\Gamma^{2(\alpha)}_{2(\alpha)2(\alpha)}\,\frac{u^{(h+2)(\alpha)}}{u^{2(\alpha)}}$ nor
	\begin{align}
		\big(\Gamma^{(l+2)(\alpha)}_{2(\alpha)2(\alpha)}-\Gamma^{l(\alpha)}_{1(\alpha)1(\alpha)}\big)\,u^{(h-l+2)(\alpha)}\qquad\text{for }1\leq l\leq h-1
		\notag
	\end{align}
	(by means of \eqref{LemmaAsc_ind} and of the requirement $\sigma\neq\alpha$) depend on $u^{1(\sigma)}$.
\end{proof}
\begin{lemma}
	Given $\alpha,\beta,\epsilon\in\{1,\dots,r\}$ with $\alpha\neq\beta\neq\epsilon\neq\alpha$ we have
	\begin{align}
		B^{s(\alpha)}_{\beta\epsilon}&:=-\overset{s+1}{\underset{t=1}{\sum}}\,\Gamma^{(s-t+2)(\alpha)}_{1(\epsilon)1(\alpha)}\,\Gamma^{t(\alpha)}_{1(\beta)1(\alpha)}+\Gamma^{(s+1)(\alpha)}_{1(\beta)1(\alpha)}\,\Gamma^{1(\beta)}_{1(\epsilon)1(\beta)}+\Gamma^{(s+1)(\alpha)}_{1(\epsilon)1(\alpha)}\,\Gamma^{1(\epsilon)}_{1(\beta)1(\epsilon)}=0
		\label{Blemma}
	\end{align}
	for every $s\in\{1,\dots,m_\alpha-1\}$.
\end{lemma}
\begin{proof}
	We will proceed by induction over $s$. For $s=1$ we get
	\begin{align}
		B^{1(\alpha)}_{\beta\epsilon}&=-\overset{2}{\underset{t=1}{\sum}}\,\Gamma^{(3-t)(\alpha)}_{1(\epsilon)1(\alpha)}\,\Gamma^{t(\alpha)}_{1(\beta)1(\alpha)}+\Gamma^{2(\alpha)}_{1(\beta)1(\alpha)}\,\Gamma^{1(\beta)}_{1(\epsilon)1(\beta)}+\Gamma^{2(\alpha)}_{1(\epsilon)1(\alpha)}\,\Gamma^{1(\epsilon)}_{1(\beta)1(\epsilon)}\notag\\&=-\Gamma^{2(\alpha)}_{1(\epsilon)1(\alpha)}\,\Gamma^{1(\alpha)}_{1(\beta)1(\alpha)}-\Gamma^{1(\alpha)}_{1(\epsilon)1(\alpha)}\,\Gamma^{2(\alpha)}_{1(\beta)1(\alpha)}+\Gamma^{2(\alpha)}_{1(\beta)1(\alpha)}\,\Gamma^{1(\beta)}_{1(\epsilon)1(\beta)}+\Gamma^{2(\alpha)}_{1(\epsilon)1(\alpha)}\,\Gamma^{1(\epsilon)}_{1(\beta)1(\epsilon)}
		\notag\\&=\Gamma^{2(\alpha)}_{1(\epsilon)1(\alpha)}\big(\Gamma^{1(\epsilon)}_{1(\beta)1(\epsilon)}-\Gamma^{1(\alpha)}_{1(\beta)1(\alpha)}\big)+\Gamma^{2(\alpha)}_{1(\beta)1(\alpha)}\big(\Gamma^{1(\beta)}_{1(\epsilon)1(\beta)}-\Gamma^{1(\alpha)}_{1(\epsilon)1(\alpha)}\big)
		\notag\\&=\Gamma^{2(\alpha)}_{1(\epsilon)1(\alpha)}\,m_\beta\epsilon_\beta\bigg(\frac{1}{u^{1(\epsilon)}-u^{1(\beta)}}-\frac{1}{u^{1(\alpha)}-u^{1(\beta)}}\bigg)\notag\\&+\Gamma^{2(\alpha)}_{1(\beta)1(\alpha)}\,m_\epsilon\epsilon_\epsilon\bigg(\frac{1}{u^{1(\beta)}-u^{1(\epsilon)}}-\frac{1}{u^{1(\alpha)}-u^{1(\epsilon)}}\bigg)
		\notag\\&=-\frac{\Gamma^{1(\alpha)}_{1(\epsilon)1(\alpha)}}{\cancel{u^{1(\alpha)}-u^{1(\epsilon)}}}\,u^{2(\alpha)}\,m_\beta\epsilon_\beta\,\frac{-\cancel{(u^{1(\alpha)}-u^{1(\epsilon)})}}{(u^{1(\beta)}-u^{1(\epsilon)})(u^{1(\alpha)}-u^{1(\beta)})}\notag\\&-\frac{\Gamma^{1(\alpha)}_{1(\beta)1(\alpha)}}{\bcancel{u^{1(\alpha)}-u^{1(\beta)}}}\,u^{2(\alpha)}\,m_\epsilon\epsilon_\epsilon\,\frac{\bcancel{u^{1(\alpha)}-u^{1(\beta)}}}{(u^{1(\beta)}-u^{1(\epsilon)})(u^{1(\alpha)}-u^{1(\epsilon)})}
		\notag\\&=\frac{m_\epsilon\epsilon_\epsilon}{u^{1(\alpha)}-u^{1(\epsilon)}}\,u^{2(\alpha)}\,\frac{m_\beta\epsilon_\beta}{(u^{1(\beta)}-u^{1(\epsilon)})(u^{1(\alpha)}-u^{1(\beta)})}\notag\\&-\frac{m_\beta\epsilon_\beta}{u^{1(\alpha)}-u^{1(\beta)}}\,u^{2(\alpha)}\,\frac{m_\epsilon\epsilon_\epsilon}{(u^{1(\beta)}-u^{1(\epsilon)})(u^{1(\alpha)}-u^{1(\epsilon)})}=0.
		\notag
	\end{align}
	Given an integer $h\geq2$, $h\leq m_\alpha-1$, let us suppose that
	\begin{align}
		\label{LemmaF_ind}
		B^{s(\alpha)}_{\beta\epsilon}=0\qquad\text{for all }s\leq h-1
	\end{align}
	an show $B^{h(\alpha)}_{\beta\epsilon}=0$. We have
	\begin{align}
		B^{h(\alpha)}_{\beta\epsilon}&=-\overset{h+1}{\underset{t=1}{\sum}}\,\Gamma^{(h-t+2)(\alpha)}_{1(\epsilon)1(\alpha)}\,\Gamma^{t(\alpha)}_{1(\beta)1(\alpha)}+\Gamma^{(h+1)(\alpha)}_{1(\beta)1(\alpha)}\,\Gamma^{1(\beta)}_{1(\epsilon)1(\beta)}+\Gamma^{(h+1)(\alpha)}_{1(\epsilon)1(\alpha)}\,\Gamma^{1(\epsilon)}_{1(\beta)1(\epsilon)}
		\notag\\&=\Gamma^{(h+1)(\alpha)}_{1(\beta)1(\alpha)}\,\big(\Gamma^{1(\beta)}_{1(\epsilon)1(\beta)}-\Gamma^{1(\alpha)}_{1(\epsilon)1(\alpha)}\big)+\Gamma^{(h+1)(\alpha)}_{1(\epsilon)1(\alpha)}\,\big(\Gamma^{1(\epsilon)}_{1(\beta)1(\epsilon)}-\Gamma^{1(\alpha)}_{1(\beta)1(\alpha)}\big)\notag\\&-\overset{h}{\underset{t=2}{\sum}}\,\Gamma^{(h-t+2)(\alpha)}_{1(\epsilon)1(\alpha)}\,\Gamma^{t(\alpha)}_{1(\beta)1(\alpha)}
		\notag\\&=-\frac{1}{u^{1(\alpha)}-u^{1(\beta)}}\,\overset{h+1}{\underset{s=2}{\sum}}\,\Gamma^{(h-s+2)(\alpha)}_{1(\beta)1(\alpha)}\,u^{s(\alpha)}\,\big(\Gamma^{1(\beta)}_{1(\epsilon)1(\beta)}-\Gamma^{1(\alpha)}_{1(\epsilon)1(\alpha)}\big)\notag\\&-\frac{1}{u^{1(\alpha)}-u^{1(\epsilon)}}\,\overset{h+1}{\underset{s=2}{\sum}}\,\Gamma^{(h-s+2)(\alpha)}_{1(\epsilon)1(\alpha)}\,u^{s(\alpha)}\,\big(\Gamma^{1(\epsilon)}_{1(\beta)1(\epsilon)}-\Gamma^{1(\alpha)}_{1(\beta)1(\alpha)}\big)\notag\\&-\overset{h}{\underset{t=2}{\sum}}\,\Gamma^{(h-t+2)(\alpha)}_{1(\epsilon)1(\alpha)}\,\Gamma^{t(\alpha)}_{1(\beta)1(\alpha)}
		\notag
	\end{align}
	where in the first two summations only the terms corresponding to $s\leq h$ survive, as the sum of their two $(s=h+1)$-th terms is
	\begin{align}
		&-\frac{1}{u^{1(\alpha)}-u^{1(\beta)}}\,\Gamma^{1(\alpha)}_{1(\beta)1(\alpha)}\,u^{(h+1)(\alpha)}\,\big(\Gamma^{1(\beta)}_{1(\epsilon)1(\beta)}-\Gamma^{1(\alpha)}_{1(\epsilon)1(\alpha)}\big)\notag\\&-\frac{1}{u^{1(\alpha)}-u^{1(\epsilon)}}\,\Gamma^{1(\alpha)}_{1(\epsilon)1(\alpha)}\,u^{(h+1)(\alpha)}\,\big(\Gamma^{1(\epsilon)}_{1(\beta)1(\epsilon)}-\Gamma^{1(\alpha)}_{1(\beta)1(\alpha)}\big)
		\notag\\&=-\frac{m_\beta\epsilon_\beta\,u^{(h+1)(\alpha)}}{(u^{1(\alpha)}-u^{1(\beta)})^2}\,\big(\Gamma^{1(\beta)}_{1(\epsilon)1(\beta)}-\Gamma^{1(\alpha)}_{1(\epsilon)1(\alpha)}\big)\notag\\&-\frac{m_\epsilon\epsilon_\epsilon\,u^{(h+1)(\alpha)}}{(u^{1(\alpha)}-u^{1(\epsilon)})^2}\,\big(\Gamma^{1(\epsilon)}_{1(\beta)1(\epsilon)}-\Gamma^{1(\alpha)}_{1(\beta)1(\alpha)}\big)
		\notag\\&=-\frac{m_\beta\epsilon_\beta\,u^{(h+1)(\alpha)}}{(u^{1(\alpha)}-u^{1(\beta)})^2}\,m_\epsilon\epsilon_\epsilon\,\bigg(\frac{1}{u^{1(\beta)}-u^{1(\epsilon)}}-\frac{1}{u^{1(\alpha)}-u^{1(\epsilon)}}\bigg)\notag\\&-\frac{m_\epsilon\epsilon_\epsilon\,u^{(h+1)(\alpha)}}{(u^{1(\alpha)}-u^{1(\epsilon)})^2}\,m_\beta\epsilon_\beta\bigg(\frac{1}{u^{1(\epsilon)}-u^{1(\beta)}}-\frac{1}{u^{1(\alpha)}-u^{1(\beta)}}\bigg)
		\notag\\&=-\frac{m_\beta\epsilon_\beta\,u^{(h+1)(\alpha)}}{(u^{1(\alpha)}-u^{1(\beta)})^{\cancel{2}}}\,m_\epsilon\epsilon_\epsilon\,\frac{\cancel{u^{1(\alpha)}-u^{1(\beta)}}}{(u^{1(\beta)}-u^{1(\epsilon)})(u^{1(\alpha)}-u^{1(\epsilon)})}\notag\\&-\frac{m_\epsilon\epsilon_\epsilon\,u^{(h+1)(\alpha)}}{(u^{1(\alpha)}-u^{1(\epsilon)})^{\bcancel{2}}}\,m_\beta\epsilon_\beta\,\frac{-\bcancel{(u^{1(\alpha)}-u^{1(\epsilon)})}}{(u^{1(\beta)}-u^{1(\epsilon)})(u^{1(\alpha)}-u^{1(\beta)})}=0.
		\notag
	\end{align}
	It follows that
	\begin{align}
		B^{h(\alpha)}_{\beta\epsilon}&=-\frac{1}{u^{1(\alpha)}-u^{1(\beta)}}\,\overset{h}{\underset{t=2}{\sum}}\,\Gamma^{(h-t+2)(\alpha)}_{1(\beta)1(\alpha)}\,u^{t(\alpha)}\,\big(\Gamma^{1(\beta)}_{1(\epsilon)1(\beta)}-\Gamma^{1(\alpha)}_{1(\epsilon)1(\alpha)}\big)\notag\\&-\frac{1}{u^{1(\alpha)}-u^{1(\epsilon)}}\,\overset{h}{\underset{t=2}{\sum}}\,\Gamma^{(h-t+2)(\alpha)}_{1(\epsilon)1(\alpha)}\,u^{t(\alpha)}\,\big(\Gamma^{1(\epsilon)}_{1(\beta)1(\epsilon)}-\Gamma^{1(\alpha)}_{1(\beta)1(\alpha)}\big)\notag\\&-\overset{h}{\underset{t=2}{\sum}}\,\Gamma^{(h-t+2)(\alpha)}_{1(\epsilon)1(\alpha)}\,\Gamma^{t(\alpha)}_{1(\beta)1(\alpha)}
		\notag
	\end{align}
	where
	\begin{align}
		-\overset{h}{\underset{t=2}{\sum}}\,\Gamma^{(h-t+2)(\alpha)}_{1(\epsilon)1(\alpha)}\,\Gamma^{t(\alpha)}_{1(\beta)1(\alpha)}&=-\overset{h}{\underset{t=2}{\sum}}\,\bigg(-\frac{1}{u^{1(\alpha)}-u^{1(\epsilon)}}\,\overset{h-t+2}{\underset{l=2}{\sum}}\,\Gamma^{(h-t-l+3)(\alpha)}_{1(\epsilon)1(\alpha)}\,u^{l(\alpha)}\bigg)\,\Gamma^{t(\alpha)}_{1(\beta)1(\alpha)}
		\notag\\&=\frac{1}{u^{1(\alpha)}-u^{1(\epsilon)}}\,\overset{h}{\underset{t=2}{\sum}}\,\overset{h-t+2}{\underset{l=2}{\sum}}\,\Gamma^{(h-t-l+3)(\alpha)}_{1(\epsilon)1(\alpha)}\,u^{l(\alpha)}\,\Gamma^{t(\alpha)}_{1(\beta)1(\alpha)}
		\notag\\&\overset{s=h-l+1}{=}\frac{1}{u^{1(\alpha)}-u^{1(\epsilon)}}\,\overset{h}{\underset{t=2}{\sum}}\,\overset{h-1}{\underset{s=t-1}{\sum}}\,\Gamma^{(s-t+2)(\alpha)}_{1(\epsilon)1(\alpha)}\,u^{(h-s+1)(\alpha)}\,\Gamma^{t(\alpha)}_{1(\beta)1(\alpha)}
		\notag\\&=\frac{1}{u^{1(\alpha)}-u^{1(\epsilon)}}\,\overset{h-1}{\underset{s=1}{\sum}}\,\overset{s+1}{\underset{t=2}{\sum}}\,\Gamma^{(s-t+2)(\alpha)}_{1(\epsilon)1(\alpha)}\,u^{(h-s+1)(\alpha)}\,\Gamma^{t(\alpha)}_{1(\beta)1(\alpha)}
		\notag\\&=\frac{1}{u^{1(\alpha)}-u^{1(\epsilon)}}\,\overset{h-1}{\underset{s=1}{\sum}}\,\Bigg(\overset{s+1}{\underset{t=1}{\sum}}\,\Gamma^{(s-t+2)(\alpha)}_{1(\epsilon)1(\alpha)}\,\Gamma^{t(\alpha)}_{1(\beta)1(\alpha)}\notag\\&-\Gamma^{(s+1)(\alpha)}_{1(\epsilon)1(\alpha)}\,\Gamma^{1(\alpha)}_{1(\beta)1(\alpha)}\Bigg)\,u^{(h-s+1)(\alpha)}
		\notag\\&\overset{\eqref{LemmaF_ind}}{=}\frac{1}{u^{1(\alpha)}-u^{1(\epsilon)}}\,\overset{h-1}{\underset{s=1}{\sum}}\,\Bigg(\Gamma^{(s+1)(\alpha)}_{1(\beta)1(\alpha)}\,\Gamma^{1(\beta)}_{1(\epsilon)1(\beta)}+\Gamma^{(s+1)(\alpha)}_{1(\epsilon)1(\alpha)}\,\Gamma^{1(\epsilon)}_{1(\beta)1(\epsilon)}\notag\\&-\Gamma^{(s+1)(\alpha)}_{1(\epsilon)1(\alpha)}\,\Gamma^{1(\alpha)}_{1(\beta)1(\alpha)}\Bigg)\,u^{(h-s+1)(\alpha)}
		\notag\\&\overset{t=h-s+1}{=}\frac{1}{u^{1(\alpha)}-u^{1(\epsilon)}}\,\overset{h}{\underset{t=2}{\sum}}\,\Bigg(\Gamma^{(h-t+2)(\alpha)}_{1(\beta)1(\alpha)}\,\Gamma^{1(\beta)}_{1(\epsilon)1(\beta)}\notag\\&+\Gamma^{(h-t+2)(\alpha)}_{1(\epsilon)1(\alpha)}\,\Gamma^{1(\epsilon)}_{1(\beta)1(\epsilon)}-\Gamma^{(h-t+2)(\alpha)}_{1(\epsilon)1(\alpha)}\,\Gamma^{1(\alpha)}_{1(\beta)1(\alpha)}\Bigg)\,u^{t(\alpha)}.
		\notag
	\end{align}
	Thus
	\begin{align}
		B^{h(\alpha)}_{\beta\epsilon}&=\overset{h}{\underset{t=2}{\sum}}\,\Bigg(-\frac{1}{u^{1(\alpha)}-u^{1(\beta)}}\,\Gamma^{(h-t+2)(\alpha)}_{1(\beta)1(\alpha)}\,\big(\Gamma^{1(\beta)}_{1(\epsilon)1(\beta)}-\Gamma^{1(\alpha)}_{1(\epsilon)1(\alpha)}\big)\notag\\&-\frac{1}{u^{1(\alpha)}-u^{1(\epsilon)}}\,\Gamma^{(h-t+2)(\alpha)}_{1(\epsilon)1(\alpha)}\,\big(\cancel{\Gamma^{1(\epsilon)}_{1(\beta)1(\epsilon)}}-\bcancel{\Gamma^{1(\alpha)}_{1(\beta)1(\alpha)}}\big)\notag\\&+\frac{1}{u^{1(\alpha)}-u^{1(\epsilon)}}\,\Gamma^{(h-t+2)(\alpha)}_{1(\beta)1(\alpha)}\,\Gamma^{1(\beta)}_{1(\epsilon)1(\beta)}+\frac{1}{u^{1(\alpha)}-u^{1(\epsilon)}}\,\Gamma^{(h-t+2)(\alpha)}_{1(\epsilon)1(\alpha)}\,\cancel{\Gamma^{1(\epsilon)}_{1(\beta)1(\epsilon)}}\notag\\&-\frac{1}{u^{1(\alpha)}-u^{1(\epsilon)}}\,\Gamma^{(h-t+2)(\alpha)}_{1(\epsilon)1(\alpha)}\,\bcancel{\Gamma^{1(\alpha)}_{1(\beta)1(\alpha)}}\Bigg)\,u^{t(\alpha)}
		\notag\\&=\overset{h}{\underset{t=2}{\sum}}\,\Gamma^{(h-t+2)(\alpha)}_{1(\beta)1(\alpha)}\,\Bigg(-\frac{1}{u^{1(\alpha)}-u^{1(\beta)}}\,\big(\Gamma^{1(\beta)}_{1(\epsilon)1(\beta)}-\Gamma^{1(\alpha)}_{1(\epsilon)1(\alpha)}\big)\notag\\&+\frac{1}{u^{1(\alpha)}-u^{1(\epsilon)}}\,\Gamma^{1(\beta)}_{1(\epsilon)1(\beta)}\Bigg)\,u^{t(\alpha)}
		\notag\\&=\overset{h}{\underset{t=2}{\sum}}\,\Gamma^{(h-t+2)(\alpha)}_{1(\beta)1(\alpha)}\,\Bigg(-\frac{1}{\cancel{u^{1(\alpha)}-u^{1(\beta)}}}\,\frac{m_\epsilon\epsilon_\epsilon\,\cancel{(u^{1(\alpha)}-u^{1(\beta)})}}{(u^{1(\beta)}-u^{1(\epsilon)})(u^{1(\alpha)}-u^{1(\epsilon)})}\notag\\&+\frac{1}{u^{1(\alpha)}-u^{1(\epsilon)}}\,\frac{m_\epsilon\epsilon_\epsilon}{u^{1(\beta)}-u^{1(\epsilon)}}\Bigg)\,u^{t(\alpha)}=0.
		\notag
	\end{align}
\end{proof}
\begin{lemma}
	Given $\alpha,\beta\in\{1,\dots,r\}$ with $\alpha\neq\beta$ we have
	\begin{align}
		C^{s(\alpha)}_{\beta}&:=\overset{s+1}{\underset{l=2}{\sum}}\,\bigg(\Gamma^{(s-l+4)(\alpha)}_{2(\alpha)2(\alpha)}-\Gamma^{(s-l+2)(\alpha)}_{1(\alpha)1(\alpha)}\bigg)\,\Gamma^{l(\alpha)}_{1(\beta)1(\alpha)}+\Gamma^{2(\alpha)}_{2(\alpha)2(\alpha)}\,\Gamma^{(s+2)(\alpha)}_{1(\beta)1(\alpha)}+\Gamma^{(s+1)(\alpha)}_{1(\beta)1(\alpha)}\,\Gamma^{1(\beta)}_{1(\alpha)1(\beta)}=0
		\label{fulm}
	\end{align}
	for every $s\in\{0,\dots,m_\alpha-2\}$\footnote{The summation is intended to be non-zero when $s\geq1$.}.
\end{lemma}
\begin{proof}
	For $s=0$ we get
	\begin{align}
		C^{0(\alpha)}_{\beta}&=\Gamma^{2(\alpha)}_{2(\alpha)2(\alpha)}\,\Gamma^{2(\alpha)}_{1(\beta)1(\alpha)}+\Gamma^{1(\alpha)}_{1(\beta)1(\alpha)}\,\Gamma^{1(\beta)}_{1(\alpha)1(\beta)}
		\notag\\&=-\Gamma^{2(\alpha)}_{2(\alpha)2(\alpha)}\,\frac{\Gamma^{1(\alpha)}_{1(\beta)1(\alpha)}}{u^{1(\alpha)}-u^{1(\beta)}}\,u^{2(\alpha)}+\Gamma^{1(\alpha)}_{1(\beta)1(\alpha)}\,\Gamma^{1(\beta)}_{1(\alpha)1(\beta)}
		\notag\\&=\Gamma^{1(\alpha)}_{1(\beta)1(\alpha)}\bigg(-\Gamma^{2(\alpha)}_{2(\alpha)2(\alpha)}\,\frac{1}{u^{1(\alpha)}-u^{1(\beta)}}\,u^{2(\alpha)}+\Gamma^{1(\beta)}_{1(\alpha)1(\beta)}\bigg)
		\notag\\&=\Gamma^{1(\alpha)}_{1(\beta)1(\alpha)}\bigg(\frac{m_\alpha\epsilon_\alpha}{u^{1(\alpha)}-u^{1(\beta)}}-\frac{m_\alpha\epsilon_\alpha}{u^{1(\alpha)}-u^{1(\beta)}}\bigg)=0.
		\notag
	\end{align}
	For $s=1$ we get
	\begin{align}
		C^{1(\alpha)}_{\beta}&=\bigg(\Gamma^{3(\alpha)}_{2(\alpha)2(\alpha)}-\Gamma^{1(\alpha)}_{1(\alpha)1(\alpha)}\bigg)\,\Gamma^{2(\alpha)}_{1(\beta)1(\alpha)}+\Gamma^{2(\alpha)}_{2(\alpha)2(\alpha)}\,\Gamma^{3(\alpha)}_{1(\beta)1(\alpha)}+\Gamma^{2(\alpha)}_{1(\beta)1(\alpha)}\,\Gamma^{1(\beta)}_{1(\alpha)1(\beta)}
		\notag\\&=\cancel{\Gamma^{2(\alpha)}_{2(\alpha)2(\alpha)}\,\frac{u^{3(\alpha)}}{u^{2(\alpha)}}\,\frac{\Gamma^{1(\alpha)}_{1(\beta)1(\alpha)}}{u^{1(\alpha)}-u^{1(\beta)}}\,u^{2(\alpha)}}-\Gamma^{2(\alpha)}_{2(\alpha)2(\alpha)}\,\frac{\Gamma^{2(\alpha)}_{1(\beta)1(\alpha)}}{u^{1(\alpha)}-u^{1(\beta)}}\,u^{2(\alpha)}\notag\\&-\cancel{\Gamma^{2(\alpha)}_{2(\alpha)2(\alpha)}\,\frac{\Gamma^{1(\alpha)}_{1(\beta)1(\alpha)}}{u^{1(\alpha)}-u^{1(\beta)}}\,u^{3(\alpha)}}+\Gamma^{2(\alpha)}_{1(\beta)1(\alpha)}\,\Gamma^{1(\beta)}_{1(\alpha)1(\beta)}
		\notag\\&=\Gamma^{2(\alpha)}_{1(\beta)1(\alpha)}\bigg(-\Gamma^{2(\alpha)}_{2(\alpha)2(\alpha)}\,\frac{1}{u^{1(\alpha)}-u^{1(\beta)}}\,u^{2(\alpha)}+\Gamma^{1(\beta)}_{1(\alpha)1(\beta)}\bigg)
		\notag\\&=\Gamma^{2(\alpha)}_{1(\beta)1(\alpha)}\bigg(\frac{m_\alpha\epsilon_\alpha}{u^{1(\alpha)}-u^{1(\beta)}}-\frac{m_\alpha\epsilon_\alpha}{u^{1(\alpha)}-u^{1(\beta)}}\bigg)=0.
		\notag
	\end{align}
	Let us now consider an integer $s\in\{1,\dots,m_\alpha-2\}$. Since
	\begin{align}
		\Gamma^{l(\alpha)}_{1(\beta)1(\alpha)}&=-\frac{1}{u^{1(\alpha)}-u^{1(\beta)}}\,\overset{l}{\underset{t=2}{\sum}}\,\Gamma^{(l-t+1)(\alpha)}_{1(\beta)1(\alpha)}\,u^{t(\alpha)}
		\notag\\&\overset{k=l-t+1}{=}-\frac{1}{u^{1(\alpha)}-u^{1(\beta)}}\,\overset{l-1}{\underset{k=1}{\sum}}\,\Gamma^{k(\alpha)}_{1(\beta)1(\alpha)}\,u^{(l-k+1)(\alpha)}
		\notag
	\end{align}
	for each $l\in\{2,\dots,s+1\}$, we have
	\begin{align}
		C^{s(\alpha)}_{\beta}&=-\frac{1}{u^{1(\alpha)}-u^{1(\beta)}}\,\overset{s+1}{\underset{l=2}{\sum}}\,\bigg(\Gamma^{(s-l+4)(\alpha)}_{2(\alpha)2(\alpha)}-\Gamma^{(s-l+2)(\alpha)}_{1(\alpha)1(\alpha)}\bigg)\,\overset{l-1}{\underset{k=1}{\sum}}\,\Gamma^{k(\alpha)}_{1(\beta)1(\alpha)}\,u^{(l-k+1)(\alpha)}\notag\\&+\Gamma^{2(\alpha)}_{2(\alpha)2(\alpha)}\,\Gamma^{(s+2)(\alpha)}_{1(\beta)1(\alpha)}+\Gamma^{(s+1)(\alpha)}_{1(\beta)1(\alpha)}\,\Gamma^{1(\beta)}_{1(\alpha)1(\beta)}
		\notag\\&=-\frac{1}{u^{1(\alpha)}-u^{1(\beta)}}\,\overset{s}{\underset{k=1}{\sum}}\,\Gamma^{k(\alpha)}_{1(\beta)1(\alpha)}\,\overset{s+1}{\underset{l=k+1}{\sum}}\,\bigg(\Gamma^{(s-l+4)(\alpha)}_{2(\alpha)2(\alpha)}-\Gamma^{(s-l+2)(\alpha)}_{1(\alpha)1(\alpha)}\bigg)\,u^{(l-k+1)(\alpha)}\notag\\&+\Gamma^{2(\alpha)}_{2(\alpha)2(\alpha)}\,\Gamma^{(s+2)(\alpha)}_{1(\beta)1(\alpha)}+\Gamma^{(s+1)(\alpha)}_{1(\beta)1(\alpha)}\,\Gamma^{1(\beta)}_{1(\alpha)1(\beta)}
		\notag
	\end{align}
	where (by taking $\overset{\sim}{s}:=s-l+2$ and $\overset{\sim}{l}:=s-k+2$)
	\begin{align}
		\overset{s+1}{\underset{l=k+1}{\sum}}\,\bigg(\Gamma^{(s-l+4)(\alpha)}_{2(\alpha)2(\alpha)}-\Gamma^{(s-l+2)(\alpha)}_{1(\alpha)1(\alpha)}\bigg)\,u^{(l-k+1)(\alpha)}&=\overset{\overset{\sim}{l}-1}{\underset{\overset{\sim}{s}=1}{\sum}}\,\bigg(\Gamma^{(\overset{\sim}{s}+2)(\alpha)}_{2(\alpha)2(\alpha)}-\Gamma^{\overset{\sim}{s}(\alpha)}_{1(\alpha)1(\alpha)}\bigg)\,u^{(\overset{\sim}{l}-\overset{\sim}{s}+1)(\alpha)}
		\notag\\&=\bigg(\Gamma^{3(\alpha)}_{2(\alpha)2(\alpha)}-\Gamma^{1(\alpha)}_{1(\alpha)1(\alpha)}\bigg)\,u^{\overset{\sim}{l}(\alpha)}\notag\\&+\overset{\overset{\sim}{l}-1}{\underset{\overset{\sim}{s}=2}{\sum}}\,\bigg(\Gamma^{(\overset{\sim}{s}+2)(\alpha)}_{2(\alpha)2(\alpha)}-\Gamma^{\overset{\sim}{s}(\alpha)}_{1(\alpha)1(\alpha)}\bigg)\,u^{(\overset{\sim}{l}-\overset{\sim}{s}+1)(\alpha)}
		\notag\\&=-\Gamma^{2(\alpha)}_{2(\alpha)2(\alpha)}\,\frac{u^{3(\alpha)}}{u^{2(\alpha)}}\,u^{\overset{\sim}{l}(\alpha)}\notag\\&+\overset{\overset{\sim}{l}-1}{\underset{\overset{\sim}{s}=2}{\sum}}\,\bigg(\Gamma^{(\overset{\sim}{s}+2)(\alpha)}_{2(\alpha)2(\alpha)}-\Gamma^{\overset{\sim}{s}(\alpha)}_{1(\alpha)1(\alpha)}\bigg)\,u^{(\overset{\sim}{l}-\overset{\sim}{s}+1)(\alpha)}
		\notag\\&\overset{Lemma\,8.9}{=}-\Gamma^{2(\alpha)}_{2(\alpha)2(\alpha)}\,\cancel{\frac{u^{3(\alpha)}}{u^{2(\alpha)}}\,u^{\overset{\sim}{l}(\alpha)}}\notag\\&+\Gamma^{2(\alpha)}_{2(\alpha)2(\alpha)}\bigg(\cancel{\frac{u^{3(\alpha)}}{u^{2(\alpha)}}\,u^{\overset{\sim}{l}(\alpha)}}-u^{(\overset{\sim}{l}+1)(\alpha)}\bigg)
		\notag\\&=-\Gamma^{2(\alpha)}_{2(\alpha)2(\alpha)}\,u^{(\overset{\sim}{l}+1)(\alpha)}
		\notag\\&=-\Gamma^{2(\alpha)}_{2(\alpha)2(\alpha)}\,u^{(s-k+3)(\alpha)}.
		\notag
	\end{align}
	It follows that
	\begin{align}
		C^{s(\alpha)}_{\beta}&=\frac{\Gamma^{2(\alpha)}_{2(\alpha)2(\alpha)}}{u^{1(\alpha)}-u^{1(\beta)}}\,\overset{s}{\underset{k=1}{\sum}}\,\Gamma^{k(\alpha)}_{1(\beta)1(\alpha)}\,u^{(s-k+3)(\alpha)}\notag\\&+\Gamma^{2(\alpha)}_{2(\alpha)2(\alpha)}\,\Gamma^{(s+2)(\alpha)}_{1(\beta)1(\alpha)}+\Gamma^{(s+1)(\alpha)}_{1(\beta)1(\alpha)}\,\Gamma^{1(\beta)}_{1(\alpha)1(\beta)}
		\notag
	\end{align}
	where
	\begin{align}
		\Gamma^{(s+2)(\alpha)}_{1(\beta)1(\alpha)}&=-\frac{1}{u^{1(\alpha)}-u^{1(\beta)}}\,\overset{s+2}{\underset{t=2}{\sum}}\,\Gamma^{(s-t+3)(\alpha)}_{1(\beta)1(\alpha)}\,u^{t(\alpha)}
		\notag\\&\overset{k=s-t+3}{=}-\frac{1}{u^{1(\alpha)}-u^{1(\beta)}}\,\overset{s+1}{\underset{k=1}{\sum}}\,\Gamma^{k(\alpha)}_{1(\beta)1(\alpha)}\,u^{(s-k+3)(\alpha)}
		\notag
	\end{align}
	thus
	\begin{align}
		C^{s(\alpha)}_{\beta}&=\frac{\Gamma^{2(\alpha)}_{2(\alpha)2(\alpha)}}{u^{1(\alpha)}-u^{1(\beta)}}\,\overset{s}{\underset{k=1}{\sum}}\,\Gamma^{k(\alpha)}_{1(\beta)1(\alpha)}\,u^{(s-k+3)(\alpha)}\notag\\&-\frac{\Gamma^{2(\alpha)}_{2(\alpha)2(\alpha)}}{u^{1(\alpha)}-u^{1(\beta)}}\,\overset{s+1}{\underset{k=1}{\sum}}\,\Gamma^{k(\alpha)}_{1(\beta)1(\alpha)}\,u^{(s-k+3)(\alpha)}\notag\\&+\Gamma^{(s+1)(\alpha)}_{1(\beta)1(\alpha)}\,\Gamma^{1(\beta)}_{1(\alpha)1(\beta)}
		\notag\\&=-\frac{\Gamma^{2(\alpha)}_{2(\alpha)2(\alpha)}}{u^{1(\alpha)}-u^{1(\beta)}}\,\Gamma^{(s+1)(\alpha)}_{1(\beta)1(\alpha)}\,u^{2(\alpha)}+\Gamma^{(s+1)(\alpha)}_{1(\beta)1(\alpha)}\,\Gamma^{1(\beta)}_{1(\alpha)1(\beta)}
		\notag\\&=\Gamma^{(s+1)(\alpha)}_{1(\beta)1(\alpha)}\,\bigg(-\frac{\Gamma^{2(\alpha)}_{2(\alpha)2(\alpha)}}{u^{1(\alpha)}-u^{1(\beta)}}\,u^{2(\alpha)}+\Gamma^{1(\beta)}_{1(\alpha)1(\beta)}\bigg)
		\notag\\&=\Gamma^{(s+1)(\alpha)}_{1(\beta)1(\alpha)}\,\bigg(\frac{m_\alpha\epsilon_\alpha}{u^{1(\alpha)}-u^{1(\beta)}}-\frac{m_\alpha\epsilon_\alpha}{u^{1(\alpha)}-u^{1(\beta)}}\bigg)=0.
		\notag
	\end{align}
\end{proof}

\subsection{The dual structure}
\begin{prop}
	For each $\alpha\in\{1,\dots,r\}$ the components of the inverse of the Euler vector field $E$ are given by
	\begin{align}
		&(E^{-1})^{1(\alpha)}=\frac{1}{u^{1(\alpha)}}\notag\\
		&(E^{-1})^{(k+1)(\alpha)}=-\frac{1}{u^{1(\alpha)}}\,\overset{k}{\underset{s=1}{\sum}}\,(E^{-1})^{(k-s+1)(\alpha)}\,u^{(s+1)(\alpha)}\qquad\text{for }1\leq k\leq m_\alpha-1.
		\notag
	\end{align}
\end{prop}
\begin{proof}
	By definition we have
	\begin{align}
		E^{-1}\,\circ\,E&=e=\overset{r}{\underset{\tau=1}{\sum}}\,\partial_{1(\tau)}
		\notag
	\end{align}
	which (by taking the $k(\alpha)$-th component) yields
	\begin{align}
		\overset{r}{\underset{\tau=1}{\sum}}\,\delta^{k(\alpha)}_{1(\tau)}&=(E^{-1}\,\circ\,E)^{k(\alpha)}=\overset{n}{\underset{i,j=1}{\sum}}\,(E^{-1})^i\,c^{k(\alpha)}_{ij}\,E^j\notag\\&=\overset{r}{\underset{\sigma,\tau=1}{\sum}}\,\overset{m_\sigma}{\underset{i=1}{\sum}}\,\overset{m_\tau}{\underset{j=1}{\sum}}\,(E^{-1})^{i(\sigma)}\,c^{k(\alpha)}_{i(\sigma)j(\tau)}\,E^{j(\tau)}\notag\\&=\overset{m_\alpha}{\underset{i,j=1}{\sum}}\,(E^{-1})^{i(\alpha)}\,c^{k(\alpha)}_{i(\alpha)j(\alpha)}\,E^{j(\alpha)}\notag\\&=\overset{m_\alpha}{\underset{i,j=1}{\sum}}\,(E^{-1})^{i(\alpha)}\,\delta^{k}_{i+j-1}\,E^{j(\alpha)}\notag\\&=\overset{m_\alpha}{\underset{j=1}{\sum}}\,(E^{-1})^{(k-j+1)(\alpha)}\,E^{j(\alpha)}.
		\notag
	\end{align}
	Since
	\begin{align}
		\overset{r}{\underset{\tau=1}{\sum}}\,\delta^{k(\alpha)}_{1(\tau)}&=\delta^k_1
		\notag
	\end{align}
	we obtain
	\begin{align}
		\label{Lemma7.4_k}
		\delta^k_1&=\overset{m_\alpha}{\underset{j=1}{\sum}}\,(E^{-1})^{(k-j+1)(\alpha)}\,E^{j(\alpha)}.
	\end{align}
	By taking $k=1$ we get
	\begin{align}
		1&=\overset{m_\alpha}{\underset{j=1}{\sum}}\,(E^{-1})^{(2-j)(\alpha)}\,E^{j(\alpha)}
		\notag
	\end{align}
	where the quantity $(E^{-1})^{(2-j)(\alpha)}$ only makes sense for $j=1$, thus
	\begin{align}
		1&=(E^{-1})^{1(\alpha)}\,E^{1(\alpha)}=(E^{-1})^{1(\alpha)}\,u^{1(\alpha)}
		\notag
	\end{align}
	which yields
	\begin{align}
		(E^{-1})^{1(\alpha)}&=\frac{1}{u^1(\alpha)}.
		\notag
	\end{align}
	By taking $k\geq2$ in \eqref{Lemma7.4_k} we get
	\begin{align}
		0&=\overset{m_\alpha}{\underset{j=1}{\sum}}\,(E^{-1})^{(k-j+1)(\alpha)}\,E^{j(\alpha)}
		\notag
	\end{align}
	where the quantity $(E^{-1})^{(k-j+1)(\alpha)}$ only makes sens for $j\leq k$, thus
	\begin{align}
		0&=\overset{k}{\underset{j=1}{\sum}}\,(E^{-1})^{(k-j+1)(\alpha)}\,E^{j(\alpha)}=\overset{k}{\underset{j=1}{\sum}}\,(E^{-1})^{(k-j+1)(\alpha)}\,u^{j(\alpha)}\notag\\&=(E^{-1})^{k(\alpha)}\,u^{1(\alpha)}+\overset{k}{\underset{j=2}{\sum}}\,(E^{-1})^{(k-j+1)(\alpha)}\,u^{j(\alpha)}\notag\\&=(E^{-1})^{k(\alpha)}\,u^{1(\alpha)}+\overset{k-1}{\underset{s=1}{\sum}}\,(E^{-1})^{(k-s)(\alpha)}\,u^{(s+1)(\alpha)}
		\notag
	\end{align}
	which yields
	\begin{align}
		(E^{-1})^{k(\alpha)}&=-\frac{1}{u^{1(\alpha)}}\,\overset{k-1}{\underset{s=1}{\sum}}\,(E^{-1})^{(k-s)(\alpha)}\,u^{(s+1)(\alpha)}.
		\notag
	\end{align}
	In particular, for $k=h+1$ this becomes
	\begin{align}
		(E^{-1})^{(h+1)(\alpha)}&=-\frac{1}{u^{1(\alpha)}}\,\overset{h}{\underset{s=1}{\sum}}\,(E^{-1})^{(h-s+1)(\alpha)}\,u^{(s+1)(\alpha)}.
		\notag
	\end{align}
\end{proof}
\begin{remark}
	By definition, the dual product $*$ must verify the following relation:
	\begin{align}
		\label{dualproductdef}
		X*Y&=E^{-1}\circ X\circ Y
	\end{align}
	for $X$, $Y$ arbitrary vector fields. This means that
	\begin{align}
		X^j{c^*}^i_{jk}Y^k&=(E^{-1})^a c^i_{ab} X^j c^b_{jk}Y^k\qquad\forall\,X\text{, }Y
		\notag
	\end{align}
	i.e.
	\begin{align}
		{c^*}^i_{jk}&=(E^{-1})^a c^i_{ab} c^b_{jk}.
		\notag
	\end{align}
	Therefore
	\begin{align}
		{c^*}^{i(\alpha)}_{j(\beta)k(\gamma)}&=\overset{r}{\underset{\sigma,\tau=1}{\sum}}\,\overset{m_\sigma}{\underset{a=1}{\sum}}\,\overset{m_\tau}{\underset{b=1}{\sum}}\,(E^{-1})^{a(\sigma)} c^{i(\alpha)}_{a(\sigma)b(\tau)} c^{b(\tau)}_{j(\beta)k(\gamma)}
		\notag\\&=\overset{m_\alpha}{\underset{a,b=1}{\sum}}\,(E^{-1})^{a(\alpha)} c^{i(\alpha)}_{a(\alpha)b(\alpha)} c^{b(\alpha)}_{j(\beta)k(\gamma)}
		\notag\\&=\overset{m_\alpha}{\underset{a,b=1}{\sum}}\,(E^{-1})^{a(\alpha)} \delta^i_{a+b-1}\,\delta^\alpha_\beta\delta^\alpha_\gamma\delta^b_{j+k-1}
		\notag\\&=\delta^\alpha_\beta\delta^\alpha_\gamma\,\overset{m_\alpha}{\underset{b=1}{\sum}}\,(E^{-1})^{(i-b+1)(\alpha)}\delta^b_{j+k-1}
		\notag\\&=\delta^\alpha_\beta\delta^\alpha_\gamma\,(E^{-1})^{(i-j-k+2)(\alpha)}.
		\label{dualproductexplicit}
	\end{align}
\end{remark}
\begin{prop}
	The Christoffel symbols of the dual connection $\nabla^*$ are given by
	\begin{align}
		{\Gamma^*}^{k(\alpha)}_{i(\beta)j(\gamma)}&={\Gamma}^{k(\alpha)}_{i(\beta)j(\gamma)}-\delta^\alpha_\beta\delta^\alpha_\gamma\,(E^{-1})^{(k-i-j+2)(\alpha)}\,\bigg[\delta^k_1\bigg(1-\underset{\sigma\neq\alpha}{\sum}\,m_\sigma\epsilon_\sigma\bigg)\notag\\&+(1-\delta^k_1)\bigg(1-\overset{r}{\underset{\tau=1}{\sum}}\,m_\tau\epsilon_\tau\bigg)\bigg]-(1-\delta^\alpha_\beta)\,\delta_{\beta\gamma}\,\delta^1_i\delta^1_j\delta^k_1\,\frac{m_\beta\epsilon_\beta}{u^{1(\beta)}}
		\notag
	\end{align}
	for every choice of $\alpha$, $\beta$, $\gamma\in\{1,\dots,r\}$ and every $k\in\{1,\dots,m_\alpha\}$, $i\in\{1,\dots,m_\beta\}$, $j\in\{1,\dots,m_\gamma\}$.
\end{prop}
\begin{proof}
	By means of \eqref{dualfromnatural} we have
	\begin{align}
		{\Gamma^*}^{k(\alpha)}_{i(\beta)j(\gamma)}&={\Gamma}^{k(\alpha)}_{i(\beta)j(\gamma)}-\overset{n}{\underset{l=1}{\sum}}\,{c^*}^{l}_{i(\beta)j(\gamma)}\,\nabla_{l}E^{k(\alpha)}
		\notag\\&={\Gamma}^{k(\alpha)}_{i(\beta)j(\gamma)}-\overset{r}{\underset{\tau=1}{\sum}}\,\overset{m_\tau}{\underset{l=1}{\sum}}\,{c^*}^{l(\tau)}_{i(\beta)j(\gamma)}\,\nabla_{l(\tau)}E^{k(\alpha)}
		\notag
	\end{align}
	where $\nabla_{l(\tau)}E^{k(\alpha)}\neq0$ only for $\tau=\alpha$ and for $\tau\neq\alpha$ and $l=k=1$ by Lemma 8.7. It follows that
	\begin{align}
		{\Gamma^*}^{k(\alpha)}_{i(\beta)j(\gamma)}&={\Gamma}^{k(\alpha)}_{i(\beta)j(\gamma)}-\overset{m_\alpha}{\underset{l=1}{\sum}}\,{c^*}^{l(\alpha)}_{i(\beta)j(\gamma)}\,\nabla_{l(\alpha)}E^{k(\alpha)}
		\notag\\&-\overset{}{\underset{\tau\neq\alpha}{\sum}}\,{c^*}^{1(\tau)}_{i(\beta)j(\gamma)}\,\nabla_{1(\tau)}E^{1(\alpha)}\,\delta^k_1
		\notag\\&={\Gamma}^{k(\alpha)}_{i(\beta)j(\gamma)}-\overset{m_\alpha}{\underset{l=1}{\sum}}\,\delta^\alpha_\beta\delta^\alpha_\gamma\,(E^{-1})^{(l-i-j+2)(\alpha)}\,\nabla_{l(\alpha)}E^{k(\alpha)}
		\notag\\&-\overset{}{\underset{\tau\neq\alpha}{\sum}}\,\delta^\tau_\beta\delta^\tau_\gamma\,(E^{-1})^{(3-i-j)(\tau)}\,\nabla_{1(\tau)}E^{1(\alpha)}\,\delta^k_1.
		\notag
	\end{align}
	which immediately implies
	\begin{align}
		{\Gamma^*}^{k(\alpha)}_{i(\beta)j(\gamma)}&={\Gamma}^{k(\alpha)}_{i(\beta)j(\gamma)}
		\notag
	\end{align}
	whenever $\beta\neq\gamma$. If $\alpha=\beta=\gamma$ then we get
	\begin{align}
		{\Gamma^*}^{k(\alpha)}_{i(\alpha)j(\alpha)}&={\Gamma}^{k(\alpha)}_{i(\alpha)j(\alpha)}-\overset{m_\alpha}{\underset{l=1}{\sum}}\,(E^{-1})^{(l-i-j+2)(\alpha)}\,\nabla_{l(\alpha)}E^{k(\alpha)}
		\notag
	\end{align}
	as $\delta^\tau_\beta=\delta^\tau_\alpha=0$ for every $\tau\neq\alpha$, where
	\begin{align}
		\nabla_{l(\alpha)}E^{k(\alpha)}&=\delta^k_l\bigg[\delta^k_1\bigg(1-\underset{\sigma\neq\alpha}{\sum}\,m_\sigma\epsilon_\sigma\bigg)+(1-\delta^k_1)\bigg(1-\overset{r}{\underset{\tau=1}{\sum}}\,m_\tau\epsilon_\tau\bigg)\bigg]
		\notag
	\end{align}
	by Lemma 8.7. It follows that
	\begin{align}
		{\Gamma^*}^{k(\alpha)}_{i(\alpha)j(\alpha)}&={\Gamma}^{k(\alpha)}_{i(\alpha)j(\alpha)}-(E^{-1})^{(k-i-j+2)(\alpha)}\,\bigg[\delta^k_1\bigg(1-\underset{\sigma\neq\alpha}{\sum}\,m_\sigma\epsilon_\sigma\bigg)\notag\\&+(1-\delta^k_1)\bigg(1-\overset{r}{\underset{\tau=1}{\sum}}\,m_\tau\epsilon_\tau\bigg)\bigg].
		\notag
	\end{align}
	If $\alpha\neq\beta=\gamma$ then we have
	\begin{align}
		{\Gamma^*}^{k(\alpha)}_{i(\beta)j(\beta)}&={\Gamma}^{k(\alpha)}_{i(\beta)j(\beta)}-\overset{}{\underset{\tau\neq\alpha}{\sum}}\,\delta^\tau_\beta\,(E^{-1})^{(3-i-j)(\tau)}\,\nabla_{1(\tau)}E^{1(\alpha)}\,\delta^k_1
		\notag\\&={\Gamma}^{k(\alpha)}_{i(\beta)j(\beta)}-(E^{-1})^{(3-i-j)(\beta)}\,\nabla_{1(\beta)}E^{1(\alpha)}\,\delta^k_1
		\notag\\&={\Gamma}^{k(\alpha)}_{i(\beta)j(\beta)}-(E^{-1})^{(3-i-j)(\beta)}\,m_\beta\epsilon_\beta\,\delta^k_1
		\notag
	\end{align}
	by Lemma 8.7, where $(E^{-1})^{(3-i-j)(\beta)}$ only makes sense for $3-i-j\geq1$ that is $i+j\leq2$ i.e. $i=j=1$. Thus $(E^{-1})^{(3-i-j)(\beta)}=\delta_i^1\delta_j^1\,(E^{-1})^{1(\beta)}=\delta_i^1\delta_j^1\,\frac{1}{u^{1(\beta)}}$ yielding
	\begin{align}
		{\Gamma^*}^{k(\alpha)}_{i(\beta)j(\beta)}&={\Gamma}^{k(\alpha)}_{i(\beta)j(\beta)}-\delta^1_i\delta^1_j\delta^k_1\,\frac{m_\beta\epsilon_\beta}{u^{1(\beta)}}.
		\notag
	\end{align}
\end{proof}

\subsection{The proof of the theorem}
\begin{theorem}
	The connection $\nabla$ associated with the product $\circ$ with structure constants $c^{k(\gamma)}_{i(\alpha)j(\beta)}=\delta^\gamma_\alpha\delta^\gamma_\beta\delta^k_{i+j-1}$, the vector field $e$ of components $e^i=\overset{r}{\underset{\alpha=1}{\sum}}\,\delta^i_{1(\alpha)}$, the dual connection $\nabla^*$ defined by the Cristoffel symbols
	\begin{align}
		{\Gamma^*}^{k(\alpha)}_{i(\beta)j(\gamma)}&={\Gamma}^{k(\alpha)}_{i(\beta)j(\gamma)}-\delta^\alpha_\beta\delta^\alpha_\gamma\,(E^{-1})^{(k-i-j+2)(\alpha)}\,\bigg[\delta^k_1\bigg(1-\underset{\sigma\neq\alpha}{\sum}\,m_\sigma\epsilon_\sigma\bigg)\notag\\&+(1-\delta^k_1)\bigg(1-\overset{r}{\underset{\tau=1}{\sum}}\,m_\tau\epsilon_\tau\bigg)\bigg]-(1-\delta^\alpha_\beta)\,\delta_{\beta\gamma}\,\delta^1_i\delta^1_j\delta^k_1\,\frac{m_\beta\epsilon_\beta}{u^{1(\beta)}},
		\label{Christoffel*}
	\end{align}
	the dual product $*$ with structure constants
	\begin{align}
		{c^*}^{i(\alpha)}_{j(\beta)k(\gamma)}&=\delta^\alpha_\beta\delta^\alpha_\gamma\,(E^{-1})^{(i-j-k+2)(\alpha)}
		\notag
	\end{align}
	and the vector field $E$ of components $E^i=u^i$ define a unique bi-flat structure.	
\end{theorem}
\begin{proof}
	In order to prove Theorem $8.13$ we have to show the following statings:
	\begin{itemize}
		\item[A.] $\nabla-\lambda\,\circ$ is flat and torsionless for each choice of $\lambda$ i.e.
		\begin{itemize}
			\item[A.1] $\nabla$ is flat and torsionless
			\item[A.2] $\nabla_i\,c^l_{jk}=\nabla_j\,c^l_{ik}$ for all of the indices
			\item[A.3] $\circ$ is commutative and associative
		\end{itemize}
		\item[B.] $X\,\circ\,e=X$ for every vector field $X$
		\item[C.] $\nabla e=0$
		\item[D.] $\nabla^*-\lambda\,*$ is flat and torsionless for each choice of $\lambda$ i.e.
		\begin{itemize}
			\item[D.1] $\nabla^*$ is flat and torsionless
			\item[D.2] ${\nabla^*}_i\,{c^*}^l_{jk}={\nabla^*}_j\,{c^*}^l_{ik}$ for all of the indices
			\item[D.3] $*$ is commutative and associative
		\end{itemize}
		\item[E.] $X\,*\,E=X$ for every vector field $X$
		\item[F.] $\nabla^* E=0$
		\item[G.] $[e,E]=e$, ${\rm Lie}_E \circ=\circ$
		\item[H.] $(d_{\nabla}-d_{\nabla^{*}})(X\,\circ)=0$
		\item[I.] $\nabla$ and $\nabla^*$ are uniquely determined.
	\end{itemize}
	Since the Christoffel symbols are such that $\Gamma^k_{ij}=\Gamma^k_{ji}$, the connection $\nabla$ is trivially proved to be torsionless. By Lemma 8.8, the connection $\nabla^*$ is torsionless as well. We are now going to prove point A.2, namely
	\begin{align}
		\nabla_{i(\alpha)}\,c^{l(\epsilon)}_{{j(\beta)}{k(\gamma)}}&=\nabla_{j(\beta)}\,c^{l(\epsilon)}_{{i(\alpha)}{k(\gamma)}}
		\notag
	\end{align}
	which is equivalent to
	\begin{align}
		\Gamma_{i(\alpha)s(\sigma)}^{l(\epsilon)}c^{s(\sigma)}_{{j(\beta)}{k(\gamma)}}-\Gamma_{i(\alpha)k(\gamma)}^{s(\sigma)}c^{l(\epsilon)}_{{j(\beta)}{s(\sigma)}}&=\Gamma_{j(\beta)s(\sigma)}^{l(\epsilon)}c^{s(\sigma)}_{{i(\alpha)}{k(\gamma)}}-\Gamma_{j(\beta)k(\gamma)}^{s(\sigma)}c^{l(\epsilon)}_{{i(\alpha)}{s(\sigma)}}
		\notag
	\end{align}
	and
	\begin{align}
		\Gamma_{i(\alpha)(j+k-1)(\beta)}^{l(\epsilon)}\delta_{\beta\gamma}-\Gamma_{i(\alpha)k(\gamma)}^{(l-j+1)(\beta)}\delta^{\epsilon}_{\beta}&=\Gamma_{j(\beta)(i+k-1)(\alpha)}^{l(\epsilon)}\delta_{\alpha\gamma}-\Gamma_{j(\beta)k(\gamma)}^{(l-i+1)(\alpha)}\delta^\epsilon_\alpha
		\label{A.2}
	\end{align}
	for all $\alpha$, $\beta$, $\gamma$, $\epsilon\in\{1,\dots,r\}$ and any suitable choice of the indices $i$, $j$, $k$, $l$. The possible cases are the following ones:
	\begin{itemize}
		\item[$1.$] $\alpha=\beta=\gamma=\epsilon$
		\item[$2.$] $\alpha=\beta=\gamma\neq\epsilon$
		\item[$3.$] $\alpha=\beta=\epsilon\neq\gamma$
		\item[$4.$] $\alpha=\gamma=\epsilon\neq\beta$
		\item[$5.$] $\beta=\gamma=\epsilon\neq\alpha$
		\item[$6.$] $\alpha=\beta\neq\gamma=\epsilon$
		\item[$7.$] $\alpha=\gamma\neq\beta=\epsilon$
		\item[$8.$] $\alpha=\epsilon\neq\beta=\gamma$
		\item[$9.$] at least three among $\alpha$, $\beta$, $\gamma$, $\epsilon$ are pairwise distinct.
	\end{itemize}
	\textbf{Case 1: $\alpha=\beta=\gamma=\epsilon$.} \eqref{A.2} becomes
	\begin{align}
		\Gamma_{i(\alpha)(j+k-1)(\alpha)}^{l(\alpha)}-\Gamma_{i(\alpha)k(\alpha)}^{(l-j+1)(\alpha)}&=\Gamma_{j(\alpha)(i+k-1)(\alpha)}^{l(\alpha)}-\Gamma_{j(\alpha)k(\alpha)}^{(l-i+1)(\alpha)}
		\label{A.2_case1}
	\end{align}
	If $i=1$ (or equivalently $j=1$) then this is
	\begin{align}
		\Gamma_{1(\alpha)(j+k-1)(\alpha)}^{l(\alpha)}-\Gamma_{1(\alpha)k(\alpha)}^{(l-j+1)(\alpha)}&=\Gamma_{j(\alpha)k(\alpha)}^{l(\alpha)}-\Gamma_{j(\alpha)k(\alpha)}^{l(\alpha)}
		\notag
	\end{align}
	where both the left and the right-hand sides vanish, as
	\begin{align}
			\Gamma_{1(\alpha)(j+k-1)(\alpha)}^{l(\alpha)}-\Gamma_{1(\alpha)k(\alpha)}^{(l-j+1)(\alpha)}&=-\underset{\sigma\neq\alpha}{\sum}\,\bigg(\Gamma_{1(\sigma)(j+k-1)(\alpha)}^{l(\alpha)}-\Gamma_{1(\sigma)k(\alpha)}^{(l-j+1)(\alpha)}\bigg)
			\notag\\&=-\underset{\sigma\neq\alpha}{\sum}\,\bigg(\Gamma_{1(\sigma)k(\alpha)}^{(l-j+1)(\alpha)}-\Gamma_{1(\sigma)k(\alpha)}^{(l-j+1)(\alpha)}\bigg)=0.
			\notag
	\end{align}
	If $k=1$ then \eqref{A.2_case1} reads
	\begin{align}
		\Gamma_{i(\alpha)j(\alpha)}^{l(\alpha)}-\Gamma_{i(\alpha)1(\alpha)}^{(l-j+1)(\alpha)}&=\Gamma_{j(\alpha)i(\alpha)}^{l(\alpha)}-\Gamma_{j(\alpha)1(\alpha)}^{(l-i+1)(\alpha)}
		\notag
	\end{align}
	that holds true, as
	\begin{align}
		\Gamma_{i(\alpha)1(\alpha)}^{(l-j+1)(\alpha)}&=-\underset{\sigma\neq\alpha}{\sum}\,\Gamma_{i(\alpha)1(\sigma)}^{(l-j+1)(\alpha)}=-\underset{\sigma\neq\alpha}{\sum}\,\Gamma_{1(\alpha)1(\sigma)}^{(l-i-j+2)(\alpha)}
		\notag\\&=-\underset{\sigma\neq\alpha}{\sum}\,\Gamma_{j(\alpha)1(\sigma)}^{(l-i+1)(\alpha)}=\Gamma_{j(\alpha)1(\alpha)}^{(l-i+1)(\alpha)}.
		\notag
	\end{align}
	If all of $i$, $j$ and $k$ are greater or equal then $2$ then \eqref{A.2_case1} reads
	\begin{align}
		\Gamma_{2(\alpha)2(\alpha)}^{(l-i-j-k+5)(\alpha)}-\Gamma_{2(\alpha)2(\alpha)}^{(l-i-j-k+5)(\alpha)}&=\Gamma_{2(\alpha)2(\alpha)}^{(l-i-j-k+5)(\alpha)}-\Gamma_{2(\alpha)2(\alpha)}^{(l-i-j-k+5)(\alpha)}
		\notag
	\end{align}
	which is trivially verified.
	\\\textbf{Case 2: $\alpha=\beta=\gamma\neq\epsilon$.} \eqref{A.2} becomes
	\begin{align}
		\Gamma_{i(\alpha)(j+k-1)(\alpha)}^{l(\epsilon)}&=\Gamma_{j(\alpha)(i+k-1)(\alpha)}^{l(\epsilon)}
		\notag
	\end{align}
	which is true by means of Proposition 8.2.
	\\\textbf{Case 3: $\alpha=\beta=\epsilon\neq\gamma$.} \eqref{A.2} becomes
	\begin{align}
		-\Gamma_{i(\alpha)k(\gamma)}^{(l-j+1)(\alpha)}&=-\Gamma_{j(\alpha)k(\gamma)}^{(l-i+1)(\alpha)}
		\notag
	\end{align}
	which is true by means of Proposition 8.2.
	\\\textbf{Case 4: $\alpha=\gamma=\epsilon\neq\beta$.} \eqref{A.2} becomes
	\begin{align}
		0&=\Gamma_{j(\beta)(i+k-1)(\alpha)}^{l(\alpha)}-\Gamma_{j(\beta)k(\alpha)}^{(l-i+1)(\alpha)}
		\notag
	\end{align}
	which is true by means of Proposition 8.2.
	\\\textbf{Case 5: $\beta=\gamma=\epsilon\neq\alpha$.} \eqref{A.2} becomes
	\begin{align}
		\Gamma_{i(\alpha)(j+k-1)(\beta)}^{l(\beta)}-\Gamma_{i(\alpha)k(\beta)}^{(l-j+1)(\beta)}&=0
		\notag
	\end{align}
	which is true by means of Proposition 8.2.
	\\\textbf{Case 6: $\alpha=\beta\neq\gamma=\epsilon$.} \eqref{A.2} becomes $0=0$.
	\\\textbf{Case 7: $\alpha=\gamma\neq\beta=\epsilon$.} \eqref{A.2} becomes
	\begin{align}
		-\Gamma_{i(\alpha)k(\alpha)}^{(l-j+1)(\beta)}&=\Gamma_{j(\beta)(i+k-1)(\alpha)}^{l(\beta)}
		\notag
	\end{align}
	which is true by means of Proposition 8.2.
	\\\textbf{Case 8: $\alpha=\epsilon\neq\beta=\gamma$.} \eqref{A.2} becomes
	\begin{align}
		\Gamma_{i(\alpha)(j+k-1)(\beta)}^{l(\alpha)}&=-\Gamma_{j(\beta)k(\beta)}^{(l-i+1)(\alpha)}
		\notag
	\end{align}
	which is true by means of Proposition 8.2.
	\\\textbf{Case 9: }{at least three among $\alpha$, $\beta$, $\gamma$, $\epsilon$ are pairwise distinct.} \eqref{A.2} becomes trivially $0=0$, as all the Dirac deltas vanish.
	\\This proves point A.2. The product $\circ$ is commutative and associative, as
	\begin{align}
		c^{i(\alpha)}_{j(\beta)k(\gamma)}&=\delta^\alpha_\beta\delta^\alpha_\gamma\delta^i_{j+k-1}=\delta^\alpha_\gamma\delta^\alpha_\beta\delta^i_{k+j-1}=c^{i(\alpha)}_{k(\gamma)j(\beta)}
		\notag
	\end{align}
	and
	\begin{align}
		c^{i(\alpha)}_{j(\beta)l(\epsilon)}c^{l(\epsilon)}_{k(\gamma)m(\mu)}&=\delta^\alpha_\beta\delta^\alpha_\epsilon\delta^i_{j+l-1}\delta^\epsilon_\gamma\delta^\epsilon_\mu\delta^l_{k+m-1}=\mathds{1}_{\{\alpha=\beta=\gamma=\mu\}}\,\delta^i_{j+k+m-2}
		\notag\\&=\delta^\alpha_\gamma\delta^\alpha_\epsilon\delta^i_{k+l-1}\delta^\epsilon_\beta\delta^\epsilon_\mu\delta^l_{j+m-1}=c^{i(\alpha)}_{k(\gamma)l(\epsilon)}c^{l(\epsilon)}_{j(\beta)m(\mu)}.
		\notag
	\end{align}
	This proves point A.3. Since $e$ is the unit of $\circ$, point B holds automatically. However, this can be seen explicitly by
	\begin{align}
		\big(X\circ e\big)^{i(\alpha)}&=X^{j(\beta)}c^{i(\alpha)}_{j(\beta)k(\gamma)}e^{k(\gamma)}=X^{j(\alpha)}\delta^{i}_{j+k-1}e^{k(\alpha)}=X^{j(\alpha)}\delta^{i}_{j+k-1}\delta^{k}_1=X^{i(\alpha)}
		\notag
	\end{align}
	for an arbitrary vector field $X$ and every $\alpha\in\{1,\dots,r\}$, $i\in\{1,\dots,m_\alpha\}$. Clearly point C holds trivially as well, since it is equivalent to the condition
	\begin{equation}
		\notag
		\overset{r}{\underset{\sigma=1}{\sum}}\,\Gamma^{i(\alpha)}_{1(\sigma)j(\beta)}=0
	\end{equation}
	expressed in Proposition 8.1. Since
	\begin{align}
		\notag
		X*Y&=E^{-1}\circ X\circ Y\qquad\forall X,Y
	\end{align}
	by definition, the product $*$ inherits the commutativity and associativity properties from the product $\circ$, proving point D.3. Given an arbitrary vector field $X$, we have
	\begin{align}
		\notag
		X*E&=E^{-1}\circ X\circ E=E^{-1}\circ E\circ X=e\circ X=X
	\end{align}
	which proves E. We are now going to prove point F, that is $\nabla_{j(\beta)}^*E^{i(\alpha)}=0$ given $\alpha$, $\beta\in\{1,\dots,r\}$ and $i\in\{1,\dots,m_\alpha\}$, $j\in\{1,\dots,m_\beta\}$. We have
	\begin{align}
		\nabla_{j(\beta)}^*E^{i(\alpha)}&=\partial_{j(\beta)}E^{i(\alpha)}+{\Gamma^*}^{i(\alpha)}_{j(\beta)k(\gamma)}\,E^{k(\gamma)}=\delta^\alpha_\beta\,\delta^i_j+{\Gamma^*}^{i(\alpha)}_{j(\beta)k(\gamma)}\,u^{k(\gamma)}
		\notag
	\end{align}
	which by means of Proposition 8.13 is
	\begin{align}
		\nabla_{j(\beta)}^*E^{i(\alpha)}&=\delta^\alpha_\beta\,\delta^i_j+{\Gamma}^{i(\alpha)}_{j(\beta)k(\gamma)}\,u^{k(\gamma)}-\delta^\alpha_\beta\delta^\alpha_\gamma\,(E^{-1})^{(i-j-k+2)(\alpha)}\,\bigg[\delta^i_1\bigg(1-\underset{\sigma\neq\alpha}{\sum}\,m_\sigma\epsilon_\sigma\bigg)\notag\\&+(1-\delta^i_1)\bigg(1-\overset{r}{\underset{\tau=1}{\sum}}\,m_\tau\epsilon_\tau\bigg)\bigg]\,u^{k(\gamma)}-(1-\delta^\alpha_\beta)\,\delta_{\beta\gamma}\,\delta^1_j\delta^1_k\delta^i_1\,\frac{m_\beta\epsilon_\beta}{u^{1(\beta)}}\,u^{k(\gamma)}
		\notag\\&=\delta^\alpha_\beta\,\delta^i_j+{\Gamma}^{i(\alpha)}_{j(\beta)k(\gamma)}\,u^{k(\gamma)}-\delta^\alpha_\beta\,(E^{-1})^{(i-j-k+2)(\alpha)}\,\bigg[\delta^i_1\bigg(1-\underset{\sigma\neq\alpha}{\sum}\,m_\sigma\epsilon_\sigma\bigg)\notag\\&+(1-\delta^i_1)\bigg(1-\overset{r}{\underset{\tau=1}{\sum}}\,m_\tau\epsilon_\tau\bigg)\bigg]\,u^{k(\alpha)}-(1-\delta^\alpha_\beta)\,\delta^1_j\delta^1_k\delta^i_1\,\frac{m_\beta\epsilon_\beta}{u^{1(\beta)}}\,u^{k(\beta)}.
		\notag
	\end{align}
	Since
	\begin{align}
		\Gamma^{i(\alpha)}_{j(\beta)k(\gamma)}\,u^{k(\gamma)}=\bigg(-\delta^\alpha_\beta\,\underset{\sigma\neq\alpha}{\sum}\,m_\sigma\epsilon_\sigma+(1-\delta^\alpha_\beta)\,m_\beta\epsilon_\beta\bigg)\delta^i_j\delta^i_1+\bigg(-\delta^\alpha_\beta\,\overset{r}{\underset{\tau=1}{\sum}}\,m_\tau\epsilon_\tau\bigg)\delta^i_j(1-\delta^i_1)
		\notag
	\end{align}
	by Lemma 8.6, we get
	\begin{align}
		\nabla_{j(\beta)}^*E^{i(\alpha)}&=\delta^\alpha_\beta\,\delta^i_j+\bigg(-\delta^\alpha_\beta\,\underset{\sigma\neq\alpha}{\sum}\,m_\sigma\epsilon_\sigma+(1-\delta^\alpha_\beta)\,m_\beta\epsilon_\beta\bigg)\delta^i_j\delta^i_1\notag\\&-\delta^\alpha_\beta\,\overset{r}{\underset{\tau=1}{\sum}}\,m_\tau\epsilon_\tau\,\delta^i_j(1-\delta^i_1)-\delta^\alpha_\beta\,(E^{-1})^{(i-j-k+2)(\alpha)}\,\bigg[\delta^i_1\bigg(1-\underset{\sigma\neq\alpha}{\sum}\,m_\sigma\epsilon_\sigma\bigg)\notag\\&+(1-\delta^i_1)\bigg(1-\overset{r}{\underset{\tau=1}{\sum}}\,m_\tau\epsilon_\tau\bigg)\bigg]\,u^{k(\alpha)}-(1-\delta^\alpha_\beta)\,\delta^1_j\delta^1_k\delta^i_1\,\frac{m_\beta\epsilon_\beta}{u^{1(\beta)}}\,u^{k(\beta)}.
		\notag
	\end{align}
	If $\alpha\neq\beta$ then
	\begin{align}
		\nabla_{j(\beta)}^*E^{i(\alpha)}&=m_\beta\epsilon_\beta\,\delta^i_j\delta^i_1-\delta^1_j\delta^1_k\delta^i_1\,\frac{m_\beta\epsilon_\beta}{u^{1(\beta)}}\,u^{k(\beta)}=m_\beta\epsilon_\beta\,\delta^i_1(\delta^i_j-\delta^1_j)=0.
		\notag
	\end{align}
	If $\alpha=\beta$ then
	\begin{align}
		\nabla_{j(\beta)}^*E^{i(\alpha)}&=\delta^i_j-\underset{\sigma\neq\alpha}{\sum}\,m_\sigma\epsilon_\sigma\,\delta^i_j\delta^i_1-\overset{r}{\underset{\tau=1}{\sum}}\,m_\tau\epsilon_\tau\,\delta^i_j(1-\delta^i_1)\label{thmF_ij}\\&-(E^{-1})^{(i-j-k+2)(\alpha)}\,\bigg[\delta^i_1\bigg(1-\underset{\sigma\neq\alpha}{\sum}\,m_\sigma\epsilon_\sigma\bigg)+(1-\delta^i_1)\bigg(1-\overset{r}{\underset{\tau=1}{\sum}}\,m_\tau\epsilon_\tau\bigg)\bigg]\,u^{k(\alpha)}
		\notag
	\end{align}
	which vanishes trivially if $i<j$. If $i=j$ then \eqref{thmF_ij} becomes
	\begin{align}
		\nabla_{j(\beta)}^*E^{i(\alpha)}&=1-\underset{\sigma\neq\alpha}{\sum}\,m_\sigma\epsilon_\sigma\delta^i_1-\overset{r}{\underset{\tau=1}{\sum}}\,m_\tau\epsilon_\tau(1-\delta^i_1)\notag\\&-(E^{-1})^{(2-k)(\alpha)}\delta^1_k\,\bigg[\delta^i_1\bigg(1-\underset{\sigma\neq\alpha}{\sum}\,m_\sigma\epsilon_\sigma\bigg)+(1-\delta^i_1)\bigg(1-\overset{r}{\underset{\tau=1}{\sum}}\,m_\tau\epsilon_\tau\bigg)\bigg]\,u^{k(\alpha)}
		\notag\\&=1-\delta^i_1\underset{\sigma\neq\alpha}{\sum}\,m_\sigma\epsilon_\sigma-(1-\delta^i_1)\overset{r}{\underset{\tau=1}{\sum}}\,m_\tau\epsilon_\tau\notag\\&-\cancel{(E^{-1})^{1(\alpha)}}\bigg[\delta^i_1\bigg(1-\underset{\sigma\neq\alpha}{\sum}\,m_\sigma\epsilon_\sigma\bigg)+(1-\delta^i_1)\bigg(1-\overset{r}{\underset{\tau=1}{\sum}}\,m_\tau\epsilon_\tau\bigg)\bigg]\,\cancel{u^{1(\alpha)}}
		\notag\\&=1-\cancel{\delta^i_1\underset{\sigma\neq\alpha}{\sum}\,m_\sigma\epsilon_\sigma}-\bcancel{(1-\delta^i_1)\overset{r}{\underset{\tau=1}{\sum}}\,m_\tau\epsilon_\tau}\notag\\&-\delta^i_1+\cancel{\delta^i_1\underset{\sigma\neq\alpha}{\sum}\,m_\sigma\epsilon_\sigma}-(1-\delta^i_1)+\bcancel{(1-\delta^i_1)\overset{r}{\underset{\tau=1}{\sum}}\,m_\tau\epsilon_\tau}
		\notag\\&=1-\delta^i_1-1+\delta^i_1=0.
		\notag
	\end{align}
	Let us introduce the constant $C_\alpha:=\delta^i_1\bigg(1-\underset{\sigma\neq\alpha}{\sum}\,m_\sigma\epsilon_\sigma\bigg)+(1-\delta^i_1)\bigg(1-\overset{r}{\underset{\tau=1}{\sum}}\,m_\tau\epsilon_\tau\bigg)$. If $i>j$ then \eqref{thmF_ij} becomes
	\begin{align}
		\nabla_{j(\beta)}^*E^{i(\alpha)}&=-(E^{-1})^{(i-j-k+2)(\alpha)}\,\bigg[\delta^i_1\bigg(1-\underset{\sigma\neq\alpha}{\sum}\,m_\sigma\epsilon_\sigma\bigg)+(1-\delta^i_1)\bigg(1-\overset{r}{\underset{\tau=1}{\sum}}\,m_\tau\epsilon_\tau\bigg)\bigg]\,u^{k(\alpha)}
		\notag\\&=-C_\alpha\,\overset{i-j+1}{\underset{k=1}{\sum}}\,(E^{-1})^{(i-j-k+2)(\alpha)}\,u^{k(\alpha)}
		\notag\\&=-C_\alpha\,(E^{-1})^{(i-j+1)(\alpha)}\,u^{1(\alpha)}-C_\alpha\,\overset{i-j+1}{\underset{k=2}{\sum}}\,(E^{-1})^{(i-j-k+2)(\alpha)}\,u^{k(\alpha)}
		\notag
	\end{align}
	where
	\begin{align}
		&(E^{-1})^{(i-j+1)(\alpha)}=-\frac{1}{u^{1(\alpha)}}\,\overset{i-j}{\underset{s=1}{\sum}}\,(E^{-1})^{(i-j-s+1)(\alpha)}\,u^{(s+1)(\alpha)}
		\notag
	\end{align}
	by Proposition 8.12. Thus
	\begin{align}
		\nabla_{j(\beta)}^*E^{i(\alpha)}&=C_\alpha\,\overset{i-j}{\underset{s=1}{\sum}}\,(E^{-1})^{(i-j-s+1)(\alpha)}\,u^{(s+1)(\alpha)}-C_\alpha\,\overset{i-j+1}{\underset{k=2}{\sum}}\,(E^{-1})^{(i-j-k+2)(\alpha)}\,u^{k(\alpha)}=0.
		\notag
	\end{align}
	This proves point F. Point G is proved by the following two computations:
	\begin{align}
		[e,E]&=e(E)-E(e)=\overset{r}{\underset{\alpha=1}{\sum}}\,\partial_{1(\alpha)}\big(u^s\,\partial_s\big)-u^s\,\partial_s\bigg(\overset{r}{\underset{\alpha=1}{\sum}}\,\partial_{1(\alpha)}\bigg)
		\notag\\&=\overset{r}{\underset{\alpha=1}{\sum}}\,\partial_{1(\alpha)}+\cancel{\overset{r}{\underset{\alpha=1}{\sum}}\,u^{s(\beta)}\,\partial_{1(\alpha)}\partial_{s(\beta)}}-\cancel{\overset{r}{\underset{\alpha=1}{\sum}}\,u^{s(\beta)}\,\partial_{s(\beta)}\partial_{1(\alpha)}}
		\notag\\&=\overset{r}{\underset{\alpha=1}{\sum}}\,\partial_{1(\alpha)}=e
		\notag
	\end{align}
	and
	\begin{align}
		{\rm Lie}_E\,c^{i(\alpha)}_{j(\beta)k(\gamma)}&=E^s\,\xcancel{\partial_s\,c^{i(\alpha)}_{j(\beta)k(\gamma)}}-c^{s(\sigma)}_{j(\beta)k(\gamma)}\,\partial_{s(\sigma)}\,E^{i(\alpha)}\notag\\&+c^{i(\alpha)}_{s(\sigma)k(\gamma)}\,\partial_{j(\beta)}\,E^{s(\sigma)}+c^{i(\alpha)}_{j(\beta)s(\sigma)}\,\partial_{k(\gamma)}\,E^{s(\sigma)}
		\notag\\&=-c^{i(\alpha)}_{j(\beta)k(\gamma)}+c^{i(\alpha)}_{j(\beta)k(\gamma)}+c^{i(\alpha)}_{j(\beta)k(\gamma)}=c^{i(\alpha)}_{j(\beta)k(\gamma)}.
		\notag
	\end{align}
	Given an arbitrary vector field $X$, we have
	\begin{align}
		(d_{\nabla}-d_{\nabla^{*}})(X\,*)^i_{jk}&=\nabla_j(X\,*)^i_k-\nabla_k(X\,*)^i_j-{\nabla^*}_j(X\,*)^i_k+{\nabla^*}_k(X\,*)^i_j
		\notag\\&=\cancel{\partial_j(X\,*)^i_k}+{\Gamma}^i_{jl}(X\,*)^l_k-\cancel{\cancel{{\Gamma}^l_{jk}(X\,*)^i_l}}\notag\\&-\bcancel{\partial_k(X\,*)^i_j}-{\Gamma}^i_{kl}(X\,*)^l_j+\cancel{\cancel{{\Gamma}^l_{kj}(X\,*)^i_l}}\notag\\&-\cancel{\partial_j(X\,*)^i_k}-{\Gamma^*}^i_{jl}(X\,*)^l_k+\bcancel{\bcancel{{\Gamma^*}^l_{jk}(X\,*)^i_l}}\notag\\&+\bcancel{\partial_k(X\,*)^i_j}+{\Gamma^*}^i_{kl}(X\,*)^l_j-\bcancel{\bcancel{{\Gamma^*}^l_{kj}(X\,*)^i_l}}
		\notag\\&=\big({\Gamma}^i_{jl}-{\Gamma^*}^i_{jl}\big)(X\,*)^l_k-\big({\Gamma}^i_{kl}-{\Gamma^*}^i_{kl}\big)(X\,*)^l_j
		\notag\\&\overset{\eqref{dualfromnatural}}{=}{c^*}^m_{jl}\,\nabla_m\,E^i\,(X\,*)^l_k-{c^*}^m_{kl}\,\nabla_m\,E^i(X\,*)^l_j
		\notag\\&={c^*}^m_{jl}\,\nabla_m\,E^i\,X^s\,{c^*}^l_{sk}-{c^*}^m_{kl}\,\nabla_m\,E^i\,X^s\,{c^*}^l_{sj}
		\notag\\&=\big({c^*}^m_{jl}\,{c^*}^l_{sk}-{c^*}^m_{kl}\,{c^*}^l_{sj}\big)\,X^s\,\nabla_m\,E^i
		\notag
	\end{align}
	where the quantity inside the brackets vanishes due to the associativity property of $*$. Therefore $(d_{\nabla}-d_{\nabla^{*}})(X\,*)^i_{jk}=0$, which is equivalent to $(d_{\nabla}-d_{\nabla^{*}})(X\,\circ)^i_{jk}=0$, This proves point H. Since all of the Christoffel symbols $\Gamma^{i(\alpha)}_{j(\beta)k(\gamma)}$ are uniquely determined by means of Proposition 8.4 and the Christoffel symbols ${\Gamma^*}^{i(\alpha)}_{j(\beta)k(\gamma)}$ are uniquely determined by \eqref{Christoffel*}, point I is proved as well.
	
	In order to complete the proof we are left to show that $\nabla$, $\nabla^*$ are flat connections. We know that if we take the flatness of $\nabla$ as already verified and assume $\nabla\nabla E=0$, then we deduce that  $(\nabla,\nabla^*,\circ,*,e,E)$ define a bi-flat structure on $M$. It is then enough for us to only prove the flatness of $\nabla$ and to verify the condition $\nabla\nabla E=0$.
	
	Let us start by proving $\nabla\nabla E=0$. We have
	\begin{align}
		\nabla_{i(\alpha)}\nabla_{j(\beta)}E^{k(\gamma)}&=\partial_{i(\alpha)}\nabla_{j(\beta)}E^{k(\gamma)}+\Gamma^{k(\gamma)}_{i(\alpha)l(\sigma)}\nabla_{j(\beta)}E^{l(\sigma)}-\Gamma^{l(\sigma)}_{i(\alpha)j(\beta)}\nabla_{l(\sigma)}E^{k(\gamma)}
		\notag
	\end{align}
	where, by means of Lemma 8.7, $\nabla_{j(\beta)}E^{k(\gamma)}$ is constant and $\nabla_{j(\beta)}E^{l(\sigma)}$, $\nabla_{l(\sigma)}E^{k(\gamma)}$ vanish respectively whenever $l\neq j$, $l\neq k$. Thus
	\begin{align}
		\nabla_{i(\alpha)}\nabla_{j(\beta)}E^{k(\gamma)}&=\Gamma^{k(\gamma)}_{i(\alpha)j(\sigma)}\nabla_{j(\beta)}E^{j(\sigma)}-\Gamma^{k(\sigma)}_{i(\alpha)j(\beta)}\nabla_{k(\sigma)}E^{k(\gamma)}.
		\label{nablanablaEdim}
	\end{align}
	The possible cases are:
	\begin{itemize}
		\item[$1.$] $\alpha=\beta=\gamma$
		\item[$2.$] $\alpha=\beta\neq\gamma$
		\item[$3.$] $\alpha=\gamma\neq\beta$
		\item[$4.$] $\beta=\gamma\neq\alpha$
		\item[$5.$] $\alpha\neq\beta\neq\gamma\neq\alpha$.
	\end{itemize}
	\textbf{Case 1: $\alpha=\beta=\gamma$.} 
	\begin{align}
		\nabla_{i(\alpha)}\nabla_{j(\beta)}E^{k(\gamma)}&=\Gamma^{k(\alpha)}_{i(\alpha)j(\sigma)}\nabla_{j(\alpha)}E^{j(\sigma)}-\Gamma^{k(\sigma)}_{i(\alpha)j(\alpha)}\nabla_{k(\sigma)}E^{k(\alpha)}
		\notag\\&=\Gamma^{k(\alpha)}_{i(\alpha)j(\alpha)}\big(\nabla_{j(\alpha)}E^{j(\alpha)}-\nabla_{k(\alpha)}E^{k(\alpha)}\big)\notag\\&+\underset{\sigma\neq\alpha}{\sum}\bigg(\Gamma^{k(\alpha)}_{i(\alpha)j(\sigma)}\nabla_{j(\alpha)}E^{j(\sigma)}-\Gamma^{k(\sigma)}_{i(\alpha)j(\alpha)}\nabla_{k(\sigma)}E^{k(\alpha)}\bigg).
		\label{nnE_case1}
	\end{align}
	If $j=k=1$ then
	\begin{align}
		\nabla_{i(\alpha)}\nabla_{j(\beta)}E^{k(\gamma)}&=\Gamma^{1(\alpha)}_{i(\alpha)1(\alpha)}\big(\nabla_{1(\alpha)}E^{1(\alpha)}-\nabla_{1(\alpha)}E^{1(\alpha)}\big)\notag\\&+\underset{\sigma\neq\alpha}{\sum}\bigg(\Gamma^{1(\alpha)}_{i(\alpha)1(\sigma)}\nabla_{1(\alpha)}E^{1(\sigma)}-\Gamma^{1(\sigma)}_{i(\alpha)1(\alpha)}\nabla_{1(\sigma)}E^{1(\alpha)}\bigg)
		\notag
	\end{align}
	which trivially vanishes if $i\geq2$ and becomes
	\begin{align}
		\nabla_{i(\alpha)}\nabla_{j(\beta)}E^{k(\gamma)}&=\Gamma^{1(\alpha)}_{1(\alpha)1(\alpha)}\big(\cancel{\nabla_{1(\alpha)}E^{1(\alpha)}}-\cancel{\nabla_{1(\alpha)}E^{1(\alpha)}}\big)\notag\\&+\underset{\sigma\neq\alpha}{\sum}\bigg(\Gamma^{1(\alpha)}_{1(\alpha)1(\sigma)}\nabla_{1(\alpha)}E^{1(\sigma)}-\Gamma^{1(\sigma)}_{1(\alpha)1(\alpha)}\nabla_{1(\sigma)}E^{1(\alpha)}\bigg)
		\notag\\&=\underset{\sigma\neq\alpha}{\sum}\bigg(\Gamma^{1(\alpha)}_{1(\alpha)1(\sigma)}\nabla_{1(\alpha)}E^{1(\sigma)}+\Gamma^{1(\sigma)}_{1(\alpha)1(\sigma)}\nabla_{1(\sigma)}E^{1(\alpha)}\bigg)
		\notag\\&=\underset{\sigma\neq\alpha}{\sum}\bigg(\frac{m_\sigma\epsilon_\sigma}{u^{1(\alpha)}-u^{1(\sigma)}}\,m_\alpha\epsilon_\alpha-\frac{m_\alpha\epsilon_\alpha}{u^{1(\alpha)}-u^{1(\sigma)}}\,m_\sigma\epsilon_\sigma\bigg)=0
		\notag
	\end{align}
	if $i=1$. If both $j$ and $k$ are both greater or equal then $2$ then \eqref{nnE_case1} becomes
	\begin{align}
		\nabla_{i(\alpha)}\nabla_{j(\beta)}E^{k(\gamma)}&=\Gamma^{k(\alpha)}_{i(\alpha)j(\alpha)}\big(\nabla_{j(\alpha)}E^{j(\alpha)}-\nabla_{k(\alpha)}E^{k(\alpha)}\big)\notag\\&+\underset{\sigma\neq\alpha}{\sum}\bigg(\Gamma^{k(\alpha)}_{i(\alpha)j(\sigma)}\nabla_{j(\alpha)}E^{j(\sigma)}-\Gamma^{k(\sigma)}_{i(\alpha)j(\alpha)}\nabla_{k(\sigma)}E^{k(\alpha)}\bigg)
		\notag\\&=\Gamma^{k(\alpha)}_{i(\alpha)j(\alpha)}\bigg(\cancel{1-\overset{r}{\underset{\tau=1}{\sum}}\,m_\tau\epsilon_\tau}\cancel{-1+\overset{r}{\underset{\tau=1}{\sum}}\,m_\tau\epsilon_\tau}\bigg)\notag\\&+\underset{\sigma\neq\alpha}{\sum}\bigg(\Gamma^{k(\alpha)}_{i(\alpha)j(\sigma)}\xcancel{\nabla_{j(\alpha)}E^{j(\sigma)}}-\Gamma^{k(\sigma)}_{i(\alpha)j(\alpha)}\xcancel{\nabla_{k(\sigma)}E^{k(\alpha)}}\bigg)=0.
		\notag
	\end{align}
	If $j=1$ and $k\geq2$ then \eqref{nnE_case1} becomes
	\begin{align}
		\nabla_{i(\alpha)}\nabla_{j(\beta)}E^{k(\gamma)}&=\Gamma^{k(\alpha)}_{i(\alpha)1(\alpha)}\big(\nabla_{1(\alpha)}E^{1(\alpha)}-\nabla_{k(\alpha)}E^{k(\alpha)}\big)\notag\\&+\underset{\sigma\neq\alpha}{\sum}\bigg(\Gamma^{k(\alpha)}_{i(\alpha)1(\sigma)}\nabla_{1(\alpha)}E^{1(\sigma)}-\Gamma^{k(\sigma)}_{i(\alpha)1(\alpha)}\nabla_{k(\sigma)}E^{k(\alpha)}\bigg)
		\notag\\&=\Gamma^{k(\alpha)}_{i(\alpha)1(\alpha)}\bigg(1-\underset{\sigma\neq\alpha}{\sum}\,m_\sigma\epsilon_\sigma-1+\overset{r}{\underset{\tau=1}{\sum}}\,m_\tau\epsilon_\tau\bigg)\notag\\&+\underset{\sigma\neq\alpha}{\sum}\bigg(\Gamma^{k(\alpha)}_{i(\alpha)1(\sigma)}\,m_\alpha\epsilon_\alpha-\Gamma^{k(\sigma)}_{i(\alpha)1(\alpha)}\xcancel{\nabla_{k(\sigma)}E^{k(\alpha)}}\bigg)
		\notag\\&=\Gamma^{k(\alpha)}_{i(\alpha)1(\alpha)}\,m_\alpha\epsilon_\alpha+\underset{\sigma\neq\alpha}{\sum}\,\Gamma^{k(\alpha)}_{i(\alpha)1(\sigma)}\,m_\alpha\epsilon_\alpha
		\notag\\&=-\underset{\sigma\neq\alpha}{\sum}\,\Gamma^{k(\alpha)}_{i(\alpha)1(\sigma)}\,m_\alpha\epsilon_\alpha+\underset{\sigma\neq\alpha}{\sum}\,\Gamma^{k(\alpha)}_{i(\alpha)1(\sigma)}\,m_\alpha\epsilon_\alpha=0.
		\notag
	\end{align}
	If $j\geq2$ and $k=1$ then \eqref{nnE_case1} becomes
	\begin{align}
		\nabla_{i(\alpha)}\nabla_{j(\beta)}E^{k(\gamma)}&=\Gamma^{1(\alpha)}_{i(\alpha)j(\alpha)}\big(\nabla_{j(\alpha)}E^{j(\alpha)}-\nabla_{1(\alpha)}E^{1(\alpha)}\big)\notag\\&+\underset{\sigma\neq\alpha}{\sum}\bigg(\Gamma^{1(\alpha)}_{i(\alpha)j(\sigma)}\nabla_{j(\alpha)}E^{j(\sigma)}-\Gamma^{1(\sigma)}_{i(\alpha)j(\alpha)}\nabla_{1(\sigma)}E^{1(\alpha)}\bigg)
		\notag\\&=-\Gamma^{1(\alpha)}_{i(\alpha)j(\alpha)}\,m_\alpha\epsilon_\alpha+\underset{\sigma\neq\alpha}{\sum}\bigg(\Gamma^{1(\alpha)}_{i(\alpha)j(\sigma)}\xcancel{\nabla_{j(\alpha)}E^{j(\sigma)}}-\Gamma^{1(\sigma)}_{i(\alpha)j(\alpha)}\,m_\sigma\epsilon_\sigma\bigg)
		\notag\\&=-\Gamma^{1(\alpha)}_{i(\alpha)j(\alpha)}\,m_\alpha\epsilon_\alpha-\underset{\sigma\neq\alpha}{\sum}\,\Gamma^{1(\sigma)}_{i(\alpha)j(\alpha)}\,m_\sigma\epsilon_\sigma
		\notag
	\end{align}
	which is
	\begin{align}
		\nabla_{i(\alpha)}\nabla_{j(\beta)}E^{k(\gamma)}&=-\Gamma^{1(\alpha)}_{1(\alpha)j(\alpha)}\,m_\alpha\epsilon_\alpha-\underset{\sigma\neq\alpha}{\sum}\,\Gamma^{1(\sigma)}_{1(\alpha)j(\alpha)}\,m_\sigma\epsilon_\sigma
		\notag\\&=\underset{\sigma\neq\alpha}{\sum}\,\xcancel{\Gamma^{1(\alpha)}_{1(\sigma)j(\alpha)}}\,m_\alpha\epsilon_\alpha+\underset{\sigma\neq\alpha}{\sum}\,\xcancel{\Gamma^{1(\sigma)}_{1(\sigma)j(\alpha)}}\,m_\sigma\epsilon_\sigma=0
		\notag
	\end{align}
	if $i=1$ and
	\begin{align}
		\nabla_{i(\alpha)}\nabla_{j(\beta)}E^{k(\gamma)}&=-\Gamma^{1(\alpha)}_{i(\alpha)j(\alpha)}\,m_\alpha\epsilon_\alpha-\underset{\sigma\neq\alpha}{\sum}\,\xcancel{\Gamma^{1(\sigma)}_{i(\alpha)j(\alpha)}}\,m_\sigma\epsilon_\sigma=0
		\notag
	\end{align}
	(as $1-i-j\leq1-2-2=-3$) if $i\geq2$.
	\\\textbf{Case 2: $\alpha=\beta\neq\gamma$.} 
	\begin{align}
		\nabla_{i(\alpha)}\nabla_{j(\alpha)}E^{k(\gamma)}&=\Gamma^{k(\gamma)}_{i(\alpha)j(\sigma)}\nabla_{j(\alpha)}E^{j(\sigma)}-\Gamma^{k(\sigma)}_{i(\alpha)j(\alpha)}\nabla_{k(\sigma)}E^{k(\gamma)}
		\notag\\&=\Gamma^{k(\gamma)}_{i(\alpha)j(\alpha)}\nabla_{j(\alpha)}E^{j(\alpha)}-\Gamma^{k(\alpha)}_{i(\alpha)j(\alpha)}\nabla_{k(\alpha)}E^{k(\gamma)}
		\notag\\&+\Gamma^{k(\gamma)}_{i(\alpha)j(\gamma)}\nabla_{j(\alpha)}E^{j(\gamma)}-\Gamma^{k(\gamma)}_{i(\alpha)j(\alpha)}\nabla_{k(\gamma)}E^{k(\gamma)}
		\notag\\&+\underset{\sigma\neq\alpha,\gamma}{\sum}\,\bigg(\xcancel{\Gamma^{k(\gamma)}_{i(\alpha)j(\sigma)}}\nabla_{j(\alpha)}E^{j(\sigma)}-\Gamma^{k(\sigma)}_{i(\alpha)j(\alpha)}\nabla_{k(\sigma)}E^{k(\gamma)}\bigg).
		\label{nnE_caso2}
	\end{align}
	If both $j$ and $k$ are greater or equal than $2$ then it becomes
	\begin{align}
		\nabla_{i(\alpha)}\nabla_{j(\alpha)}E^{k(\gamma)}&=\Gamma^{k(\gamma)}_{i(\alpha)j(\alpha)}\nabla_{j(\alpha)}E^{j(\alpha)}-\Gamma^{k(\alpha)}_{i(\alpha)j(\alpha)}\xcancel{\nabla_{k(\alpha)}E^{k(\gamma)}}
		\notag\\&+\Gamma^{k(\gamma)}_{i(\alpha)j(\gamma)}\xcancel{\nabla_{j(\alpha)}E^{j(\gamma)}}-\Gamma^{k(\gamma)}_{i(\alpha)j(\alpha)}\nabla_{k(\gamma)}E^{k(\gamma)}
		\notag\\&-\underset{\sigma\neq\alpha,\gamma}{\sum}\,\Gamma^{k(\sigma)}_{i(\alpha)j(\alpha)}\xcancel{\nabla_{k(\sigma)}E^{k(\gamma)}}
		\notag\\&=\Gamma^{k(\gamma)}_{i(\alpha)j(\alpha)}\big(\cancel{\nabla_{j(\alpha)}E^{j(\alpha)}}-\cancel{\nabla_{k(\gamma)}E^{k(\gamma)}}\big)=0.
		\notag
	\end{align}
	If $j=k=1$ then \eqref{nnE_caso2} becomes
	\begin{align}
		\nabla_{i(\alpha)}\nabla_{j(\alpha)}E^{k(\gamma)}&=\Gamma^{1(\gamma)}_{i(\alpha)1(\alpha)}\nabla_{1(\alpha)}E^{1(\alpha)}-\Gamma^{1(\alpha)}_{i(\alpha)1(\alpha)}\nabla_{1(\alpha)}E^{1(\gamma)}
		\notag\\&+\Gamma^{1(\gamma)}_{i(\alpha)1(\gamma)}\nabla_{1(\alpha)}E^{1(\gamma)}-\Gamma^{1(\gamma)}_{i(\alpha)1(\alpha)}\nabla_{1(\gamma)}E^{1(\gamma)}
		\notag\\&-\underset{\sigma\neq\alpha,\gamma}{\sum}\,\Gamma^{1(\sigma)}_{i(\alpha)1(\alpha)}\nabla_{1(\sigma)}E^{1(\gamma)}
		\notag\\&=-\Gamma^{1(\gamma)}_{i(\alpha)1(\gamma)}\nabla_{1(\alpha)}E^{1(\alpha)}\notag\\&+\Gamma^{1(\alpha)}_{i(\alpha)1(\gamma)}\nabla_{1(\alpha)}E^{1(\gamma)}+\underset{\sigma\neq\alpha,\gamma}{\sum}\,\Gamma^{1(\alpha)}_{i(\alpha)1(\sigma)}\nabla_{1(\alpha)}E^{1(\gamma)}
		\notag\\&+\Gamma^{1(\gamma)}_{i(\alpha)1(\gamma)}\nabla_{1(\alpha)}E^{1(\gamma)}+\Gamma^{1(\gamma)}_{i(\alpha)1(\gamma)}\nabla_{1(\gamma)}E^{1(\gamma)}
		\notag\\&+\underset{\sigma\neq\alpha,\gamma}{\sum}\,\Gamma^{1(\sigma)}_{i(\alpha)1(\sigma)}\nabla_{1(\sigma)}E^{1(\gamma)}
		\notag\\&=-\Gamma^{1(\gamma)}_{i(\alpha)1(\gamma)}\bigg(1-\underset{\sigma\neq\alpha}{\sum}\,m_\sigma\epsilon_\sigma\bigg)\notag\\&+\Gamma^{1(\alpha)}_{i(\alpha)1(\gamma)}\,m_\alpha\epsilon_\alpha+\underset{\sigma\neq\alpha,\gamma}{\sum}\,\Gamma^{1(\alpha)}_{i(\alpha)1(\sigma)}\,m_\alpha\epsilon_\alpha
		\notag\\&+\Gamma^{1(\gamma)}_{i(\alpha)1(\gamma)}\,m_\alpha\epsilon_\alpha+\Gamma^{1(\gamma)}_{i(\alpha)1(\gamma)}\bigg(1-\underset{\sigma\neq\gamma}{\sum}\,m_\sigma\epsilon_\sigma\bigg)
		\notag\\&+\underset{\sigma\neq\alpha,\gamma}{\sum}\,\Gamma^{1(\sigma)}_{i(\alpha)1(\sigma)}\,m_\sigma\epsilon_\sigma
		\notag\\&=\Gamma^{1(\gamma)}_{i(\alpha)1(\gamma)}\,m_\gamma\epsilon_\gamma+\Gamma^{1(\alpha)}_{i(\alpha)1(\gamma)}\,m_\alpha\epsilon_\alpha+\underset{\sigma\neq\alpha,\gamma}{\sum}\,\Gamma^{1(\alpha)}_{i(\alpha)1(\sigma)}\,m_\alpha\epsilon_\alpha
		\notag\\&+\underset{\sigma\neq\alpha,\gamma}{\sum}\,\Gamma^{1(\sigma)}_{i(\alpha)1(\sigma)}\,m_\sigma\epsilon_\sigma
		\notag
	\end{align}
	which trivially vanishes if $i\geq2$ and becomes
	\begin{align}
		\nabla_{i(\alpha)}\nabla_{j(\alpha)}E^{k(\gamma)}&=\Gamma^{1(\gamma)}_{1(\alpha)1(\gamma)}\,m_\gamma\epsilon_\gamma+\Gamma^{1(\alpha)}_{1(\alpha)1(\gamma)}\,m_\alpha\epsilon_\alpha
		\notag\\&+\underset{\sigma\neq\alpha,\gamma}{\sum}\,\big(\Gamma^{1(\alpha)}_{1(\alpha)1(\sigma)}\,m_\alpha\epsilon_\alpha+\Gamma^{1(\sigma)}_{1(\alpha)1(\sigma)}\,m_\sigma\epsilon_\sigma\big)
		\notag\\&=-\frac{m_\alpha\epsilon_\alpha}{u^{1(\alpha)}-u^{1(\gamma)}}\,m_\gamma\epsilon_\gamma+\frac{m_\gamma\epsilon_\gamma}{u^{1(\alpha)}-u^{1(\gamma)}}\,m_\alpha\epsilon_\alpha
		\notag\\&+\underset{\sigma\neq\alpha,\gamma}{\sum}\,\bigg(\frac{m_\sigma\epsilon_\sigma}{u^{1(\alpha)}-u^{1(\sigma)}}\,m_\alpha\epsilon_\alpha-\frac{m_\alpha\epsilon_\alpha}{u^{1(\alpha)}-u^{1(\sigma)}}\,m_\sigma\epsilon_\sigma\bigg)=0
		\notag
	\end{align}
	if $i=1$. If $j=1$ and $k\geq2$ then \eqref{nnE_caso2} becomes
	\begin{align}
		\nabla_{i(\alpha)}\nabla_{j(\alpha)}E^{k(\gamma)}&=\Gamma^{k(\gamma)}_{i(\alpha)1(\alpha)}\nabla_{1(\alpha)}E^{1(\alpha)}-\Gamma^{k(\alpha)}_{i(\alpha)1(\alpha)}\xcancel{\nabla_{k(\alpha)}E^{k(\gamma)}}
		\notag\\&+\Gamma^{k(\gamma)}_{i(\alpha)1(\gamma)}\nabla_{1(\alpha)}E^{1(\gamma)}-\Gamma^{k(\gamma)}_{i(\alpha)1(\alpha)}\nabla_{k(\gamma)}E^{k(\gamma)}
		\notag\\&-\underset{\sigma\neq\alpha,\gamma}{\sum}\,\Gamma^{k(\sigma)}_{i(\alpha)1(\alpha)}\xcancel{\nabla_{k(\sigma)}E^{k(\gamma)}}
		\notag\\&=-\Gamma^{k(\gamma)}_{i(\alpha)1(\gamma)}\bigg(1-\underset{\sigma\neq\alpha}{\sum}\,m_\sigma\epsilon_\sigma\bigg)
		\notag\\&+\Gamma^{k(\gamma)}_{i(\alpha)1(\gamma)}\,m_\alpha\epsilon_\alpha+\Gamma^{k(\gamma)}_{i(\alpha)1(\gamma)}\bigg(1-\overset{r}{\underset{\tau=1}{\sum}}\,m_\tau\epsilon_\tau\bigg)=0.
		\notag
	\end{align}
	If $j\geq2$ and $k=1$ then \eqref{nnE_caso2} becomes
	\begin{align}
		\nabla_{i(\alpha)}\nabla_{j(\alpha)}E^{k(\gamma)}&=\xcancel{\Gamma^{1(\gamma)}_{i(\alpha)j(\alpha)}}\nabla_{j(\alpha)}E^{j(\alpha)}-\xcancel{\Gamma^{1(\alpha)}_{i(\alpha)j(\alpha)}}\nabla_{1(\alpha)}E^{1(\gamma)}
		\notag\\&+\xcancel{\Gamma^{1(\gamma)}_{i(\alpha)j(\gamma)}}\nabla_{j(\alpha)}E^{j(\gamma)}-\xcancel{\Gamma^{1(\gamma)}_{i(\alpha)j(\alpha)}}\nabla_{1(\gamma)}E^{1(\gamma)}
		\notag\\&-\underset{\sigma\neq\alpha,\gamma}{\sum}\,\xcancel{\Gamma^{1(\sigma)}_{i(\alpha)j(\alpha)}}\nabla_{1(\sigma)}E^{1(\gamma)}=0.
		\notag
	\end{align}
	\\\textbf{Case 3: $\alpha=\gamma\neq\beta$.} 
	\begin{align}
		\nabla_{i(\alpha)}\nabla_{j(\beta)}E^{k(\gamma)}&=\Gamma^{k(\alpha)}_{i(\alpha)j(\sigma)}\nabla_{j(\beta)}E^{j(\sigma)}-\Gamma^{k(\sigma)}_{i(\alpha)j(\beta)}\nabla_{k(\sigma)}E^{k(\alpha)}
		\notag\\&=\Gamma^{k(\alpha)}_{i(\alpha)j(\alpha)}\nabla_{j(\beta)}E^{j(\alpha)}-\Gamma^{k(\alpha)}_{i(\alpha)j(\beta)}\nabla_{k(\alpha)}E^{k(\alpha)}
		\notag\\&+\Gamma^{k(\alpha)}_{i(\alpha)j(\beta)}\nabla_{j(\beta)}E^{j(\beta)}-\Gamma^{k(\beta)}_{i(\alpha)j(\beta)}\nabla_{k(\beta)}E^{k(\alpha)}
		\notag\\&+\underset{\sigma\neq\alpha,\beta}{\sum}\,\bigg(\Gamma^{k(\alpha)}_{i(\alpha)j(\sigma)}\nabla_{j(\beta)}E^{j(\sigma)}-\xcancel{\Gamma^{k(\sigma)}_{i(\alpha)j(\beta)}}\nabla_{k(\sigma)}E^{k(\alpha)}\bigg).
		\label{nnE_caso3}
	\end{align}
	If both $j$ and $k$ are greater or equal than $2$ then it becomes
	\begin{align}
		\nabla_{i(\alpha)}\nabla_{j(\beta)}E^{k(\gamma)}&=\Gamma^{k(\alpha)}_{i(\alpha)j(\alpha)}\xcancel{\nabla_{j(\beta)}E^{j(\alpha)}}-\Gamma^{k(\alpha)}_{i(\alpha)j(\beta)}\nabla_{k(\alpha)}E^{k(\alpha)}
		\notag\\&+\Gamma^{k(\alpha)}_{i(\alpha)j(\beta)}\nabla_{j(\beta)}E^{j(\beta)}-\Gamma^{k(\beta)}_{i(\alpha)j(\beta)}\xcancel{\nabla_{k(\beta)}E^{k(\alpha)}}
		\notag\\&+\underset{\sigma\neq\alpha,\beta}{\sum}\,\Gamma^{k(\alpha)}_{i(\alpha)j(\sigma)}\xcancel{\nabla_{j(\beta)}E^{j(\sigma)}}
		\notag\\&=\Gamma^{k(\alpha)}_{i(\alpha)j(\beta)}\big(\cancel{\nabla_{j(\beta)}E^{j(\beta)}}-\cancel{\nabla_{k(\alpha)}E^{k(\alpha)}}\big)=0.
		\notag
	\end{align}
	If $j=k=1$ then \eqref{nnE_caso3} becomes
	\begin{align}
		\nabla_{i(\alpha)}\nabla_{j(\beta)}E^{k(\gamma)}&=\Gamma^{1(\alpha)}_{i(\alpha)1(\alpha)}\nabla_{1(\beta)}E^{1(\alpha)}-\Gamma^{1(\alpha)}_{i(\alpha)1(\beta)}\nabla_{1(\alpha)}E^{1(\alpha)}
		\notag\\&+\Gamma^{1(\alpha)}_{i(\alpha)1(\beta)}\nabla_{1(\beta)}E^{1(\beta)}-\Gamma^{1(\beta)}_{i(\alpha)1(\beta)}\nabla_{1(\beta)}E^{1(\alpha)}
		\notag\\&+\underset{\sigma\neq\alpha,\beta}{\sum}\,\Gamma^{1(\alpha)}_{i(\alpha)1(\sigma)}\nabla_{1(\beta)}E^{1(\sigma)}
		\notag
	\end{align}
	which trivially vanishes if $i\geq2$ and is
	\begin{align}
		\nabla_{i(\alpha)}\nabla_{j(\beta)}E^{k(\gamma)}&=\Gamma^{1(\alpha)}_{1(\alpha)1(\alpha)}\nabla_{1(\beta)}E^{1(\alpha)}-\Gamma^{1(\alpha)}_{1(\alpha)1(\beta)}\nabla_{1(\alpha)}E^{1(\alpha)}
		\notag\\&+\Gamma^{1(\alpha)}_{1(\alpha)1(\beta)}\nabla_{1(\beta)}E^{1(\beta)}-\Gamma^{1(\beta)}_{1(\alpha)1(\beta)}\nabla_{1(\beta)}E^{1(\alpha)}
		\notag\\&+\underset{\sigma\neq\alpha,\beta}{\sum}\,\Gamma^{1(\alpha)}_{1(\alpha)1(\sigma)}\nabla_{1(\beta)}E^{1(\sigma)}
		\notag\\&=-\Gamma^{1(\alpha)}_{1(\beta)1(\alpha)}\,m_\beta\epsilon_\beta-\cancel{\underset{\sigma\neq\alpha,\beta}{\sum}\,\Gamma^{1(\alpha)}_{1(\sigma)1(\alpha)}\,m_\beta\epsilon_\beta}\notag\\&-\Gamma^{1(\alpha)}_{1(\alpha)1(\beta)}\bigg(\bcancel{1}-\underset{\sigma\neq\alpha}{\sum}\,m_\sigma\epsilon_\sigma\bigg)
		\notag\\&+\Gamma^{1(\alpha)}_{1(\alpha)1(\beta)}\bigg(\bcancel{1}-\underset{\sigma\neq\beta}{\sum}\,m_\sigma\epsilon_\sigma\bigg)-\Gamma^{1(\beta)}_{1(\alpha)1(\beta)}\,m_\beta\epsilon_\beta
		\notag\\&+\cancel{\underset{\sigma\neq\alpha,\beta}{\sum}\,\Gamma^{1(\alpha)}_{1(\alpha)1(\sigma)}\,m_\beta\epsilon_\beta}
		\notag\\&=\Gamma^{1(\alpha)}_{1(\beta)1(\alpha)}\bigg(-m_\beta\epsilon_\beta+\underset{\sigma\neq\alpha}{\sum}\,m_\sigma\epsilon_\sigma-\underset{\sigma\neq\beta}{\sum}\,m_\sigma\epsilon_\sigma\bigg)-\Gamma^{1(\beta)}_{1(\alpha)1(\beta)}\,m_\beta\epsilon_\beta
		\notag\\&=-\Gamma^{1(\alpha)}_{1(\beta)1(\alpha)}\,m_\alpha\epsilon_\alpha-\Gamma^{1(\beta)}_{1(\alpha)1(\beta)}\,m_\beta\epsilon_\beta
		\notag\\&=-\frac{m_\beta\epsilon_\beta}{u^{1(\alpha)}-u^{1(\beta)}}\,m_\alpha\epsilon_\alpha+\frac{m_\alpha\epsilon_\alpha}{u^{1(\alpha)}-u^{1(\beta)}}\,m_\beta\epsilon_\beta=0
		\notag
	\end{align}
	if $i=1$. If $j=1$ and $k\geq2$ then \eqref{nnE_caso3} becomes
	\begin{align}
		\nabla_{i(\alpha)}\nabla_{j(\beta)}E^{k(\gamma)}&=\Gamma^{k(\alpha)}_{i(\alpha)1(\alpha)}\nabla_{1(\beta)}E^{1(\alpha)}-\Gamma^{k(\alpha)}_{i(\alpha)1(\beta)}\nabla_{k(\alpha)}E^{k(\alpha)}
		\notag\\&+\Gamma^{k(\alpha)}_{i(\alpha)1(\beta)}\nabla_{1(\beta)}E^{1(\beta)}-\Gamma^{k(\beta)}_{i(\alpha)1(\beta)}\xcancel{\nabla_{k(\beta)}E^{k(\alpha)}}
		\notag\\&+\underset{\sigma\neq\alpha,\beta}{\sum}\,\Gamma^{k(\alpha)}_{i(\alpha)1(\sigma)}\nabla_{1(\beta)}E^{1(\sigma)}
		\notag\\&=-\Gamma^{k(\alpha)}_{i(\alpha)1(\beta)}\,m_\beta\epsilon_\beta-\cancel{\underset{\sigma\neq\alpha,\beta}{\sum}\,\Gamma^{k(\alpha)}_{i(\alpha)1(\sigma)}\,m_\beta\epsilon_\beta}\notag\\&-\Gamma^{k(\alpha)}_{i(\alpha)1(\beta)}\bigg(\bcancel{1}-\overset{r}{\underset{\tau=1}{\sum}}\,m_\tau\epsilon_\tau\bigg)+\Gamma^{k(\alpha)}_{i(\alpha)1(\beta)}\bigg(\bcancel{1}-\underset{\sigma\neq\beta}{\sum}\,m_\sigma\epsilon_\sigma\bigg)
		\notag\\&+\cancel{\underset{\sigma\neq\alpha,\beta}{\sum}\,\Gamma^{k(\alpha)}_{i(\alpha)1(\sigma)}\,m_\beta\epsilon_\beta}
		\notag\\&=\Gamma^{k(\alpha)}_{i(\alpha)1(\beta)}\bigg(-m_\beta\epsilon_\beta+\overset{r}{\underset{\tau=1}{\sum}}\,m_\tau\epsilon_\tau-\underset{\sigma\neq\beta}{\sum}\,m_\sigma\epsilon_\sigma\bigg)=0.
		\notag
	\end{align}
	If $j\geq2$ and $k=1$ then \eqref{nnE_caso3} becomes
	\begin{align}
		\nabla_{i(\alpha)}\nabla_{j(\beta)}E^{k(\gamma)}&=\Gamma^{1(\alpha)}_{i(\alpha)j(\alpha)}\xcancel{\nabla_{j(\beta)}E^{j(\alpha)}}-\xcancel{\Gamma^{1(\alpha)}_{i(\alpha)j(\beta)}}\nabla_{1(\alpha)}E^{1(\alpha)}
		\notag\\&+\xcancel{\Gamma^{1(\alpha)}_{i(\alpha)j(\beta)}}\nabla_{j(\beta)}E^{j(\beta)}-\xcancel{\Gamma^{1(\beta)}_{i(\alpha)j(\beta)}}\nabla_{1(\beta)}E^{1(\alpha)}
		\notag\\&+\underset{\sigma\neq\alpha,\beta}{\sum}\,\Gamma^{1(\alpha)}_{i(\alpha)j(\sigma)}\xcancel{\nabla_{j(\beta)}E^{j(\sigma)}}=0.
		\notag
	\end{align}
	\\\textbf{Case 4: $\beta=\gamma\neq\alpha$.} 
	\begin{align}
		\nabla_{i(\alpha)}\nabla_{j(\beta)}E^{k(\gamma)}&=\Gamma^{k(\beta)}_{i(\alpha)j(\sigma)}\nabla_{j(\beta)}E^{j(\sigma)}-\Gamma^{k(\sigma)}_{i(\alpha)j(\beta)}\nabla_{k(\sigma)}E^{k(\beta)}
		\notag\\&=\Gamma^{k(\beta)}_{i(\alpha)j(\alpha)}\nabla_{j(\beta)}E^{j(\alpha)}-\Gamma^{k(\alpha)}_{i(\alpha)j(\beta)}\nabla_{k(\alpha)}E^{k(\beta)}
		\notag\\&+\Gamma^{k(\beta)}_{i(\alpha)j(\beta)}\nabla_{j(\beta)}E^{j(\beta)}-\Gamma^{k(\beta)}_{i(\alpha)j(\beta)}\nabla_{k(\beta)}E^{k(\beta)}
		\notag\\&+\underset{\sigma\neq\alpha,\beta}{\sum}\,\big(\xcancel{\Gamma^{k(\beta)}_{i(\alpha)j(\sigma)}}\nabla_{j(\beta)}E^{j(\sigma)}-\xcancel{\Gamma^{k(\sigma)}_{i(\alpha)j(\beta)}}\nabla_{k(\sigma)}E^{k(\beta)}\big).
		\label{nnE_caso4}
	\end{align}
	If both $j$ and $k$ are greater or equal than $2$ then it becomes
	\begin{align}
		\nabla_{i(\alpha)}\nabla_{j(\beta)}E^{k(\gamma)}&=\Gamma^{k(\beta)}_{i(\alpha)j(\alpha)}\xcancel{\nabla_{j(\beta)}E^{j(\alpha)}}-\Gamma^{k(\alpha)}_{i(\alpha)j(\beta)}\xcancel{\nabla_{k(\alpha)}E^{k(\beta)}}
		\notag\\&+\Gamma^{k(\beta)}_{i(\alpha)j(\beta)}\nabla_{j(\beta)}E^{j(\beta)}-\Gamma^{k(\beta)}_{i(\alpha)j(\beta)}\nabla_{k(\beta)}E^{k(\beta)}
		\notag\\&=\Gamma^{k(\beta)}_{i(\alpha)j(\beta)}\bigg(1-\overset{r}{\underset{\tau=1}{\sum}}\,m_\tau\epsilon_\tau-1+\overset{r}{\underset{\tau=1}{\sum}}\,m_\tau\epsilon_\tau\bigg)=0.
		\notag
	\end{align}
	If $j=k=1$ then \eqref{nnE_caso4} becomes
	\begin{align}
		\nabla_{i(\alpha)}\nabla_{j(\beta)}E^{k(\gamma)}&=\Gamma^{1(\beta)}_{i(\alpha)1(\alpha)}\nabla_{1(\beta)}E^{1(\alpha)}-\Gamma^{1(\alpha)}_{i(\alpha)1(\beta)}\nabla_{1(\alpha)}E^{1(\beta)}
		\notag\\&+\Gamma^{1(\beta)}_{i(\alpha)1(\beta)}\nabla_{1(\beta)}E^{1(\beta)}-\Gamma^{1(\beta)}_{i(\alpha)1(\beta)}\nabla_{1(\beta)}E^{1(\beta)}
		\notag\\&=-\Gamma^{1(\beta)}_{i(\alpha)1(\beta)}\,m_\beta\epsilon_\beta-\Gamma^{1(\alpha)}_{i(\alpha)1(\beta)}\,m_\alpha\epsilon_\alpha
		\notag\\&+\Gamma^{1(\beta)}_{i(\alpha)1(\beta)}\bigg(\bcancel{1}-\cancel{\underset{\sigma\neq\beta}{\sum}\,m_\sigma\epsilon_\sigma}\bigg)-\Gamma^{1(\beta)}_{i(\alpha)1(\beta)}\bigg(\bcancel{1}-\cancel{\underset{\sigma\neq\beta}{\sum}\,m_\sigma\epsilon_\sigma}\bigg)
		\notag\\&=\frac{m_\alpha\epsilon_\alpha}{u^{1(\alpha)}-u^{1(\beta)}}\,m_\beta\epsilon_\beta-\frac{m_\beta\epsilon_\beta}{u^{1(\alpha)}-u^{1(\beta)}}\,m_\alpha\epsilon_\alpha=0.
		\notag
	\end{align}
	If $j=1$ and $k\geq2$ then \eqref{nnE_caso4} becomes
	\begin{align}
		\nabla_{i(\alpha)}\nabla_{j(\beta)}E^{k(\gamma)}&=\Gamma^{k(\beta)}_{i(\alpha)1(\alpha)}\nabla_{1(\beta)}E^{1(\alpha)}-\Gamma^{k(\alpha)}_{i(\alpha)1(\beta)}\xcancel{\nabla_{k(\alpha)}E^{k(\beta)}}
		\notag\\&+\Gamma^{k(\beta)}_{i(\alpha)1(\beta)}\nabla_{1(\beta)}E^{1(\beta)}-\Gamma^{k(\beta)}_{i(\alpha)1(\beta)}\nabla_{k(\beta)}E^{k(\beta)}
		\notag\\&=\Gamma^{k(\beta)}_{i(\alpha)1(\beta)}\bigg(-m_\beta\epsilon_\beta+1-\underset{\sigma\neq\beta}{\sum}\,m_\sigma\epsilon_\sigma-1+\overset{r}{\underset{\tau=1}{\sum}}\,m_\tau\epsilon_\tau\bigg)=0.
		\notag
	\end{align}
	If $j\geq2$ and $k=1$ then \eqref{nnE_caso4} becomes
	\begin{align}
		\nabla_{i(\alpha)}\nabla_{j(\beta)}E^{k(\gamma)}&=\Gamma^{1(\beta)}_{i(\alpha)j(\alpha)}\xcancel{\nabla_{j(\beta)}E^{j(\alpha)}}-\xcancel{\Gamma^{1(\alpha)}_{i(\alpha)j(\beta)}}\nabla_{1(\alpha)}E^{1(\beta)}
		\notag\\&+\xcancel{\Gamma^{1(\beta)}_{i(\alpha)j(\beta)}}\nabla_{j(\beta)}E^{j(\beta)}-\xcancel{\Gamma^{1(\beta)}_{i(\alpha)j(\beta)}}\nabla_{1(\beta)}E^{1(\beta)}=0.
		\notag
	\end{align}
	\\\textbf{Case 5: $\alpha\neq\beta\neq\gamma\neq\alpha$.} 
	\begin{align}
		\nabla_{i(\alpha)}\nabla_{j(\beta)}E^{k(\gamma)}&=\Gamma^{k(\gamma)}_{i(\alpha)j(\sigma)}\nabla_{j(\beta)}E^{j(\sigma)}-\Gamma^{k(\sigma)}_{i(\alpha)j(\beta)}\nabla_{k(\sigma)}E^{k(\gamma)}
		\notag\\&=\Gamma^{k(\gamma)}_{i(\alpha)j(\alpha)}\nabla_{j(\beta)}E^{j(\alpha)}-\Gamma^{k(\alpha)}_{i(\alpha)j(\beta)}\nabla_{k(\alpha)}E^{k(\gamma)}
		\notag\\&+\xcancel{\Gamma^{k(\gamma)}_{i(\alpha)j(\beta)}}\nabla_{j(\beta)}E^{j(\beta)}-\Gamma^{k(\beta)}_{i(\alpha)j(\beta)}\nabla_{k(\beta)}E^{k(\gamma)}
		\notag\\&+\Gamma^{k(\gamma)}_{i(\alpha)j(\gamma)}\nabla_{j(\beta)}E^{j(\gamma)}-\xcancel{\Gamma^{k(\gamma)}_{i(\alpha)j(\beta)}}\nabla_{k(\gamma)}E^{k(\gamma)}.
		\label{nnE_caso5}
	\end{align}
	If both $j$ and $k$ are greater or equal than $2$ then it becomes
	\begin{align}
		\nabla_{i(\alpha)}\nabla_{j(\beta)}E^{k(\gamma)}&=\Gamma^{k(\gamma)}_{i(\alpha)j(\alpha)}\xcancel{\nabla_{j(\beta)}E^{j(\alpha)}}-\Gamma^{k(\alpha)}_{i(\alpha)j(\beta)}\xcancel{\nabla_{k(\alpha)}E^{k(\gamma)}}
		\notag\\&-\Gamma^{k(\beta)}_{i(\alpha)j(\beta)}\xcancel{\nabla_{k(\beta)}E^{k(\gamma)}}+\Gamma^{k(\gamma)}_{i(\alpha)j(\gamma)}\xcancel{\nabla_{j(\beta)}E^{j(\gamma)}}=0.
		\notag
	\end{align}
	If $j=k=1$ then \eqref{nnE_caso5} becomes
	\begin{align}
		\nabla_{i(\alpha)}\nabla_{j(\beta)}E^{k(\gamma)}&=\Gamma^{1(\gamma)}_{i(\alpha)1(\alpha)}\nabla_{1(\beta)}E^{1(\alpha)}-\Gamma^{1(\alpha)}_{i(\alpha)1(\beta)}\nabla_{1(\alpha)}E^{1(\gamma)}
		\notag\\&-\Gamma^{1(\beta)}_{i(\alpha)1(\beta)}\nabla_{1(\beta)}E^{1(\gamma)}+\Gamma^{1(\gamma)}_{i(\alpha)1(\gamma)}\nabla_{1(\beta)}E^{1(\gamma)}
		\notag\\&=-\cancel{\Gamma^{1(\gamma)}_{i(\alpha)1(\gamma)}\,m_\beta\epsilon_\beta}-\Gamma^{1(\alpha)}_{i(\alpha)1(\beta)}\,m_\alpha\epsilon_\alpha
		\notag\\&-\Gamma^{1(\beta)}_{i(\alpha)1(\beta)}\,m_\beta\epsilon_\beta+\cancel{\Gamma^{1(\gamma)}_{i(\alpha)1(\gamma)}\,m_\beta\epsilon_\beta}
		\notag\\&=-\frac{m_\beta\epsilon_\beta}{u^{1(\alpha)}-u^{1(\beta)}}\,m_\alpha\epsilon_\alpha+\frac{m_\alpha\epsilon_\alpha}{u^{1(\alpha)}-u^{1(\beta)}}\,m_\beta\epsilon_\beta=0.
		\notag
	\end{align}
	If $j=1$ and $k\geq2$ then \eqref{nnE_caso5} becomes
	\begin{align}
		\nabla_{i(\alpha)}\nabla_{j(\beta)}E^{k(\gamma)}&=\Gamma^{k(\gamma)}_{i(\alpha)1(\alpha)}\nabla_{1(\beta)}E^{1(\alpha)}-\Gamma^{k(\alpha)}_{i(\alpha)1(\beta)}\xcancel{\nabla_{k(\alpha)}E^{k(\gamma)}}
		\notag\\&-\Gamma^{k(\beta)}_{i(\alpha)1(\beta)}\xcancel{\nabla_{k(\beta)}E^{k(\gamma)}}+\Gamma^{k(\gamma)}_{i(\alpha)1(\gamma)}\nabla_{1(\beta)}E^{1(\gamma)}
		\notag\\&=-\Gamma^{k(\gamma)}_{i(\alpha)1(\gamma)}\,m_\beta\epsilon_\beta+\Gamma^{k(\gamma)}_{i(\alpha)1(\gamma)}\,m_\beta\epsilon_\beta=0.
		\notag
	\end{align}
	If $j\geq2$ and $k=1$ then \eqref{nnE_caso5} becomes
	\begin{align}
		\nabla_{i(\alpha)}\nabla_{j(\beta)}E^{k(\gamma)}&=\Gamma^{1(\gamma)}_{i(\alpha)j(\alpha)}\xcancel{\nabla_{j(\beta)}E^{j(\alpha)}}-\xcancel{\Gamma^{1(\alpha)}_{i(\alpha)j(\beta)}}\nabla_{1(\alpha)}E^{1(\gamma)}
		\notag\\&-\xcancel{\Gamma^{1(\beta)}_{i(\alpha)j(\beta)}}\nabla_{1(\beta)}E^{1(\gamma)}+\Gamma^{1(\gamma)}_{i(\alpha)j(\gamma)}\xcancel{\nabla_{j(\beta)}E^{j(\gamma)}}=0.
		\notag
	\end{align}
	This proves that $\nabla\nabla E=0$.
	
	We are now left with proving the flatness of $\nabla$, that is $R=0$. Due to the simmetries of
	\begin{align}
		\label{Rdef}
		R^{i(\alpha)}_{h(\epsilon)k(\gamma)j(\beta)}&=\partial_{k(\gamma)}\Gamma^{i(\alpha)}_{h(\epsilon)j(\beta)}-\partial_{h(\epsilon)}\Gamma^{i(\alpha)}_{k(\gamma)j(\beta)}\notag\\&+\overset{r}{\underset{\sigma=1}{\sum}}\,\overset{m_\sigma}{\underset{l=1}{\sum}}\,\bigg(\Gamma^{i(\alpha)}_{k(\gamma)l(\sigma)}\Gamma^{l(\sigma)}_{h(\epsilon)j(\beta)}-\Gamma^{i(\alpha)}_{h(\epsilon)l(\sigma)}\Gamma^{l(\sigma)}_{k(\gamma)j(\beta)}\bigg)
	\end{align}
	the cases to be considered are the following ones:
	\begin{itemize}
		\item[$1.$] $\alpha=\beta=\gamma=\epsilon$
		\item[$2.$] $\alpha=\beta=\gamma\neq\epsilon$
		\item[$3.$] $\alpha=\gamma=\epsilon\neq\beta$
		\item[$4.$] $\beta=\gamma=\epsilon\neq\alpha$
		\item[$5.$] $\alpha=\beta\neq\gamma=\epsilon$
		\item[$6.$] $\alpha=\gamma\neq\beta=\epsilon$
		\item[$7.$] $\alpha=\beta\notin\{\gamma,\epsilon\}$, $\gamma\neq\epsilon$
		\item[$8.$] $\alpha=\gamma\notin\{\beta,\epsilon\}$, $\beta\neq\epsilon$
		\item[$9.$] $\beta=\gamma\notin\{\alpha,\epsilon\}$, $\alpha\neq\epsilon$
		\item[$10.$] $\gamma=\epsilon\notin\{\alpha,\beta\}$, $\alpha\neq\beta$
		\item[$11.$] $\alpha$, $\beta$, $\gamma$ and $\epsilon$ are pairwise distinct.
	\end{itemize}
	\textbf{Case 1: $\alpha=\beta=\gamma=\epsilon$.} Our goal is to prove that
	\begin{align}
		\label{Rdef_caso1}
		R^{i(\alpha)}_{h(\alpha)k(\alpha)j(\alpha)}&=\partial_{k(\alpha)}\Gamma^{i(\alpha)}_{h(\alpha)j(\alpha)}-\partial_{h(\alpha)}\Gamma^{i(\alpha)}_{k(\alpha)j(\alpha)}\notag\\&+\overset{r}{\underset{\sigma=1}{\sum}}\,\overset{m_\sigma}{\underset{l=1}{\sum}}\,\bigg(\Gamma^{i(\alpha)}_{k(\alpha)l(\sigma)}\Gamma^{l(\sigma)}_{h(\alpha)j(\alpha)}-\Gamma^{i(\alpha)}_{h(\alpha)l(\sigma)}\Gamma^{l(\sigma)}_{k(\alpha)j(\alpha)}\bigg)
		\notag\\&=\partial_{k(\alpha)}\Gamma^{i(\alpha)}_{h(\alpha)j(\alpha)}-\partial_{h(\alpha)}\Gamma^{i(\alpha)}_{k(\alpha)j(\alpha)}\notag\\&+\overset{m_\alpha}{\underset{l=1}{\sum}}\,\bigg(\Gamma^{i(\alpha)}_{k(\alpha)l(\alpha)}\Gamma^{l(\alpha)}_{h(\alpha)j(\alpha)}-\Gamma^{i(\alpha)}_{h(\alpha)l(\alpha)}\Gamma^{l(\alpha)}_{k(\alpha)j(\alpha)}\bigg)\notag\\&+\overset{}{\underset{\sigma\neq\alpha}{\sum}}\,\overset{m_\sigma}{\underset{l=1}{\sum}}\,\bigg(\Gamma^{i(\alpha)}_{k(\alpha)l(\sigma)}\Gamma^{l(\sigma)}_{h(\alpha)j(\alpha)}-\Gamma^{i(\alpha)}_{h(\alpha)l(\sigma)}\Gamma^{l(\sigma)}_{k(\alpha)j(\alpha)}\bigg)
	\end{align}
	vanishes. Let us first note that for each integer $n\geq2$ it is possible to recover part of the Christoffel symbols for the case where $m_\alpha=n+1$ starting from the ones for the case where $m_\alpha=n$. More precisely, let us denote by $(\Gamma^{n+1})^k_{ij}$ the Christoffel symbols in the case where $m_\alpha=n+1$ and by $(\Gamma^{n})^k_{ij}$ the Christoffel symbols in the case where $m_\alpha=n$, where the sizes $m_\sigma$ of the remaining blocks $\sigma\neq\alpha$ are the same and where the constant $\epsilon_{1(\alpha)}$ has been replaced by $\frac{n+1}{n}\,\epsilon_{1(\alpha)}$. Then
	\begin{align}
		(\Gamma^{n+1})^k_{ij}&=(\Gamma^{n})^k_{ij}
		\notag
	\end{align}
	whenever the indices $i,j,k$ are different from $(n+1)(\alpha)$. In the wake of this property, we will proceed by induction over $m_\alpha$. Let us first consider the case\footnote{The case where $m_\alpha=1$ is trivial, as $k=h=1$ directly implies $R^{i(\alpha)}_{h(\alpha)k(\alpha)j(\alpha)}=0$.} where $m_\alpha=2$, so that the indices $i$, $j$, $k$ and $h$ run from $1$ to $2$. In particular, since $R^{i(\alpha)}_{h(\alpha)k(\alpha)j(\alpha)}$ automatically vanishes when $k=h$, the only relevant cases are the one where $k=1$, $h=2$ and the one where $k=2$, $k=1$. By using the simmetries of $R$, we only consider the case where $k=1$ and $h=2$, hence obtaining
	\begin{align}
		\notag
		R^{i(\alpha)}_{2(\alpha)1(\alpha)j(\alpha)}&=\partial_{1(\alpha)}\Gamma^{i(\alpha)}_{2(\alpha)j(\alpha)}-\partial_{2(\alpha)}\Gamma^{i(\alpha)}_{1(\alpha)j(\alpha)}\notag\\&+\overset{2}{\underset{l=1}{\sum}}\,\bigg(\Gamma^{i(\alpha)}_{1(\alpha)l(\alpha)}\Gamma^{l(\alpha)}_{2(\alpha)j(\alpha)}-\Gamma^{i(\alpha)}_{2(\alpha)l(\alpha)}\Gamma^{l(\alpha)}_{1(\alpha)j(\alpha)}\bigg)\notag\\&+\overset{}{\underset{\sigma\neq\alpha}{\sum}}\,\overset{m_\sigma}{\underset{l=1}{\sum}}\,\bigg(\Gamma^{i(\alpha)}_{1(\alpha)l(\sigma)}\Gamma^{l(\sigma)}_{2(\alpha)j(\alpha)}-\Gamma^{i(\alpha)}_{2(\alpha)l(\sigma)}\Gamma^{l(\sigma)}_{1(\alpha)j(\alpha)}\bigg)
	\end{align}
	where both $\Gamma^{i(\alpha)}_{1(\alpha)l(\sigma)}$ and $\Gamma^{i(\alpha)}_{2(\alpha)l(\sigma)}$ survive only for $l=1$. This yields
	\begin{align}
		R^{i(\alpha)}_{2(\alpha)1(\alpha)j(\alpha)}&=\partial_{1(\alpha)}\Gamma^{i(\alpha)}_{2(\alpha)j(\alpha)}-\partial_{2(\alpha)}\Gamma^{i(\alpha)}_{1(\alpha)j(\alpha)}\notag\\&+\Gamma^{i(\alpha)}_{1(\alpha)1(\alpha)}\Gamma^{1(\alpha)}_{2(\alpha)j(\alpha)}-\Gamma^{i(\alpha)}_{2(\alpha)1(\alpha)}\Gamma^{1(\alpha)}_{1(\alpha)j(\alpha)}\notag\\&+\Gamma^{i(\alpha)}_{1(\alpha)2(\alpha)}\Gamma^{2(\alpha)}_{2(\alpha)j(\alpha)}-\Gamma^{i(\alpha)}_{2(\alpha)2(\alpha)}\Gamma^{2(\alpha)}_{1(\alpha)j(\alpha)}\notag\\&+\overset{}{\underset{\sigma\neq\alpha}{\sum}}\,\bigg(\Gamma^{i(\alpha)}_{1(\alpha)1(\sigma)}\Gamma^{1(\sigma)}_{2(\alpha)j(\alpha)}-\Gamma^{i(\alpha)}_{2(\alpha)1(\sigma)}\Gamma^{1(\sigma)}_{1(\alpha)j(\alpha)}\bigg).\label{Rdef_caso1_k1h2}
	\end{align}
	If $i=1$ we get
	\begin{align}
		\notag
		R^{1(\alpha)}_{2(\alpha)1(\alpha)j(\alpha)}&=\partial_{1(\alpha)}\Gamma^{1(\alpha)}_{2(\alpha)j(\alpha)}-\partial_{2(\alpha)}\Gamma^{1(\alpha)}_{1(\alpha)j(\alpha)}\notag\\&+\Gamma^{1(\alpha)}_{1(\alpha)1(\alpha)}\Gamma^{1(\alpha)}_{2(\alpha)j(\alpha)}-\Gamma^{1(\alpha)}_{2(\alpha)1(\alpha)}\Gamma^{1(\alpha)}_{1(\alpha)j(\alpha)}\notag\\&+\Gamma^{1(\alpha)}_{1(\alpha)2(\alpha)}\Gamma^{2(\alpha)}_{2(\alpha)j(\alpha)}-\Gamma^{1(\alpha)}_{2(\alpha)2(\alpha)}\Gamma^{2(\alpha)}_{1(\alpha)j(\alpha)}\notag\\&+\overset{}{\underset{\sigma\neq\alpha}{\sum}}\,\bigg(\Gamma^{1(\alpha)}_{1(\alpha)1(\sigma)}\Gamma^{1(\sigma)}_{2(\alpha)j(\alpha)}-\xcancel{\Gamma^{1(\alpha)}_{2(\alpha)1(\sigma)}}\Gamma^{1(\sigma)}_{1(\alpha)j(\alpha)}\bigg)\notag
	\end{align}
	that becomes
	\begin{align}
		\notag
		R^{1(\alpha)}_{2(\alpha)1(\alpha)1(\alpha)}&=\partial_{1(\alpha)}\Gamma^{1(\alpha)}_{2(\alpha)1(\alpha)}-\partial_{2(\alpha)}\Gamma^{1(\alpha)}_{1(\alpha)1(\alpha)}\notag\\&+\Gamma^{1(\alpha)}_{1(\alpha)1(\alpha)}\Gamma^{1(\alpha)}_{2(\alpha)1(\alpha)}-\Gamma^{1(\alpha)}_{2(\alpha)1(\alpha)}\Gamma^{1(\alpha)}_{1(\alpha)1(\alpha)}\notag\\&+\Gamma^{1(\alpha)}_{1(\alpha)2(\alpha)}\Gamma^{2(\alpha)}_{2(\alpha)1(\alpha)}-\Gamma^{1(\alpha)}_{2(\alpha)2(\alpha)}\Gamma^{2(\alpha)}_{1(\alpha)1(\alpha)}\notag\\&+\overset{}{\underset{\sigma\neq\alpha}{\sum}}\,\Gamma^{1(\alpha)}_{1(\alpha)1(\sigma)}\Gamma^{1(\sigma)}_{2(\alpha)1(\alpha)}
		\notag\\&=-\underset{\sigma\neq\alpha}{\sum}\,\partial_{1(\alpha)}\xcancel{\Gamma^{1(\alpha)}_{2(\alpha)1(\sigma)}}-\partial_{2(\alpha)}\Gamma^{1(\alpha)}_{1(\alpha)1(\alpha)}\notag\\&-\underset{\sigma\neq\alpha}{\sum}\,\Gamma^{1(\alpha)}_{1(\alpha)1(\alpha)}\xcancel{\Gamma^{1(\alpha)}_{2(\alpha)1(\sigma)}}+\underset{\sigma\neq\alpha}{\sum}\,\xcancel{\Gamma^{1(\alpha)}_{2(\alpha)1(\sigma)}}\Gamma^{1(\alpha)}_{1(\alpha)1(\alpha)}\notag\\&-\underset{\sigma\neq\alpha}{\sum}\,\xcancel{\Gamma^{1(\alpha)}_{1(\sigma)2(\alpha)}}\Gamma^{2(\alpha)}_{2(\alpha)1(\alpha)}-\xcancel{\Gamma^{1(\alpha)}_{2(\alpha)2(\alpha)}}\Gamma^{2(\alpha)}_{1(\alpha)1(\alpha)}\notag\\&-\overset{}{\underset{\sigma\neq\alpha}{\sum}}\,\Gamma^{1(\alpha)}_{1(\alpha)1(\sigma)}\xcancel{\Gamma^{1(\sigma)}_{2(\alpha)1(\sigma)}}
		\notag\\&=\underset{\sigma\neq\alpha}{\sum}\,\partial_{2(\alpha)}\Gamma^{1(\alpha)}_{1(\sigma)1(\alpha)}=\underset{\sigma\neq\alpha}{\sum}\,\partial_{2(\alpha)}\bigg(\frac{m_\sigma\epsilon_\sigma}{u^{1(\alpha)}-u^{1(\sigma)}}\bigg)=0\notag
	\end{align}
	when $j=1$ and
	\begin{align}
		\notag
		R^{1(\alpha)}_{2(\alpha)1(\alpha)2(\alpha)}&=\partial_{1(\alpha)}\xcancel{\Gamma^{1(\alpha)}_{2(\alpha)2(\alpha)}}-\partial_{2(\alpha)}\xcancel{\Gamma^{1(\alpha)}_{1(\alpha)2(\alpha)}}\notag\\&+\Gamma^{1(\alpha)}_{1(\alpha)1(\alpha)}\xcancel{\Gamma^{1(\alpha)}_{2(\alpha)2(\alpha)}}-\xcancel{\Gamma^{1(\alpha)}_{2(\alpha)1(\alpha)}}\Gamma^{1(\alpha)}_{1(\alpha)2(\alpha)}\notag\\&+\xcancel{\Gamma^{1(\alpha)}_{1(\alpha)2(\alpha)}}\Gamma^{2(\alpha)}_{2(\alpha)2(\alpha)}-\xcancel{\Gamma^{1(\alpha)}_{2(\alpha)2(\alpha)}}\Gamma^{2(\alpha)}_{1(\alpha)2(\alpha)}\notag\\&+\overset{}{\underset{\sigma\neq\alpha}{\sum}}\,\Gamma^{1(\alpha)}_{1(\alpha)1(\sigma)}\xcancel{\Gamma^{1(\sigma)}_{2(\alpha)2(\alpha)}}=0\notag
	\end{align}
	when $j=2$. If $i=2$ then \eqref{Rdef_caso1_k1h2} reads
	\begin{align}
		R^{2(\alpha)}_{2(\alpha)1(\alpha)j(\alpha)}&=\partial_{1(\alpha)}\Gamma^{2(\alpha)}_{2(\alpha)j(\alpha)}-\partial_{2(\alpha)}\Gamma^{2(\alpha)}_{1(\alpha)j(\alpha)}\notag\\&+\Gamma^{2(\alpha)}_{1(\alpha)1(\alpha)}\Gamma^{1(\alpha)}_{2(\alpha)j(\alpha)}-\Gamma^{2(\alpha)}_{2(\alpha)1(\alpha)}\Gamma^{1(\alpha)}_{1(\alpha)j(\alpha)}\notag\\&+\Gamma^{2(\alpha)}_{1(\alpha)2(\alpha)}\Gamma^{2(\alpha)}_{2(\alpha)j(\alpha)}-\Gamma^{2(\alpha)}_{2(\alpha)2(\alpha)}\Gamma^{2(\alpha)}_{1(\alpha)j(\alpha)}\notag\\&+\overset{}{\underset{\sigma\neq\alpha}{\sum}}\,\bigg(\Gamma^{2(\alpha)}_{1(\alpha)1(\sigma)}\Gamma^{1(\sigma)}_{2(\alpha)j(\alpha)}-\Gamma^{2(\alpha)}_{2(\alpha)1(\sigma)}\Gamma^{1(\sigma)}_{1(\alpha)j(\alpha)}\bigg)\notag
	\end{align}
	that becomes
	\begin{align}
		R^{2(\alpha)}_{2(\alpha)1(\alpha)1(\alpha)}&=\partial_{1(\alpha)}\Gamma^{2(\alpha)}_{2(\alpha)1(\alpha)}-\partial_{2(\alpha)}\Gamma^{2(\alpha)}_{1(\alpha)1(\alpha)}\notag\\&+\Gamma^{2(\alpha)}_{1(\alpha)1(\alpha)}\xcancel{\Gamma^{1(\alpha)}_{2(\alpha)1(\alpha)}}-\cancel{\Gamma^{2(\alpha)}_{2(\alpha)1(\alpha)}\Gamma^{1(\alpha)}_{1(\alpha)1(\alpha)}}\notag\\&+\cancel{\Gamma^{2(\alpha)}_{1(\alpha)2(\alpha)}\Gamma^{2(\alpha)}_{2(\alpha)1(\alpha)}}-\Gamma^{2(\alpha)}_{2(\alpha)2(\alpha)}\Gamma^{2(\alpha)}_{1(\alpha)1(\alpha)}\notag\\&+\overset{}{\underset{\sigma\neq\alpha}{\sum}}\,\bigg(\Gamma^{2(\alpha)}_{1(\alpha)1(\sigma)}\xcancel{\Gamma^{1(\sigma)}_{2(\alpha)1(\alpha)}}-\Gamma^{2(\alpha)}_{2(\alpha)1(\sigma)}\Gamma^{1(\sigma)}_{1(\alpha)1(\alpha)}\bigg)
		\notag\\&=\partial_{1(\alpha)}\Gamma^{2(\alpha)}_{2(\alpha)1(\alpha)}-\partial_{2(\alpha)}\Gamma^{2(\alpha)}_{1(\alpha)1(\alpha)}\notag\\&+\overset{}{\underset{\sigma\neq\alpha}{\sum}}\bigg(\Gamma^{2(\alpha)}_{2(\alpha)2(\alpha)}\Gamma^{2(\alpha)}_{1(\sigma)1(\alpha)}+\Gamma^{2(\alpha)}_{2(\alpha)1(\sigma)}\Gamma^{1(\sigma)}_{1(\alpha)1(\sigma)}\bigg)
		\notag\\&=\partial_{1(\alpha)}\Gamma^{2(\alpha)}_{2(\alpha)1(\alpha)}-\partial_{2(\alpha)}\Gamma^{2(\alpha)}_{1(\alpha)1(\alpha)}\notag\\&+\overset{}{\underset{\sigma\neq\alpha}{\sum}}\bigg(\cancel{\frac{m_\alpha\epsilon_\alpha}{u^{2(\alpha)}}\,\frac{1}{u^{1(\alpha)}-u^{1(\sigma)}}\,\Gamma^{1(\alpha)}_{1(\sigma)1(\alpha)}\,u^{2(\alpha)}}\notag\\&-\cancel{\Gamma^{1(\alpha)}_{1(\alpha)1(\sigma)}\,\frac{m_\alpha\epsilon_\alpha}{u^{1(\alpha)}-u^{1(\sigma)}}}\bigg)
		\notag\\&=\partial_{1(\alpha)}\Gamma^{2(\alpha)}_{2(\alpha)1(\alpha)}-\partial_{2(\alpha)}\Gamma^{2(\alpha)}_{1(\alpha)1(\alpha)}
		\notag\\&=\underset{\sigma\neq\alpha}{\sum}\bigg(-\partial_{1(\alpha)}\Gamma^{2(\alpha)}_{2(\alpha)1(\sigma)}+\partial_{2(\alpha)}\Gamma^{2(\alpha)}_{1(\sigma)1(\alpha)}\bigg)
		\notag\\&=\underset{\sigma\neq\alpha}{\sum}\bigg[-\partial_{1(\alpha)}\bigg(\frac{m_\sigma\epsilon_\sigma}{u^{1(\alpha)}-u^{1(\sigma)}}\bigg)+\partial_{2(\alpha)}\bigg(-\frac{1}{u^{1(\alpha)}-u^{1(\sigma)}}\,\Gamma^{1(\alpha)}_{1(\sigma)1(\alpha)}\,u^{2(\alpha)}\bigg)\bigg]
		\notag\\&=\underset{\sigma\neq\alpha}{\sum}\bigg[\frac{m_\sigma\epsilon_\sigma}{(u^{1(\alpha)}-u^{1(\sigma)})^2}-\frac{1}{u^{1(\alpha)}-u^{1(\sigma)}}\,\Gamma^{1(\alpha)}_{1(\sigma)1(\alpha)}\bigg]
		\notag\\&=\underset{\sigma\neq\alpha}{\sum}\bigg[\frac{m_\sigma\epsilon_\sigma}{(u^{1(\alpha)}-u^{1(\sigma)})^2}-\frac{1}{u^{1(\alpha)}-u^{1(\sigma)}}\,\frac{m_\sigma\epsilon_\sigma}{u^{1(\alpha)}-u^{1(\sigma)}}\bigg]=0
		\notag
	\end{align}
	when $j=1$ and
	\begin{align}
		R^{2(\alpha)}_{2(\alpha)1(\alpha)2(\alpha)}&=\partial_{1(\alpha)}\Gamma^{2(\alpha)}_{2(\alpha)2(\alpha)}-\partial_{2(\alpha)}\Gamma^{2(\alpha)}_{1(\alpha)2(\alpha)}\notag\\&+\Gamma^{2(\alpha)}_{1(\alpha)1(\alpha)}\xcancel{\Gamma^{1(\alpha)}_{2(\alpha)2(\alpha)}}-\Gamma^{2(\alpha)}_{2(\alpha)1(\alpha)}\xcancel{\Gamma^{1(\alpha)}_{1(\alpha)2(\alpha)}}\notag\\&+\cancel{\Gamma^{2(\alpha)}_{1(\alpha)2(\alpha)}\Gamma^{2(\alpha)}_{2(\alpha)2(\alpha)}}-\cancel{\Gamma^{2(\alpha)}_{2(\alpha)2(\alpha)}\Gamma^{2(\alpha)}_{1(\alpha)2(\alpha)}}\notag\\&+\overset{}{\underset{\sigma\neq\alpha}{\sum}}\,\bigg(\Gamma^{2(\alpha)}_{1(\alpha)1(\sigma)}\xcancel{\Gamma^{1(\sigma)}_{2(\alpha)2(\alpha)}}-\Gamma^{2(\alpha)}_{2(\alpha)1(\sigma)}\xcancel{\Gamma^{1(\sigma)}_{1(\alpha)2(\alpha)}}\bigg)\notag
		\notag\\&=\xcancel{\partial_{1(\alpha)}\bigg(-\frac{m_\alpha\epsilon_\alpha}{u^{2(\alpha)}}\bigg)}+\underset{\sigma\neq\alpha}{\sum}\,\partial_{2(\alpha)}\Gamma^{2(\alpha)}_{1(\sigma)2(\alpha)}
		\notag\\&=\underset{\sigma\neq\alpha}{\sum}\,\partial_{2(\alpha)}\bigg(\frac{m_\sigma\epsilon_\sigma}{u^{1(\alpha)}-u^{1(\sigma)}}\bigg)=0\notag
	\end{align}
	when $j=2$. Therefore we proved that \eqref{Rdef_caso1} vanishes when $m_\alpha=2$. Given an integer $n\geq2$, let us now suppose that \eqref{Rdef_caso1} vanishes for $m_\alpha=n$ and show it vanishes for $m_\alpha=n+1$ as well. In other words, we are supposing that
	\begin{align}
		(R^n)^{i(\alpha)}_{h(\alpha)k(\alpha)j(\alpha)}&:=\partial_{k(\alpha)}\Gamma^{i(\alpha)}_{h(\alpha)j(\alpha)}-\partial_{h(\alpha)}\Gamma^{i(\alpha)}_{k(\alpha)j(\alpha)}\notag\\&+\overset{n}{\underset{l=1}{\sum}}\,\bigg(\Gamma^{i(\alpha)}_{k(\alpha)l(\alpha)}\Gamma^{l(\alpha)}_{h(\alpha)j(\alpha)}-\Gamma^{i(\alpha)}_{h(\alpha)l(\alpha)}\Gamma^{l(\alpha)}_{k(\alpha)j(\alpha)}\bigg)\notag\\&+\overset{}{\underset{\sigma\neq\alpha}{\sum}}\,\overset{m_\sigma}{\underset{l=1}{\sum}}\,\bigg(\Gamma^{i(\alpha)}_{k(\alpha)l(\sigma)}\Gamma^{l(\sigma)}_{h(\alpha)j(\alpha)}-\Gamma^{i(\alpha)}_{h(\alpha)l(\sigma)}\Gamma^{l(\sigma)}_{k(\alpha)j(\alpha)}\bigg)=0
		\label{Rdef_caso1_n}
	\end{align}
	for every $i$, $j$, $k$, $h\in\{1,\dots,m_\alpha\}$ for each $m_\alpha\leq n$ and we want to prove that
	\begin{align}
		(R^{n+1})^{i(\alpha)}_{h(\alpha)k(\alpha)j(\alpha)}&:=\partial_{k(\alpha)}\Gamma^{i(\alpha)}_{h(\alpha)j(\alpha)}-\partial_{h(\alpha)}\Gamma^{i(\alpha)}_{k(\alpha)j(\alpha)}\notag\\&+\overset{n+1}{\underset{l=1}{\sum}}\,\bigg(\Gamma^{i(\alpha)}_{k(\alpha)l(\alpha)}\Gamma^{l(\alpha)}_{h(\alpha)j(\alpha)}-\Gamma^{i(\alpha)}_{h(\alpha)l(\alpha)}\Gamma^{l(\alpha)}_{k(\alpha)j(\alpha)}\bigg)\notag\\&+\overset{}{\underset{\sigma\neq\alpha}{\sum}}\,\overset{m_\sigma}{\underset{l=1}{\sum}}\,\bigg(\Gamma^{i(\alpha)}_{k(\alpha)l(\sigma)}\Gamma^{l(\sigma)}_{h(\alpha)j(\alpha)}-\Gamma^{i(\alpha)}_{h(\alpha)l(\sigma)}\Gamma^{l(\sigma)}_{k(\alpha)j(\alpha)}\bigg)=0
		\label{Rdef_caso1_n+1}
	\end{align}
	for every $i$, $j$, $k$, $h\in\{1,\dots,n+1\}$ for $m_\alpha=n+1$. Let us start by considering the case where $i\leq n$ and observe that
	\begin{align}
		\Gamma^{i(\alpha)}_{j(\alpha)(n+1)(\alpha)}&=0\qquad\forall j\in\{1,\dots,n+1\}
		\label{helpful_ileqn}
	\end{align}
	as $i-j-(n+1)\leq n-2-n-1=-3$ for $j\geq 2$ and \[\Gamma^{i(\alpha)}_{j(\alpha)(n+1)(\alpha)}=-\underset{\sigma\neq\alpha}{\sum}\,\Gamma^{i(\alpha)}_{1(\sigma)(n+1)(\alpha)}=0\]
	($i\leq n<n+1$) for $j=1$. If all of $j$, $k$, $h$ are less or equal than $n$ then
	\begin{align}
		(R^{n+1})^{i(\alpha)}_{h(\alpha)k(\alpha)j(\alpha)}&=\partial_{k(\alpha)}\Gamma^{i(\alpha)}_{h(\alpha)j(\alpha)}-\partial_{h(\alpha)}\Gamma^{i(\alpha)}_{k(\alpha)j(\alpha)}\notag\\&+\overset{n+1}{\underset{l=1}{\sum}}\,\bigg(\Gamma^{i(\alpha)}_{k(\alpha)l(\alpha)}\Gamma^{l(\alpha)}_{h(\alpha)j(\alpha)}-\Gamma^{i(\alpha)}_{h(\alpha)l(\alpha)}\Gamma^{l(\alpha)}_{k(\alpha)j(\alpha)}\bigg)\notag\\&+\overset{}{\underset{\sigma\neq\alpha}{\sum}}\,\overset{m_\sigma}{\underset{l=1}{\sum}}\,\bigg(\Gamma^{i(\alpha)}_{k(\alpha)l(\sigma)}\Gamma^{l(\sigma)}_{h(\alpha)j(\alpha)}-\Gamma^{i(\alpha)}_{h(\alpha)l(\sigma)}\Gamma^{l(\sigma)}_{k(\alpha)j(\alpha)}\bigg)
		\notag
	\end{align}
	where in the first summation only the terms for $l\leq n$ survive, as both $\Gamma^{i(\alpha)}_{k(\alpha)(n+1)(\alpha)}$ and $\Gamma^{i(\alpha)}_{h(\alpha)(n+1)(\alpha)}$ vanish due to \eqref{helpful_ileqn}. This yields
	\begin{align}
		(R^{n+1})^{i(\alpha)}_{h(\alpha)k(\alpha)j(\alpha)}&=\partial_{k(\alpha)}\Gamma^{i(\alpha)}_{h(\alpha)j(\alpha)}-\partial_{h(\alpha)}\Gamma^{i(\alpha)}_{k(\alpha)j(\alpha)}\notag\\&+\overset{n}{\underset{l=1}{\sum}}\,\bigg(\Gamma^{i(\alpha)}_{k(\alpha)l(\alpha)}\Gamma^{l(\alpha)}_{h(\alpha)j(\alpha)}-\Gamma^{i(\alpha)}_{h(\alpha)l(\alpha)}\Gamma^{l(\alpha)}_{k(\alpha)j(\alpha)}\bigg)\notag\\&+\overset{}{\underset{\sigma\neq\alpha}{\sum}}\,\overset{m_\sigma}{\underset{l=1}{\sum}}\,\bigg(\Gamma^{i(\alpha)}_{k(\alpha)l(\sigma)}\Gamma^{l(\sigma)}_{h(\alpha)j(\alpha)}-\Gamma^{i(\alpha)}_{h(\alpha)l(\sigma)}\Gamma^{l(\sigma)}_{k(\alpha)j(\alpha)}\bigg)
		\notag\\&=(R^{n})^{i(\alpha)}_{h(\alpha)k(\alpha)j(\alpha)}\overset{\eqref{Rdef_caso1_n}}{=}0.
		\notag
	\end{align}
	If $k$, $h\leq n$ and $j=n+1$ then
	\begin{align}
		(R^{n+1})^{i(\alpha)}_{h(\alpha)k(\alpha)(n+1)(\alpha)}&=\partial_{k(\alpha)}\xcancel{\Gamma^{i(\alpha)}_{h(\alpha)(n+1)(\alpha)}}-\partial_{h(\alpha)}\xcancel{\Gamma^{i(\alpha)}_{k(\alpha)(n+1)(\alpha)}}\notag\\&+\overset{n}{\underset{l=1}{\sum}}\,\bigg(\Gamma^{i(\alpha)}_{k(\alpha)l(\alpha)}\xcancel{\Gamma^{l(\alpha)}_{h(\alpha)(n+1)(\alpha)}}-\Gamma^{i(\alpha)}_{h(\alpha)l(\alpha)}\xcancel{\Gamma^{l(\alpha)}_{k(\alpha)(n+1)(\alpha)}}\bigg)\notag\\&+\xcancel{\Gamma^{i(\alpha)}_{k(\alpha)(n+1)(\alpha)}}\Gamma^{(n+1)(\alpha)}_{h(\alpha)(n+1)(\alpha)}-\xcancel{\Gamma^{i(\alpha)}_{h(\alpha)(n+1)(\alpha)}}\Gamma^{(n+1)(\alpha)}_{k(\alpha)(n+1)(\alpha)}\notag\\&+\overset{}{\underset{\sigma\neq\alpha}{\sum}}\,\overset{m_\sigma}{\underset{l=1}{\sum}}\,\bigg(\Gamma^{i(\alpha)}_{k(\alpha)l(\sigma)}\Gamma^{l(\sigma)}_{h(\alpha)(n+1)(\alpha)}-\Gamma^{i(\alpha)}_{h(\alpha)l(\sigma)}\Gamma^{l(\sigma)}_{k(\alpha)(n+1)(\alpha)}\bigg)
		\notag
	\end{align}
	where the erased terms vanish because of \eqref{helpful_ileqn}. This yields
	\begin{align}
		(R^{n+1})^{i(\alpha)}_{h(\alpha)k(\alpha)(n+1)(\alpha)}&=\overset{}{\underset{\sigma\neq\alpha}{\sum}}\,\overset{m_\sigma}{\underset{l=1}{\sum}}\,\bigg(\Gamma^{i(\alpha)}_{k(\alpha)l(\sigma)}\xcancel{\Gamma^{l(\sigma)}_{h(\alpha)(n+1)(\alpha)}}-\Gamma^{i(\alpha)}_{h(\alpha)l(\sigma)}\xcancel{\Gamma^{l(\sigma)}_{k(\alpha)(n+1)(\alpha)}}\bigg)=0.
		\notag
	\end{align}
	If $k$, $j\leq n$ and $h=n+1$ (due to the simmetries of $R$, this covers the case where $h$, $j\leq n$ and $k=n+1$ as well) then
	\begin{align}
		(R^{n+1})^{i(\alpha)}_{(n+1)(\alpha)k(\alpha)j(\alpha)}&=\partial_{k(\alpha)}\xcancel{\Gamma^{i(\alpha)}_{(n+1)(\alpha)j(\alpha)}}-\partial_{(n+1)(\alpha)}\Gamma^{i(\alpha)}_{k(\alpha)j(\alpha)}\notag\\&+\overset{n}{\underset{l=1}{\sum}}\,\bigg(\Gamma^{i(\alpha)}_{k(\alpha)l(\alpha)}\xcancel{\Gamma^{l(\alpha)}_{(n+1)(\alpha)j(\alpha)}}-\xcancel{\Gamma^{i(\alpha)}_{(n+1)(\alpha)l(\alpha)}}\Gamma^{l(\alpha)}_{k(\alpha)j(\alpha)}\bigg)\notag\\&+\xcancel{\Gamma^{i(\alpha)}_{k(\alpha)(n+1)(\alpha)}}\Gamma^{(n+1)(\alpha)}_{(n+1)(\alpha)j(\alpha)}-\xcancel{\Gamma^{i(\alpha)}_{(n+1)(\alpha)(n+1)(\alpha)}}\Gamma^{(n+1)(\alpha)}_{k(\alpha)j(\alpha)}\notag\\&+\overset{}{\underset{\sigma\neq\alpha}{\sum}}\,\overset{m_\sigma}{\underset{l=1}{\sum}}\,\bigg(\Gamma^{i(\alpha)}_{k(\alpha)l(\sigma)}\Gamma^{l(\sigma)}_{(n+1)(\alpha)j(\alpha)}-\Gamma^{i(\alpha)}_{(n+1)(\alpha)l(\sigma)}\Gamma^{l(\sigma)}_{k(\alpha)j(\alpha)}\bigg)
		\notag
	\end{align}
	where the erased terms vanish because of \eqref{helpful_ileqn}. This yields
	\begin{align}
		(R^{n+1})^{i(\alpha)}_{(n+1)(\alpha)k(\alpha)j(\alpha)}&=-\partial_{(n+1)(\alpha)}\Gamma^{i(\alpha)}_{k(\alpha)j(\alpha)}\notag\\&+\overset{}{\underset{\sigma\neq\alpha}{\sum}}\,\overset{m_\sigma}{\underset{l=1}{\sum}}\,\bigg(\Gamma^{i(\alpha)}_{k(\alpha)l(\sigma)}\xcancel{\Gamma^{l(\sigma)}_{(n+1)(\alpha)j(\alpha)}}-\Gamma^{i(\alpha)}_{(n+1)(\alpha)l(\sigma)}\xcancel{\Gamma^{l(\sigma)}_{k(\alpha)j(\alpha)}}\bigg)
		\notag\\&=-\partial_{(n+1)(\alpha)}\Gamma^{i(\alpha)}_{k(\alpha)j(\alpha)}
		\notag
	\end{align}
	that becomes
	\begin{align}
		(R^{n+1})^{i(\alpha)}_{(n+1)(\alpha)k(\alpha)1(\alpha)}&=-\partial_{(n+1)(\alpha)}\Gamma^{i(\alpha)}_{k(\alpha)1(\alpha)}=0
		\notag
	\end{align}
	($\Gamma^{i(\alpha)}_{k(\alpha)1(\alpha)}=-\underset{\sigma\neq\alpha}{\sum}\,\Gamma^{i(\alpha)}_{k(\alpha)1(\sigma)}$ only depends on $\{u^{1(\alpha)}-u^{1(\sigma)}\,|\,\sigma\neq\alpha\}$ and $\{u^{s(\alpha)}\,|\,2\leq s\leq i-k+1\}$, thus it does not depend on $u^{(n+1)(\alpha)}$ as $i-k+1\leq n-1+1=n<n+1$) when $j=1$,
	\begin{align}
		(R^{n+1})^{i(\alpha)}_{(n+1)(\alpha)1(\alpha)j(\alpha)}&=-\partial_{(n+1)(\alpha)}\Gamma^{i(\alpha)}_{1(\alpha)j(\alpha)}=0
		\notag
	\end{align}
	(analogously) when $k=1$ and
	\begin{align}
		(R^{n+1})^{i(\alpha)}_{(n+1)(\alpha)k(\alpha)j(\alpha)}&=-\partial_{(n+1)(\alpha)}\Gamma^{i(\alpha)}_{k(\alpha)j(\alpha)}=0
		\notag
	\end{align}
	when both $j$ and $k$ are greater or equal than $2$, as\footnote{Without loss of generality we assume $i-j-k\geq-2$, as $\Gamma^{i(\alpha)}_{k(\alpha)j(\alpha)}=0$ automatically when $i-j-k\leq-3$.}
	\begin{align}
		\Gamma^{i(\alpha)}_{k(\alpha)j(\alpha)}&=\Gamma^{(i-j-k+4)(\alpha)}_{2(\alpha)2(\alpha)}=\Gamma^{(i-j-k+2)(\alpha)}_{1(\alpha)1(\alpha)}-\Gamma^{2(\alpha)}_{2(\alpha)2(\alpha)}\frac{u^{(i-j-k+4)(\alpha)}}{u^{2(\alpha)}}\notag\\&-\frac{1}{u^{2(\alpha)}}\overset{i-j-k+1}{\underset{l=1}{\sum}}\big(\Gamma^{(l+2)(\alpha)}_{2(\alpha)2(\alpha)}-\Gamma^{l(\alpha)}_{1(\alpha)1(\alpha)}\big)u^{(i-j-k+4-l)(\alpha)}
		\notag
	\end{align}
	does not depend on $u^{(n+1)(\alpha)}$ ($\Gamma^{(i-j-k+2)(\alpha)}_{1(\alpha)1(\alpha)}$ only depends on $\{u^{1(\alpha)}-u^{1(\sigma)}\,|\,\sigma\neq\alpha\}$ and $\{u^{s(\alpha)}\,|\,2\leq s\leq i-j-k+2\}$ where $i-j-k+2\leq n-2-2+2=n-2<n+1$, $\Gamma^{2(\alpha)}_{2(\alpha)2(\alpha)}$ only depends on $u^{2(\alpha)}$, $i-j-k+4\leq n-2-2+4=n<n+1$ and for every $1\leq l\leq i-j-k+1$ we have $i-j-k+4-l<i-j-k+4<n+1$ and $\Gamma^{(l+2)(\alpha)}_{2(\alpha)2(\alpha)}-\Gamma^{l(\alpha)}_{1(\alpha)1(\alpha)}$ only depends on quantities that correspond to lower indices so it does not depend on $u^{(n+1)(\alpha)}$ a fortiori). If $k=h=n+1$ then $(R^{n+1})^{i(\alpha)}_{h(\alpha)k(\alpha)j(\alpha)}=0$ for every value of $j$ due to the simmetries of $R$. If $k\leq n$ and $j=h=n+1$ (due to the simmetries of $R$, this covers the case where $h\leq n$ and $j=k=n+1$ as well) then
	\begin{align}
		(R^{n+1})^{i(\alpha)}_{(n+1)(\alpha)k(\alpha)(n+1)(\alpha)}&:=\partial_{k(\alpha)}\xcancel{\Gamma^{i(\alpha)}_{(n+1)(\alpha)(n+1)(\alpha)}}-\partial_{(n+1)(\alpha)}\xcancel{\Gamma^{i(\alpha)}_{k(\alpha)(n+1)(\alpha)}}\notag\\&+\overset{n}{\underset{l=1}{\sum}}\,\bigg(\Gamma^{i(\alpha)}_{k(\alpha)l(\alpha)}\xcancel{\Gamma^{l(\alpha)}_{(n+1)(\alpha)(n+1)(\alpha)}}-\xcancel{\Gamma^{i(\alpha)}_{(n+1)(\alpha)l(\alpha)}}\Gamma^{l(\alpha)}_{k(\alpha)(n+1)(\alpha)}\bigg)\notag\\&+\xcancel{\Gamma^{i(\alpha)}_{k(\alpha)(n+1)(\alpha)}}\Gamma^{(n+1)(\alpha)}_{(n+1)(\alpha)(n+1)(\alpha)}-\xcancel{\Gamma^{i(\alpha)}_{(n+1)(\alpha)(n+1)(\alpha)}}\Gamma^{(n+1)(\alpha)}_{k(\alpha)(n+1)(\alpha)}\notag\\&+\overset{}{\underset{\sigma\neq\alpha}{\sum}}\,\overset{m_\sigma}{\underset{l=1}{\sum}}\,\bigg(\Gamma^{i(\alpha)}_{k(\alpha)l(\sigma)}\Gamma^{l(\sigma)}_{(n+1)(\alpha)(n+1)(\alpha)}-\Gamma^{i(\alpha)}_{(n+1)(\alpha)l(\sigma)}\Gamma^{l(\sigma)}_{k(\alpha)(n+1)(\alpha)}\bigg)
		\notag
	\end{align}
	where the erased terms vanish because of \eqref{helpful_ileqn}. This yields
	\begin{align}
		(R^{n+1})^{i(\alpha)}_{(n+1)(\alpha)k(\alpha)(n+1)(\alpha)}&:=\overset{}{\underset{\sigma\neq\alpha}{\sum}}\,\overset{m_\sigma}{\underset{l=1}{\sum}}\,\bigg(\Gamma^{i(\alpha)}_{k(\alpha)l(\sigma)}\xcancel{\Gamma^{l(\sigma)}_{(n+1)(\alpha)(n+1)(\alpha)}}-\Gamma^{i(\alpha)}_{(n+1)(\alpha)l(\sigma)}\xcancel{\Gamma^{l(\sigma)}_{k(\alpha)(n+1)(\alpha)}}\bigg)\notag\\&=0.
		\notag
	\end{align}
	We have therefore proved \eqref{Rdef_caso1_n+1} under the assumption that $i\leq n$. Let us now fix $i=n+1$. We have
	\begin{align}
		(R^{n+1})^{(n+1)(\alpha)}_{h(\alpha)k(\alpha)j(\alpha)}&=\partial_{k(\alpha)}\Gamma^{(n+1)(\alpha)}_{h(\alpha)j(\alpha)}-\partial_{h(\alpha)}\Gamma^{(n+1)(\alpha)}_{k(\alpha)j(\alpha)}\notag\\&+\overset{n+1}{\underset{l=1}{\sum}}\,\bigg(\Gamma^{(n+1)(\alpha)}_{k(\alpha)l(\alpha)}\Gamma^{l(\alpha)}_{h(\alpha)j(\alpha)}-\Gamma^{(n+1)(\alpha)}_{h(\alpha)l(\alpha)}\Gamma^{l(\alpha)}_{k(\alpha)j(\alpha)}\bigg)\notag\\&+\overset{}{\underset{\sigma\neq\alpha}{\sum}}\,\overset{m_\sigma}{\underset{l=1}{\sum}}\,\bigg(\Gamma^{(n+1)(\alpha)}_{k(\alpha)l(\sigma)}\Gamma^{l(\sigma)}_{h(\alpha)j(\alpha)}-\Gamma^{(n+1)(\alpha)}_{h(\alpha)l(\sigma)}\Gamma^{l(\sigma)}_{k(\alpha)j(\alpha)}\bigg)
		\notag
	\end{align}
	where in the last summation only survive the terms for $l=1$, yielding
	\begin{align}
		(R^{n+1})^{(n+1)(\alpha)}_{h(\alpha)k(\alpha)j(\alpha)}&=\partial_{k(\alpha)}\Gamma^{(n+1)(\alpha)}_{h(\alpha)j(\alpha)}-\partial_{h(\alpha)}\Gamma^{(n+1)(\alpha)}_{k(\alpha)j(\alpha)}\notag\\&+\overset{n+1}{\underset{l=1}{\sum}}\,\bigg(\Gamma^{(n+1)(\alpha)}_{k(\alpha)l(\alpha)}\Gamma^{l(\alpha)}_{h(\alpha)j(\alpha)}-\Gamma^{(n+1)(\alpha)}_{h(\alpha)l(\alpha)}\Gamma^{l(\alpha)}_{k(\alpha)j(\alpha)}\bigg)\notag\\&+\overset{}{\underset{\sigma\neq\alpha}{\sum}}\,\bigg(\Gamma^{(n+1)(\alpha)}_{k(\alpha)1(\sigma)}\Gamma^{1(\sigma)}_{h(\alpha)j(\alpha)}-\Gamma^{(n+1)(\alpha)}_{h(\alpha)1(\sigma)}\Gamma^{1(\sigma)}_{k(\alpha)j(\alpha)}\bigg).
		\label{Rdel_caso1_n+1_in+1}
	\end{align}
	We distinguish between the following subcases:
	\begin{itemize}
		\item[$a.$] both $k$ and $h$ are greater or equal than $3$
		\item[$b.$] $k=1$, $h\geq3$ (this covers $h=1$, $k\geq3$ as well)
		\item[$c.$] $k=2$, $h\geq3$ (this covers $h=2$, $k\geq3$ as well)
		\item[$d.$] $k=1$, $h=2$ (this covers $h=1$, $k=2$ as well)
	\end{itemize}
	observing that $(R^{n+1})^{(n+1)(\alpha)}_{h(\alpha)k(\alpha)j(\alpha)}=0$ automatically whenever $k=h$.
	\\\textbf{Subcase a:} both $k$ and $h$ are greater or equal than $3$. We have
	\begin{align}
		(R^{n+1})^{(n+1)(\alpha)}_{h(\alpha)k(\alpha)j(\alpha)}&=\partial_{k(\alpha)}\Gamma^{(n+1)(\alpha)}_{h(\alpha)j(\alpha)}-\partial_{h(\alpha)}\Gamma^{(n+1)(\alpha)}_{k(\alpha)j(\alpha)}\notag\\&+\overset{n+1}{\underset{l=1}{\sum}}\,\bigg(\Gamma^{(n+1)(\alpha)}_{k(\alpha)l(\alpha)}\Gamma^{l(\alpha)}_{h(\alpha)j(\alpha)}-\Gamma^{(n+1)(\alpha)}_{h(\alpha)l(\alpha)}\Gamma^{l(\alpha)}_{k(\alpha)j(\alpha)}\bigg)\notag\\&+\overset{}{\underset{\sigma\neq\alpha}{\sum}}\,\bigg(\Gamma^{(n+1)(\alpha)}_{k(\alpha)1(\sigma)}\Gamma^{1(\sigma)}_{h(\alpha)j(\alpha)}-\Gamma^{(n+1)(\alpha)}_{h(\alpha)1(\sigma)}\Gamma^{1(\sigma)}_{k(\alpha)j(\alpha)}\bigg)
		\notag\\&\overset{\eqref{Lemma7.1eq}}{=}\partial_{(k-1)(\alpha)}\Gamma^{n(\alpha)}_{h(\alpha)j(\alpha)}-\partial_{(h-1)(\alpha)}\Gamma^{n(\alpha)}_{k(\alpha)j(\alpha)}\notag\\&+\Gamma^{n(\alpha)}_{(k-1)(\alpha)1(\alpha)}\xcancel{\Gamma^{1(\alpha)}_{h(\alpha)j(\alpha)}}-\Gamma^{n(\alpha)}_{(h-1)(\alpha)1(\alpha)}\xcancel{\Gamma^{1(\alpha)}_{k(\alpha)j(\alpha)}}\notag\\&+\overset{n+1}{\underset{l=2}{\sum}}\,\bigg(\Gamma^{n(\alpha)}_{(k-1)(\alpha)l(\alpha)}\Gamma^{l(\alpha)}_{h(\alpha)j(\alpha)}-\Gamma^{n(\alpha)}_{(h-1)(\alpha)l(\alpha)}\Gamma^{l(\alpha)}_{k(\alpha)j(\alpha)}\bigg)\notag\\&+\overset{}{\underset{\sigma\neq\alpha}{\sum}}\,\bigg(\Gamma^{n(\alpha)}_{(k-1)(\alpha)1(\sigma)}\Gamma^{1(\sigma)}_{h(\alpha)j(\alpha)}-\Gamma^{n(\alpha)}_{(h-1)(\alpha)1(\sigma)}\Gamma^{1(\sigma)}_{k(\alpha)j(\alpha)}\bigg).
		\label{Rdel_caso1_n+1_in+1_subcasea}
	\end{align}
	If $j\leq n+1$ we get
	\begin{align}
		(R^{n+1})^{(n+1)(\alpha)}_{h(\alpha)k(\alpha)j(\alpha)}&=\partial_{(k-1)(\alpha)}\Gamma^{n(\alpha)}_{(h-1)(\alpha)(j+1)(\alpha)}-\partial_{(h-1)(\alpha)}\Gamma^{n(\alpha)}_{(k-1)(\alpha)(j+1)(\alpha)}\notag\\&+\overset{n+1}{\underset{l=\cancelto{1}{2}}{\sum}}\,\bigg(\Gamma^{n(\alpha)}_{(k-1)(\alpha)l(\alpha)}\Gamma^{l(\alpha)}_{(h-1)(\alpha)(j+1)(\alpha)}-\Gamma^{n(\alpha)}_{(h-1)(\alpha)l(\alpha)}\Gamma^{l(\alpha)}_{(k-1)(\alpha)(j+1)(\alpha)}\bigg)\notag\\&+\overset{}{\underset{\sigma\neq\alpha}{\sum}}\,\bigg(\Gamma^{n(\alpha)}_{(k-1)(\alpha)1(\sigma)}\Gamma^{1(\sigma)}_{(h-1)(\alpha)(j+1)(\alpha)}-\Gamma^{n(\alpha)}_{(h-1)(\alpha)1(\sigma)}\Gamma^{1(\sigma)}_{(k-1)(\alpha)(j+1)(\alpha)}\bigg)
		\notag\\&=(R^{n})^{n(\alpha)}_{(h-1)(\alpha)(k-1)(\alpha)(j+1)(\alpha)}\notag\\&+\Gamma^{n(\alpha)}_{(k-1)(\alpha)(n+1)(\alpha)}\Gamma^{(n+1)(\alpha)}_{(h-1)(\alpha)(j+1)(\alpha)}-\Gamma^{n(\alpha)}_{(h-1)(\alpha)(n+1)(\alpha)}\Gamma^{(n+1)(\alpha)}_{(k-1)(\alpha)(j+1)(\alpha)}=0
		\notag
	\end{align}
	as $(R^{n})^{n(\alpha)}_{(h-1)(\alpha)(k-1)(\alpha)(j+1)(\alpha)}$ vanishes because of \eqref{Rdef_caso1_n} and $\Gamma^{n(\alpha)}_{(k-1)(\alpha)(n+1)(\alpha)}$, $\Gamma^{n(\alpha)}_{(h-1)(\alpha)(n+1)(\alpha)}$ vanish because of \eqref{helpful_ileqn}. If $j=n$ then \eqref{Rdel_caso1_n+1_in+1_subcasea} becomes
	\begin{align}
		(R^{n+1})^{(n+1)(\alpha)}_{h(\alpha)k(\alpha)n(\alpha)}&=\partial_{(k-1)(\alpha)}\xcancel{\Gamma^{n(\alpha)}_{h(\alpha)n(\alpha)}}-\partial_{(h-1)(\alpha)}\xcancel{\Gamma^{n(\alpha)}_{k(\alpha)n(\alpha)}}\notag\\&+\overset{n}{\underset{l=2}{\sum}}\,\bigg(\Gamma^{n(\alpha)}_{(k-1)(\alpha)l(\alpha)}\Gamma^{l(\alpha)}_{(h-1)(\alpha)(n+1)(\alpha)}-\Gamma^{n(\alpha)}_{(h-1)(\alpha)l(\alpha)}\Gamma^{l(\alpha)}_{(k-1)(\alpha)(n+1)(\alpha)}\bigg)\notag\\&+\Gamma^{n(\alpha)}_{(k-1)(\alpha)(n+1)(\alpha)}\Gamma^{(n+1)(\alpha)}_{h(\alpha)n(\alpha)}-\Gamma^{n(\alpha)}_{(h-1)(\alpha)(n+1)(\alpha)}\Gamma^{(n+1)(\alpha)}_{k(\alpha)n(\alpha)}\notag\\&+\overset{}{\underset{\sigma\neq\alpha}{\sum}}\,\bigg(\Gamma^{n(\alpha)}_{(k-1)(\alpha)1(\sigma)}\xcancel{\Gamma^{1(\sigma)}_{h(\alpha)n(\alpha)}}-\Gamma^{n(\alpha)}_{(h-1)(\alpha)1(\sigma)}\xcancel{\Gamma^{1(\sigma)}_{k(\alpha)n(\alpha)}}\bigg)
		\notag\\&\overset{\eqref{helpful_ileqn}}{=}\overset{n}{\underset{l=2}{\sum}}\,\bigg(\Gamma^{n(\alpha)}_{(k-1)(\alpha)l(\alpha)}\xcancel{\Gamma^{l(\alpha)}_{(h-1)(\alpha)(n+1)(\alpha)}}-\Gamma^{n(\alpha)}_{(h-1)(\alpha)l(\alpha)}\xcancel{\Gamma^{l(\alpha)}_{(k-1)(\alpha)(n+1)(\alpha)}}\bigg)\notag\\&+\xcancel{\Gamma^{n(\alpha)}_{(k-1)(\alpha)(n+1)(\alpha)}}\Gamma^{(n+1)(\alpha)}_{h(\alpha)n(\alpha)}-\xcancel{\Gamma^{n(\alpha)}_{(h-1)(\alpha)(n+1)(\alpha)}}\Gamma^{(n+1)(\alpha)}_{k(\alpha)n(\alpha)}=0.
		\notag
	\end{align}
	If $j=n+1$ then \eqref{Rdel_caso1_n+1_in+1_subcasea} becomes
	\begin{align}
		(R^{n+1})^{(n+1)(\alpha)}_{h(\alpha)k(\alpha)(n+1)(\alpha)}&=\partial_{(k-1)(\alpha)}\Gamma^{n(\alpha)}_{h(\alpha)(n+1)(\alpha)}-\partial_{(h-1)(\alpha)}\Gamma^{n(\alpha)}_{k(\alpha)(n+1)(\alpha)}\notag\\&+\overset{n}{\underset{l=2}{\sum}}\,\bigg(\Gamma^{n(\alpha)}_{(k-1)(\alpha)l(\alpha)}\Gamma^{l(\alpha)}_{h(\alpha)(n+1)(\alpha)}-\Gamma^{n(\alpha)}_{(h-1)(\alpha)l(\alpha)}\Gamma^{l(\alpha)}_{k(\alpha)(n+1)(\alpha)}\bigg)\notag\\&+\Gamma^{n(\alpha)}_{(k-1)(\alpha)(n+1)(\alpha)}\Gamma^{(n+1)(\alpha)}_{h(\alpha)(n+1)(\alpha)}-\Gamma^{n(\alpha)}_{(h-1)(\alpha)(n+1)(\alpha)}\Gamma^{(n+1)(\alpha)}_{k(\alpha)(n+1)(\alpha)}\notag\\&+\overset{}{\underset{\sigma\neq\alpha}{\sum}}\,\bigg(\Gamma^{n(\alpha)}_{(k-1)(\alpha)1(\sigma)}\xcancel{\Gamma^{1(\sigma)}_{h(\alpha)(n+1)(\alpha)}}-\Gamma^{n(\alpha)}_{(h-1)(\alpha)1(\sigma)}\xcancel{\Gamma^{1(\sigma)}_{k(\alpha)(n+1)(\alpha)}}\bigg)
		\notag\\&\overset{\eqref{helpful_ileqn}}{=}\partial_{(k-1)(\alpha)}\xcancel{\Gamma^{n(\alpha)}_{h(\alpha)(n+1)(\alpha)}}-\partial_{(h-1)(\alpha)}\xcancel{\Gamma^{n(\alpha)}_{k(\alpha)(n+1)(\alpha)}}\notag\\&+\overset{n}{\underset{l=2}{\sum}}\,\bigg(\Gamma^{n(\alpha)}_{(k-1)(\alpha)l(\alpha)}\xcancel{\Gamma^{l(\alpha)}_{h(\alpha)(n+1)(\alpha)}}-\Gamma^{n(\alpha)}_{(h-1)(\alpha)l(\alpha)}\xcancel{\Gamma^{l(\alpha)}_{k(\alpha)(n+1)(\alpha)}}\bigg)\notag\\&+\xcancel{\Gamma^{n(\alpha)}_{(k-1)(\alpha)(n+1)(\alpha)}}\Gamma^{(n+1)(\alpha)}_{h(\alpha)(n+1)(\alpha)}-\xcancel{\Gamma^{n(\alpha)}_{(h-1)(\alpha)(n+1)(\alpha)}}\Gamma^{(n+1)(\alpha)}_{k(\alpha)(n+1)(\alpha)}=0.
		\notag
	\end{align}
	\\\textbf{Subcase b:} $k=1$, $h\geq3$. We have
	\begin{align}
		(R^{n+1})^{(n+1)(\alpha)}_{h(\alpha)1(\alpha)j(\alpha)}&=\partial_{1(\alpha)}\Gamma^{(n+1)(\alpha)}_{h(\alpha)j(\alpha)}-\partial_{h(\alpha)}\Gamma^{(n+1)(\alpha)}_{1(\alpha)j(\alpha)}\notag\\&+\overset{n+1}{\underset{l=1}{\sum}}\,\bigg(\Gamma^{(n+1)(\alpha)}_{1(\alpha)l(\alpha)}\Gamma^{l(\alpha)}_{h(\alpha)j(\alpha)}-\Gamma^{(n+1)(\alpha)}_{h(\alpha)l(\alpha)}\Gamma^{l(\alpha)}_{1(\alpha)j(\alpha)}\bigg)\notag\\&+\overset{}{\underset{\sigma\neq\alpha}{\sum}}\,\bigg(\Gamma^{(n+1)(\alpha)}_{1(\alpha)1(\sigma)}\Gamma^{1(\sigma)}_{h(\alpha)j(\alpha)}-\Gamma^{(n+1)(\alpha)}_{h(\alpha)1(\sigma)}\Gamma^{1(\sigma)}_{1(\alpha)j(\alpha)}\bigg)
		\notag\\&\overset{\eqref{Lemma7.1eq}}{=}\partial_{1(\alpha)}\Gamma^{n(\alpha)}_{(h-1)(\alpha)j(\alpha)}-\partial_{(h-1)(\alpha)}\Gamma^{n(\alpha)}_{1(\alpha)j(\alpha)}\notag\\&+\overset{n+1}{\underset{l=1}{\sum}}\,\Gamma^{(n+1)(\alpha)}_{1(\alpha)l(\alpha)}\Gamma^{l(\alpha)}_{h(\alpha)j(\alpha)}-\overset{n+1}{\underset{l=1}{\sum}}\,\Gamma^{(n+1)(\alpha)}_{h(\alpha)l(\alpha)}\Gamma^{l(\alpha)}_{1(\alpha)j(\alpha)}\notag\\&+\overset{}{\underset{\sigma\neq\alpha}{\sum}}\,\bigg(\Gamma^{(n+1)(\alpha)}_{1(\alpha)1(\sigma)}\Gamma^{1(\sigma)}_{h(\alpha)j(\alpha)}-\Gamma^{(n+1)(\alpha)}_{h(\alpha)1(\sigma)}\Gamma^{1(\sigma)}_{1(\alpha)j(\alpha)}\bigg)\notag\\&\underline{+\overset{}{\underset{\sigma\neq\alpha}{\sum}}\,\bigg(\Gamma^{n(\alpha)}_{1(\alpha)1(\sigma)}\Gamma^{1(\sigma)}_{(h-1)(\alpha)j(\alpha)}-\Gamma^{n(\alpha)}_{(h-1)(\alpha)1(\sigma)}\Gamma^{1(\sigma)}_{1(\alpha)j(\alpha)}\bigg)}\notag\\&\underline{-\overset{}{\underset{\sigma\neq\alpha}{\sum}}\,\bigg(\Gamma^{n(\alpha)}_{1(\alpha)1(\sigma)}\Gamma^{1(\sigma)}_{(h-1)(\alpha)j(\alpha)}-\Gamma^{n(\alpha)}_{(h-1)(\alpha)1(\sigma)}\Gamma^{1(\sigma)}_{1(\alpha)j(\alpha)}\bigg)}
		\notag
	\end{align}
	where the underlined terms cancel out and where the terms $$\partial_{1(\alpha)}\Gamma^{n(\alpha)}_{(h-1)(\alpha)j(\alpha)}\text{,}\,-\partial_{(h-1)(\alpha)}\Gamma^{n(\alpha)}_{1(\alpha)j(\alpha)}\text{,}$$
	\begin{align}
		\overset{n+1}{\underset{l=2}{\sum}}\,\Gamma^{(n+1)(\alpha)}_{1(\alpha)l(\alpha)}\Gamma^{l(\alpha)}_{h(\alpha)j(\alpha)}&=\overset{n+1}{\underset{l=2}{\sum}}\,\Gamma^{n(\alpha)}_{1(\alpha)(l-1)(\alpha)}\Gamma^{(l-1)(\alpha)}_{(h-1)(\alpha)j(\alpha)}
		\notag\\&=\overset{n}{\underset{l=1}{\sum}}\,\Gamma^{n(\alpha)}_{1(\alpha)l(\alpha)}\Gamma^{l(\alpha)}_{(h-1)(\alpha)j(\alpha)}\text{,}
		\notag
	\end{align}
	\begin{align}
		-\overset{n}{\underset{l=1}{\sum}}\,\Gamma^{(n+1)(\alpha)}_{h(\alpha)l(\alpha)}\Gamma^{l(\alpha)}_{1(\alpha)j(\alpha)}&=-\overset{n}{\underset{l=1}{\sum}}\,\Gamma^{n(\alpha)}_{(h-1)(\alpha)l(\alpha)}\Gamma^{l(\alpha)}_{1(\alpha)j(\alpha)}
		\notag
	\end{align}
	and
	\begin{align}
		\overset{}{\underset{\sigma\neq\alpha}{\sum}}\,\bigg(\Gamma^{n(\alpha)}_{1(\alpha)1(\sigma)}\Gamma^{1(\sigma)}_{(h-1)(\alpha)j(\alpha)}-\Gamma^{n(\alpha)}_{(h-1)(\alpha)1(\sigma)}\Gamma^{1(\sigma)}_{1(\alpha)j(\alpha)}\bigg)
		\notag
	\end{align}
	combine to form $(R^{n})^{n(\alpha)}_{(h-1)(\alpha)1(\alpha)j(\alpha)}$, which vanishes by \eqref{Rdef_caso1_n}. Thus
	\begin{align}
		(R^{n+1})^{(n+1)(\alpha)}_{h(\alpha)1(\alpha)j(\alpha)}&=\Gamma^{(n+1)(\alpha)}_{1(\alpha)1(\alpha)}\xcancel{\Gamma^{1(\alpha)}_{h(\alpha)j(\alpha)}}-\xcancel{\Gamma^{(n+1)(\alpha)}_{h(\alpha)(n+1)(\alpha)}}\Gamma^{(n+1)(\alpha)}_{1(\alpha)j(\alpha)}\notag\\&+\overset{}{\underset{\sigma\neq\alpha}{\sum}}\,\bigg(\Gamma^{(n+1)(\alpha)}_{1(\alpha)1(\sigma)}\xcancel{\Gamma^{1(\sigma)}_{h(\alpha)j(\alpha)}}-\Gamma^{(n+1)(\alpha)}_{h(\alpha)1(\sigma)}\Gamma^{1(\sigma)}_{1(\alpha)j(\alpha)}\bigg)\notag\\&-\overset{}{\underset{\sigma\neq\alpha}{\sum}}\,\bigg(\Gamma^{n(\alpha)}_{1(\alpha)1(\sigma)}\xcancel{\Gamma^{1(\sigma)}_{(h-1)(\alpha)j(\alpha)}}-\Gamma^{n(\alpha)}_{(h-1)(\alpha)1(\sigma)}\Gamma^{1(\sigma)}_{1(\alpha)j(\alpha)}\bigg)
		\notag\\&=\overset{}{\underset{\sigma\neq\alpha}{\sum}}\,\bigg(-\cancel{\Gamma^{n(\alpha)}_{(h-1)(\alpha)1(\sigma)}\Gamma^{1(\sigma)}_{1(\alpha)j(\alpha)}}+\cancel{\Gamma^{n(\alpha)}_{(h-1)(\alpha)1(\sigma)}\Gamma^{1(\sigma)}_{1(\alpha)j(\alpha)}}\bigg)=0.
		\notag
	\end{align}
	\textbf{Subcase c:} $k=2$, $h\geq3$. We have
	\begin{align}
		(R^{n+1})^{(n+1)(\alpha)}_{h(\alpha)2(\alpha)j(\alpha)}&=\partial_{2(\alpha)}\Gamma^{(n+1)(\alpha)}_{h(\alpha)j(\alpha)}-\partial_{h(\alpha)}\Gamma^{(n+1)(\alpha)}_{2(\alpha)j(\alpha)}\notag\\&+\overset{n+1}{\underset{l=1}{\sum}}\,\bigg(\Gamma^{(n+1)(\alpha)}_{2(\alpha)l(\alpha)}\Gamma^{l(\alpha)}_{h(\alpha)j(\alpha)}-\Gamma^{(n+1)(\alpha)}_{h(\alpha)l(\alpha)}\Gamma^{l(\alpha)}_{2(\alpha)j(\alpha)}\bigg)\notag\\&+\overset{}{\underset{\sigma\neq\alpha}{\sum}}\,\bigg(\Gamma^{(n+1)(\alpha)}_{2(\alpha)1(\sigma)}\Gamma^{1(\sigma)}_{h(\alpha)j(\alpha)}-\Gamma^{(n+1)(\alpha)}_{h(\alpha)1(\sigma)}\Gamma^{1(\sigma)}_{2(\alpha)j(\alpha)}\bigg)
		\notag\\&\overset{\eqref{Lemma7.1eq}}{=}\partial_{2(\alpha)}\Gamma^{n(\alpha)}_{(h-1)(\alpha)j(\alpha)}-\partial_{(h-1)(\alpha)}\Gamma^{n(\alpha)}_{2(\alpha)j(\alpha)}\notag\\&+\overset{n+1}{\underset{l=1}{\sum}}\,\Gamma^{(n+1)(\alpha)}_{2(\alpha)l(\alpha)}\Gamma^{l(\alpha)}_{h(\alpha)j(\alpha)}-\overset{n+1}{\underset{l=1}{\sum}}\,\Gamma^{(n+1)(\alpha)}_{h(\alpha)l(\alpha)}\Gamma^{l(\alpha)}_{2(\alpha)j(\alpha)}\notag\\&+\overset{}{\underset{\sigma\neq\alpha}{\sum}}\,\bigg(\Gamma^{(n+1)(\alpha)}_{2(\alpha)1(\sigma)}\Gamma^{1(\sigma)}_{h(\alpha)j(\alpha)}-\Gamma^{(n+1)(\alpha)}_{h(\alpha)1(\sigma)}\Gamma^{1(\sigma)}_{2(\alpha)j(\alpha)}\bigg)\notag\\&\underline{+\overset{}{\underset{\sigma\neq\alpha}{\sum}}\,\bigg(\Gamma^{n(\alpha)}_{2(\alpha)1(\sigma)}\Gamma^{1(\sigma)}_{(h-1)(\alpha)j(\alpha)}-\Gamma^{n(\alpha)}_{(h-1)(\alpha)1(\sigma)}\Gamma^{1(\sigma)}_{2(\alpha)j(\alpha)}\bigg)}\notag\\&\underline{-\overset{}{\underset{\sigma\neq\alpha}{\sum}}\,\bigg(\Gamma^{n(\alpha)}_{2(\alpha)1(\sigma)}\Gamma^{1(\sigma)}_{(h-1)(\alpha)j(\alpha)}-\Gamma^{n(\alpha)}_{(h-1)(\alpha)1(\sigma)}\Gamma^{1(\sigma)}_{2(\alpha)j(\alpha)}\bigg)}
		\notag
	\end{align}
	where the underlined terms cancel out and where the terms $$\partial_{2(\alpha)}\Gamma^{n(\alpha)}_{(h-1)(\alpha)j(\alpha)}\text{,}\,-\partial_{(h-1)(\alpha)}\Gamma^{n(\alpha)}_{2(\alpha)j(\alpha)}\text{,}$$
	\begin{align}
		\overset{n+1}{\underset{l=2}{\sum}}\,\Gamma^{(n+1)(\alpha)}_{2(\alpha)l(\alpha)}\Gamma^{l(\alpha)}_{h(\alpha)j(\alpha)}&=\overset{n+1}{\underset{l=2}{\sum}}\,\Gamma^{n(\alpha)}_{2(\alpha)(l-1)(\alpha)}\Gamma^{(l-1)(\alpha)}_{(h-1)(\alpha)j(\alpha)}
		\notag\\&=\overset{n}{\underset{l=1}{\sum}}\,\Gamma^{n(\alpha)}_{2(\alpha)l(\alpha)}\Gamma^{l(\alpha)}_{(h-1)(\alpha)j(\alpha)}
		\text{,}\notag
	\end{align}
	\begin{align}
		-\overset{n}{\underset{l=1}{\sum}}\,\Gamma^{(n+1)(\alpha)}_{h(\alpha)l(\alpha)}\Gamma^{l(\alpha)}_{2(\alpha)j(\alpha)}&=-\overset{n}{\underset{l=1}{\sum}}\,\Gamma^{n(\alpha)}_{(h-1)(\alpha)l(\alpha)}\Gamma^{l(\alpha)}_{2(\alpha)j(\alpha)}
		\notag
	\end{align}
	and
	\begin{align}
		\overset{}{\underset{\sigma\neq\alpha}{\sum}}\,\bigg(\Gamma^{n(\alpha)}_{2(\alpha)1(\sigma)}\Gamma^{1(\sigma)}_{(h-1)(\alpha)j(\alpha)}-\Gamma^{n(\alpha)}_{(h-1)(\alpha)1(\sigma)}\Gamma^{1(\sigma)}_{2(\alpha)j(\alpha)}\bigg)
		\notag
	\end{align}
	combine to form $(R^{n})^{n(\alpha)}_{(h-1)(\alpha)2(\alpha)j(\alpha)}$, which vanishes by \eqref{Rdef_caso1_n}. Thus
	\begin{align}
		(R^{n+1})^{(n+1)(\alpha)}_{h(\alpha)2(\alpha)j(\alpha)}&=\Gamma^{(n+1)(\alpha)}_{2(\alpha)1(\alpha)}\xcancel{\Gamma^{1(\alpha)}_{h(\alpha)j(\alpha)}}-\xcancel{\Gamma^{(n+1)(\alpha)}_{h(\alpha)(n+1)(\alpha)}}\Gamma^{(n+1)(\alpha)}_{2(\alpha)j(\alpha)}\notag\\&+\overset{}{\underset{\sigma\neq\alpha}{\sum}}\,\bigg(\Gamma^{(n+1)(\alpha)}_{2(\alpha)1(\sigma)}\xcancel{\Gamma^{1(\sigma)}_{h(\alpha)j(\alpha)}}-\Gamma^{(n+1)(\alpha)}_{h(\alpha)1(\sigma)}\xcancel{\Gamma^{1(\sigma)}_{2(\alpha)j(\alpha)}}\bigg)\notag\\&-\overset{}{\underset{\sigma\neq\alpha}{\sum}}\,\bigg(\Gamma^{n(\alpha)}_{2(\alpha)1(\sigma)}\xcancel{\Gamma^{1(\sigma)}_{(h-1)(\alpha)j(\alpha)}}-\Gamma^{n(\alpha)}_{(h-1)(\alpha)1(\sigma)}\xcancel{\Gamma^{1(\sigma)}_{2(\alpha)j(\alpha)}}\bigg)=0.
		\notag
	\end{align}
	\textbf{Subcase d:} $k=1$, $h=2$. We have
	\begin{align}
		(R^{n+1})^{(n+1)(\alpha)}_{2(\alpha)1(\alpha)j(\alpha)}&=\partial_{1(\alpha)}\Gamma^{(n+1)(\alpha)}_{2(\alpha)j(\alpha)}-\partial_{2(\alpha)}\Gamma^{(n+1)(\alpha)}_{1(\alpha)j(\alpha)}\notag\\&+\overset{n+1}{\underset{l=1}{\sum}}\,\bigg(\Gamma^{(n+1)(\alpha)}_{1(\alpha)l(\alpha)}\Gamma^{l(\alpha)}_{2(\alpha)j(\alpha)}-\Gamma^{(n+1)(\alpha)}_{2(\alpha)l(\alpha)}\Gamma^{l(\alpha)}_{1(\alpha)j(\alpha)}\bigg)\notag\\&+\overset{}{\underset{\sigma\neq\alpha}{\sum}}\,\bigg(\Gamma^{(n+1)(\alpha)}_{1(\alpha)1(\sigma)}\Gamma^{1(\sigma)}_{2(\alpha)j(\alpha)}-\Gamma^{(n+1)(\alpha)}_{2(\alpha)1(\sigma)}\Gamma^{1(\sigma)}_{1(\alpha)j(\alpha)}\bigg).
		\label{Rdel_caso1_n+1_in+1_subcased}
	\end{align}
	If $j\geq3$ we get
	\begin{align}
		(R^{n+1})^{(n+1)(\alpha)}_{2(\alpha)1(\alpha)j(\alpha)}&=\partial_{1(\alpha)}\Gamma^{(n+1)(\alpha)}_{2(\alpha)j(\alpha)}-\partial_{2(\alpha)}\Gamma^{(n+1)(\alpha)}_{1(\alpha)j(\alpha)}\notag\\&+\overset{n+1}{\underset{l=1}{\sum}}\,\bigg(\Gamma^{(n+1)(\alpha)}_{1(\alpha)l(\alpha)}\Gamma^{l(\alpha)}_{2(\alpha)j(\alpha)}-\Gamma^{(n+1)(\alpha)}_{2(\alpha)l(\alpha)}\Gamma^{l(\alpha)}_{1(\alpha)j(\alpha)}\bigg)\notag\\&+\overset{}{\underset{\sigma\neq\alpha}{\sum}}\,\bigg(\Gamma^{(n+1)(\alpha)}_{1(\alpha)1(\sigma)}\Gamma^{1(\sigma)}_{2(\alpha)j(\alpha)}-\Gamma^{(n+1)(\alpha)}_{2(\alpha)1(\sigma)}\Gamma^{1(\sigma)}_{1(\alpha)j(\alpha)}\bigg)
		\notag\\&=\partial_{1(\alpha)}\Gamma^{n(\alpha)}_{2(\alpha)(j-1)(\alpha)}-\partial_{2(\alpha)}\Gamma^{n(\alpha)}_{1(\alpha)(j-1)(\alpha)}\notag\\&+\overset{n+1}{\underset{l=1}{\sum}}\,\Gamma^{(n+1)(\alpha)}_{1(\alpha)l(\alpha)}\Gamma^{l(\alpha)}_{2(\alpha)j(\alpha)}-\overset{n+1}{\underset{l=1}{\sum}}\,\Gamma^{(n+1)(\alpha)}_{2(\alpha)l(\alpha)}\Gamma^{l(\alpha)}_{1(\alpha)j(\alpha)}\notag\\&+\overset{}{\underset{\sigma\neq\alpha}{\sum}}\,\bigg(\Gamma^{(n+1)(\alpha)}_{1(\alpha)1(\sigma)}\Gamma^{1(\sigma)}_{2(\alpha)j(\alpha)}-\Gamma^{(n+1)(\alpha)}_{2(\alpha)1(\sigma)}\Gamma^{1(\sigma)}_{1(\alpha)j(\alpha)}\bigg)\notag\\&\underline{+\overset{}{\underset{\sigma\neq\alpha}{\sum}}\,\bigg(\Gamma^{n(\alpha)}_{1(\alpha)1(\sigma)}\Gamma^{1(\sigma)}_{2(\alpha)(j-1)(\alpha)}-\Gamma^{n(\alpha)}_{2(\alpha)1(\sigma)}\Gamma^{1(\sigma)}_{1(\alpha)(j-1)(\alpha)}\bigg)}\notag\\&\underline{-\overset{}{\underset{\sigma\neq\alpha}{\sum}}\,\bigg(\Gamma^{n(\alpha)}_{1(\alpha)1(\sigma)}\Gamma^{1(\sigma)}_{2(\alpha)(j-1)(\alpha)}-\Gamma^{n(\alpha)}_{2(\alpha)1(\sigma)}\Gamma^{1(\sigma)}_{1(\alpha)(j-1)(\alpha)}\bigg)}
		\notag
	\end{align}
	where the underlined terms cancel out and where the terms $$\partial_{1(\alpha)}\Gamma^{n(\alpha)}_{2(\alpha)(j-1)(\alpha)}\text{,}\,-\partial_{2(\alpha)}\Gamma^{n(\alpha)}_{1(\alpha)(j-1)(\alpha)}\text{,}$$
	\begin{align}
		\overset{n+1}{\underset{l=2}{\sum}}\,\Gamma^{(n+1)(\alpha)}_{1(\alpha)l(\alpha)}\Gamma^{l(\alpha)}_{2(\alpha)j(\alpha)}&=\overset{n+1}{\underset{l=2}{\sum}}\,\Gamma^{n(\alpha)}_{1(\alpha)(l-1)(\alpha)}\Gamma^{(l-1)(\alpha)}_{2(\alpha)(j-1)(\alpha)}
		=\overset{n}{\underset{l=1}{\sum}}\,\Gamma^{n(\alpha)}_{1(\alpha)l(\alpha)}\Gamma^{l(\alpha)}_{2(\alpha)(j-1)(\alpha)}\text{,}
		\notag
	\end{align}
	\begin{align}
		-\overset{n+1}{\underset{l=2}{\sum}}\,\Gamma^{(n+1)(\alpha)}_{2(\alpha)l(\alpha)}\Gamma^{l(\alpha)}_{1(\alpha)j(\alpha)}&=-\overset{n+1}{\underset{l=2}{\sum}}\,\Gamma^{n(\alpha)}_{2(\alpha)(l-1)(\alpha)}\Gamma^{(l-1)(\alpha)}_{1(\alpha)(j-1)(\alpha)}
		=-\overset{n}{\underset{l=1}{\sum}}\,\Gamma^{n(\alpha)}_{2(\alpha)l(\alpha)}\Gamma^{l(\alpha)}_{1(\alpha)(j-1)(\alpha)}
		\notag
	\end{align}
	and
	\begin{align}
		\overset{}{\underset{\sigma\neq\alpha}{\sum}}\,\bigg(\Gamma^{n(\alpha)}_{1(\alpha)1(\sigma)}\Gamma^{1(\sigma)}_{2(\alpha)(j-1)(\alpha)}-\Gamma^{n(\alpha)}_{2(\alpha)1(\sigma)}\Gamma^{1(\sigma)}_{1(\alpha)(j-1)(\alpha)}\bigg)
		\notag
	\end{align}
	combine to form $(R^{n})^{n(\alpha)}_{2(\alpha)1(\alpha)(j-1)(\alpha)}$, which vanishes by \eqref{Rdef_caso1_n}. Thus
	\begin{align}
		(R^{n+1})^{(n+1)(\alpha)}_{2(\alpha)1(\alpha)j(\alpha)}&=\Gamma^{(n+1)(\alpha)}_{1(\alpha)1(\alpha)}\xcancel{\Gamma^{1(\alpha)}_{2(\alpha)j(\alpha)}}-\Gamma^{(n+1)(\alpha)}_{2(\alpha)1(\alpha)}\xcancel{\Gamma^{1(\alpha)}_{1(\alpha)j(\alpha)}}\notag\\&+\overset{}{\underset{\sigma\neq\alpha}{\sum}}\,\bigg(\Gamma^{(n+1)(\alpha)}_{1(\alpha)1(\sigma)}\xcancel{\Gamma^{1(\sigma)}_{2(\alpha)j(\alpha)}}-\Gamma^{(n+1)(\alpha)}_{2(\alpha)1(\sigma)}\xcancel{\Gamma^{1(\sigma)}_{1(\alpha)j(\alpha)}}\bigg)\notag\\&-\overset{}{\underset{\sigma\neq\alpha}{\sum}}\,\bigg(\Gamma^{n(\alpha)}_{1(\alpha)1(\sigma)}\xcancel{\Gamma^{1(\sigma)}_{2(\alpha)(j-1)(\alpha)}}-\Gamma^{n(\alpha)}_{2(\alpha)1(\sigma)}\xcancel{\Gamma^{1(\sigma)}_{1(\alpha)(j-1)(\alpha)}}\bigg)=0.
		\notag
	\end{align}
	If $j=2$ then \eqref{Rdel_caso1_n+1_in+1_subcased} becomes
	\begin{align}
		(R^{n+1})^{(n+1)(\alpha)}_{2(\alpha)1(\alpha)2(\alpha)}&=\partial_{1(\alpha)}\Gamma^{(n+1)(\alpha)}_{2(\alpha)2(\alpha)}-\partial_{2(\alpha)}\Gamma^{(n+1)(\alpha)}_{1(\alpha)2(\alpha)}\notag\\&+\overset{n+1}{\underset{l=1}{\sum}}\,\bigg(\Gamma^{(n+1)(\alpha)}_{1(\alpha)l(\alpha)}\Gamma^{l(\alpha)}_{2(\alpha)2(\alpha)}-\Gamma^{(n+1)(\alpha)}_{2(\alpha)l(\alpha)}\Gamma^{l(\alpha)}_{1(\alpha)2(\alpha)}\bigg)\notag\\&+\overset{}{\underset{\sigma\neq\alpha}{\sum}}\,\bigg(\Gamma^{(n+1)(\alpha)}_{1(\alpha)1(\sigma)}\Gamma^{1(\sigma)}_{2(\alpha)2(\alpha)}-\Gamma^{(n+1)(\alpha)}_{2(\alpha)1(\sigma)}\Gamma^{1(\sigma)}_{1(\alpha)2(\alpha)}\bigg)
		\notag
	\end{align}
	where in the first summation only the terms for $l\geq2$ survive, as $\Gamma^{1(\alpha)}_{2(\alpha)2(\alpha)}=0$ and $\Gamma^{1(\alpha)}_{1(\alpha)2(\alpha)}=-\underset{\sigma\neq\alpha}{\sum}\,\Gamma^{1(\alpha)}_{1(\sigma)2(\alpha)}=0$. This yields
	\begin{align}
		(R^{n+1})^{(n+1)(\alpha)}_{2(\alpha)1(\alpha)2(\alpha)}&=\partial_{1(\alpha)}\Gamma^{(n+1)(\alpha)}_{2(\alpha)2(\alpha)}-\partial_{2(\alpha)}\Gamma^{(n+1)(\alpha)}_{1(\alpha)2(\alpha)}\notag\\&+\overset{n+1}{\underset{l=2}{\sum}}\,\bigg(\Gamma^{(n+1)(\alpha)}_{1(\alpha)l(\alpha)}\Gamma^{l(\alpha)}_{2(\alpha)2(\alpha)}-\Gamma^{(n+1)(\alpha)}_{2(\alpha)l(\alpha)}\Gamma^{l(\alpha)}_{1(\alpha)2(\alpha)}\bigg)\notag\\&+\overset{}{\underset{\sigma\neq\alpha}{\sum}}\,\bigg(\Gamma^{(n+1)(\alpha)}_{1(\alpha)1(\sigma)}\xcancel{\Gamma^{1(\sigma)}_{2(\alpha)2(\alpha)}}-\Gamma^{(n+1)(\alpha)}_{2(\alpha)1(\sigma)}\xcancel{\Gamma^{1(\sigma)}_{1(\alpha)2(\alpha)}}\bigg)
		\notag\\&=\partial_{1(\alpha)}\Gamma^{(n+1)(\alpha)}_{2(\alpha)2(\alpha)}-\partial_{2(\alpha)}\Gamma^{n(\alpha)}_{1(\alpha)1(\alpha)}\notag\\&-\underset{\sigma\neq\alpha}{\sum}\,\overset{n+1}{\underset{l=2}{\sum}}\,\Gamma^{(n-l+2)(\alpha)}_{1(\sigma)1(\alpha)}\Gamma^{l(\alpha)}_{2(\alpha)2(\alpha)}-\overset{n+1}{\underset{l=2}{\sum}}\,\Gamma^{(n-l+3)(\alpha)}_{2(\alpha)2(\alpha)}\Gamma^{l(\alpha)}_{1(\alpha)2(\alpha)}
		\notag
	\end{align}
	where
	\begin{align}
		-\overset{n+1}{\underset{l=2}{\sum}}\,\Gamma^{(n-l+3)(\alpha)}_{2(\alpha)2(\alpha)}\Gamma^{l(\alpha)}_{1(\alpha)2(\alpha)}&=\underset{\sigma\neq\alpha}{\sum}\,\overset{n+1}{\underset{l=2}{\sum}}\,\Gamma^{(n-l+3)(\alpha)}_{2(\alpha)2(\alpha)}\Gamma^{(l-1)(\alpha)}_{1(\sigma)1(\alpha)}
		\notag\\&\overset{t:=n-l+3}{=}\underset{\sigma\neq\alpha}{\sum}\,\overset{n+1}{\underset{t=2}{\sum}}\,\Gamma^{t(\alpha)}_{2(\alpha)2(\alpha)}\Gamma^{(n-t+2)(\alpha)}_{1(\sigma)1(\alpha)}.
		\notag
	\end{align}
	Thus
	\begin{align}
		(R^{n+1})^{(n+1)(\alpha)}_{2(\alpha)1(\alpha)2(\alpha)}&=\partial_{1(\alpha)}\Gamma^{(n+1)(\alpha)}_{2(\alpha)2(\alpha)}-\partial_{2(\alpha)}\Gamma^{n(\alpha)}_{1(\alpha)1(\alpha)}\notag\\&-\cancel{\underset{\sigma\neq\alpha}{\sum}\,\overset{n+1}{\underset{l=2}{\sum}}\,\Gamma^{(n-l+2)(\alpha)}_{1(\sigma)1(\alpha)}\Gamma^{l(\alpha)}_{2(\alpha)2(\alpha)}}+\cancel{\underset{\sigma\neq\alpha}{\sum}\,\overset{n+1}{\underset{t=2}{\sum}}\,\Gamma^{t(\alpha)}_{2(\alpha)2(\alpha)}\Gamma^{(n-t+2)(\alpha)}_{1(\sigma)1(\alpha)}}
		\notag\\&=\partial_{1(\alpha)}\big[\underline{\Gamma^{(n-1)(\alpha)}_{1(\alpha)1(\alpha)}}+\big(\Gamma^{(n+1)(\alpha)}_{2(\alpha)2(\alpha)}\underline{-\Gamma^{(n-1)(\alpha)}_{1(\alpha)1(\alpha)}}\big)\big]-\partial_{2(\alpha)}\Gamma^{n(\alpha)}_{1(\alpha)1(\alpha)}
		\notag\\&=\partial_{1(\alpha)}\Gamma^{(n-1)(\alpha)}_{1(\alpha)1(\alpha)}+\partial_{1(\alpha)}\big(\Gamma^{(n+1)(\alpha)}_{2(\alpha)2(\alpha)}-\Gamma^{(n-1)(\alpha)}_{1(\alpha)1(\alpha)}\big)-\partial_{2(\alpha)}\Gamma^{n(\alpha)}_{1(\alpha)1(\alpha)}
		\notag\\&=\partial_{1(\alpha)}\Gamma^{(n-1)(\alpha)}_{1(\alpha)1(\alpha)}+\partial_{1(\alpha)}\big(\Gamma^{(n+1)(\alpha)}_{2(\alpha)2(\alpha)}-\Gamma^{(n-1)(\alpha)}_{1(\alpha)1(\alpha)}\big)-\partial_{2(\alpha)}\Gamma^{n(\alpha)}_{1(\alpha)1(\alpha)}
		\notag
	\end{align}
	where the underlined terms cancel out and where
	\begin{align}
		\partial_{2(\alpha)}\Gamma^{n(\alpha)}_{1(\alpha)1(\alpha)}&=-\underset{\sigma\neq\alpha}{\sum}\,\partial_{2(\alpha)}\Gamma^{n(\alpha)}_{1(\sigma)1(\alpha)}\overset{\eqref{Lemma7.1eq}}{=}-\underset{\sigma\neq\alpha}{\sum}\,\partial_{1(\alpha)}\Gamma^{(n-1)(\alpha)}_{1(\sigma)1(\alpha)}=\partial_{1(\alpha)}\Gamma^{(n-1)(\alpha)}_{1(\alpha)1(\alpha)}.
		\notag
	\end{align}
	This implies
	\begin{align}
		(R^{n+1})^{(n+1)(\alpha)}_{2(\alpha)1(\alpha)2(\alpha)}&=\cancel{\partial_{1(\alpha)}\Gamma^{(n-1)(\alpha)}_{1(\alpha)1(\alpha)}}+\partial_{1(\alpha)}\big(\Gamma^{(n+1)(\alpha)}_{2(\alpha)2(\alpha)}-\Gamma^{(n-1)(\alpha)}_{1(\alpha)1(\alpha)}\big)-\cancel{\partial_{2(\alpha)}\Gamma^{n(\alpha)}_{1(\alpha)1(\alpha)}}
		\notag\\&=\partial_{1(\alpha)}\big(\Gamma^{(n+1)(\alpha)}_{2(\alpha)2(\alpha)}-\Gamma^{(n-1)(\alpha)}_{1(\alpha)1(\alpha)}\big)=0
		\notag
	\end{align}
	because
	\begin{align}
		\Gamma^{(n+1)(\alpha)}_{2(\alpha)2(\alpha)}-\Gamma^{(n-1)(\alpha)}_{1(\alpha)1(\alpha)}&=-\Gamma^{2(\alpha)}_{2(\alpha)2(\alpha)}\frac{u^{(n+1)(\alpha)}}{u^{2(\alpha)}}-\frac{1}{u^{2(\alpha)}}\overset{n-2}{\underset{l=1}{\sum}}\big(\Gamma^{(l+2)(\alpha)}_{2(\alpha)2(\alpha)}-\Gamma^{l(\alpha)}_{1(\alpha)1(\alpha)}\big)u^{(n-l+1)(\alpha)}
		\notag
	\end{align}
	does not depend on $u^{1(\alpha)}$ (it only depends on $\{u^{1(\sigma)}\,|\,\sigma\neq\alpha\}$ and $\{u^{2(\alpha)}\,|\,2\leq s\leq n\}$). If $j=1$ then \eqref{Rdel_caso1_n+1_in+1_subcased} becomes
	\begin{align}
		(R^{n+1})^{(n+1)(\alpha)}_{2(\alpha)1(\alpha)1(\alpha)}&=\partial_{1(\alpha)}\Gamma^{(n+1)(\alpha)}_{2(\alpha)1(\alpha)}-\partial_{2(\alpha)}\Gamma^{(n+1)(\alpha)}_{1(\alpha)1(\alpha)}\notag\\&+\overset{n+1}{\underset{l=1}{\sum}}\,\bigg(\Gamma^{(n+1)(\alpha)}_{1(\alpha)l(\alpha)}\Gamma^{l(\alpha)}_{2(\alpha)1(\alpha)}-\Gamma^{(n+1)(\alpha)}_{2(\alpha)l(\alpha)}\Gamma^{l(\alpha)}_{1(\alpha)1(\alpha)}\bigg)\notag\\&+\overset{}{\underset{\sigma\neq\alpha}{\sum}}\,\bigg(\Gamma^{(n+1)(\alpha)}_{1(\alpha)1(\sigma)}\xcancel{\Gamma^{1(\sigma)}_{2(\alpha)1(\alpha)}}-\Gamma^{(n+1)(\alpha)}_{2(\alpha)1(\sigma)}\Gamma^{1(\sigma)}_{1(\alpha)1(\alpha)}\bigg)
		\notag
	\end{align}
	where
	\begin{align}
		\partial_{1(\alpha)}\Gamma^{(n+1)(\alpha)}_{2(\alpha)1(\alpha)}&=\partial_{1(\alpha)}\Gamma^{n(\alpha)}_{1(\alpha)1(\alpha)}
		\notag
	\end{align}
	and
	\begin{align}
		\partial_{2(\alpha)}\Gamma^{(n+1)(\alpha)}_{1(\alpha)1(\alpha)}&=-\underset{\sigma\neq\alpha}{\sum}\,\partial_{2(\alpha)}\Gamma^{(n+1)(\alpha)}_{1(\sigma)1(\alpha)}
		\overset{\eqref{Lemma7.1eq}}{=}-\underset{\sigma\neq\alpha}{\sum}\,\partial_{1(\alpha)}\Gamma^{n(\alpha)}_{1(\sigma)1(\alpha)}
		=\partial_{1(\alpha)}\Gamma^{n(\alpha)}_{1(\alpha)1(\alpha)}
		\notag
	\end{align}
	mutually cancel out. This yields
	\begin{align}
		(R^{n+1})^{(n+1)(\alpha)}_{2(\alpha)1(\alpha)1(\alpha)}&=\overset{n+1}{\underset{l=1}{\sum}}\,\bigg(\Gamma^{(n+1)(\alpha)}_{1(\alpha)l(\alpha)}\Gamma^{l(\alpha)}_{2(\alpha)1(\alpha)}-\Gamma^{(n+1)(\alpha)}_{2(\alpha)l(\alpha)}\Gamma^{l(\alpha)}_{1(\alpha)1(\alpha)}\bigg)\notag\\&-\overset{}{\underset{\sigma\neq\alpha}{\sum}}\,\Gamma^{(n+1)(\alpha)}_{2(\alpha)1(\sigma)}\Gamma^{1(\sigma)}_{1(\alpha)1(\alpha)}
		\notag\\&=\Gamma^{(n+1)(\alpha)}_{1(\alpha)1(\alpha)}\xcancel{\Gamma^{1(\alpha)}_{2(\alpha)1(\alpha)}}-\Gamma^{(n+1)(\alpha)}_{2(\alpha)1(\alpha)}\Gamma^{1(\alpha)}_{1(\alpha)1(\alpha)}\notag\\&+\overset{n+1}{\underset{l=2}{\sum}}\,\Gamma^{n(\alpha)}_{1(\alpha)(l-1)(\alpha)}\Gamma^{(l-1)(\alpha)}_{1(\alpha)1(\alpha)}-\overset{n+1}{\underset{l=2}{\sum}}\,\Gamma^{(n-l+3)(\alpha)}_{2(\alpha)2(\alpha)}\Gamma^{l(\alpha)}_{1(\alpha)1(\alpha)}\notag\\&-\overset{}{\underset{\sigma\neq\alpha}{\sum}}\,\Gamma^{n(\alpha)}_{1(\alpha)1(\sigma)}\Gamma^{1(\sigma)}_{1(\alpha)1(\alpha)}
		\notag\\&=-\Gamma^{(n+1)(\alpha)}_{2(\alpha)1(\alpha)}\Gamma^{1(\alpha)}_{1(\alpha)1(\alpha)}+\overset{n}{\underset{l=1}{\sum}}\,\Gamma^{n(\alpha)}_{1(\alpha)l(\alpha)}\Gamma^{l(\alpha)}_{1(\alpha)1(\alpha)}\notag\\&-\overset{n+1}{\underset{l=2}{\sum}}\,\Gamma^{(n-l+3)(\alpha)}_{2(\alpha)2(\alpha)}\Gamma^{l(\alpha)}_{1(\alpha)1(\alpha)}-\overset{}{\underset{\sigma\neq\alpha}{\sum}}\,\Gamma^{n(\alpha)}_{1(\alpha)1(\sigma)}\Gamma^{1(\sigma)}_{1(\alpha)1(\alpha)}
		\notag\\&=-\cancel{\Gamma^{n(\alpha)}_{1(\alpha)1(\alpha)}\Gamma^{1(\alpha)}_{1(\alpha)1(\alpha)}}+\cancel{\Gamma^{n(\alpha)}_{1(\alpha)1(\alpha)}\Gamma^{1(\alpha)}_{1(\alpha)1(\alpha)}}+\overset{n}{\underset{l=2}{\sum}}\,\Gamma^{n(\alpha)}_{1(\alpha)l(\alpha)}\Gamma^{l(\alpha)}_{1(\alpha)1(\alpha)}\notag\\&-\overset{n}{\underset{l=2}{\sum}}\,\Gamma^{(n-l+3)(\alpha)}_{2(\alpha)2(\alpha)}\Gamma^{l(\alpha)}_{1(\alpha)1(\alpha)}-\Gamma^{2(\alpha)}_{2(\alpha)2(\alpha)}\Gamma^{(n+1)(\alpha)}_{1(\alpha)1(\alpha)}-\overset{}{\underset{\sigma\neq\alpha}{\sum}}\,\Gamma^{n(\alpha)}_{1(\alpha)1(\sigma)}\Gamma^{1(\sigma)}_{1(\alpha)1(\alpha)}
		\notag\\&=\overset{n}{\underset{l=2}{\sum}}\,\big(\Gamma^{n(\alpha)}_{1(\alpha)l(\alpha)}-\Gamma^{(n-l+3)(\alpha)}_{2(\alpha)2(\alpha)}\big)\Gamma^{l(\alpha)}_{1(\alpha)1(\alpha)}\notag\\&-\Gamma^{2(\alpha)}_{2(\alpha)2(\alpha)}\Gamma^{(n+1)(\alpha)}_{1(\alpha)1(\alpha)}+\overset{}{\underset{\sigma\neq\alpha}{\sum}}\,\Gamma^{n(\alpha)}_{1(\alpha)1(\sigma)}\Gamma^{1(\sigma)}_{1(\alpha)1(\sigma)}
		\notag\\&=\overset{}{\underset{\sigma\neq\alpha}{\sum}}\big[\Gamma^{n(\alpha)}_{1(\alpha)1(\sigma)}\Gamma^{1(\sigma)}_{1(\alpha)1(\sigma)}+\Gamma^{2(\alpha)}_{2(\alpha)2(\alpha)}\Gamma^{(n+1)(\alpha)}_{1(\sigma)1(\alpha)}\notag\\&+\overset{n}{\underset{l=2}{\sum}}\,\big(\Gamma^{(n-l+3)(\alpha)}_{2(\alpha)2(\alpha)}-\Gamma^{(n-l+1)(\alpha)}_{1(\alpha)1(\alpha)}\big)\Gamma^{l(\alpha)}_{1(\sigma)1(\alpha)}\big]=0
		\notag
	\end{align}
	because for each $\sigma\neq\alpha$ we have
	\begin{align}
		&\Gamma^{n(\alpha)}_{1(\alpha)1(\sigma)}\Gamma^{1(\sigma)}_{1(\alpha)1(\sigma)}+\Gamma^{2(\alpha)}_{2(\alpha)2(\alpha)}\Gamma^{(n+1)(\alpha)}_{1(\sigma)1(\alpha)}+\overset{n}{\underset{l=2}{\sum}}\,\big(\Gamma^{(n-l+3)(\alpha)}_{2(\alpha)2(\alpha)}-\Gamma^{(n-l+1)(\alpha)}_{1(\alpha)1(\alpha)}\big)\Gamma^{l(\alpha)}_{1(\sigma)1(\alpha)}
		\notag\\&=\Gamma^{n(\alpha)}_{1(\alpha)1(\sigma)}\bigg[-\frac{m_\alpha\epsilon_\alpha}{u^{1(\alpha)}-u^{1(\sigma)}}\bigg]+\Gamma^{2(\alpha)}_{2(\alpha)2(\alpha)}\bigg[-\frac{1}{u^{1(\alpha)}-u^{1(\sigma)}}\,\overset{n+1}{\underset{s=2}{\sum}}\,\Gamma^{(n-s+2)(\alpha)}_{1(\sigma)1(\alpha)}\,u^{s(\alpha)}\bigg]\notag\\&+\big(\Gamma^{(n+1)(\alpha)}_{2(\alpha)2(\alpha)}-\Gamma^{(n-1)(\alpha)}_{1(\alpha)1(\alpha)}\big)\Gamma^{2(\alpha)}_{1(\sigma)1(\alpha)}+\overset{n}{\underset{l=3}{\sum}}\,\big(\Gamma^{(n-l+3)(\alpha)}_{2(\alpha)2(\alpha)}-\Gamma^{(n-l+1)(\alpha)}_{1(\alpha)1(\alpha)}\big)\Gamma^{l(\alpha)}_{1(\sigma)1(\alpha)}
		\notag\\&=\cancel{\Gamma^{n(\alpha)}_{1(\alpha)1(\sigma)}\bigg[\frac{u^{2(\alpha)}\,\Gamma^{2(\alpha)}_{2(\alpha)2(\alpha)}}{u^{1(\alpha)}-u^{1(\sigma)}}\bigg]}+\Gamma^{2(\alpha)}_{2(\alpha)2(\alpha)}\bigg[\cancel{-\frac{1}{u^{1(\alpha)}-u^{1(\sigma)}}\,\Gamma^{n(\alpha)}_{1(\sigma)1(\alpha)}\,u^{2(\alpha)}}\notag\\&-\frac{1}{u^{1(\alpha)}-u^{1(\sigma)}}\,\overset{n}{\underset{s=3}{\sum}}\,\Gamma^{(n-s+2)(\alpha)}_{1(\sigma)1(\alpha)}\,u^{s(\alpha)}-\frac{1}{u^{1(\alpha)}-u^{1(\sigma)}}\,\Gamma^{1(\alpha)}_{1(\sigma)1(\alpha)}\,u^{(n+1)(\alpha)}\bigg]\notag\\&+\big(\Gamma^{(n+1)(\alpha)}_{2(\alpha)2(\alpha)}-\Gamma^{(n-1)(\alpha)}_{1(\alpha)1(\alpha)}\big)\Gamma^{2(\alpha)}_{1(\sigma)1(\alpha)}+\overset{n}{\underset{l=3}{\sum}}\,\big(\Gamma^{(n-l+3)(\alpha)}_{2(\alpha)2(\alpha)}-\Gamma^{(n-l+1)(\alpha)}_{1(\alpha)1(\alpha)}\big)\Gamma^{l(\alpha)}_{1(\sigma)1(\alpha)}
		\notag\\&=-\frac{\Gamma^{2(\alpha)}_{2(\alpha)2(\alpha)}}{u^{1(\alpha)}-u^{1(\sigma)}}\,\overset{n}{\underset{s=3}{\sum}}\,\Gamma^{(n-s+2)(\alpha)}_{1(\sigma)1(\alpha)}\,u^{s(\alpha)}-\cancel{\frac{\Gamma^{2(\alpha)}_{2(\alpha)2(\alpha)}}{u^{1(\alpha)}-u^{1(\sigma)}}\,\Gamma^{1(\alpha)}_{1(\sigma)1(\alpha)}\,u^{(n+1)(\alpha)}}\notag\\&+\bigg(\cancel{-\Gamma^{2(\alpha)}_{2(\alpha)2(\alpha)}\frac{u^{(n+1)(\alpha)}}{u^{2(\alpha)}}}-\frac{1}{u^{2(\alpha)}}\overset{n-2}{\underset{l=1}{\sum}}\big(\Gamma^{(l+2)(\alpha)}_{2(\alpha)2(\alpha)}-\Gamma^{l(\alpha)}_{1(\alpha)1(\alpha)}\big)u^{(n-l+1)(\alpha)}\bigg)\bigg[-\frac{\Gamma^{1(\alpha)}_{1(\sigma)1(\alpha)}\,u^{2(\alpha)}}{u^{1(\alpha)}-u^{1(\sigma)}}\bigg]\notag\\&+\overset{n-2}{\underset{s=1}{\sum}}\,\big(\Gamma^{(s+2)(\alpha)}_{2(\alpha)2(\alpha)}-\Gamma^{s(\alpha)}_{1(\alpha)1(\alpha)}\big)\Gamma^{(n-s+1)(\alpha)}_{1(\sigma)1(\alpha)}
		\notag\\&=-\frac{\Gamma^{2(\alpha)}_{2(\alpha)2(\alpha)}}{u^{1(\alpha)}-u^{1(\sigma)}}\,\overset{n-1}{\underset{s=3}{\sum}}\,\Gamma^{(n-s+2)(\alpha)}_{1(\sigma)1(\alpha)}\,u^{s(\alpha)}-\frac{\Gamma^{2(\alpha)}_{2(\alpha)2(\alpha)}}{u^{1(\alpha)}-u^{1(\sigma)}}\,\Gamma^{2(\alpha)}_{1(\sigma)1(\alpha)}\,u^{n(\alpha)}\notag\\&+\overset{n-2}{\underset{l=1}{\sum}}\big(\Gamma^{(l+2)(\alpha)}_{2(\alpha)2(\alpha)}-\Gamma^{l(\alpha)}_{1(\alpha)1(\alpha)}\big)u^{(n-l+1)(\alpha)}\bigg[\frac{\Gamma^{1(\alpha)}_{1(\sigma)1(\alpha)}}{u^{1(\alpha)}-u^{1(\sigma)}}\bigg]\notag\\&+\overset{n-2}{\underset{s=1}{\sum}}\,\big(\Gamma^{(s+2)(\alpha)}_{2(\alpha)2(\alpha)}-\Gamma^{s(\alpha)}_{1(\alpha)1(\alpha)}\big)\Gamma^{(n-s+1)(\alpha)}_{1(\sigma)1(\alpha)}
		\notag\\&=-\frac{\Gamma^{2(\alpha)}_{2(\alpha)2(\alpha)}}{u^{1(\alpha)}-u^{1(\sigma)}}\,\overset{n-1}{\underset{s=3}{\sum}}\,\Gamma^{(n-s+2)(\alpha)}_{1(\sigma)1(\alpha)}\,u^{s(\alpha)}-\frac{\Gamma^{2(\alpha)}_{2(\alpha)2(\alpha)}}{u^{1(\alpha)}-u^{1(\sigma)}}\,\Gamma^{2(\alpha)}_{1(\sigma)1(\alpha)}\,u^{n(\alpha)}\notag\\&+\overset{n-3}{\underset{s=1}{\sum}}\,\big(\Gamma^{(s+2)(\alpha)}_{2(\alpha)2(\alpha)}-\Gamma^{s(\alpha)}_{1(\alpha)1(\alpha)}\big)\bigg[\Gamma^{(n-s+1)(\alpha)}_{1(\sigma)1(\alpha)}+\frac{\Gamma^{1(\alpha)}_{1(\sigma)1(\alpha)}}{u^{1(\alpha)}-u^{1(\sigma)}}\,u^{(n-s+1)(\alpha)}\bigg]\notag\\&+\big(\Gamma^{n(\alpha)}_{2(\alpha)2(\alpha)}-\Gamma^{(n-2)(\alpha)}_{1(\alpha)1(\alpha)}\big)\bigg[\Gamma^{3(\alpha)}_{1(\sigma)1(\alpha)}+\frac{\Gamma^{1(\alpha)}_{1(\sigma)1(\alpha)}}{u^{1(\alpha)}-u^{1(\sigma)}}\,u^{3(\alpha)}\bigg]
		\notag\\&=-\frac{\Gamma^{2(\alpha)}_{2(\alpha)2(\alpha)}}{u^{1(\alpha)}-u^{1(\sigma)}}\,\overset{n-1}{\underset{s=3}{\sum}}\,\Gamma^{(n-s+2)(\alpha)}_{1(\sigma)1(\alpha)}\,u^{s(\alpha)}-\cancel{\frac{\Gamma^{2(\alpha)}_{2(\alpha)2(\alpha)}}{u^{1(\alpha)}-u^{1(\sigma)}}\,\Gamma^{2(\alpha)}_{1(\sigma)1(\alpha)}\,u^{n(\alpha)}}\notag\\&+\overset{n-3}{\underset{s=1}{\sum}}\,\big(\Gamma^{(s+2)(\alpha)}_{2(\alpha)2(\alpha)}-\Gamma^{s(\alpha)}_{1(\alpha)1(\alpha)}\big)\bigg[-\frac{1}{u^{1(\alpha)}-u^{1(\sigma)}}\overset{n-s}{\underset{l=2}{\sum}}\,\Gamma^{(n-s-l+2)(\alpha)}_{1(\sigma)1(\alpha)}\,u^{l(\alpha)}\bigg]\notag\\&+\bigg(-\cancel{\Gamma^{2(\alpha)}_{2(\alpha)2(\alpha)}\frac{u^{n(\alpha)}}{u^{2(\alpha)}}}-\frac{1}{u^{2(\alpha)}}\overset{n-3}{\underset{l=1}{\sum}}\big(\Gamma^{(l+2)(\alpha)}_{2(\alpha)2(\alpha)}-\Gamma^{l(\alpha)}_{1(\alpha)1(\alpha)}\big)u^{(n-l)(\alpha)}\bigg)\bigg[-\frac{\Gamma^{2(\alpha)}_{1(\sigma)1(\alpha)}}{u^{1(\alpha)}-u^{1(\sigma)}}\,u^{2(\alpha)}\bigg]
		\notag\\&=-\frac{\Gamma^{2(\alpha)}_{2(\alpha)2(\alpha)}}{u^{1(\alpha)}-u^{1(\sigma)}}\,\overset{n-1}{\underset{s=3}{\sum}}\,\Gamma^{(n-s+2)(\alpha)}_{1(\sigma)1(\alpha)}\,u^{s(\alpha)}\notag\\&+\overset{n-3}{\underset{s=1}{\sum}}\,\big(\Gamma^{(s+2)(\alpha)}_{2(\alpha)2(\alpha)}-\Gamma^{s(\alpha)}_{1(\alpha)1(\alpha)}\big)\bigg[-\frac{1}{u^{1(\alpha)}-u^{1(\sigma)}}\overset{n-s-1}{\underset{l=2}{\sum}}\,\Gamma^{(n-s-l+2)(\alpha)}_{1(\sigma)1(\alpha)}\,u^{l(\alpha)}-\cancel{\frac{\Gamma^{2(\alpha)}_{1(\sigma)1(\alpha)}\,u^{(n-s)(\alpha)}}{u^{1(\alpha)}-u^{1(\sigma)}}}\bigg]\notag\\&+\bigg(\overset{n-3}{\underset{l=1}{\sum}}\big(\Gamma^{(l+2)(\alpha)}_{2(\alpha)2(\alpha)}-\Gamma^{l(\alpha)}_{1(\alpha)1(\alpha)}\big)u^{(n-l)(\alpha)}\bigg)\bigg[\cancel{\frac{\Gamma^{2(\alpha)}_{1(\sigma)1(\alpha)}}{u^{1(\alpha)}-u^{1(\sigma)}}}\bigg]
		\notag\\&=-\frac{\Gamma^{2(\alpha)}_{2(\alpha)2(\alpha)}}{u^{1(\alpha)}-u^{1(\sigma)}}\,\overset{n-1}{\underset{s=3}{\sum}}\,\Gamma^{(n-s+2)(\alpha)}_{1(\sigma)1(\alpha)}\,u^{s(\alpha)}\notag\\&-\frac{1}{u^{1(\alpha)}-u^{1(\sigma)}}\,\overset{n-3}{\underset{s=1}{\sum}}\,\big(\Gamma^{(s+2)(\alpha)}_{2(\alpha)2(\alpha)}-\Gamma^{s(\alpha)}_{1(\alpha)1(\alpha)}\big)\,\overset{n-s-1}{\underset{l=2}{\sum}}\,\Gamma^{(n-s-l+2)(\alpha)}_{1(\sigma)1(\alpha)}\,u^{l(\alpha)}
		\notag\\&=-\frac{\Gamma^{2(\alpha)}_{2(\alpha)2(\alpha)}}{u^{1(\alpha)}-u^{1(\sigma)}}\,\Gamma^{(n-1)(\alpha)}_{1(\sigma)1(\alpha)}\,u^{3(\alpha)}-\frac{\Gamma^{2(\alpha)}_{2(\alpha)2(\alpha)}}{u^{1(\alpha)}-u^{1(\sigma)}}\,\overset{n-1}{\underset{s=4}{\sum}}\,\Gamma^{(n-s+2)(\alpha)}_{1(\sigma)1(\alpha)}\,u^{s(\alpha)}\notag\\&-\frac{1}{u^{1(\alpha)}-u^{1(\sigma)}}\,\big(\Gamma^{3(\alpha)}_{2(\alpha)2(\alpha)}-\Gamma^{1(\alpha)}_{1(\alpha)1(\alpha)}\big)\,\overset{n-2}{\underset{l=2}{\sum}}\,\Gamma^{(n-l+1)(\alpha)}_{1(\sigma)1(\alpha)}\,u^{l(\alpha)}\notag\\&-\frac{1}{u^{1(\alpha)}-u^{1(\sigma)}}\,\overset{n-3}{\underset{s=2}{\sum}}\,\big(\Gamma^{(s+2)(\alpha)}_{2(\alpha)2(\alpha)}-\Gamma^{s(\alpha)}_{1(\alpha)1(\alpha)}\big)\,\overset{n-s-1}{\underset{l=2}{\sum}}\,\Gamma^{(n-s-l+2)(\alpha)}_{1(\sigma)1(\alpha)}\,u^{l(\alpha)}
		\notag\\&=-\cancel{\frac{\Gamma^{2(\alpha)}_{2(\alpha)2(\alpha)}}{u^{1(\alpha)}-u^{1(\sigma)}}\,\Gamma^{(n-1)(\alpha)}_{1(\sigma)1(\alpha)}\,u^{3(\alpha)}}-\frac{\Gamma^{2(\alpha)}_{2(\alpha)2(\alpha)}}{u^{1(\alpha)}-u^{1(\sigma)}}\,\overset{n-2}{\underset{l=3}{\sum}}\,\Gamma^{(n-l+1)(\alpha)}_{1(\sigma)1(\alpha)}\,u^{(l+1)(\alpha)}\notag\\&+\frac{1}{u^{1(\alpha)}-u^{1(\sigma)}}\,\bigg(\Gamma^{2(\alpha)}_{2(\alpha)2(\alpha)}\,\frac{u^{3(\alpha)}}{u^{2(\alpha)}}\bigg)\,\bigg[\cancel{\Gamma^{(n-1)(\alpha)}_{1(\sigma)1(\alpha)}\,u^{2(\alpha)}}+\overset{n-2}{\underset{l=3}{\sum}}\,\Gamma^{(n-l+1)(\alpha)}_{1(\sigma)1(\alpha)}\,u^{l(\alpha)}\bigg]\notag\\&-\frac{1}{u^{1(\alpha)}-u^{1(\sigma)}}\,\overset{n-3}{\underset{s=2}{\sum}}\,\big(\Gamma^{(s+2)(\alpha)}_{2(\alpha)2(\alpha)}-\Gamma^{s(\alpha)}_{1(\alpha)1(\alpha)}\big)\,\overset{n-s-1}{\underset{l=2}{\sum}}\,\Gamma^{(n-s-l+2)(\alpha)}_{1(\sigma)1(\alpha)}\,u^{l(\alpha)}
		\notag\\&=\frac{\Gamma^{2(\alpha)}_{2(\alpha)2(\alpha)}}{u^{1(\alpha)}-u^{1(\sigma)}}\,\overset{n-2}{\underset{l=3}{\sum}}\,\Gamma^{(n-l+1)(\alpha)}_{1(\sigma)1(\alpha)}\,\bigg(\frac{u^{3(\alpha)}}{u^{2(\alpha)}}\,u^{l(\alpha)}-u^{(l+1)(\alpha)}\bigg)\notag\\&-\frac{1}{u^{1(\alpha)}-u^{1(\sigma)}}\,\overset{n-3}{\underset{s=2}{\sum}}\,\big(\Gamma^{(s+2)(\alpha)}_{2(\alpha)2(\alpha)}-\Gamma^{s(\alpha)}_{1(\alpha)1(\alpha)}\big)\,\overset{n-s-1}{\underset{l=2}{\sum}}\,\Gamma^{(n-s-l+2)(\alpha)}_{1(\sigma)1(\alpha)}\,u^{l(\alpha)}
		\notag\\&=\frac{\Gamma^{2(\alpha)}_{2(\alpha)2(\alpha)}}{u^{1(\alpha)}-u^{1(\sigma)}}\,\overset{n-2}{\underset{l=3}{\sum}}\,\Gamma^{(n-l+1)(\alpha)}_{1(\sigma)1(\alpha)}\,\bigg(\frac{u^{3(\alpha)}}{u^{2(\alpha)}}\,u^{l(\alpha)}-u^{(l+1)(\alpha)}\bigg)\notag\\&-\frac{1}{u^{1(\alpha)}-u^{1(\sigma)}}\,\overset{n-3}{\underset{s=2}{\sum}}\,\big(\Gamma^{(s+2)(\alpha)}_{2(\alpha)2(\alpha)}-\Gamma^{s(\alpha)}_{1(\alpha)1(\alpha)}\big)\,\overset{n-2}{\underset{t=s+1}{\sum}}\,\Gamma^{(n-t+1)(\alpha)}_{1(\sigma)1(\alpha)}\,u^{(t-s+1)(\alpha)}
		\notag\\&=\frac{\Gamma^{2(\alpha)}_{2(\alpha)2(\alpha)}}{u^{1(\alpha)}-u^{1(\sigma)}}\,\overset{n-2}{\underset{l=3}{\sum}}\,\Gamma^{(n-l+1)(\alpha)}_{1(\sigma)1(\alpha)}\,\bigg(\frac{u^{3(\alpha)}}{u^{2(\alpha)}}\,u^{l(\alpha)}-u^{(l+1)(\alpha)}\bigg)\notag\\&-\frac{1}{u^{1(\alpha)}-u^{1(\sigma)}}\,\overset{n-2}{\underset{t=3}{\sum}}\,\overset{t-1}{\underset{s=2}{\sum}}\,\big(\Gamma^{(s+2)(\alpha)}_{2(\alpha)2(\alpha)}-\Gamma^{s(\alpha)}_{1(\alpha)1(\alpha)}\big)\,\Gamma^{(n-t+1)(\alpha)}_{1(\sigma)1(\alpha)}\,u^{(t-s+1)(\alpha)}
		\notag\\&=\frac{\Gamma^{2(\alpha)}_{2(\alpha)2(\alpha)}}{u^{1(\alpha)}-u^{1(\sigma)}}\,\overset{n-2}{\underset{l=3}{\sum}}\,\Gamma^{(n-l+1)(\alpha)}_{1(\sigma)1(\alpha)}\,\bigg(\frac{u^{3(\alpha)}}{u^{2(\alpha)}}\,u^{l(\alpha)}-u^{(l+1)(\alpha)}\bigg)\notag\\&-\frac{1}{u^{1(\alpha)}-u^{1(\sigma)}}\,\overset{n-2}{\underset{l=3}{\sum}}\,\overset{l-1}{\underset{s=2}{\sum}}\,\big(\Gamma^{(s+2)(\alpha)}_{2(\alpha)2(\alpha)}-\Gamma^{s(\alpha)}_{1(\alpha)1(\alpha)}\big)\,\Gamma^{(n-l+1)(\alpha)}_{1(\sigma)1(\alpha)}\,u^{(l-s+1)(\alpha)}
		\notag\\&=\frac{1}{u^{1(\alpha)}-u^{1(\sigma)}}\,\overset{n-2}{\underset{l=3}{\sum}}\,\Gamma^{(n-l+1)(\alpha)}_{1(\sigma)1(\alpha)}\,\bigg[\Gamma^{2(\alpha)}_{2(\alpha)2(\alpha)}\,\bigg(\frac{u^{3(\alpha)}}{u^{2(\alpha)}}\,u^{l(\alpha)}-u^{(l+1)(\alpha)}\bigg)\notag\\&-\overset{l-1}{\underset{s=2}{\sum}}\,\big(\Gamma^{(s+2)(\alpha)}_{2(\alpha)2(\alpha)}-\Gamma^{s(\alpha)}_{1(\alpha)1(\alpha)}\big)\,u^{(l-s+1)(\alpha)}\bigg]
		\overset{\eqref{atlas}}{=}0.
		\notag
	\end{align}
	\textbf{Case 2: $\alpha=\beta=\gamma\neq\epsilon$.} Our goal is to prove that
	\begin{align}
		R^{i(\alpha)}_{h(\epsilon)k(\alpha)j(\alpha)}&=\partial_{k(\alpha)}\Gamma^{i(\alpha)}_{h(\epsilon)j(\alpha)}-\partial_{h(\epsilon)}\Gamma^{i(\alpha)}_{k(\alpha)j(\alpha)}\notag\\&+\overset{r}{\underset{\sigma=1}{\sum}}\,\overset{m_\sigma}{\underset{l=1}{\sum}}\,\bigg(\Gamma^{i(\alpha)}_{k(\alpha)l(\sigma)}\Gamma^{l(\sigma)}_{h(\epsilon)j(\alpha)}-\Gamma^{i(\alpha)}_{h(\epsilon)l(\sigma)}\Gamma^{l(\sigma)}_{k(\alpha)j(\alpha)}\bigg)
		\label{Rdef_caso2}
	\end{align}
	vanishes. Let us first consider the case where $h\geq2$. We have
	\begin{align}
		R^{i(\alpha)}_{h(\epsilon)k(\alpha)j(\alpha)}&=\partial_{k(\alpha)}\xcancel{\Gamma^{i(\alpha)}_{h(\epsilon)j(\alpha)}}-\partial_{h(\epsilon)}\Gamma^{i(\alpha)}_{k(\alpha)j(\alpha)}\notag\\&+\overset{r}{\underset{\sigma=1}{\sum}}\,\overset{m_\sigma}{\underset{l=1}{\sum}}\,\bigg(\Gamma^{i(\alpha)}_{k(\alpha)l(\sigma)}\Gamma^{l(\sigma)}_{h(\epsilon)j(\alpha)}-\Gamma^{i(\alpha)}_{h(\epsilon)l(\sigma)}\Gamma^{l(\sigma)}_{k(\alpha)j(\alpha)}\bigg)
		\notag
	\end{align}
	where $\Gamma^{i(\alpha)}_{k(\alpha)j(\alpha)}$ only depends on $\{u^{s(\alpha)}\,|\,1\leq s\leq m_\alpha\}$ and $\{u^{1(\sigma)}\,|\,\sigma\neq\alpha\}$ (thus it does not depend on $u^{h(\epsilon)}$ as $h\geq2$, that is $\partial_{h(\epsilon)}\Gamma^{i(\alpha)}_{k(\alpha)j(\alpha)}=0$). This yields
	\begin{align}
		R^{i(\alpha)}_{h(\epsilon)k(\alpha)j(\alpha)}&=\overset{r}{\underset{\sigma=1}{\sum}}\,\overset{m_\sigma}{\underset{l=1}{\sum}}\,\bigg(\Gamma^{i(\alpha)}_{k(\alpha)l(\sigma)}\Gamma^{l(\sigma)}_{h(\epsilon)j(\alpha)}-\Gamma^{i(\alpha)}_{h(\epsilon)l(\sigma)}\Gamma^{l(\sigma)}_{k(\alpha)j(\alpha)}\bigg)
		\notag
	\end{align}
	where the terms $\Gamma^{l(\sigma)}_{h(\epsilon)j(\alpha)}$ and $\Gamma^{i(\alpha)}_{h(\epsilon)l(\sigma)}$ trivially vanishes for $\sigma\notin\{\alpha,\epsilon\}$. Thus
	\begin{align}
		R^{i(\alpha)}_{h(\epsilon)k(\alpha)j(\alpha)}&=\overset{m_\alpha}{\underset{l=1}{\sum}}\,\bigg(\Gamma^{i(\alpha)}_{k(\alpha)l(\alpha)}\xcancel{\Gamma^{l(\alpha)}_{h(\epsilon)j(\alpha)}}-\xcancel{\Gamma^{i(\alpha)}_{h(\epsilon)l(\alpha)}}\Gamma^{l(\alpha)}_{k(\alpha)j(\alpha)}\bigg)\notag\\&+\overset{m_\epsilon}{\underset{l=1}{\sum}}\,\bigg(\Gamma^{i(\alpha)}_{k(\alpha)l(\epsilon)}\Gamma^{l(\epsilon)}_{h(\epsilon)j(\alpha)}-\xcancel{\Gamma^{i(\alpha)}_{h(\epsilon)l(\epsilon)}}\Gamma^{l(\epsilon)}_{k(\alpha)j(\alpha)}\bigg)
		\notag\\&=\overset{m_\epsilon}{\underset{l=h}{\sum}}\,\Gamma^{i(\alpha)}_{k(\alpha)l(\epsilon)}\Gamma^{l(\epsilon)}_{h(\epsilon)j(\alpha)}=0
		\notag
	\end{align}
	as $\Gamma^{i(\alpha)}_{k(\alpha)l(\epsilon)}=0$ for every $l\geq h$ ($h\geq2$ implies $l\geq 2$). Let us now fix $h=1$. We have
	\begin{align}
		R^{i(\alpha)}_{1(\epsilon)k(\alpha)j(\alpha)}&=\partial_{k(\alpha)}\Gamma^{i(\alpha)}_{1(\epsilon)j(\alpha)}-\partial_{1(\epsilon)}\Gamma^{i(\alpha)}_{k(\alpha)j(\alpha)}\notag\\&+\overset{r}{\underset{\sigma=1}{\sum}}\,\overset{m_\sigma}{\underset{l=1}{\sum}}\,\bigg(\Gamma^{i(\alpha)}_{k(\alpha)l(\sigma)}\Gamma^{l(\sigma)}_{1(\epsilon)j(\alpha)}-\Gamma^{i(\alpha)}_{1(\epsilon)l(\sigma)}\Gamma^{l(\sigma)}_{k(\alpha)j(\alpha)}\bigg)
		\notag\\&=\partial_{k(\alpha)}\Gamma^{i(\alpha)}_{1(\epsilon)j(\alpha)}-\partial_{1(\epsilon)}\Gamma^{i(\alpha)}_{k(\alpha)j(\alpha)}\notag\\&+\overset{m_\alpha}{\underset{l=1}{\sum}}\,\bigg(\Gamma^{i(\alpha)}_{k(\alpha)l(\alpha)}\Gamma^{l(\alpha)}_{1(\epsilon)j(\alpha)}-\Gamma^{i(\alpha)}_{1(\epsilon)l(\alpha)}\Gamma^{l(\alpha)}_{k(\alpha)j(\alpha)}\bigg)\notag\\&+\overset{m_\epsilon}{\underset{l=1}{\sum}}\,\bigg(\Gamma^{i(\alpha)}_{k(\alpha)l(\epsilon)}\Gamma^{l(\epsilon)}_{1(\epsilon)j(\alpha)}-\Gamma^{i(\alpha)}_{1(\epsilon)l(\epsilon)}\Gamma^{l(\epsilon)}_{k(\alpha)j(\alpha)}\bigg)
		\notag\\&=\partial_{k(\alpha)}\Gamma^{i(\alpha)}_{1(\epsilon)j(\alpha)}-\partial_{1(\epsilon)}\Gamma^{i(\alpha)}_{k(\alpha)j(\alpha)}\notag\\&+\overset{m_\alpha}{\underset{l=1}{\sum}}\,\bigg(\Gamma^{i(\alpha)}_{k(\alpha)l(\alpha)}\Gamma^{l(\alpha)}_{1(\epsilon)j(\alpha)}-\Gamma^{i(\alpha)}_{1(\epsilon)l(\alpha)}\Gamma^{l(\alpha)}_{k(\alpha)j(\alpha)}\bigg)\notag\\&+\Gamma^{i(\alpha)}_{k(\alpha)1(\epsilon)}\Gamma^{1(\epsilon)}_{1(\epsilon)j(\alpha)}-\Gamma^{i(\alpha)}_{1(\epsilon)1(\epsilon)}\Gamma^{1(\epsilon)}_{k(\alpha)j(\alpha)}.
		\label{Rdef_caso2_h1}
	\end{align}
	We distinguish between the following subcases:
	\begin{itemize}
		\item[$a.$] both $j$ and $k$ are greater or equal than $2$
		\item[$b.$] $j=k=1$
		\item[$c.$] $j=1$, $k\geq2$
		\item[$d.$] $j\geq2$, $k=1$.
	\end{itemize}
	\textbf{Subcase a:} both $j$ and $k$ are greater or equal than $2$. Let us first claim that in this case
	\begin{align}
		\partial_{k(\alpha)}\Gamma^{i(\alpha)}_{1(\epsilon)j(\alpha)}-\partial_{1(\epsilon)}\Gamma^{i(\alpha)}_{k(\alpha)j(\alpha)}&=0.
		\label{claim_Rcaso2_jk2}
	\end{align}
	Indeed, if $i<j$ then both $\Gamma^{i(\alpha)}_{1(\epsilon)j(\alpha)}$ and $\Gamma^{i(\alpha)}_{k(\alpha)j(\alpha)}$ trivially vanish. If $i=j$ then
	\begin{align}
		\partial_{k(\alpha)}\Gamma^{i(\alpha)}_{1(\epsilon)j(\alpha)}-\partial_{1(\epsilon)}\Gamma^{i(\alpha)}_{k(\alpha)j(\alpha)}&=\partial_{k(\alpha)}\Gamma^{j(\alpha)}_{1(\epsilon)j(\alpha)}-\partial_{1(\epsilon)}\Gamma^{j(\alpha)}_{k(\alpha)j(\alpha)}
		\notag\\&=\xcancel{\partial_{k(\alpha)}\bigg[\frac{m_\epsilon\epsilon_\epsilon}{u^{1(\alpha)}-u^{1(\epsilon)}}\bigg]}-\partial_{1(\epsilon)}\Gamma^{(4-k)(\alpha)}_{2(\alpha)2(\alpha)}
		\notag\\&=-\partial_{1(\epsilon)}\Gamma^{2(\alpha)}_{2(\alpha)2(\alpha)}\,\delta^2_k
		\notag\\&=\delta^2_k\,\xcancel{\partial_{1(\epsilon)}\bigg[\frac{m_\alpha\epsilon_\alpha}{u^{2(\alpha)}}\bigg]}=0.
		\notag
	\end{align}
	We are going to prove \eqref{claim_Rcaso2_jk2} when $i>j$ by induction over $i$ (starting from the case $i=j$ that we just proved). Given an integer $s\geq1$, let us suppose that \eqref{claim_Rcaso2_jk2} holds true when $i=j+t$ for each $t\leq s-1$ that is
	\begin{align}
		\partial_{k(\alpha)}\Gamma^{i(\alpha)}_{1(\epsilon)j(\alpha)}-\partial_{1(\epsilon)}\Gamma^{i(\alpha)}_{k(\alpha)j(\alpha)}&=0\quad\text{for }i=j+t\quad\forall t\leq s-1
		\label{claim_Rcaso2_jk2_induz}
	\end{align}
	and show it holds for $t=s$ as well. We are thus considering $i=j+s$, so that
	\begin{align}
		\partial_{k(\alpha)}\Gamma^{i(\alpha)}_{1(\epsilon)j(\alpha)}-\partial_{1(\epsilon)}\Gamma^{i(\alpha)}_{k(\alpha)j(\alpha)}&=\partial_{k(\alpha)}\Gamma^{(j+s)(\alpha)}_{1(\epsilon)j(\alpha)}-\partial_{1(\epsilon)}\Gamma^{(j+s)(\alpha)}_{k(\alpha)j(\alpha)}
		\notag\\&=\partial_{k(\alpha)}\Gamma^{(s+1)(\alpha)}_{1(\epsilon)1(\alpha)}-\partial_{1(\epsilon)}\Gamma^{(s-k+4)(\alpha)}_{2(\alpha)2(\alpha)}
		\notag\\&\overset{\eqref{ascendente}}{=}\partial_{k(\alpha)}\bigg[-\frac{1}{u^{1(\alpha)}-u^{1(\epsilon)}}\,\overset{s+1}{\underset{l=2}{\sum}}\,\Gamma^{(s-l+2)(\alpha)}_{1(\epsilon)1(\alpha)}\,u^{l(\alpha)}\bigg]\notag\\&-\partial_{1(\epsilon)}\Gamma^{(s-k+2)(\alpha)}_{1(\alpha)1(\alpha)}
		\notag\\&=-\frac{1}{u^{1(\alpha)}-u^{1(\epsilon)}}\,\overset{s+1}{\underset{l=2}{\sum}}\,\bigg(\partial_{k(\alpha)}\Gamma^{(s-l+2)(\alpha)}_{1(\epsilon)1(\alpha)}\bigg)\,u^{l(\alpha)}\notag\\&-\frac{1}{u^{1(\alpha)}-u^{1(\epsilon)}}\,\Gamma^{(s-k+2)(\alpha)}_{1(\epsilon)1(\alpha)}+\partial_{1(\epsilon)}\Gamma^{(s-k+2)(\alpha)}_{1(\epsilon)1(\alpha)}
		\notag\\&\overset{\eqref{claim_Rcaso2_jk2_induz}}{=}-\frac{1}{u^{1(\alpha)}-u^{1(\epsilon)}}\,\overset{s+1}{\underset{l=2}{\sum}}\,\bigg(\partial_{1(\epsilon)}\Gamma^{(s-l-k+5)(\alpha)}_{2(\alpha)2(\alpha)}\bigg)\,u^{l(\alpha)}\notag\\&-\frac{1}{u^{1(\alpha)}-u^{1(\epsilon)}}\,\Gamma^{(s-k+2)(\alpha)}_{1(\epsilon)1(\alpha)}+\partial_{1(\epsilon)}\Gamma^{(s-k+2)(\alpha)}_{1(\epsilon)1(\alpha)}
		\notag\\&\overset{\eqref{ascendente}}{=}-\frac{1}{u^{1(\alpha)}-u^{1(\epsilon)}}\,\overset{s+1}{\underset{l=2}{\sum}}\,\bigg(\partial_{1(\epsilon)}\Gamma^{(s-l-k+3)(\alpha)}_{1(\alpha)1(\alpha)}\bigg)\,u^{l(\alpha)}\notag\\&-\frac{1}{u^{1(\alpha)}-u^{1(\epsilon)}}\,\Gamma^{(s-k+2)(\alpha)}_{1(\epsilon)1(\alpha)}+\partial_{1(\epsilon)}\Gamma^{(s-k+2)(\alpha)}_{1(\epsilon)1(\alpha)}
		\notag
	\end{align}
	where $\Gamma^{(s-l-k+3)(\alpha)}_{1(\alpha)1(\alpha)}$ vanishes for each $l\geq s-k+3$. It follows that
	\begin{align}
		\partial_{k(\alpha)}\Gamma^{i(\alpha)}_{1(\epsilon)j(\alpha)}-\partial_{1(\epsilon)}\Gamma^{i(\alpha)}_{k(\alpha)j(\alpha)}&=\frac{1}{u^{1(\alpha)}-u^{1(\epsilon)}}\,\overset{s-k+2}{\underset{l=2}{\sum}}\,\bigg(\partial_{1(\epsilon)}\Gamma^{(s-l-k+2)(\alpha)}_{1(\epsilon)1(\alpha)}\bigg)\,u^{l(\alpha)}\notag\\&-\frac{1}{u^{1(\alpha)}-u^{1(\epsilon)}}\,\Gamma^{(s-k+2)(\alpha)}_{1(\epsilon)1(\alpha)}\notag\\&+\partial_{1(\epsilon)}\bigg[-\frac{1}{u^{1(\alpha)}-u^{1(\epsilon)}}\,\overset{s-k+2}{\underset{l=2}{\sum}}\,\Gamma^{(s-l-k+3)(\alpha)}_{1(\epsilon)1(\alpha)}\,u^{l(\alpha)}\bigg]
		\notag\\&=\cancel{\frac{1}{u^{1(\alpha)}-u^{1(\epsilon)}}\,\overset{s-k+2}{\underset{l=2}{\sum}}\,\bigg(\partial_{1(\epsilon)}\Gamma^{(s-l-k+3)(\alpha)}_{1(\epsilon)1(\alpha)}\bigg)\,u^{l(\alpha)}}\notag\\&-\frac{1}{u^{1(\alpha)}-u^{1(\epsilon)}}\,\Gamma^{(s-k+2)(\alpha)}_{1(\epsilon)1(\alpha)}\notag\\&-\frac{1}{(u^{1(\alpha)}-u^{1(\epsilon)})^2}\,\overset{s-k+2}{\underset{l=2}{\sum}}\,\Gamma^{(s-l-k+3)(\alpha)}_{1(\epsilon)1(\alpha)}\,u^{l(\alpha)}\notag\\&-\cancel{\frac{1}{u^{1(\alpha)}-u^{1(\epsilon)}}\,\overset{s-k+2}{\underset{l=2}{\sum}}\,\bigg(\partial_{1(\epsilon)}\Gamma^{(s-l-k+3)(\alpha)}_{1(\epsilon)1(\alpha)}\bigg)\,u^{l(\alpha)}}
		\notag\\&=\frac{1}{(u^{1(\alpha)}-u^{1(\epsilon)})^2}\,\overset{s-k+2}{\underset{l=2}{\sum}}\,\Gamma^{(s-l-k+3)(\alpha)}_{1(\epsilon)1(\alpha)}\,u^{l(\alpha)}\notag\\&-\frac{1}{(u^{1(\alpha)}-u^{1(\epsilon)})^2}\,\overset{s-k+2}{\underset{l=2}{\sum}}\,\Gamma^{(s-l-k+3)(\alpha)}_{1(\epsilon)1(\alpha)}\,u^{l(\alpha)}=0.
		\notag
	\end{align}
	This proves \eqref{claim_Rcaso2_jk2}, thus \eqref{Rdef_caso2_h1} becomes
	\begin{align}
		R^{i(\alpha)}_{1(\epsilon)k(\alpha)j(\alpha)}&=\overset{m_\alpha}{\underset{l=1}{\sum}}\,\bigg(\Gamma^{i(\alpha)}_{k(\alpha)l(\alpha)}\Gamma^{l(\alpha)}_{1(\epsilon)j(\alpha)}-\Gamma^{i(\alpha)}_{1(\epsilon)l(\alpha)}\Gamma^{l(\alpha)}_{k(\alpha)j(\alpha)}\bigg)\notag\\&+\Gamma^{i(\alpha)}_{k(\alpha)1(\epsilon)}\xcancel{\Gamma^{1(\epsilon)}_{1(\epsilon)j(\alpha)}}-\Gamma^{i(\alpha)}_{1(\epsilon)1(\epsilon)}\xcancel{\Gamma^{1(\epsilon)}_{k(\alpha)j(\alpha)}}
		\notag
	\end{align}
	where in the first summation only the terms for $j\leq l\leq i$ survive. We get
	\begin{align}
		R^{i(\alpha)}_{1(\epsilon)k(\alpha)j(\alpha)}&=\overset{i}{\underset{l=j}{\sum}}\,\bigg(\Gamma^{i(\alpha)}_{k(\alpha)l(\alpha)}\Gamma^{l(\alpha)}_{1(\epsilon)j(\alpha)}-\Gamma^{i(\alpha)}_{1(\epsilon)l(\alpha)}\Gamma^{l(\alpha)}_{k(\alpha)j(\alpha)}\bigg)
		\notag
	\end{align}
	which trivially vanishes for $i<j$. For $i\geq j$ we have
	\begin{align}
		R^{i(\alpha)}_{1(\epsilon)k(\alpha)j(\alpha)}&=\overset{i}{\underset{l=j}{\sum}}\,\bigg(\Gamma^{(i-k-l+4)(\alpha)}_{2(\alpha)2(\alpha)}\Gamma^{(l-j+1)(\alpha)}_{1(\epsilon)1(\alpha)}-\Gamma^{(i-l+1)(\alpha)}_{1(\epsilon)1(\alpha)}\Gamma^{(l-k-j+4)(\alpha)}_{2(\alpha)2(\alpha)}\bigg)
		\notag\\&=\overset{i}{\underset{l=j}{\sum}}\,\Gamma^{(i-k-l+4)(\alpha)}_{2(\alpha)2(\alpha)}\Gamma^{(l-j+1)(\alpha)}_{1(\epsilon)1(\alpha)}-\overset{i}{\underset{l=j}{\sum}}\,\Gamma^{(i-l+1)(\alpha)}_{1(\epsilon)1(\alpha)}\Gamma^{(l-k-j+4)(\alpha)}_{2(\alpha)2(\alpha)}
		\notag\\&=\cancel{\overset{i}{\underset{l=j}{\sum}}\,\Gamma^{(i-k-l+4)(\alpha)}_{2(\alpha)2(\alpha)}\Gamma^{(l-j+1)(\alpha)}_{1(\epsilon)1(\alpha)}}-\cancel{\overset{i}{\underset{t=j}{\sum}}\,\Gamma^{(t-j+1)(\alpha)}_{1(\epsilon)1(\alpha)}\Gamma^{(i-k-t+4)(\alpha)}_{2(\alpha)2(\alpha)}}=0.
		\notag
	\end{align}
	\textbf{Subcase b:} $j=k=1$. We have
	\begin{align}
		R^{i(\alpha)}_{1(\epsilon)1(\alpha)1(\alpha)}&=\partial_{1(\alpha)}\Gamma^{i(\alpha)}_{1(\epsilon)1(\alpha)}-\partial_{1(\epsilon)}\Gamma^{i(\alpha)}_{1(\alpha)1(\alpha)}\notag\\&+\overset{m_\alpha}{\underset{l=1}{\sum}}\,\bigg(\Gamma^{i(\alpha)}_{1(\alpha)l(\alpha)}\Gamma^{l(\alpha)}_{1(\epsilon)1(\alpha)}-\Gamma^{i(\alpha)}_{1(\epsilon)l(\alpha)}\Gamma^{l(\alpha)}_{1(\alpha)1(\alpha)}\bigg)\notag\\&+\Gamma^{i(\alpha)}_{1(\alpha)1(\epsilon)}\Gamma^{1(\epsilon)}_{1(\epsilon)1(\alpha)}-\Gamma^{i(\alpha)}_{1(\epsilon)1(\epsilon)}\Gamma^{1(\epsilon)}_{1(\alpha)1(\alpha)}
		\notag
	\end{align}
	where in the summation only the terms for $l\leq i$ survive and
	\begin{align}
		\partial_{1(\epsilon)}\Gamma^{i(\alpha)}_{1(\alpha)1(\alpha)}&=-\partial_{1(\epsilon)}\Gamma^{i(\alpha)}_{1(\epsilon)1(\alpha)}=\partial_{1(\alpha)}\Gamma^{i(\alpha)}_{1(\epsilon)1(\alpha)}
		\notag
	\end{align}
	as $\Gamma^{i(\alpha)}_{1(\epsilon)1(\alpha)}$ only depends on both $u^{1(\epsilon)}$ and $u^{1(\alpha)}$ by means of the term $u^{1(\alpha)}-u^{1(\epsilon)}$. It follows that
	\begin{align}
		R^{i(\alpha)}_{1(\epsilon)1(\alpha)1(\alpha)}&=\overset{i}{\underset{l=1}{\sum}}\,\bigg(\Gamma^{i(\alpha)}_{1(\alpha)l(\alpha)}\Gamma^{l(\alpha)}_{1(\epsilon)1(\alpha)}-\Gamma^{i(\alpha)}_{1(\epsilon)l(\alpha)}\Gamma^{l(\alpha)}_{1(\alpha)1(\alpha)}\bigg)\notag\\&+\Gamma^{i(\alpha)}_{1(\alpha)1(\epsilon)}\Gamma^{1(\epsilon)}_{1(\epsilon)1(\alpha)}-\Gamma^{i(\alpha)}_{1(\epsilon)1(\alpha)}\Gamma^{1(\epsilon)}_{1(\alpha)1(\epsilon)}
		\notag\\&=-\underset{\sigma\neq\alpha}{\sum}\,\overset{i}{\underset{l=1}{\sum}}\,\bigg(\Gamma^{i(\alpha)}_{1(\sigma)l(\alpha)}\Gamma^{l(\alpha)}_{1(\epsilon)1(\alpha)}-\Gamma^{i(\alpha)}_{1(\epsilon)l(\alpha)}\Gamma^{l(\alpha)}_{1(\sigma)1(\alpha)}\bigg)\notag\\&+\cancel{\Gamma^{i(\alpha)}_{1(\alpha)1(\epsilon)}\Gamma^{1(\epsilon)}_{1(\epsilon)1(\alpha)}}-\cancel{\Gamma^{i(\alpha)}_{1(\epsilon)1(\alpha)}\Gamma^{1(\epsilon)}_{1(\alpha)1(\epsilon)}}
		\notag\\&=-\overset{i}{\underset{l=1}{\sum}}\,\bigg(\cancel{\Gamma^{i(\alpha)}_{1(\epsilon)l(\alpha)}\Gamma^{l(\alpha)}_{1(\epsilon)1(\alpha)}}-\cancel{\Gamma^{i(\alpha)}_{1(\epsilon)l(\alpha)}\Gamma^{l(\alpha)}_{1(\epsilon)1(\alpha)}}\bigg)\notag\\&-\underset{\sigma\notin\{\alpha,\epsilon\}}{\sum}\,\overset{i}{\underset{l=1}{\sum}}\,\bigg(\Gamma^{i(\alpha)}_{1(\sigma)l(\alpha)}\Gamma^{l(\alpha)}_{1(\epsilon)1(\alpha)}-\Gamma^{i(\alpha)}_{1(\epsilon)l(\alpha)}\Gamma^{l(\alpha)}_{1(\sigma)1(\alpha)}\bigg)\notag\\&+\Gamma^{i(\alpha)}_{1(\alpha)1(\epsilon)}\Gamma^{1(\epsilon)}_{1(\epsilon)1(\alpha)}-\Gamma^{i(\alpha)}_{1(\epsilon)1(\alpha)}\Gamma^{1(\epsilon)}_{1(\alpha)1(\epsilon)}
		\notag\\&=-\underset{\sigma\notin\{\alpha,\epsilon\}}{\sum}\,\bigg(\overset{i}{\underset{l=1}{\sum}}\,\Gamma^{(i-l+1)(\alpha)}_{1(\sigma)1(\alpha)}\Gamma^{l(\alpha)}_{1(\epsilon)1(\alpha)}-\overset{i}{\underset{l=1}{\sum}}\,\Gamma^{(i-l+1)(\alpha)}_{1(\epsilon)1(\alpha)}\Gamma^{l(\alpha)}_{1(\sigma)1(\alpha)}\bigg)
		\notag\\&=-\underset{\sigma\notin\{\alpha,\epsilon\}}{\sum}\,\bigg(\cancel{\overset{i}{\underset{l=1}{\sum}}\,\Gamma^{(i-l+1)(\alpha)}_{1(\sigma)1(\alpha)}\Gamma^{l(\alpha)}_{1(\epsilon)1(\alpha)}}-\cancel{\overset{i}{\underset{t=1}{\sum}}\,\Gamma^{t(\alpha)}_{1(\epsilon)1(\alpha)}\Gamma^{(i-t+1)(\alpha)}_{1(\sigma)1(\alpha)}}\bigg)=0.
		\notag
	\end{align}
	\textbf{Subcase c:} $j=1$, $k\geq2$. We have
	\begin{align}
		R^{i(\alpha)}_{1(\epsilon)k(\alpha)1(\alpha)}&=\partial_{k(\alpha)}\Gamma^{i(\alpha)}_{1(\epsilon)1(\alpha)}-\partial_{1(\epsilon)}\Gamma^{i(\alpha)}_{k(\alpha)1(\alpha)}\notag\\&+\overset{m_\alpha}{\underset{l=1}{\sum}}\,\bigg(\Gamma^{i(\alpha)}_{k(\alpha)l(\alpha)}\Gamma^{l(\alpha)}_{1(\epsilon)1(\alpha)}-\Gamma^{i(\alpha)}_{1(\epsilon)l(\alpha)}\Gamma^{l(\alpha)}_{k(\alpha)1(\alpha)}\bigg)\notag\\&+\Gamma^{i(\alpha)}_{k(\alpha)1(\epsilon)}\Gamma^{1(\epsilon)}_{1(\epsilon)1(\alpha)}-\Gamma^{i(\alpha)}_{1(\epsilon)1(\epsilon)}\Gamma^{1(\epsilon)}_{k(\alpha)1(\alpha)}
		\notag\\&=\partial_{k(\alpha)}\Gamma^{i(\alpha)}_{1(\epsilon)1(\alpha)}+\partial_{1(\epsilon)}\Gamma^{(i-k+1)(\alpha)}_{1(\epsilon)1(\alpha)}\notag\\&+\overset{i}{\underset{l=1}{\sum}}\,\bigg(\Gamma^{i(\alpha)}_{k(\alpha)l(\alpha)}\Gamma^{l(\alpha)}_{1(\epsilon)1(\alpha)}-\Gamma^{i(\alpha)}_{1(\epsilon)l(\alpha)}\Gamma^{l(\alpha)}_{k(\alpha)1(\alpha)}\bigg)\notag\\&+\Gamma^{i(\alpha)}_{k(\alpha)1(\epsilon)}\Gamma^{1(\epsilon)}_{1(\epsilon)1(\alpha)}-\Gamma^{i(\alpha)}_{1(\epsilon)1(\alpha)}\xcancel{\Gamma^{1(\epsilon)}_{k(\alpha)1(\epsilon)}}.
		\notag
	\end{align}
	Let us first claim that in this case
	\begin{align}
		\partial_{k(\alpha)}\Gamma^{i(\alpha)}_{1(\epsilon)1(\alpha)}+\partial_{1(\epsilon)}\Gamma^{(i-k+1)(\alpha)}_{1(\epsilon)1(\alpha)}&=0.
		\label{claim_Rcaso2_j1k2}
	\end{align}
	Indeed, if $i=1$ then
	\begin{align}
		\partial_{k(\alpha)}\Gamma^{1(\alpha)}_{1(\epsilon)1(\alpha)}+\partial_{1(\epsilon)}\xcancel{\Gamma^{(2-k)(\alpha)}_{1(\epsilon)1(\alpha)}}&=\partial_{k(\alpha)}\bigg[\frac{m_\epsilon\epsilon_\epsilon}{u^{1(\alpha)}-u^{1(\epsilon)}}\bigg]=0.
		\notag
	\end{align}
	We are going to prove \eqref{claim_Rcaso2_j1k2} when $i\geq1$ by induction over $i$ (starting from the case $i=1$ that we just proved). Given an integer $s\geq1$, let us suppose that \eqref{claim_Rcaso2_j1k2} holds true when $i=1+t$ for each $t\leq s-1$ that is
	\begin{align}
		\partial_{k(\alpha)}\Gamma^{i(\alpha)}_{1(\epsilon)1(\alpha)}+\partial_{1(\epsilon)}\Gamma^{(i-k+1)(\alpha)}_{1(\epsilon)1(\alpha)}&=0\quad\text{for }i=1+t\quad\forall t\leq s-1
		\label{claim_Rcaso2_j1k2_induz}
	\end{align}
	and show it holds for $t=s$ as well. We are thus considering $i=1+s$, so that
	\begin{align}
		\partial_{k(\alpha)}\Gamma^{i(\alpha)}_{1(\epsilon)1(\alpha)}&=\partial_{k(\alpha)}\Gamma^{(1+s)(\alpha)}_{1(\epsilon)1(\alpha)}\overset{\eqref{Lemma7.1eq}}{=}\partial_{(k-1)(\alpha)}\Gamma^{s(\alpha)}_{1(\epsilon)1(\alpha)}.
		\label{partial1_Rcaso2_j1k2_induz}
	\end{align}
	If $k=2$ then \eqref{partial1_Rcaso2_j1k2_induz} becomes
	\begin{align}
		\partial_{k(\alpha)}\Gamma^{i(\alpha)}_{1(\epsilon)1(\alpha)}&=\partial_{1(\alpha)}\Gamma^{s(\alpha)}_{1(\epsilon)1(\alpha)}=-\partial_{1(\epsilon)}\Gamma^{s(\alpha)}_{1(\epsilon)1(\alpha)}=-\partial_{1(\epsilon)}\Gamma^{(i-k+1)(\alpha)}_{1(\epsilon)1(\alpha)}
		\notag
	\end{align}
	as $\Gamma^{s(\alpha)}_{1(\epsilon)1(\alpha)}$ only depends on both $u^{1(\epsilon)}$ and $u^{1(\alpha)}$ by means of the term $u^{1(\alpha)}-u^{1(\epsilon)}$. If $k\geq3$ then \eqref{partial1_Rcaso2_j1k2_induz} becomes
	\begin{align}
		\partial_{k(\alpha)}\Gamma^{i(\alpha)}_{1(\epsilon)1(\alpha)}&=\partial_{(k-1)(\alpha)}\Gamma^{s(\alpha)}_{1(\epsilon)1(\alpha)}\overset{\eqref{claim_Rcaso2_j1k2_induz}}{=}-\partial_{1(\epsilon)}\Gamma^{(s-k+2)(\alpha)}_{1(\epsilon)1(\alpha)}=-\partial_{1(\epsilon)}\Gamma^{(i-k+1)(\alpha)}_{1(\epsilon)1(\alpha)}.
		\notag
	\end{align}
	Thus \eqref{claim_Rcaso2_j1k2} holds, yielding
	\begin{align}
		R^{i(\alpha)}_{1(\epsilon)k(\alpha)1(\alpha)}&=\overset{i}{\underset{l=1}{\sum}}\,\bigg(\Gamma^{i(\alpha)}_{k(\alpha)l(\alpha)}\Gamma^{l(\alpha)}_{1(\epsilon)1(\alpha)}-\Gamma^{i(\alpha)}_{1(\epsilon)l(\alpha)}\Gamma^{l(\alpha)}_{k(\alpha)1(\alpha)}\bigg)\notag\\&+\Gamma^{i(\alpha)}_{k(\alpha)1(\epsilon)}\Gamma^{1(\epsilon)}_{1(\epsilon)1(\alpha)}
		\notag\\&=\overset{i-k+2}{\underset{l=1}{\sum}}\,\Gamma^{i(\alpha)}_{k(\alpha)l(\alpha)}\Gamma^{l(\alpha)}_{1(\epsilon)1(\alpha)}-\overset{i}{\underset{l=k}{\sum}}\,\Gamma^{i(\alpha)}_{1(\epsilon)l(\alpha)}\Gamma^{l(\alpha)}_{k(\alpha)1(\alpha)}\notag\\&+\Gamma^{i(\alpha)}_{k(\alpha)1(\epsilon)}\Gamma^{1(\epsilon)}_{1(\epsilon)1(\alpha)}
		\notag
	\end{align}
	which trivially vanishes for $i<k$. For $i=k$ we get
	\begin{align}
		R^{i(\alpha)}_{1(\epsilon)k(\alpha)1(\alpha)}&=\cancel{\Gamma^{k(\alpha)}_{k(\alpha)1(\alpha)}\Gamma^{1(\alpha)}_{1(\epsilon)1(\alpha)}}+\Gamma^{k(\alpha)}_{k(\alpha)2(\alpha)}\Gamma^{2(\alpha)}_{1(\epsilon)1(\alpha)}\notag\\&-\cancel{\Gamma^{k(\alpha)}_{1(\epsilon)k(\alpha)}\Gamma^{k(\alpha)}_{k(\alpha)1(\alpha)}}+\Gamma^{k(\alpha)}_{k(\alpha)1(\epsilon)}\Gamma^{1(\epsilon)}_{1(\epsilon)1(\alpha)}
		\notag\\&=\Gamma^{2(\alpha)}_{2(\alpha)2(\alpha)}\Gamma^{2(\alpha)}_{1(\epsilon)1(\alpha)}+\Gamma^{1(\alpha)}_{1(\alpha)1(\epsilon)}\Gamma^{1(\epsilon)}_{1(\epsilon)1(\alpha)}
		\notag\\&=\frac{m_\alpha\epsilon_\alpha}{u^{2(\alpha)}}\,\frac{\Gamma^{1(\alpha)}_{1(\epsilon)1(\alpha)}\,u^{2(\alpha)}}{u^{1(\alpha)}-u^{1(\epsilon)}}-\Gamma^{1(\alpha)}_{1(\alpha)1(\epsilon)}\,\frac{m_\alpha\epsilon_\alpha}{u^{1(\alpha)}-u^{1(\epsilon)}}=0.
		\notag
	\end{align}
	For $i\geq k$ we get
	\begin{align}
		R^{i(\alpha)}_{1(\epsilon)k(\alpha)1(\alpha)}&=\overset{i-k+2}{\underset{l=1}{\sum}}\,\Gamma^{i(\alpha)}_{k(\alpha)l(\alpha)}\Gamma^{l(\alpha)}_{1(\epsilon)1(\alpha)}-\overset{i}{\underset{l=k}{\sum}}\,\Gamma^{i(\alpha)}_{1(\epsilon)l(\alpha)}\Gamma^{l(\alpha)}_{k(\alpha)1(\alpha)}\notag\\&+\Gamma^{i(\alpha)}_{k(\alpha)1(\epsilon)}\Gamma^{1(\epsilon)}_{1(\epsilon)1(\alpha)}
		\notag\\&=\Gamma^{(i-k+1)(\alpha)}_{1(\alpha)1(\alpha)}\Gamma^{1(\alpha)}_{1(\epsilon)1(\alpha)}+\overset{i-k+2}{\underset{l=2}{\sum}}\,\Gamma^{(i-k-l+4)(\alpha)}_{2(\alpha)2(\alpha)}\Gamma^{l(\alpha)}_{1(\epsilon)1(\alpha)}\notag\\&-\overset{i}{\underset{l=k}{\sum}}\,\Gamma^{(i-l+1)(\alpha)}_{1(\epsilon)1(\alpha)}\Gamma^{l(\alpha)}_{k(\alpha)1(\alpha)}+\Gamma^{(i-k+1)(\alpha)}_{1(\alpha)1(\epsilon)}\Gamma^{1(\epsilon)}_{1(\epsilon)1(\alpha)}
		\notag\\&=\cancel{\Gamma^{(i-k+1)(\alpha)}_{1(\alpha)1(\alpha)}\Gamma^{1(\alpha)}_{1(\epsilon)1(\alpha)}}+\overset{i-k+2}{\underset{l=2}{\sum}}\,\Gamma^{(i-k-l+4)(\alpha)}_{2(\alpha)2(\alpha)}\Gamma^{l(\alpha)}_{1(\epsilon)1(\alpha)}\notag\\&-\overset{i-k+1}{\underset{t=\cancelto{2}{1}}{\sum}}\,\Gamma^{t(\alpha)}_{1(\epsilon)1(\alpha)}\Gamma^{(i-t+1)(\alpha)}_{k(\alpha)1(\alpha)}+\Gamma^{(i-k+1)(\alpha)}_{1(\alpha)1(\epsilon)}\Gamma^{1(\epsilon)}_{1(\epsilon)1(\alpha)}
		\notag\\&=\overset{i-k+1}{\underset{l=2}{\sum}}\,\Gamma^{l(\alpha)}_{1(\epsilon)1(\alpha)}\,\bigg(\Gamma^{(i-k-l+4)(\alpha)}_{2(\alpha)2(\alpha)}-\Gamma^{(i-l+1)(\alpha)}_{k(\alpha)1(\alpha)}\bigg)\notag\\&+\Gamma^{2(\alpha)}_{2(\alpha)2(\alpha)}\Gamma^{(i-k+2)(\alpha)}_{1(\epsilon)1(\alpha)}+\Gamma^{(i-k+1)(\alpha)}_{1(\alpha)1(\epsilon)}\Gamma^{1(\epsilon)}_{1(\epsilon)1(\alpha)}
		\notag\\&=\overset{i-k+1}{\underset{l=2}{\sum}}\,\Gamma^{l(\alpha)}_{1(\epsilon)1(\alpha)}\,\bigg(\Gamma^{(i-k-l+4)(\alpha)}_{2(\alpha)2(\alpha)}-\Gamma^{(i-k-l+2)(\alpha)}_{1(\alpha)1(\alpha)}\bigg)\notag\\&+\Gamma^{2(\alpha)}_{2(\alpha)2(\alpha)}\Gamma^{(i-k+2)(\alpha)}_{1(\epsilon)1(\alpha)}+\Gamma^{(i-k+1)(\alpha)}_{1(\alpha)1(\epsilon)}\Gamma^{1(\epsilon)}_{1(\epsilon)1(\alpha)}
		\notag
	\end{align}
	where
	\begin{align}
		\Gamma^{2(\alpha)}_{2(\alpha)2(\alpha)}\Gamma^{(i-k+2)(\alpha)}_{1(\epsilon)1(\alpha)}+\Gamma^{(i-k+1)(\alpha)}_{1(\alpha)1(\epsilon)}\Gamma^{1(\epsilon)}_{1(\epsilon)1(\alpha)}&=\frac{m_\alpha\epsilon_\alpha}{u^{2(\alpha)}}\,\frac{1}{u^{1(\alpha)}-u^{1(\epsilon)}}\,\overset{i-k+2}{\underset{s=2}{\sum}}\,\Gamma^{(i-k-s+3)(\alpha)}_{1(\epsilon)1(\alpha)}\,u^{s(\alpha)}\notag\\&+\Gamma^{(i-k+1)(\alpha)}_{1(\alpha)1(\epsilon)}\Gamma^{1(\epsilon)}_{1(\epsilon)1(\alpha)}
		\notag\\&=-\frac{\Gamma^{1(\epsilon)}_{1(\alpha)1(\epsilon)}}{u^{2(\alpha)}}\,\bigg[\cancel{\Gamma^{(i-k+1)(\alpha)}_{1(\epsilon)1(\alpha)}\,u^{2(\alpha)}}\notag\\&+\overset{i-k+2}{\underset{s=3}{\sum}}\,\Gamma^{(i-k-s+3)(\alpha)}_{1(\epsilon)1(\alpha)}\,u^{s(\alpha)}\bigg]\notag\\&+\cancel{\Gamma^{(i-k+1)(\alpha)}_{1(\alpha)1(\epsilon)}\Gamma^{1(\epsilon)}_{1(\epsilon)1(\alpha)}}
		\notag\\&=-\frac{\Gamma^{1(\epsilon)}_{1(\alpha)1(\epsilon)}}{u^{2(\alpha)}}\,\overset{i-k+2}{\underset{s=3}{\sum}}\,\Gamma^{(i-k-s+3)(\alpha)}_{1(\epsilon)1(\alpha)}\,u^{s(\alpha)}
		\notag
	\end{align}
	thus
	\begin{align}
		R^{i(\alpha)}_{1(\epsilon)k(\alpha)1(\alpha)}&=\overset{i-k+1}{\underset{l=2}{\sum}}\,\Gamma^{l(\alpha)}_{1(\epsilon)1(\alpha)}\,\bigg(\Gamma^{(i-k-l+4)(\alpha)}_{2(\alpha)2(\alpha)}-\Gamma^{(i-k-l+2)(\alpha)}_{1(\alpha)1(\alpha)}\bigg)\notag\\&-\frac{\Gamma^{1(\epsilon)}_{1(\alpha)1(\epsilon)}}{u^{2(\alpha)}}\,\overset{i-k+2}{\underset{s=3}{\sum}}\,\Gamma^{(i-k-s+3)(\alpha)}_{1(\epsilon)1(\alpha)}\,u^{s(\alpha)}.
		\label{Rcaso2_j1k2_bis}
	\end{align}
	We are going to prove that \eqref{Rcaso2_j1k2_bis} vanishes for each $i\geq k$ by induction over $i$ (starting from the case $i=k$, where \eqref{Rcaso2_j1k2_bis} vanishes trivially). Given an integer $s\geq1$, let us suppose that \eqref{Rcaso2_j1k2_bis} vanishes when $i=k+t$ for each $t\leq s-1$ that is
	\begin{align}
		&\overset{t+1}{\underset{l=2}{\sum}}\,\Gamma^{l(\alpha)}_{1(\epsilon)1(\alpha)}\,\bigg(\Gamma^{(t-l+4)(\alpha)}_{2(\alpha)2(\alpha)}-\Gamma^{(t-l+2)(\alpha)}_{1(\alpha)1(\alpha)}\bigg)\notag\\&-\frac{\Gamma^{1(\epsilon)}_{1(\alpha)1(\epsilon)}}{u^{2(\alpha)}}\,\overset{t+2}{\underset{s=3}{\sum}}\,\Gamma^{(t-s+3)(\alpha)}_{1(\epsilon)1(\alpha)}\,u^{s(\alpha)}=0\qquad\forall t\leq s-1
		\label{claim_Rcaso2_j1k2_induz_bis}
	\end{align}
	and show that \eqref{Rcaso2_j1k2_bis} vanishes for $t=s$ as well. We are thus considering $i=k+s$, so that
	\begin{align}
		R^{i(\alpha)}_{1(\epsilon)k(\alpha)1(\alpha)}&=\overset{s+1}{\underset{l=2}{\sum}}\,\Gamma^{l(\alpha)}_{1(\epsilon)1(\alpha)}\,\bigg(\Gamma^{(s-l+4)(\alpha)}_{2(\alpha)2(\alpha)}-\Gamma^{(s-l+2)(\alpha)}_{1(\alpha)1(\alpha)}\bigg)\notag\\&-\frac{\Gamma^{1(\epsilon)}_{1(\alpha)1(\epsilon)}}{u^{2(\alpha)}}\,\overset{s+2}{\underset{l=3}{\sum}}\,\Gamma^{(s-l+3)(\alpha)}_{1(\epsilon)1(\alpha)}\,u^{l(\alpha)}
		\notag
	\end{align}
	where
	\begin{align}
		&\overset{s+1}{\underset{l=2}{\sum}}\,\Gamma^{l(\alpha)}_{1(\epsilon)1(\alpha)}\,\bigg(\Gamma^{(s-l+4)(\alpha)}_{2(\alpha)2(\alpha)}-\Gamma^{(s-l+2)(\alpha)}_{1(\alpha)1(\alpha)}\bigg)
		\notag\\&=\overset{s}{\underset{t=1}{\sum}}\,\Gamma^{(s-t+2)(\alpha)}_{1(\epsilon)1(\alpha)}\,\bigg(\Gamma^{(t+2)(\alpha)}_{2(\alpha)2(\alpha)}-\Gamma^{t(\alpha)}_{1(\alpha)1(\alpha)}\bigg)
		\notag\\&=\overset{s}{\underset{t=1}{\sum}}\,\bigg[-\frac{1}{u^{1(\alpha)}-u^{1(\epsilon)}}\,\overset{s-t+2}{\underset{r=2}{\sum}}\,\Gamma^{(s-t-r+3)(\alpha)}_{1(\epsilon)1(\alpha)}\,u^{r(\alpha)}\bigg]\,\bigg(\Gamma^{(t+2)(\alpha)}_{2(\alpha)2(\alpha)}-\Gamma^{t(\alpha)}_{1(\alpha)1(\alpha)}\bigg)
		\notag\\&=-\frac{1}{u^{1(\alpha)}-u^{1(\epsilon)}}\,\overset{s}{\underset{t=1}{\sum}}\,\bigg[\overset{s-t+1}{\underset{l=1}{\sum}}\,\Gamma^{l(\alpha)}_{1(\epsilon)1(\alpha)}\,u^{(s-t-l+3)(\alpha)}\bigg]\,\bigg(\Gamma^{(t+2)(\alpha)}_{2(\alpha)2(\alpha)}-\Gamma^{t(\alpha)}_{1(\alpha)1(\alpha)}\bigg)
		\notag\\&=-\frac{1}{u^{1(\alpha)}-u^{1(\epsilon)}}\,\overset{s}{\underset{l=1}{\sum}}\,\Gamma^{l(\alpha)}_{1(\epsilon)1(\alpha)}\,\overset{s-l+1}{\underset{t=1}{\sum}}\,\bigg(\Gamma^{(t+2)(\alpha)}_{2(\alpha)2(\alpha)}-\Gamma^{t(\alpha)}_{1(\alpha)1(\alpha)}\bigg)\,u^{(s-t-l+3)(\alpha)}
		\notag\\&=-\frac{1}{u^{1(\alpha)}-u^{1(\epsilon)}}\,\overset{s}{\underset{l=1}{\sum}}\,\Gamma^{l(\alpha)}_{1(\epsilon)1(\alpha)}\,\bigg[\bigg(\Gamma^{3(\alpha)}_{2(\alpha)2(\alpha)}-\Gamma^{1(\alpha)}_{1(\alpha)1(\alpha)}\bigg)\,u^{(s-l+2)(\alpha)}\notag\\&+\overset{s-l+1}{\underset{t=2}{\sum}}\,\bigg(\Gamma^{(t+2)(\alpha)}_{2(\alpha)2(\alpha)}-\Gamma^{t(\alpha)}_{1(\alpha)1(\alpha)}\bigg)\,u^{(s-t-l+3)(\alpha)}\bigg]
		\notag\\&\overset{\eqref{atlas}}{=}-\frac{1}{u^{1(\alpha)}-u^{1(\epsilon)}}\,\overset{s}{\underset{l=1}{\sum}}\,\Gamma^{l(\alpha)}_{1(\epsilon)1(\alpha)}\,\bigg[-\cancel{\Gamma^{2(\alpha)}_{2(\alpha)2(\alpha)}\,\frac{u^{3(\alpha)}}{u^{2(\alpha)}}\,u^{(s-l+2)(\alpha)}}\notag\\&+\Gamma^{2(\alpha)}_{2(\alpha)2(\alpha)}\,\bigg(\cancel{\frac{u^{3(\alpha)}}{u^{2(\alpha)}}\,u^{(s-l+2)(\alpha)}}-u^{(s-l+3)(\alpha)}\bigg)\bigg]
		\notag\\&\overset{\eqref{atlas}}{=}\frac{\Gamma^{2(\alpha)}_{2(\alpha)2(\alpha)}}{u^{1(\alpha)}-u^{1(\epsilon)}}\,\overset{s}{\underset{l=1}{\sum}}\,\Gamma^{l(\alpha)}_{1(\epsilon)1(\alpha)}\,u^{(s-l+3)(\alpha)}
		\notag
	\end{align}
	thus
	\begin{align}
		R^{i(\alpha)}_{1(\epsilon)k(\alpha)1(\alpha)}&=\frac{\Gamma^{2(\alpha)}_{2(\alpha)2(\alpha)}}{u^{1(\alpha)}-u^{1(\epsilon)}}\,\overset{s}{\underset{l=1}{\sum}}\,\Gamma^{l(\alpha)}_{1(\epsilon)1(\alpha)}\,u^{(s-l+3)(\alpha)}\notag\\&-\frac{\Gamma^{1(\epsilon)}_{1(\alpha)1(\epsilon)}}{u^{2(\alpha)}}\,\overset{s+2}{\underset{l=3}{\sum}}\,\Gamma^{(s-l+3)(\alpha)}_{1(\epsilon)1(\alpha)}\,u^{l(\alpha)}
		\notag\\&=\frac{\Gamma^{2(\alpha)}_{2(\alpha)2(\alpha)}}{u^{1(\alpha)}-u^{1(\epsilon)}}\,\overset{s}{\underset{l=1}{\sum}}\,\Gamma^{l(\alpha)}_{1(\epsilon)1(\alpha)}\,u^{(s-l+3)(\alpha)}\notag\\&-\frac{\Gamma^{1(\epsilon)}_{1(\alpha)1(\epsilon)}}{u^{2(\alpha)}}\,\overset{s}{\underset{t=1}{\sum}}\,\Gamma^{t(\alpha)}_{1(\epsilon)1(\alpha)}\,u^{(s-t+3)(\alpha)}
		\notag\\&=\overset{s}{\underset{l=1}{\sum}}\,\Gamma^{l(\alpha)}_{1(\epsilon)1(\alpha)}\,u^{(s-l+3)(\alpha)}\,\bigg(\frac{\Gamma^{2(\alpha)}_{2(\alpha)2(\alpha)}}{u^{1(\alpha)}-u^{1(\epsilon)}}-\frac{\Gamma^{1(\epsilon)}_{1(\alpha)1(\epsilon)}}{u^{2(\alpha)}}\bigg)
		\notag\\&=\overset{s}{\underset{l=1}{\sum}}\,\Gamma^{l(\alpha)}_{1(\epsilon)1(\alpha)}\,u^{(s-l+3)(\alpha)}\,\bigg(-\frac{m_\alpha\epsilon_\alpha}{u^{2(\alpha)}}\,\frac{1}{u^{1(\alpha)}-u^{1(\epsilon)}}+\frac{1}{u^{2(\alpha)}}\,\frac{m_\alpha\epsilon_\alpha}{u^{1(\alpha)}-u^{1(\epsilon)}}\bigg)=0.
		\notag
	\end{align}
	\textbf{Subcase d:} $j\geq2$, $k=1$. We have
	\begin{align}
		R^{i(\alpha)}_{1(\epsilon)1(\alpha)j(\alpha)}&=\partial_{1(\alpha)}\Gamma^{i(\alpha)}_{1(\epsilon)j(\alpha)}-\partial_{1(\epsilon)}\Gamma^{i(\alpha)}_{1(\alpha)j(\alpha)}\notag\\&+\overset{m_\alpha}{\underset{l=1}{\sum}}\,\bigg(\Gamma^{i(\alpha)}_{1(\alpha)l(\alpha)}\Gamma^{l(\alpha)}_{1(\epsilon)j(\alpha)}-\Gamma^{i(\alpha)}_{1(\epsilon)l(\alpha)}\Gamma^{l(\alpha)}_{1(\alpha)j(\alpha)}\bigg)\notag\\&+\Gamma^{i(\alpha)}_{1(\alpha)1(\epsilon)}\xcancel{\Gamma^{1(\epsilon)}_{1(\epsilon)j(\alpha)}}-\Gamma^{i(\alpha)}_{1(\epsilon)1(\epsilon)}\xcancel{\Gamma^{1(\epsilon)}_{1(\alpha)j(\alpha)}}
		\notag
	\end{align}
	where
	\begin{align}
		\partial_{1(\epsilon)}\Gamma^{i(\alpha)}_{1(\alpha)j(\alpha)}&=-\underset{\sigma\neq\alpha}{\sum}\,\partial_{1(\epsilon)}\Gamma^{i(\alpha)}_{1(\sigma)j(\alpha)}=-\partial_{1(\epsilon)}\Gamma^{i(\alpha)}_{1(\epsilon)j(\alpha)}-\underset{\sigma\notin\{\alpha,\epsilon\}}{\sum}\,\xcancel{\partial_{1(\epsilon)}\Gamma^{i(\alpha)}_{1(\sigma)j(\alpha)}}
		\notag\\&=\partial_{1(\alpha)}\Gamma^{i(\alpha)}_{1(\epsilon)j(\alpha)}
		\notag
	\end{align}
	as $\Gamma^{i(\alpha)}_{1(\epsilon)j(\alpha)}$ only depends on both $u^{1(\epsilon)}$ and $u^{1(\alpha)}$ by means of the term $u^{1(\alpha)}-u^{1(\epsilon)}$. It follows that
	\begin{align}
		R^{i(\alpha)}_{1(\epsilon)1(\alpha)j(\alpha)}&=\overset{m_\alpha}{\underset{l=1}{\sum}}\,\bigg(\Gamma^{i(\alpha)}_{1(\alpha)l(\alpha)}\Gamma^{l(\alpha)}_{1(\epsilon)j(\alpha)}-\Gamma^{i(\alpha)}_{1(\epsilon)l(\alpha)}\Gamma^{l(\alpha)}_{1(\alpha)j(\alpha)}\bigg)
		\notag\\&=\overset{i}{\underset{l=j}{\sum}}\,\Gamma^{(i-l+1)(\alpha)}_{1(\alpha)1(\alpha)}\Gamma^{(l-j+1)(\alpha)}_{1(\epsilon)1(\alpha)}-\overset{i}{\underset{l=j}{\sum}}\,\Gamma^{(i-l+1)(\alpha)}_{1(\epsilon)1(\alpha)}\Gamma^{(l-j+1)(\alpha)}_{1(\alpha)1(\alpha)}
		\notag\\&=\overset{i}{\underset{l=j}{\sum}}\,\Gamma^{(i-l+1)(\alpha)}_{1(\alpha)1(\alpha)}\Gamma^{(l-j+1)(\alpha)}_{1(\epsilon)1(\alpha)}-\overset{i}{\underset{t=j}{\sum}}\,\Gamma^{(t-j+1)(\alpha)}_{1(\epsilon)1(\alpha)}\Gamma^{(i-t+1)(\alpha)}_{1(\alpha)1(\alpha)}=0.
		\notag
	\end{align}
	\textbf{Case 3: $\alpha=\gamma=\epsilon\neq\beta$.}
	Out goal is to prove that
	\begin{align}
		R^{i(\alpha)}_{h(\alpha)k(\alpha)j(\beta)}&=\partial_{k(\alpha)}\Gamma^{i(\alpha)}_{h(\alpha)j(\beta)}-\partial_{h(\alpha)}\Gamma^{i(\alpha)}_{k(\alpha)j(\beta)}\notag\\&+\overset{r}{\underset{\sigma=1}{\sum}}\,\overset{m_\sigma}{\underset{l=1}{\sum}}\,\bigg(\Gamma^{i(\alpha)}_{k(\alpha)l(\sigma)}\Gamma^{l(\sigma)}_{h(\alpha)j(\beta)}-\Gamma^{i(\alpha)}_{h(\alpha)l(\sigma)}\Gamma^{l(\sigma)}_{k(\alpha)j(\beta)}\bigg)
		\notag\\&=\partial_{k(\alpha)}\Gamma^{i(\alpha)}_{h(\alpha)j(\beta)}-\partial_{h(\alpha)}\Gamma^{i(\alpha)}_{k(\alpha)j(\beta)}\notag\\&+\overset{m_\alpha}{\underset{l=1}{\sum}}\,\bigg(\Gamma^{i(\alpha)}_{k(\alpha)l(\alpha)}\Gamma^{l(\alpha)}_{h(\alpha)j(\beta)}-\Gamma^{i(\alpha)}_{h(\alpha)l(\alpha)}\Gamma^{l(\alpha)}_{k(\alpha)j(\beta)}\bigg)\notag\\&+\overset{m_\beta}{\underset{l=1}{\sum}}\,\bigg(\Gamma^{i(\alpha)}_{k(\alpha)l(\beta)}\Gamma^{l(\beta)}_{h(\alpha)j(\beta)}-\Gamma^{i(\alpha)}_{h(\alpha)l(\beta)}\Gamma^{l(\beta)}_{k(\alpha)j(\beta)}\bigg)
		\label{Rdef_caso3}
	\end{align}
	vanishes. If $j\geq2$ we get
	\begin{align}
		R^{i(\alpha)}_{h(\epsilon)k(\gamma)j(\beta)}&=\partial_{k(\alpha)}\xcancel{\Gamma^{i(\alpha)}_{h(\alpha)j(\beta)}}-\partial_{h(\alpha)}\xcancel{\Gamma^{i(\alpha)}_{k(\alpha)j(\beta)}}\notag\\&+\overset{m_\alpha}{\underset{l=1}{\sum}}\,\bigg(\Gamma^{i(\alpha)}_{k(\alpha)l(\alpha)}\xcancel{\Gamma^{l(\alpha)}_{h(\alpha)j(\beta)}}-\Gamma^{i(\alpha)}_{h(\alpha)l(\alpha)}\xcancel{\Gamma^{l(\alpha)}_{k(\alpha)j(\beta)}}\bigg)\notag\\&+\overset{m_\beta}{\underset{l=1}{\sum}}\,\bigg(\Gamma^{i(\alpha)}_{k(\alpha)l(\beta)}\Gamma^{l(\beta)}_{h(\alpha)j(\beta)}-\Gamma^{i(\alpha)}_{h(\alpha)l(\beta)}\Gamma^{l(\beta)}_{k(\alpha)j(\beta)}\bigg)
		\notag
	\end{align}
	where in the last summation only the terms for $l=1$ survive. Thus
	\begin{align}
		R^{i(\alpha)}_{h(\epsilon)k(\gamma)j(\beta)}&=\Gamma^{i(\alpha)}_{k(\alpha)1(\beta)}\xcancel{\Gamma^{1(\beta)}_{h(\alpha)j(\beta)}}-\Gamma^{i(\alpha)}_{h(\alpha)1(\beta)}\xcancel{\Gamma^{1(\beta)}_{k(\alpha)j(\beta)}}=0
		\notag
	\end{align}
	for $j\geq2$. Let us then fix $j=1$. We have
	\begin{align}
		R^{i(\alpha)}_{h(\epsilon)k(\gamma)1(\beta)}&=\partial_{k(\alpha)}\Gamma^{i(\alpha)}_{h(\alpha)1(\beta)}-\partial_{h(\alpha)}\Gamma^{i(\alpha)}_{k(\alpha)1(\beta)}\notag\\&+\overset{m_\alpha}{\underset{l=1}{\sum}}\,\bigg(\Gamma^{i(\alpha)}_{k(\alpha)l(\alpha)}\Gamma^{l(\alpha)}_{h(\alpha)1(\beta)}-\Gamma^{i(\alpha)}_{h(\alpha)l(\alpha)}\Gamma^{l(\alpha)}_{k(\alpha)1(\beta)}\bigg)\notag\\&+\overset{m_\beta}{\underset{l=1}{\sum}}\,\bigg(\Gamma^{i(\alpha)}_{k(\alpha)l(\beta)}\Gamma^{l(\beta)}_{h(\alpha)1(\beta)}-\Gamma^{i(\alpha)}_{h(\alpha)l(\beta)}\Gamma^{l(\beta)}_{k(\alpha)1(\beta)}\bigg).
		\label{Rdef_caso3_j1}
	\end{align}
	Let us recall that, as we have already seen in \eqref{claim_Rcaso2_j1k2},
	\begin{align}
		\partial_{t(\alpha)}\Gamma^{i(\alpha)}_{1(\epsilon)1(\alpha)}&=-\partial_{1(\epsilon)}\Gamma^{(i-t+1)(\alpha)}_{1(\epsilon)1(\alpha)}
		\label{helpful_deriv}
	\end{align}
	for every $\epsilon\neq\alpha$ and $t\geq2$. We distinguish between the following subcases:
	\begin{itemize}
		\item[$a.$] both $h$ and $k$ are greater or equal than $2$
		\item[$b.$] $h=1$, $k\geq2$ (this covers $h\geq2$, $k=1$ as well)
	\end{itemize}
	observing that $R^{i(\alpha)}_{h(\epsilon)k(\gamma)1(\beta)}=0$ automatically whenever $k=h$.
	\\\textbf{Subcase a:} both $k$ and $h$ are greater or equal than $2$. We have
	\begin{align}
		R^{i(\alpha)}_{h(\epsilon)k(\gamma)1(\beta)}&=\partial_{k(\alpha)}\Gamma^{i(\alpha)}_{h(\alpha)1(\beta)}-\partial_{h(\alpha)}\Gamma^{i(\alpha)}_{k(\alpha)1(\beta)}\notag\\&+\overset{i-k+2}{\underset{l=h}{\sum}}\,\Gamma^{i(\alpha)}_{k(\alpha)l(\alpha)}\Gamma^{l(\alpha)}_{h(\alpha)1(\beta)}-\overset{i-h+2}{\underset{l=k}{\sum}}\,\Gamma^{i(\alpha)}_{h(\alpha)l(\alpha)}\Gamma^{l(\alpha)}_{k(\alpha)1(\beta)}\notag\\&+\overset{m_\beta}{\underset{l=1}{\sum}}\,\bigg(\Gamma^{i(\alpha)}_{k(\alpha)l(\beta)}\xcancel{\Gamma^{l(\beta)}_{h(\alpha)1(\beta)}}-\Gamma^{i(\alpha)}_{h(\alpha)l(\beta)}\xcancel{\Gamma^{l(\beta)}_{k(\alpha)1(\beta)}}\bigg)
		\notag
	\end{align}
	where
	\begin{align}
		\partial_{k(\alpha)}\Gamma^{i(\alpha)}_{h(\alpha)1(\beta)}-\partial_{h(\alpha)}\Gamma^{i(\alpha)}_{k(\alpha)1(\beta)}&=\partial_{k(\alpha)}\Gamma^{(i-h+1)(\alpha)}_{1(\alpha)1(\beta)}-\partial_{h(\alpha)}\Gamma^{(i-k+1)(\alpha)}_{1(\alpha)1(\beta)}
		\notag\\&\overset{\eqref{helpful_deriv}}{=}-\partial_{1(\beta)}\Gamma^{(i-h-k+2)(\alpha)}_{1(\alpha)1(\beta)}+\partial_{1(\beta)}\Gamma^{(i-k-h+2)(\alpha)}_{1(\alpha)1(\beta)}=0.
		\notag
	\end{align}
	This yields
	\begin{align}
		R^{i(\alpha)}_{h(\epsilon)k(\gamma)1(\beta)}&=\overset{i-k+2}{\underset{l=h}{\sum}}\,\Gamma^{i(\alpha)}_{k(\alpha)l(\alpha)}\Gamma^{l(\alpha)}_{h(\alpha)1(\beta)}-\overset{i-h+2}{\underset{l=k}{\sum}}\,\Gamma^{i(\alpha)}_{h(\alpha)l(\alpha)}\Gamma^{l(\alpha)}_{k(\alpha)1(\beta)}
		\label{Rdef_caso3_j1_hk2}
	\end{align}
	which automatically vanishes for $i<k$. For $i=k$ we get
	\begin{align}
		R^{k(\alpha)}_{h(\epsilon)k(\gamma)1(\beta)}&=\overset{2}{\underset{l=h}{\sum}}\,\Gamma^{k(\alpha)}_{k(\alpha)l(\alpha)}\Gamma^{l(\alpha)}_{h(\alpha)1(\beta)}-\overset{k-h+2}{\underset{l=k}{\sum}}\,\Gamma^{k(\alpha)}_{h(\alpha)l(\alpha)}\Gamma^{l(\alpha)}_{k(\alpha)1(\beta)}
		\notag
	\end{align}
	which trivially vanishes for $h\geq3$ and reads
	\begin{align}
		R^{k(\alpha)}_{2(\epsilon)k(\gamma)1(\beta)}&=\Gamma^{k(\alpha)}_{k(\alpha)2(\alpha)}\Gamma^{2(\alpha)}_{2(\alpha)1(\beta)}-\Gamma^{k(\alpha)}_{2(\alpha)k(\alpha)}\Gamma^{k(\alpha)}_{k(\alpha)1(\beta)}
		\notag\\&=\Gamma^{2(\alpha)}_{2(\alpha)2(\alpha)}\Gamma^{1(\alpha)}_{1(\alpha)1(\beta)}-\Gamma^{2(\alpha)}_{2(\alpha)2(\alpha)}\Gamma^{1(\alpha)}_{1(\alpha)1(\beta)}=0
		\notag
	\end{align}
	for $h=2$. For $i>k$ \eqref{Rdef_caso3_j1_hk2} becomes
	\begin{align}
		R^{i(\alpha)}_{h(\epsilon)k(\gamma)1(\beta)}&=\overset{i-k+2}{\underset{l=h}{\sum}}\,\Gamma^{(i-k-l+4)(\alpha)}_{2(\alpha)2(\alpha)}\Gamma^{(l-h+1)(\alpha)}_{1(\alpha)1(\beta)}-\overset{i-h+2}{\underset{l=k}{\sum}}\,\Gamma^{(i-h-l+4)(\alpha)}_{2(\alpha)2(\alpha)}\Gamma^{(l-k+1)(\alpha)}_{1(\alpha)1(\beta)}
		\notag\\&=\overset{i-k+2}{\underset{l=h}{\sum}}\,\Gamma^{(i-k-l+4)(\alpha)}_{2(\alpha)2(\alpha)}\Gamma^{(l-h+1)(\alpha)}_{1(\alpha)1(\beta)}-\overset{i-k+2}{\underset{t=h}{\sum}}\,\Gamma^{(i-k-t+4)(\alpha)}_{2(\alpha)2(\alpha)}\Gamma^{(t-h+1)(\alpha)}_{1(\alpha)1(\beta)}=0.
		\notag
	\end{align}
	\textbf{Subcase b:} $h=1$, $k\geq2$. We have
	\begin{align}
		R^{i(\alpha)}_{1(\epsilon)k(\gamma)1(\beta)}&=\partial_{k(\alpha)}\Gamma^{i(\alpha)}_{1(\alpha)1(\beta)}-\partial_{1(\alpha)}\Gamma^{i(\alpha)}_{k(\alpha)1(\beta)}\notag\\&+\overset{i-k+2}{\underset{l=1}{\sum}}\,\Gamma^{i(\alpha)}_{k(\alpha)l(\alpha)}\Gamma^{l(\alpha)}_{1(\alpha)1(\beta)}-\overset{i}{\underset{l=k}{\sum}}\,\Gamma^{i(\alpha)}_{1(\alpha)l(\alpha)}\Gamma^{l(\alpha)}_{k(\alpha)1(\beta)}\notag\\&+\overset{m_\beta}{\underset{l=1}{\sum}}\,\bigg(\Gamma^{i(\alpha)}_{k(\alpha)l(\beta)}\Gamma^{l(\beta)}_{1(\alpha)1(\beta)}-\Gamma^{i(\alpha)}_{1(\alpha)l(\beta)}\xcancel{\Gamma^{l(\beta)}_{k(\alpha)1(\beta)}}\bigg)
		\notag
	\end{align}
	where in the last summation only the term for $l=1$ survives and
	\begin{align}
		\partial_{k(\alpha)}\Gamma^{i(\alpha)}_{1(\alpha)1(\beta)}-\partial_{1(\alpha)}\Gamma^{i(\alpha)}_{k(\alpha)1(\beta)}&\overset{\eqref{helpful_deriv}}{=}-\partial_{1(\beta)}\Gamma^{(i-k+1)(\alpha)}_{1(\alpha)1(\beta)}+\partial_{1(\beta)}\Gamma^{i(\alpha)}_{k(\alpha)1(\beta)}=0.
		\notag
	\end{align}
	This yields
	\begin{align}
		R^{i(\alpha)}_{1(\epsilon)k(\gamma)1(\beta)}&=\overset{i-k+2}{\underset{l=1}{\sum}}\,\Gamma^{i(\alpha)}_{k(\alpha)l(\alpha)}\Gamma^{l(\alpha)}_{1(\alpha)1(\beta)}-\overset{i}{\underset{l=k}{\sum}}\,\Gamma^{i(\alpha)}_{1(\alpha)l(\alpha)}\Gamma^{l(\alpha)}_{k(\alpha)1(\beta)}\notag\\&+\Gamma^{i(\alpha)}_{k(\alpha)1(\beta)}\Gamma^{1(\beta)}_{1(\alpha)1(\beta)}
		\notag\\&=\Gamma^{i(\alpha)}_{k(\alpha)1(\alpha)}\Gamma^{1(\alpha)}_{1(\alpha)1(\beta)}+\overset{i-k+2}{\underset{l=2}{\sum}}\,\Gamma^{(i-k-l+4)(\alpha)}_{2(\alpha)2(\alpha)}\Gamma^{l(\alpha)}_{1(\alpha)1(\beta)}-\overset{i}{\underset{l=k}{\sum}}\,\Gamma^{(i-l+1)(\alpha)}_{1(\alpha)1(\alpha)}\Gamma^{(l-k+1)(\alpha)}_{1(\alpha)1(\beta)}\notag\\&+\Gamma^{i(\alpha)}_{k(\alpha)1(\beta)}\Gamma^{1(\beta)}_{1(\alpha)1(\beta)}
		\notag\\&=\Gamma^{i(\alpha)}_{k(\alpha)1(\alpha)}\Gamma^{1(\alpha)}_{1(\alpha)1(\beta)}+\overset{i-k+2}{\underset{l=2}{\sum}}\,\Gamma^{(i-k-l+4)(\alpha)}_{2(\alpha)2(\alpha)}\Gamma^{l(\alpha)}_{1(\alpha)1(\beta)}-\overset{i-k+1}{\underset{t=1}{\sum}}\,\Gamma^{(i-k-t+2)(\alpha)}_{1(\alpha)1(\alpha)}\Gamma^{t(\alpha)}_{1(\alpha)1(\beta)}\notag\\&+\Gamma^{i(\alpha)}_{k(\alpha)1(\beta)}\Gamma^{1(\beta)}_{1(\alpha)1(\beta)}
		\notag\\&=\cancel{\Gamma^{(i-k+1)(\alpha)}_{1(\alpha)1(\alpha)}\Gamma^{1(\alpha)}_{1(\alpha)1(\beta)}}+\overset{i-k+2}{\underset{l=2}{\sum}}\,\Gamma^{(i-k-l+4)(\alpha)}_{2(\alpha)2(\alpha)}\Gamma^{l(\alpha)}_{1(\alpha)1(\beta)}\notag\\&-\cancel{\Gamma^{(i-k+1)(\alpha)}_{1(\alpha)1(\alpha)}\Gamma^{1(\alpha)}_{1(\alpha)1(\beta)}}-\overset{i-k+1}{\underset{l=2}{\sum}}\,\Gamma^{(i-k-l+2)(\alpha)}_{1(\alpha)1(\alpha)}\Gamma^{l(\alpha)}_{1(\alpha)1(\beta)}\notag\\&+\Gamma^{i(\alpha)}_{k(\alpha)1(\beta)}\Gamma^{1(\beta)}_{1(\alpha)1(\beta)}
		\notag\\&=\overset{i-k+1}{\underset{l=2}{\sum}}\,\bigg(\Gamma^{(i-k-l+4)(\alpha)}_{2(\alpha)2(\alpha)}-\Gamma^{(i-k-l+2)(\alpha)}_{1(\alpha)1(\alpha)}\bigg)\Gamma^{l(\alpha)}_{1(\alpha)1(\beta)}+\Gamma^{2(\alpha)}_{2(\alpha)2(\alpha)}\Gamma^{(i-k+2)(\alpha)}_{1(\alpha)1(\beta)}\notag\\&+\Gamma^{(i-k+1)(\alpha)}_{1(\alpha)1(\beta)}\Gamma^{1(\beta)}_{1(\alpha)1(\beta)}\overset{\eqref{fulm}}{=}0.
		\label{Rdef_caso3_j1_h1k2}
	\end{align}
	\textbf{Case 4: $\beta=\gamma=\epsilon\neq\alpha$.} Our goal is to prove that
	\begin{align}
		R^{i(\alpha)}_{h(\beta)k(\beta)j(\beta)}&=\partial_{k(\beta)}\Gamma^{i(\alpha)}_{h(\beta)j(\beta)}-\partial_{h(\beta)}\Gamma^{i(\alpha)}_{k(\beta)j(\beta)}\notag\\&+\overset{r}{\underset{\sigma=1}{\sum}}\,\overset{m_\sigma}{\underset{l=1}{\sum}}\,\bigg(\Gamma^{i(\alpha)}_{k(\beta)l(\sigma)}\Gamma^{l(\sigma)}_{h(\beta)j(\beta)}-\Gamma^{i(\alpha)}_{h(\beta)l(\sigma)}\Gamma^{l(\sigma)}_{k(\beta)j(\beta)}\bigg)
		\notag\\&=\partial_{k(\beta)}\Gamma^{i(\alpha)}_{h(\beta)j(\beta)}-\partial_{h(\beta)}\Gamma^{i(\alpha)}_{k(\beta)j(\beta)}\notag\\&+\overset{m_\alpha}{\underset{l=1}{\sum}}\,\bigg(\Gamma^{i(\alpha)}_{k(\beta)l(\alpha)}\Gamma^{l(\alpha)}_{h(\beta)j(\beta)}-\Gamma^{i(\alpha)}_{h(\beta)l(\alpha)}\Gamma^{l(\alpha)}_{k(\beta)j(\beta)}\bigg)\notag\\&+\overset{m_\beta}{\underset{l=1}{\sum}}\,\bigg(\Gamma^{i(\alpha)}_{k(\beta)l(\beta)}\Gamma^{l(\beta)}_{h(\beta)j(\beta)}-\Gamma^{i(\alpha)}_{h(\beta)l(\beta)}\Gamma^{l(\beta)}_{k(\beta)j(\beta)}\bigg)
		\notag
	\end{align}
	vanishes. Without loss of generality, by the simmetries of $R$, we can set $h>k$. In particular, $h\geq2$. We get
	\begin{align}
		R^{i(\alpha)}_{h(\beta)k(\beta)j(\beta)}&=\partial_{k(\beta)}\xcancel{\Gamma^{i(\alpha)}_{h(\beta)j(\beta)}}-\partial_{h(\beta)}\Gamma^{i(\alpha)}_{k(\beta)j(\beta)}\notag\\&+\overset{m_\alpha}{\underset{l=1}{\sum}}\,\bigg(\Gamma^{i(\alpha)}_{k(\beta)l(\alpha)}\xcancel{\Gamma^{l(\alpha)}_{h(\beta)j(\beta)}}-\Gamma^{i(\alpha)}_{h(\beta)l(\alpha)}\xcancel{\Gamma^{l(\alpha)}_{k(\beta)j(\beta)}}\bigg)\notag\\&+\overset{m_\beta}{\underset{l=1}{\sum}}\,\bigg(\Gamma^{i(\alpha)}_{k(\beta)l(\beta)}\Gamma^{l(\beta)}_{h(\beta)j(\beta)}-\xcancel{\Gamma^{i(\alpha)}_{h(\beta)l(\beta)}}\Gamma^{l(\beta)}_{k(\beta)j(\beta)}\bigg)
		\label{Rdef_caso4}
	\end{align}
	which reads
	\begin{align}
		R^{i(\alpha)}_{h(\beta)k(\beta)j(\beta)}&=-\partial_{h(\beta)}\xcancel{\Gamma^{i(\alpha)}_{k(\beta)j(\beta)}}+\overset{m_\beta}{\underset{l=1}{\sum}}\,\xcancel{\Gamma^{i(\alpha)}_{k(\beta)l(\beta)}}\Gamma^{l(\beta)}_{h(\beta)j(\beta)}=0
		\notag
	\end{align}
	for $k\geq2$. For $k=1$ \eqref{Rdef_caso4} becomes
	\begin{align}
		R^{i(\alpha)}_{h(\beta)1(\beta)j(\beta)}&=-\partial_{h(\beta)}\Gamma^{i(\alpha)}_{1(\beta)j(\beta)}+\overset{m_\beta}{\underset{l=1}{\sum}}\,\Gamma^{i(\alpha)}_{1(\beta)l(\beta)}\Gamma^{l(\beta)}_{h(\beta)j(\beta)}
		\notag\\&=\partial_{h(\beta)}\Gamma^{i(\alpha)}_{1(\alpha)j(\beta)}-\Gamma^{i(\alpha)}_{1(\alpha)1(\beta)}\Gamma^{1(\beta)}_{h(\beta)j(\beta)}
		\notag
	\end{align}
	which is
	\begin{align}
		R^{i(\alpha)}_{h(\beta)1(\beta)j(\beta)}&=\partial_{h(\beta)}\xcancel{\Gamma^{i(\alpha)}_{1(\alpha)j(\beta)}}-\Gamma^{i(\alpha)}_{1(\alpha)1(\beta)}\xcancel{\Gamma^{1(\beta)}_{h(\beta)j(\beta)}}=0
		\notag
	\end{align}
	for $j\geq2$ and
	\begin{align}
		R^{i(\alpha)}_{h(\beta)1(\beta)1(\beta)}&=\partial_{h(\beta)}\Gamma^{i(\alpha)}_{1(\alpha)1(\beta)}-\Gamma^{i(\alpha)}_{1(\alpha)1(\beta)}\Gamma^{1(\beta)}_{h(\beta)1(\beta)}=0
		\notag
	\end{align}
	for $j=1$, as $\Gamma^{i(\alpha)}_{1(\alpha)1(\beta)}$ does not depend on any of the $u^{h(\beta)}$s for $h\geq2$ and ($1<2\leq h$)
	\begin{align}
		\Gamma^{1(\beta)}_{h(\beta)1(\beta)}&=-\underset{\sigma\neq\beta}{\sum}\,\Gamma^{1(\beta)}_{h(\beta)1(\sigma)}=0.
		\notag
	\end{align}
	\textbf{Case 5: $\alpha=\beta\neq\gamma=\epsilon$.} Our goal is to prove that
	\begin{align}
		R^{i(\alpha)}_{h(\gamma)k(\gamma)j(\alpha)}&=\partial_{k(\gamma)}\Gamma^{i(\alpha)}_{h(\gamma)j(\alpha)}-\partial_{h(\gamma)}\Gamma^{i(\alpha)}_{k(\gamma)j(\alpha)}\notag\\&+\overset{r}{\underset{\sigma=1}{\sum}}\,\overset{m_\sigma}{\underset{l=1}{\sum}}\,\bigg(\Gamma^{i(\alpha)}_{k(\gamma)l(\sigma)}\Gamma^{l(\sigma)}_{h(\gamma)j(\alpha)}-\Gamma^{i(\alpha)}_{h(\gamma)l(\sigma)}\Gamma^{l(\sigma)}_{k(\gamma)j(\alpha)}\bigg)
		\notag\\&=\partial_{k(\gamma)}\Gamma^{i(\alpha)}_{h(\gamma)j(\alpha)}-\partial_{h(\gamma)}\Gamma^{i(\alpha)}_{k(\gamma)j(\alpha)}\notag\\&+\overset{m_\alpha}{\underset{l=1}{\sum}}\,\bigg(\Gamma^{i(\alpha)}_{k(\gamma)l(\alpha)}\Gamma^{l(\alpha)}_{h(\gamma)j(\alpha)}-\Gamma^{i(\alpha)}_{h(\gamma)l(\alpha)}\Gamma^{l(\alpha)}_{k(\gamma)j(\alpha)}\bigg)\notag\\&+\overset{m_\gamma}{\underset{l=1}{\sum}}\,\bigg(\Gamma^{i(\alpha)}_{k(\gamma)l(\gamma)}\Gamma^{l(\gamma)}_{h(\gamma)j(\alpha)}-\Gamma^{i(\alpha)}_{h(\gamma)l(\gamma)}\Gamma^{l(\gamma)}_{k(\gamma)j(\alpha)}\bigg)
	\end{align}
	vanishes. Without loss of generality, by the simmetries of $R$, we can set $h>k$. In particular, $h\geq2$. We get
	\begin{align}
		R^{i(\alpha)}_{h(\gamma)k(\gamma)j(\alpha)}&=\partial_{k(\gamma)}\xcancel{\Gamma^{i(\alpha)}_{h(\gamma)j(\alpha)}}-\partial_{h(\gamma)}\Gamma^{i(\alpha)}_{k(\gamma)j(\alpha)}\notag\\&+\overset{m_\alpha}{\underset{l=1}{\sum}}\,\bigg(\Gamma^{i(\alpha)}_{k(\gamma)l(\alpha)}\xcancel{\Gamma^{l(\alpha)}_{h(\gamma)j(\alpha)}}-\xcancel{\Gamma^{i(\alpha)}_{h(\gamma)l(\alpha)}}\Gamma^{l(\alpha)}_{k(\gamma)j(\alpha)}\bigg)\notag\\&+\overset{m_\gamma}{\underset{l=1}{\sum}}\,\bigg(\Gamma^{i(\alpha)}_{k(\gamma)l(\gamma)}\Gamma^{l(\gamma)}_{h(\gamma)j(\alpha)}-\xcancel{\Gamma^{i(\alpha)}_{h(\gamma)l(\gamma)}}\Gamma^{l(\gamma)}_{k(\gamma)j(\alpha)}\bigg)
		\notag
	\end{align}
	where
	\begin{align}
		\overset{m_\gamma}{\underset{l=1}{\sum}}\,\Gamma^{i(\alpha)}_{k(\gamma)l(\gamma)}\Gamma^{l(\gamma)}_{h(\gamma)j(\alpha)}&=\overset{m_\gamma}{\underset{l=h}{\sum}}\,\xcancel{\Gamma^{i(\alpha)}_{k(\gamma)l(\gamma)}}\Gamma^{l(\gamma)}_{h(\gamma)j(\alpha)}=0.
		\notag
	\end{align}
	It follows that
	\begin{align}
		R^{i(\alpha)}_{h(\gamma)k(\gamma)j(\alpha)}&=-\partial_{h(\gamma)}\Gamma^{i(\alpha)}_{k(\gamma)j(\alpha)}=0
		\notag
	\end{align}
	as $\Gamma^{i(\alpha)}_{k(\gamma)j(\alpha)}$ does not depend on any of the $u^{h(\gamma)}$s for $h\geq2$.
	\\\textbf{Case 6: $\alpha=\gamma\neq\beta=\epsilon$.} Our goal is to prove that
	\begin{align}
		R^{i(\alpha)}_{h(\beta)k(\alpha)j(\beta)}&=\partial_{k(\alpha)}\Gamma^{i(\alpha)}_{h(\beta)j(\beta)}-\partial_{h(\beta)}\Gamma^{i(\alpha)}_{k(\alpha)j(\beta)}\notag\\&+\overset{r}{\underset{\sigma=1}{\sum}}\,\overset{m_\sigma}{\underset{l=1}{\sum}}\,\bigg(\Gamma^{i(\alpha)}_{k(\alpha)l(\sigma)}\Gamma^{l(\sigma)}_{h(\beta)j(\beta)}-\Gamma^{i(\alpha)}_{h(\beta)l(\sigma)}\Gamma^{l(\sigma)}_{k(\alpha)j(\beta)}\bigg)
		\notag\\&=\partial_{k(\alpha)}\Gamma^{i(\alpha)}_{h(\beta)j(\beta)}-\partial_{h(\beta)}\Gamma^{i(\alpha)}_{k(\alpha)j(\beta)}\notag\\&+\overset{m_\alpha}{\underset{l=1}{\sum}}\,\bigg(\Gamma^{i(\alpha)}_{k(\alpha)l(\alpha)}\Gamma^{l(\alpha)}_{h(\beta)j(\beta)}-\Gamma^{i(\alpha)}_{h(\beta)l(\alpha)}\Gamma^{l(\alpha)}_{k(\alpha)j(\beta)}\bigg)\notag\\&+\overset{m_\beta}{\underset{l=1}{\sum}}\,\bigg(\Gamma^{i(\alpha)}_{k(\alpha)l(\beta)}\Gamma^{l(\beta)}_{h(\beta)j(\beta)}-\Gamma^{i(\alpha)}_{h(\beta)l(\beta)}\Gamma^{l(\beta)}_{k(\alpha)j(\beta)}\bigg)\notag\\&+\overset{}{\underset{\sigma\notin\{\alpha,\beta\}}{\sum}}\,\overset{m_\sigma}{\underset{l=1}{\sum}}\,\bigg(\Gamma^{i(\alpha)}_{k(\alpha)l(\sigma)}\Gamma^{l(\sigma)}_{h(\beta)j(\beta)}-\xcancel{\Gamma^{i(\alpha)}_{h(\beta)l(\sigma)}}\Gamma^{l(\sigma)}_{k(\alpha)j(\beta)}\bigg)
		\label{Rdef_caso6}
	\end{align}
	vanishes. For $j\geq2$ \eqref{Rdef_caso6} vanishes trivially, as
	\begin{align}
		R^{i(\alpha)}_{h(\beta)k(\alpha)j(\beta)}&=\partial_{k(\alpha)}\xcancel{\Gamma^{i(\alpha)}_{h(\beta)j(\beta)}}-\partial_{h(\beta)}\xcancel{\Gamma^{i(\alpha)}_{k(\alpha)j(\beta)}}\notag\\&+\overset{m_\alpha}{\underset{l=1}{\sum}}\,\bigg(\Gamma^{i(\alpha)}_{k(\alpha)l(\alpha)}\xcancel{\Gamma^{l(\alpha)}_{h(\beta)j(\beta)}}-\Gamma^{i(\alpha)}_{h(\beta)l(\alpha)}\xcancel{\Gamma^{l(\alpha)}_{k(\alpha)j(\beta)}}\bigg)\notag\\&+\overset{m_\beta}{\underset{l=1}{\sum}}\,\bigg(\Gamma^{i(\alpha)}_{k(\alpha)l(\beta)}\Gamma^{l(\beta)}_{h(\beta)j(\beta)}-\Gamma^{i(\alpha)}_{h(\beta)l(\beta)}\Gamma^{l(\beta)}_{k(\alpha)j(\beta)}\bigg)\notag\\&+\overset{}{\underset{\sigma\notin\{\alpha,\beta\}}{\sum}}\,\overset{m_\sigma}{\underset{l=1}{\sum}}\,\Gamma^{i(\alpha)}_{k(\alpha)l(\sigma)}\xcancel{\Gamma^{l(\sigma)}_{h(\beta)j(\beta)}}
		\notag\\&=\Gamma^{i(\alpha)}_{k(\alpha)1(\beta)}\Gamma^{1(\beta)}_{h(\beta)j(\beta)}-\Gamma^{i(\alpha)}_{h(\beta)1(\beta)}\xcancel{\Gamma^{1(\beta)}_{k(\alpha)j(\beta)}}
		\notag\\&=\Gamma^{i(\alpha)}_{k(\alpha)1(\beta)}\Gamma^{1(\beta)}_{h(\beta)j(\beta)}
		\notag\\&=
		\begin{cases}
			\Gamma^{i(\alpha)}_{k(\alpha)1(\beta)}\Gamma^{1(\beta)}_{1(\beta)j(\beta)}=-\Gamma^{i(\alpha)}_{k(\alpha)1(\beta)}\underset{\sigma\neq\beta}{\sum}\,\xcancel{\Gamma^{1(\beta)}_{1(\sigma)j(\beta)}}=0\qquad&\text{if }h=1\\
			\Gamma^{i(\alpha)}_{k(\alpha)1(\beta)}\Gamma^{1(\beta)}_{h(\beta)j(\beta)}\overset{1-h-j\leq-3}{=}0\qquad&\text{if }h\geq2.
		\end{cases}
		\notag
	\end{align}
	Let us then fix $j=1$. \eqref{Rdef_caso6} becomes
	\begin{align}
		R^{i(\alpha)}_{h(\beta)k(\alpha)1(\beta)}&=\partial_{k(\alpha)}\Gamma^{i(\alpha)}_{h(\beta)1(\beta)}-\partial_{h(\beta)}\Gamma^{i(\alpha)}_{k(\alpha)1(\beta)}\notag\\&+\overset{m_\alpha}{\underset{l=1}{\sum}}\,\bigg(\Gamma^{i(\alpha)}_{k(\alpha)l(\alpha)}\Gamma^{l(\alpha)}_{h(\beta)1(\beta)}-\Gamma^{i(\alpha)}_{h(\beta)l(\alpha)}\Gamma^{l(\alpha)}_{k(\alpha)1(\beta)}\bigg)\notag\\&+\overset{m_\beta}{\underset{l=1}{\sum}}\,\bigg(\Gamma^{i(\alpha)}_{k(\alpha)l(\beta)}\Gamma^{l(\beta)}_{h(\beta)1(\beta)}-\Gamma^{i(\alpha)}_{h(\beta)l(\beta)}\Gamma^{l(\beta)}_{k(\alpha)1(\beta)}\bigg)\notag\\&+\overset{}{\underset{\sigma\notin\{\alpha,\beta\}}{\sum}}\,\overset{m_\sigma}{\underset{l=1}{\sum}}\,\Gamma^{i(\alpha)}_{k(\alpha)l(\sigma)}\Gamma^{l(\sigma)}_{h(\beta)1(\beta)}.
		\label{Rdef_caso6_bis}
	\end{align}
	We distinguish between the following subcases:
	\begin{itemize}
		\item[$a.$] both $h$ and $k$ are greater or equal than $2$
		\item[$b.$] $h=k=1$
		\item[$c.$] $h\geq2$, $k=1$
		\item[$d.$] $h=1$, $k\geq2$
	\end{itemize}
	\textbf{Subcase a:} both $k$ and $h$ are greater or equal than $2$. We have
	\begin{align}
		R^{i(\alpha)}_{h(\beta)k(\alpha)1(\beta)}&=\partial_{k(\alpha)}\Gamma^{i(\alpha)}_{h(\beta)1(\beta)}-\partial_{h(\beta)}\Gamma^{i(\alpha)}_{k(\alpha)1(\beta)}\notag\\&+\overset{m_\alpha}{\underset{l=1}{\sum}}\,\bigg(\Gamma^{i(\alpha)}_{k(\alpha)l(\alpha)}\xcancel{\Gamma^{l(\alpha)}_{h(\beta)1(\beta)}}-\xcancel{\Gamma^{i(\alpha)}_{h(\beta)l(\alpha)}}\Gamma^{l(\alpha)}_{k(\alpha)1(\beta)}\bigg)\notag\\&+\overset{m_\beta}{\underset{l=1}{\sum}}\,\bigg(\Gamma^{i(\alpha)}_{k(\alpha)l(\beta)}\Gamma^{l(\beta)}_{h(\beta)1(\beta)}-\xcancel{\Gamma^{i(\alpha)}_{h(\beta)l(\beta)}}\Gamma^{l(\beta)}_{k(\alpha)1(\beta)}\bigg)\notag\\&+\overset{}{\underset{\sigma\notin\{\alpha,\beta\}}{\sum}}\,\overset{m_\sigma}{\underset{l=1}{\sum}}\,\Gamma^{i(\alpha)}_{k(\alpha)l(\sigma)}\xcancel{\Gamma^{l(\sigma)}_{h(\beta)1(\beta)}}
		\notag\\&=-\partial_{k(\alpha)}\xcancel{\Gamma^{i(\alpha)}_{h(\beta)1(\alpha)}}-\partial_{h(\beta)}\Gamma^{i(\alpha)}_{k(\alpha)1(\beta)}+\Gamma^{i(\alpha)}_{k(\alpha)1(\beta)}\xcancel{\Gamma^{1(\beta)}_{h(\beta)1(\beta)}}
		\notag
	\end{align}
	where $\Gamma^{i(\alpha)}_{k(\alpha)1(\beta)}$ does not depend on $u^{h(\beta)}$ for $h\geq2$. Thus $R^{i(\alpha)}_{h(\beta)k(\alpha)1(\beta)}=0$.
	\\\textbf{Subcase b:} $h=k=1$. \eqref{Rdef_caso6_bis} reads
	\begin{align}
		R^{i(\alpha)}_{1(\beta)1(\alpha)1(\beta)}&=\partial_{1(\alpha)}\Gamma^{i(\alpha)}_{1(\beta)1(\beta)}-\partial_{1(\beta)}\Gamma^{i(\alpha)}_{1(\alpha)1(\beta)}\notag\\&+\overset{m_\alpha}{\underset{l=1}{\sum}}\,\bigg(\Gamma^{i(\alpha)}_{1(\alpha)l(\alpha)}\Gamma^{l(\alpha)}_{1(\beta)1(\beta)}-\Gamma^{i(\alpha)}_{1(\beta)l(\alpha)}\Gamma^{l(\alpha)}_{1(\alpha)1(\beta)}\bigg)\notag\\&+\overset{m_\beta}{\underset{l=1}{\sum}}\,\bigg(\Gamma^{i(\alpha)}_{1(\alpha)l(\beta)}\Gamma^{l(\beta)}_{1(\beta)1(\beta)}-\Gamma^{i(\alpha)}_{1(\beta)l(\beta)}\Gamma^{l(\beta)}_{1(\alpha)1(\beta)}\bigg)\notag\\&+\overset{}{\underset{\sigma\notin\{\alpha,\beta\}}{\sum}}\,\overset{m_\sigma}{\underset{l=1}{\sum}}\,\Gamma^{i(\alpha)}_{1(\alpha)l(\sigma)}\Gamma^{l(\sigma)}_{1(\beta)1(\beta)}
		\notag
	\end{align}
	where in the second and in the third summation only the terms for $l=1$ survive and
	\begin{align}
		\partial_{1(\alpha)}\Gamma^{i(\alpha)}_{1(\beta)1(\beta)}-\partial_{1(\beta)}\Gamma^{i(\alpha)}_{1(\alpha)1(\beta)}&=-\partial_{1(\alpha)}\Gamma^{i(\alpha)}_{1(\alpha)1(\beta)}+\partial_{1(\alpha)}\Gamma^{i(\alpha)}_{1(\alpha)1(\beta)}=0.
		\notag
	\end{align}
	We get
	\begin{align}
		R^{i(\alpha)}_{1(\beta)1(\alpha)1(\beta)}&=\overset{m_\alpha}{\underset{l=1}{\sum}}\,\bigg(\Gamma^{i(\alpha)}_{1(\alpha)l(\alpha)}\Gamma^{l(\alpha)}_{1(\beta)1(\beta)}-\Gamma^{i(\alpha)}_{1(\beta)l(\alpha)}\Gamma^{l(\alpha)}_{1(\alpha)1(\beta)}\bigg)\notag\\&+\Gamma^{i(\alpha)}_{1(\alpha)1(\beta)}\Gamma^{1(\beta)}_{1(\beta)1(\beta)}-\Gamma^{i(\alpha)}_{1(\beta)1(\beta)}\Gamma^{1(\beta)}_{1(\alpha)1(\beta)}\notag\\&+\overset{}{\underset{\sigma\notin\{\alpha,\beta\}}{\sum}}\,\Gamma^{i(\alpha)}_{1(\alpha)1(\sigma)}\Gamma^{1(\sigma)}_{1(\beta)1(\beta)}
		\notag\\&=\overset{i}{\underset{l=1}{\sum}}\,\bigg(\Gamma^{i(\alpha)}_{1(\alpha)l(\alpha)}\Gamma^{l(\alpha)}_{1(\beta)1(\beta)}-\Gamma^{i(\alpha)}_{1(\beta)l(\alpha)}\Gamma^{l(\alpha)}_{1(\alpha)1(\beta)}\bigg)\notag\\&+\Gamma^{i(\alpha)}_{1(\alpha)1(\beta)}\Gamma^{1(\beta)}_{1(\beta)1(\beta)}-\Gamma^{i(\alpha)}_{1(\beta)1(\beta)}\Gamma^{1(\beta)}_{1(\alpha)1(\beta)}\notag\\&+\overset{}{\underset{\sigma\notin\{\alpha,\beta\}}{\sum}}\,\Gamma^{i(\alpha)}_{1(\alpha)1(\sigma)}\Gamma^{1(\sigma)}_{1(\beta)1(\beta)}
		\notag\\&=\overset{i}{\underset{l=1}{\sum}}\,\bigg(-\cancel{\Gamma^{i(\alpha)}_{1(\beta)l(\alpha)}\Gamma^{l(\alpha)}_{1(\beta)1(\beta)}}-\underset{\sigma\notin\{\alpha,\beta\}}{\sum}\,\Gamma^{i(\alpha)}_{1(\sigma)l(\alpha)}\Gamma^{l(\alpha)}_{1(\beta)1(\beta)}+\cancel{\Gamma^{i(\alpha)}_{1(\beta)l(\alpha)}\Gamma^{l(\alpha)}_{1(\beta)1(\beta)}}\bigg)\notag\\&-\Gamma^{i(\alpha)}_{1(\alpha)1(\beta)}\bigg(\bcancel{\Gamma^{1(\beta)}_{1(\alpha)1(\beta)}}+\underset{\sigma\notin\{\alpha,\beta\}}{\sum}\,\Gamma^{1(\beta)}_{1(\sigma)1(\beta)}\bigg)+\bcancel{\Gamma^{i(\alpha)}_{1(\beta)1(\alpha)}\Gamma^{1(\beta)}_{1(\alpha)1(\beta)}}\notag\\&+\overset{}{\underset{\sigma\notin\{\alpha,\beta\}}{\sum}}\,\Gamma^{i(\alpha)}_{1(\alpha)1(\sigma)}\Gamma^{1(\sigma)}_{1(\beta)1(\beta)}
		\notag\\&=\overset{}{\underset{\sigma\notin\{\alpha,\beta\}}{\sum}}\,\bigg(-\overset{i}{\underset{l=1}{\sum}}\,\Gamma^{(i-l+1)(\alpha)}_{1(\sigma)1(\alpha)}\Gamma^{l(\alpha)}_{1(\beta)1(\beta)}-\Gamma^{i(\alpha)}_{1(\alpha)1(\beta)}\Gamma^{1(\beta)}_{1(\sigma)1(\beta)}+\Gamma^{i(\alpha)}_{1(\alpha)1(\sigma)}\Gamma^{1(\sigma)}_{1(\beta)1(\beta)}\bigg)
		\notag\\&=\overset{}{\underset{\sigma\notin\{\alpha,\beta\}}{\sum}}\,\bigg(\overset{i}{\underset{l=1}{\sum}}\,\Gamma^{(i-l+1)(\alpha)}_{1(\sigma)1(\alpha)}\Gamma^{l(\alpha)}_{1(\beta)1(\alpha)}-\Gamma^{i(\alpha)}_{1(\alpha)1(\beta)}\Gamma^{1(\beta)}_{1(\sigma)1(\beta)}-\Gamma^{i(\alpha)}_{1(\alpha)1(\sigma)}\Gamma^{1(\sigma)}_{1(\beta)1(\sigma)}\bigg)
		\notag
	\end{align}
	which vanishes by \eqref{Blemma}.
	\\\textbf{Subcase c:} $h\geq2$, $k=1$. The argument of \emph{subcase a} applies here as well.
	\\\textbf{Subcase d:} $h=1$, $k\geq2$. \eqref{Rdef_caso6_bis} reads
	\begin{align}
		R^{i(\alpha)}_{1(\beta)k(\alpha)1(\beta)}&=\partial_{k(\alpha)}\Gamma^{i(\alpha)}_{1(\beta)1(\beta)}-\partial_{1(\beta)}\Gamma^{i(\alpha)}_{k(\alpha)1(\beta)}\notag\\&+\overset{m_\alpha}{\underset{l=1}{\sum}}\,\bigg(\Gamma^{i(\alpha)}_{k(\alpha)l(\alpha)}\Gamma^{l(\alpha)}_{1(\beta)1(\beta)}-\Gamma^{i(\alpha)}_{1(\beta)l(\alpha)}\Gamma^{l(\alpha)}_{k(\alpha)1(\beta)}\bigg)\notag\\&+\overset{m_\beta}{\underset{l=1}{\sum}}\,\bigg(\Gamma^{i(\alpha)}_{k(\alpha)l(\beta)}\Gamma^{l(\beta)}_{1(\beta)1(\beta)}-\Gamma^{i(\alpha)}_{1(\beta)l(\beta)}\Gamma^{l(\beta)}_{k(\alpha)1(\beta)}\bigg)\notag\\&+\overset{}{\underset{\sigma\notin\{\alpha,\beta\}}{\sum}}\,\overset{m_\sigma}{\underset{l=1}{\sum}}\,\Gamma^{i(\alpha)}_{k(\alpha)l(\sigma)}\Gamma^{l(\sigma)}_{1(\beta)1(\beta)}\notag
	\end{align}
	where in the second and in the third summation only the terms for $l=1$ survive and
	\begin{align}
		\partial_{k(\alpha)}\Gamma^{i(\alpha)}_{1(\beta)1(\beta)}-\partial_{1(\beta)}\Gamma^{i(\alpha)}_{k(\alpha)1(\beta)}&=-\partial_{k(\alpha)}\Gamma^{i(\alpha)}_{1(\beta)1(\alpha)}-\partial_{1(\beta)}\Gamma^{(i-k+1)(\alpha)}_{1(\alpha)1(\beta)}
		\notag\\&\overset{\eqref{claim_Rcaso2_j1k2}}{=}\partial_{1(\beta)}\Gamma^{(i-k+1)(\alpha)}_{1(\beta)1(\alpha)}-\partial_{1(\beta)}\Gamma^{(i-k+1)(\alpha)}_{1(\alpha)1(\beta)}=0.
		\notag
	\end{align}
	We get
	\begin{align}
		R^{i(\alpha)}_{1(\beta)k(\alpha)1(\beta)}&=\overset{i-k+2}{\underset{l=1}{\sum}}\,\Gamma^{i(\alpha)}_{k(\alpha)l(\alpha)}\Gamma^{l(\alpha)}_{1(\beta)1(\beta)}-\overset{i}{\underset{l=k}{\sum}}\,\Gamma^{i(\alpha)}_{1(\beta)l(\alpha)}\Gamma^{l(\alpha)}_{k(\alpha)1(\beta)}\notag\\&+\Gamma^{i(\alpha)}_{k(\alpha)1(\beta)}\Gamma^{1(\beta)}_{1(\beta)1(\beta)}-\Gamma^{i(\alpha)}_{1(\beta)1(\beta)}\xcancel{\Gamma^{1(\beta)}_{k(\alpha)1(\beta)}}\notag\\&+\overset{}{\underset{\sigma\notin\{\alpha,\beta\}}{\sum}}\,\Gamma^{i(\alpha)}_{k(\alpha)1(\sigma)}\Gamma^{1(\sigma)}_{1(\beta)1(\beta)}
		\notag\\&=-\Gamma^{(i-k+1)(\alpha)}_{1(\alpha)1(\alpha)}\Gamma^{1(\alpha)}_{1(\beta)1(\alpha)}-\overset{i-k+2}{\underset{l=2}{\sum}}\,\Gamma^{(i-k-l+4)(\alpha)}_{2(\alpha)2(\alpha)}\Gamma^{l(\alpha)}_{1(\beta)1(\alpha)}\notag\\&-\overset{i}{\underset{l=k}{\sum}}\,\Gamma^{(i-l+1)(\alpha)}_{1(\beta)1(\alpha)}\Gamma^{(l-k+1)(\alpha)}_{1(\alpha)1(\beta)}\notag\\&-\Gamma^{(i-k+1)(\alpha)}_{1(\alpha)1(\beta)}\Gamma^{1(\beta)}_{1(\alpha)1(\beta)}-\overset{}{\underset{\sigma\notin\{\alpha,\beta\}}{\sum}}\,\Gamma^{(i-k+1)(\alpha)}_{1(\alpha)1(\beta)}\Gamma^{1(\beta)}_{1(\sigma)1(\beta)}\notag\\&-\overset{}{\underset{\sigma\notin\{\alpha,\beta\}}{\sum}}\,\Gamma^{(i-k+1)(\alpha)}_{1(\alpha)1(\sigma)}\Gamma^{1(\sigma)}_{1(\beta)1(\sigma)}
		\notag\\&=\Gamma^{(i-k+1)(\alpha)}_{1(\beta)1(\alpha)}\Gamma^{1(\alpha)}_{1(\beta)1(\alpha)}+\overset{}{\underset{\sigma\notin\{\alpha,\beta\}}{\sum}}\,\Gamma^{(i-k+1)(\alpha)}_{1(\sigma)1(\alpha)}\Gamma^{1(\alpha)}_{1(\beta)1(\alpha)}\notag\\&-\overset{i-k+2}{\underset{l=2}{\sum}}\,\Gamma^{(i-k-l+4)(\alpha)}_{2(\alpha)2(\alpha)}\Gamma^{l(\alpha)}_{1(\beta)1(\alpha)}-\overset{i}{\underset{l=k}{\sum}}\,\Gamma^{(i-l+1)(\alpha)}_{1(\beta)1(\alpha)}\Gamma^{(l-k+1)(\alpha)}_{1(\alpha)1(\beta)}\notag\\&-\Gamma^{(i-k+1)(\alpha)}_{1(\alpha)1(\beta)}\Gamma^{1(\beta)}_{1(\alpha)1(\beta)}-\overset{}{\underset{\sigma\notin\{\alpha,\beta\}}{\sum}}\,\Gamma^{(i-k+1)(\alpha)}_{1(\alpha)1(\beta)}\Gamma^{1(\beta)}_{1(\sigma)1(\beta)}\notag\\&-\overset{}{\underset{\sigma\notin\{\alpha,\beta\}}{\sum}}\,\Gamma^{(i-k+1)(\alpha)}_{1(\alpha)1(\sigma)}\Gamma^{1(\sigma)}_{1(\beta)1(\sigma)}
		\notag\\&=\cancel{\Gamma^{(i-k+1)(\alpha)}_{1(\beta)1(\alpha)}\Gamma^{1(\alpha)}_{1(\beta)1(\alpha)}}-\overset{i-k+2}{\underset{l=2}{\sum}}\,\Gamma^{(i-k-l+4)(\alpha)}_{2(\alpha)2(\alpha)}\Gamma^{l(\alpha)}_{1(\beta)1(\alpha)}\notag\\&-\cancel{\Gamma^{(i-k+1)(\alpha)}_{1(\beta)1(\alpha)}\Gamma^{1(\alpha)}_{1(\alpha)1(\beta)}}-\overset{i}{\underset{l=k+1}{\sum}}\,\Gamma^{(i-l+1)(\alpha)}_{1(\beta)1(\alpha)}\Gamma^{(l-k+1)(\alpha)}_{1(\alpha)1(\beta)}\notag\\&-\Gamma^{(i-k+1)(\alpha)}_{1(\alpha)1(\beta)}\Gamma^{1(\beta)}_{1(\alpha)1(\beta)}
		\notag\\&+\overset{}{\underset{\sigma\notin\{\alpha,\beta\}}{\sum}}\,\bigg(\Gamma^{(i-k+1)(\alpha)}_{1(\sigma)1(\alpha)}\Gamma^{1(\alpha)}_{1(\beta)1(\alpha)}-\Gamma^{(i-k+1)(\alpha)}_{1(\alpha)1(\beta)}\Gamma^{1(\beta)}_{1(\sigma)1(\beta)}-\Gamma^{(i-k+1)(\alpha)}_{1(\alpha)1(\sigma)}\Gamma^{1(\sigma)}_{1(\beta)1(\sigma)}\bigg)
		\notag\\&=-\overset{i-k+2}{\underset{l=2}{\sum}}\,\Gamma^{(i-k-l+4)(\alpha)}_{2(\alpha)2(\alpha)}\Gamma^{l(\alpha)}_{1(\beta)1(\alpha)}-\overset{i-k+1}{\underset{t=2}{\sum}}\,\Gamma^{(i-k-t+2)(\alpha)}_{1(\beta)1(\alpha)}\Gamma^{t(\alpha)}_{1(\alpha)1(\beta)}\notag\\&-\Gamma^{(i-k+1)(\alpha)}_{1(\alpha)1(\beta)}\Gamma^{1(\beta)}_{1(\alpha)1(\beta)}
		\notag\\&+\overset{}{\underset{\sigma\notin\{\alpha,\beta\}}{\sum}}\,\bigg(\Gamma^{(i-k+1)(\alpha)}_{1(\sigma)1(\alpha)}\Gamma^{1(\alpha)}_{1(\beta)1(\alpha)}-\Gamma^{(i-k+1)(\alpha)}_{1(\alpha)1(\beta)}\Gamma^{1(\beta)}_{1(\sigma)1(\beta)}-\Gamma^{(i-k+1)(\alpha)}_{1(\alpha)1(\sigma)}\Gamma^{1(\sigma)}_{1(\beta)1(\sigma)}\bigg)
		\notag\\&=-\Gamma^{2(\alpha)}_{2(\alpha)2(\alpha)}\Gamma^{(i-k+2)(\alpha)}_{1(\beta)1(\alpha)}-\overset{i-k+1}{\underset{l=2}{\sum}}\,\bigg(\Gamma^{(i-k-l+4)(\alpha)}_{2(\alpha)2(\alpha)}+\Gamma^{(i-k-l+2)(\alpha)}_{1(\beta)1(\alpha)}\bigg)\Gamma^{l(\alpha)}_{1(\beta)1(\alpha)}\notag\\&-\Gamma^{(i-k+1)(\alpha)}_{1(\alpha)1(\beta)}\Gamma^{1(\beta)}_{1(\alpha)1(\beta)}
		\notag\\&+\overset{}{\underset{\sigma\notin\{\alpha,\beta\}}{\sum}}\,\bigg(\Gamma^{(i-k+1)(\alpha)}_{1(\sigma)1(\alpha)}\Gamma^{1(\alpha)}_{1(\beta)1(\alpha)}-\Gamma^{(i-k+1)(\alpha)}_{1(\alpha)1(\beta)}\Gamma^{1(\beta)}_{1(\sigma)1(\beta)}-\Gamma^{(i-k+1)(\alpha)}_{1(\alpha)1(\sigma)}\Gamma^{1(\sigma)}_{1(\beta)1(\sigma)}\bigg)
		\notag\\&=-\Gamma^{2(\alpha)}_{2(\alpha)2(\alpha)}\Gamma^{(i-k+2)(\alpha)}_{1(\beta)1(\alpha)}-\overset{i-k+1}{\underset{l=2}{\sum}}\,\bigg(\Gamma^{(i-k-l+4)(\alpha)}_{2(\alpha)2(\alpha)}-\Gamma^{(i-k-l+2)(\alpha)}_{1(\alpha)1(\alpha)}\bigg)\Gamma^{l(\alpha)}_{1(\beta)1(\alpha)}\notag\\&+\underset{\sigma\notin\{\alpha,\beta\}}{\sum}\,\overset{i-k+1}{\underset{l=2}{\sum}}\,\Gamma^{(i-k-l+2)(\alpha)}_{1(\sigma)1(\alpha)}\Gamma^{l(\alpha)}_{1(\beta)1(\alpha)}-\Gamma^{(i-k+1)(\alpha)}_{1(\alpha)1(\beta)}\Gamma^{1(\beta)}_{1(\alpha)1(\beta)}
		\notag\\&+\overset{}{\underset{\sigma\notin\{\alpha,\beta\}}{\sum}}\,\bigg(\Gamma^{(i-k+1)(\alpha)}_{1(\sigma)1(\alpha)}\Gamma^{1(\alpha)}_{1(\beta)1(\alpha)}-\Gamma^{(i-k+1)(\alpha)}_{1(\alpha)1(\beta)}\Gamma^{1(\beta)}_{1(\sigma)1(\beta)}-\Gamma^{(i-k+1)(\alpha)}_{1(\alpha)1(\sigma)}\Gamma^{1(\sigma)}_{1(\beta)1(\sigma)}\bigg)
		\notag
	\end{align}
	where
	\begin{align}
		-\Gamma^{2(\alpha)}_{2(\alpha)2(\alpha)}\Gamma^{(i-k+2)(\alpha)}_{1(\beta)1(\alpha)}-\Gamma^{(i-k+1)(\alpha)}_{1(\alpha)1(\beta)}\Gamma^{1(\beta)}_{1(\alpha)1(\beta)}&=\overset{i-k+1}{\underset{l=2}{\sum}}\,\bigg(\Gamma^{(i-k-l+4)(\alpha)}_{2(\alpha)2(\alpha)}-\Gamma^{(i-k-l+2)(\alpha)}_{1(\alpha)1(\alpha)}\bigg)\,\Gamma^{l(\alpha)}_{1(\beta)1(\alpha)}
		\notag
	\end{align}
	by means of \eqref{fulm}. This yields
	\begin{align}
		R^{i(\alpha)}_{1(\beta)k(\alpha)1(\beta)}&=\underset{\sigma\notin\{\alpha,\beta\}}{\sum}\,\overset{i-k+1}{\underset{l=2}{\sum}}\,\Gamma^{(i-k-l+2)(\alpha)}_{1(\sigma)1(\alpha)}\Gamma^{l(\alpha)}_{1(\beta)1(\alpha)}
		\notag\\&+\overset{}{\underset{\sigma\notin\{\alpha,\beta\}}{\sum}}\,\bigg(\Gamma^{(i-k+1)(\alpha)}_{1(\sigma)1(\alpha)}\Gamma^{1(\alpha)}_{1(\beta)1(\alpha)}-\Gamma^{(i-k+1)(\alpha)}_{1(\alpha)1(\beta)}\Gamma^{1(\beta)}_{1(\sigma)1(\beta)}-\Gamma^{(i-k+1)(\alpha)}_{1(\alpha)1(\sigma)}\Gamma^{1(\sigma)}_{1(\beta)1(\sigma)}\bigg)
		\notag
	\end{align}
	where for each $\sigma\notin\{\alpha,\beta\}$ we have
	\begin{align}
		-\Gamma^{(i-k+1)(\alpha)}_{1(\alpha)1(\beta)}\Gamma^{1(\beta)}_{1(\sigma)1(\beta)}-\Gamma^{(i-k+1)(\alpha)}_{1(\alpha)1(\sigma)}\Gamma^{1(\sigma)}_{1(\beta)1(\sigma)}&=-\overset{i-k+1}{\underset{t=1}{\sum}}\,\Gamma^{(i-k-t+2)(\alpha)}_{1(\epsilon)1(\alpha)}\,\Gamma^{t(\alpha)}_{1(\beta)1(\alpha)}
		\notag
	\end{align}
	by means of \eqref{Blemma}. Thus
	\begin{align}
		R^{i(\alpha)}_{1(\beta)k(\alpha)1(\beta)}&=\underset{\sigma\notin\{\alpha,\beta\}}{\sum}\,\overset{i-k+1}{\underset{l=2}{\sum}}\,\Gamma^{(i-k-l+2)(\alpha)}_{1(\sigma)1(\alpha)}\Gamma^{l(\alpha)}_{1(\beta)1(\alpha)}
		\notag\\&+\overset{}{\underset{\sigma\notin\{\alpha,\beta\}}{\sum}}\,\bigg(\Gamma^{(i-k+1)(\alpha)}_{1(\sigma)1(\alpha)}\Gamma^{1(\alpha)}_{1(\beta)1(\alpha)}-\overset{i-k+1}{\underset{t=1}{\sum}}\,\Gamma^{(i-k-t+2)(\alpha)}_{1(\sigma)1(\alpha)}\,\Gamma^{t(\alpha)}_{1(\beta)1(\alpha)}\bigg)
		\notag\\&=\underset{\sigma\notin\{\alpha,\beta\}}{\sum}\,\bigg(\cancel{\overset{i-k+1}{\underset{l=2}{\sum}}\,\Gamma^{(i-k-l+2)(\alpha)}_{1(\sigma)1(\alpha)}\Gamma^{l(\alpha)}_{1(\beta)1(\alpha)}}+\bcancel{\Gamma^{(i-k+1)(\alpha)}_{1(\sigma)1(\alpha)}\Gamma^{1(\alpha)}_{1(\beta)1(\alpha)}}
		\notag\\&-\bcancel{\Gamma^{(i-k+1)(\alpha)}_{1(\sigma)1(\alpha)}\,\Gamma^{1(\alpha)}_{1(\beta)1(\alpha)}}-\cancel{\overset{i-k+1}{\underset{t=2}{\sum}}\,\Gamma^{(i-k-t+2)(\alpha)}_{1(\sigma)1(\alpha)}\,\Gamma^{t(\alpha)}_{1(\beta)1(\alpha)}}\bigg)=0.
		\notag
	\end{align}
	\textbf{Case 7: $\alpha=\beta\notin\{\gamma,\epsilon\}$, $\gamma\neq\epsilon$.} Our goal is to prove that
	\begin{align}
		R^{i(\alpha)}_{h(\epsilon)k(\gamma)j(\alpha)}&=\xcancel{\partial_{k(\gamma)}\Gamma^{i(\alpha)}_{h(\epsilon)j(\alpha)}}-\xcancel{\partial_{h(\epsilon)}\Gamma^{i(\alpha)}_{k(\gamma)j(\alpha)}}\notag\\&+\overset{r}{\underset{\sigma=1}{\sum}}\,\overset{m_\sigma}{\underset{l=1}{\sum}}\,\bigg(\Gamma^{i(\alpha)}_{k(\gamma)l(\sigma)}\Gamma^{l(\sigma)}_{h(\epsilon)j(\alpha)}-\Gamma^{i(\alpha)}_{h(\epsilon)l(\sigma)}\Gamma^{l(\sigma)}_{k(\gamma)j(\alpha)}\bigg)
		\notag\\&=\overset{m_\alpha}{\underset{l=1}{\sum}}\,\bigg(\Gamma^{i(\alpha)}_{k(\gamma)l(\alpha)}\Gamma^{l(\alpha)}_{h(\epsilon)j(\alpha)}-\Gamma^{i(\alpha)}_{h(\epsilon)l(\alpha)}\Gamma^{l(\alpha)}_{k(\gamma)j(\alpha)}\bigg)
		\label{Rdef_caso7}
	\end{align}
	vanishes. This trivially holds when $h\geq2$ or $k\geq2$. Let us then fix $h=k=1$. We have
	\begin{align}
		R^{i(\alpha)}_{1(\epsilon)1(\gamma)j(\alpha)}&=\overset{m_\alpha}{\underset{l=1}{\sum}}\,\bigg(\Gamma^{i(\alpha)}_{1(\gamma)l(\alpha)}\Gamma^{l(\alpha)}_{1(\epsilon)j(\alpha)}-\Gamma^{i(\alpha)}_{1(\epsilon)l(\alpha)}\Gamma^{l(\alpha)}_{1(\gamma)j(\alpha)}\bigg)
		\notag\\&=\overset{i}{\underset{l=j}{\sum}}\,\Gamma^{(i-l+1)(\alpha)}_{1(\gamma)1(\alpha)}\Gamma^{(l-j+1)(\alpha)}_{1(\epsilon)1(\alpha)}-\overset{i}{\underset{l=j}{\sum}}\,\Gamma^{(i-l+1)(\alpha)}_{1(\epsilon)1(\alpha)}\Gamma^{(l-j+1)(\alpha)}_{1(\gamma)1(\alpha)}
		\notag\\&=\overset{i}{\underset{l=j}{\sum}}\,\Gamma^{(i-l+1)(\alpha)}_{1(\gamma)1(\alpha)}\Gamma^{(l-j+1)(\alpha)}_{1(\epsilon)1(\alpha)}-\overset{i}{\underset{t=j}{\sum}}\,\Gamma^{(t-j+1)(\alpha)}_{1(\epsilon)1(\alpha)}\Gamma^{(i-t+1)(\alpha)}_{1(\gamma)1(\alpha)}=0.
		\notag
	\end{align}
	\textbf{Case 8: $\alpha=\gamma\notin\{\beta,\epsilon\}$, $\beta\neq\epsilon$.} Our goal is to prove that
	\begin{align}
		R^{i(\alpha)}_{h(\epsilon)k(\alpha)j(\beta)}&=\partial_{k(\alpha)}\xcancel{\Gamma^{i(\alpha)}_{h(\epsilon)j(\beta)}}-\xcancel{\partial_{h(\epsilon)}\Gamma^{i(\alpha)}_{k(\alpha)j(\beta)}}\notag\\&+\overset{r}{\underset{\sigma=1}{\sum}}\,\overset{m_\sigma}{\underset{l=1}{\sum}}\,\bigg(\Gamma^{i(\alpha)}_{k(\alpha)l(\sigma)}\Gamma^{l(\sigma)}_{h(\epsilon)j(\beta)}-\Gamma^{i(\alpha)}_{h(\epsilon)l(\sigma)}\Gamma^{l(\sigma)}_{k(\alpha)j(\beta)}\bigg)
		\notag\\&=\overset{m_\alpha}{\underset{l=1}{\sum}}\,\bigg(\Gamma^{i(\alpha)}_{k(\alpha)l(\alpha)}\xcancel{\Gamma^{l(\alpha)}_{h(\epsilon)j(\beta)}}-\Gamma^{i(\alpha)}_{h(\epsilon)l(\alpha)}\Gamma^{l(\alpha)}_{k(\alpha)j(\beta)}\bigg)\notag\\&+\overset{m_\beta}{\underset{l=1}{\sum}}\,\bigg(\Gamma^{i(\alpha)}_{k(\alpha)l(\beta)}\Gamma^{l(\beta)}_{h(\epsilon)j(\beta)}-\xcancel{\Gamma^{i(\alpha)}_{h(\epsilon)l(\beta)}}\Gamma^{l(\beta)}_{k(\alpha)j(\beta)}\bigg)\notag\\&+\overset{m_\epsilon}{\underset{l=1}{\sum}}\,\bigg(\Gamma^{i(\alpha)}_{k(\alpha)l(\epsilon)}\Gamma^{l(\epsilon)}_{h(\epsilon)j(\beta)}-\Gamma^{i(\alpha)}_{h(\epsilon)l(\epsilon)}\xcancel{\Gamma^{l(\epsilon)}_{k(\alpha)j(\beta)}}\bigg)
		\notag\\&=-\overset{i}{\underset{l=k}{\sum}}\,\Gamma^{i(\alpha)}_{h(\epsilon)l(\alpha)}\Gamma^{l(\alpha)}_{k(\alpha)j(\beta)}+\Gamma^{i(\alpha)}_{k(\alpha)1(\beta)}\Gamma^{1(\beta)}_{h(\epsilon)j(\beta)}+\Gamma^{i(\alpha)}_{k(\alpha)1(\epsilon)}\Gamma^{1(\epsilon)}_{h(\epsilon)j(\beta)}
		\label{Rdef_caso8}
	\end{align}
	vanishes. If $j\geq2$ or $h\geq2$ then \eqref{Rdef_caso8} trivially vanishes. Let us then fix $j=1$ and $h=1$. We have
	\begin{align}
		R^{i(\alpha)}_{1(\epsilon)k(\alpha)1(\beta)}&=-\overset{i}{\underset{l=k}{\sum}}\,\Gamma^{i(\alpha)}_{1(\epsilon)l(\alpha)}\Gamma^{l(\alpha)}_{k(\alpha)1(\beta)}+\Gamma^{i(\alpha)}_{k(\alpha)1(\beta)}\Gamma^{1(\beta)}_{1(\epsilon)1(\beta)}+\Gamma^{i(\alpha)}_{k(\alpha)1(\epsilon)}\Gamma^{1(\epsilon)}_{1(\epsilon)1(\beta)}
		\notag\\&=-\overset{i}{\underset{l=k}{\sum}}\,\Gamma^{(i-l+1)(\alpha)}_{1(\epsilon)1(\alpha)}\Gamma^{(l-k+1)(\alpha)}_{1(\alpha)1(\beta)}+\Gamma^{(i-k+1)(\alpha)}_{1(\alpha)1(\beta)}\Gamma^{1(\beta)}_{1(\epsilon)1(\beta)}+\Gamma^{(i-k+1)(\alpha)}_{1(\alpha)1(\epsilon)}\Gamma^{1(\epsilon)}_{1(\epsilon)1(\beta)}
		\notag\\&=-\overset{i-k+1}{\underset{t=1}{\sum}}\,\Gamma^{(i-t-k+2)(\alpha)}_{1(\epsilon)1(\alpha)}\Gamma^{t(\alpha)}_{1(\alpha)1(\beta)}+\Gamma^{(i-k+1)(\alpha)}_{1(\alpha)1(\beta)}\Gamma^{1(\beta)}_{1(\epsilon)1(\beta)}+\Gamma^{(i-k+1)(\alpha)}_{1(\alpha)1(\epsilon)}\Gamma^{1(\epsilon)}_{1(\epsilon)1(\beta)}
		\notag
	\end{align}
	which vanishes by means of \eqref{Blemma}.
	\\\textbf{Case 9: $\beta=\gamma\notin\{\alpha,\epsilon\}$, $\alpha\neq\epsilon$.} Our goal is to prove that
	\begin{align}
		R^{i(\alpha)}_{h(\epsilon)k(\beta)j(\beta)}&=\partial_{k(\beta)}\xcancel{\Gamma^{i(\alpha)}_{h(\epsilon)j(\beta)}}-\xcancel{\partial_{h(\epsilon)}\Gamma^{i(\alpha)}_{k(\beta)j(\beta)}}\notag\\&+\overset{r}{\underset{\sigma=1}{\sum}}\,\overset{m_\sigma}{\underset{l=1}{\sum}}\,\bigg(\Gamma^{i(\alpha)}_{k(\beta)l(\sigma)}\Gamma^{l(\sigma)}_{h(\epsilon)j(\beta)}-\Gamma^{i(\alpha)}_{h(\epsilon)l(\sigma)}\Gamma^{l(\sigma)}_{k(\beta)j(\beta)}\bigg)
		\notag\\&=\overset{m_\alpha}{\underset{l=1}{\sum}}\,\bigg(\Gamma^{i(\alpha)}_{k(\beta)l(\alpha)}\xcancel{\Gamma^{l(\alpha)}_{h(\epsilon)j(\beta)}}-\Gamma^{i(\alpha)}_{h(\epsilon)l(\alpha)}\Gamma^{l(\alpha)}_{k(\beta)j(\beta)}\bigg)\notag\\&+\overset{m_\beta}{\underset{l=1}{\sum}}\,\bigg(\Gamma^{i(\alpha)}_{k(\beta)l(\beta)}\Gamma^{l(\beta)}_{h(\epsilon)j(\beta)}-\xcancel{\Gamma^{i(\alpha)}_{h(\epsilon)l(\beta)}}\Gamma^{l(\beta)}_{k(\beta)j(\beta)}\bigg)\notag\\&+\overset{m_\epsilon}{\underset{l=1}{\sum}}\,\bigg(\xcancel{\Gamma^{i(\alpha)}_{k(\beta)l(\epsilon)}}\Gamma^{l(\epsilon)}_{h(\epsilon)j(\beta)}-\Gamma^{i(\alpha)}_{h(\epsilon)l(\epsilon)}\Gamma^{l(\epsilon)}_{k(\beta)j(\beta)}\bigg)
		\notag\\&=-\overset{i}{\underset{l=1}{\sum}}\,\Gamma^{i(\alpha)}_{h(\epsilon)l(\alpha)}\Gamma^{l(\alpha)}_{k(\beta)j(\beta)}+\Gamma^{i(\alpha)}_{k(\beta)1(\beta)}\Gamma^{1(\beta)}_{h(\epsilon)j(\beta)}-\Gamma^{i(\alpha)}_{h(\epsilon)1(\epsilon)}\Gamma^{1(\epsilon)}_{k(\beta)j(\beta)}
		\label{Rdef_caso9}
	\end{align}
	vanishes. If $j\geq2$ or $h\geq2$ or $k\geq2$ then \eqref{Rdef_caso9} trivially vanishes. Let us then fix $j=h=k=1$. We have
	\begin{align}
		R^{i(\alpha)}_{1(\epsilon)1(\beta)1(\beta)}&=-\overset{i}{\underset{l=1}{\sum}}\,\Gamma^{(i-l+1)(\alpha)}_{1(\epsilon)1(\alpha)}\Gamma^{l(\alpha)}_{1(\beta)1(\beta)}+\Gamma^{i(\alpha)}_{1(\beta)1(\beta)}\Gamma^{1(\beta)}_{1(\epsilon)1(\beta)}-\Gamma^{i(\alpha)}_{1(\epsilon)1(\epsilon)}\Gamma^{1(\epsilon)}_{1(\beta)1(\beta)}
		\notag\\&=\overset{i}{\underset{l=1}{\sum}}\,\Gamma^{(i-l+1)(\alpha)}_{1(\epsilon)1(\alpha)}\Gamma^{l(\alpha)}_{1(\beta)1(\alpha)}-\Gamma^{i(\alpha)}_{1(\beta)1(\alpha)}\Gamma^{1(\beta)}_{1(\epsilon)1(\beta)}-\Gamma^{i(\alpha)}_{1(\epsilon)1(\alpha)}\Gamma^{1(\epsilon)}_{1(\beta)1(\epsilon)}
		\notag
	\end{align}
	which vanishes by means of \eqref{Blemma}.
	\\\textbf{Case 10: $\gamma=\epsilon\notin\{\alpha,\beta\}$, $\alpha\neq\beta$.} Our goal is to prove that
	\begin{align}
		R^{i(\alpha)}_{h(\gamma)k(\gamma)j(\beta)}&=\partial_{k(\gamma)}\xcancel{\Gamma^{i(\alpha)}_{h(\gamma)j(\beta)}}-\partial_{h(\gamma)}\xcancel{\Gamma^{i(\alpha)}_{k(\gamma)j(\beta)}}\notag\\&+\overset{r}{\underset{\sigma=1}{\sum}}\,\overset{m_\sigma}{\underset{l=1}{\sum}}\,\bigg(\Gamma^{i(\alpha)}_{k(\gamma)l(\sigma)}\Gamma^{l(\sigma)}_{h(\gamma)j(\beta)}-\Gamma^{i(\alpha)}_{h(\gamma)l(\sigma)}\Gamma^{l(\sigma)}_{k(\gamma)j(\beta)}\bigg)
		\notag\\&=\overset{m_\gamma}{\underset{l=1}{\sum}}\,\bigg(\Gamma^{i(\alpha)}_{k(\gamma)l(\gamma)}\Gamma^{l(\gamma)}_{h(\gamma)j(\beta)}-\Gamma^{i(\alpha)}_{h(\gamma)l(\gamma)}\Gamma^{l(\gamma)}_{k(\gamma)j(\beta)}\bigg)
		\notag\\&=\Gamma^{i(\alpha)}_{k(\gamma)1(\gamma)}\Gamma^{1(\gamma)}_{h(\gamma)j(\beta)}-\Gamma^{i(\alpha)}_{h(\gamma)1(\gamma)}\Gamma^{1(\gamma)}_{k(\gamma)j(\beta)}
		\label{Rdef_caso10}
	\end{align}
	vanishes. If $j\geq2$ or $h\geq2$ or $k\geq2$ then \eqref{Rdef_caso10} trivially vanishes. Let us then fix $j=h=k=1$. We have
	\begin{align}
		R^{i(\alpha)}_{1(\gamma)1(\gamma)1(\beta)}&=\Gamma^{i(\alpha)}_{1(\gamma)1(\gamma)}\Gamma^{1(\gamma)}_{1(\gamma)1(\beta)}-\Gamma^{i(\alpha)}_{1(\gamma)1(\gamma)}\Gamma^{1(\gamma)}_{1(\gamma)1(\beta)}=0.
		\notag
	\end{align}
	\\\textbf{Case 11:} $\alpha$, $\beta$, $\gamma$ and $\epsilon$ are pairwise distinct. Our goal is to prove that
	\begin{align}
		R^{i(\alpha)}_{h(\epsilon)k(\gamma)j(\beta)}&=\partial_{k(\gamma)}\Gamma^{i(\alpha)}_{h(\epsilon)j(\beta)}-\partial_{h(\epsilon)}\Gamma^{i(\alpha)}_{k(\gamma)j(\beta)}\notag\\&+\overset{r}{\underset{\sigma=1}{\sum}}\,\overset{m_\sigma}{\underset{l=1}{\sum}}\,\bigg(\Gamma^{i(\alpha)}_{k(\gamma)l(\sigma)}\Gamma^{l(\sigma)}_{h(\epsilon)j(\beta)}-\Gamma^{i(\alpha)}_{h(\epsilon)l(\sigma)}\Gamma^{l(\sigma)}_{k(\gamma)j(\beta)}\bigg)
		\label{Rdef_caso11}
	\end{align}
	vanishes. Since the only Christoffel symbols appearing have each of the three indices belonging to a different Jordan block, \eqref{Rdef_caso11} trivially vanishes.
	
	This concludes the proof about the flatness of $\nabla$.
	\end{proof}

\end{document}